%% file: HVMC4Dv2.tex
\documentclass[reqno,12pt,letterpaper]{amsart}
\usepackage{amsmath,amssymb,amsthm,graphicx,mathrsfs,url,slashed}
\usepackage[usenames,dvipsnames,svgnames,table]{xcolor}
\usepackage[colorlinks = true,linkcolor = BrickRed,citecolor = Green, pdfencoding=auto, psdextra]{hyperref}
\hypersetup{linktocpage}
\usepackage{amsxtra}
\usepackage{cancel}
\usepackage{bbm}
\usepackage{tikz}
\usepackage{tikz-cd}
\usepackage{pgfplots}

\usepackage{tensor}
\usepackage[color=yellow]{todonotes}
\usepackage{bookmark}

\usepackage{enumitem}

\definecolor{green}{rgb}{0,0.8,0} % Redefines the color green.
\usetikzlibrary{arrows.meta}% arrow tip library
\usetikzlibrary{bending}% better arrow head for bended lines

\usepackage{xcolor}
\definecolor{amber}{rgb}{1.0, 0.49, 0.0}
\definecolor{cadmiumgreen}{rgb}{0.0, 0.42, 0.24}

\usepackage{mathbbol}
\usepackage{amssymb}

\DeclareSymbolFontAlphabet{\amsmathbb}{AMSb}

\newtheorem{theorem}{Theorem}[section]
\newtheorem{corollary}[theorem]{Corollary}
\newtheorem{lemma}[theorem]{Lemma}
\newtheorem{proposition}[theorem]{Proposition}
\theoremstyle{definition}
\newtheorem{definition}[theorem]{Definition}

\theoremstyle{remark}
\newtheorem{remark}[theorem]{Remark}

\numberwithin{equation}{section}

\newcommand{\brk}[1]{\langle#1\rangle}

\newcommand{\dist}{\mathrm{dist}}

\newcommand{\tr}{\textrm{tr}}
\newcommand{\supp}{{\mathrm{supp}\,}}

\newcommand{\ud}{\mathrm{d}}

\newcommand{\gmm}{\gamma}
\newcommand{\Gmm}{\Gamma}
\newcommand{\dlt}{\delta}

\newcommand{\eps}{\epsilon}
\newcommand{\veps}{\varepsilon}
\newcommand{\kpp}{\kappa}

\newcommand{\omg}{\omega}
\newcommand{\Omg}{\Omega}

\newcommand{\scB}{{\mathscr{B}}}

\newcommand{\scE}{{\mathscr{E}}}

\newcommand{\bbN}{\mathbb N}

\newcommand{\bbP}{\mathbb P}

\newcommand{\bbR}{\mathbb R}
\newcommand{\R}{\mathbb R}
\newcommand{\bbS}{\mathbb S}

\newcommand{\bbZ}{\mathbb Z}

\newcommand{\calB}{\mathcal B}
\newcommand{\calC}{\mathcal C}
\newcommand{\calD}{\mathcal D}
\newcommand{\calE}{\mathcal E}
\newcommand{\calF}{\mathcal F}
\newcommand{\calG}{\mathcal G}
\newcommand{\calH}{\mathcal H}

\newcommand{\calL}{\mathcal L}
\newcommand{\calM}{\mathcal M}

\newcommand{\calO}{\mathcal O}
\newcommand{\calP}{\mathcal P}
\newcommand{\calQ}{\mathcal Q}
\newcommand{\calR}{\mathcal R}
\newcommand{\calS}{\mathcal S}

\newcommand{\calZ}{\mathcal Z}

\newcommand{\II}{{\mathrm{I\!I}}}

\newcommand{\Vol}{\,\text{Vol}}

\setlist[enumerate]{leftmargin=2em, label=(\arabic*)}
\setlist[itemize]{leftmargin=2em}
\newcommand{\p}{\partial}

\newcommand{\trap}{\textrm{trap}}
\usepackage{slashed}
\usepackage{upgreek, bm}

\newcommand{\ringsg}{\mathring{\slashed{g}}}
\newcommand{\sDelta}{{\slashed{\Delta}}}
\newcommand{\snabla}{{\slashed{\nabla}}}

\newcommand{\graph}{{\mathrm{graph}}}
\newcommand{\frakf}{\mathfrak f}
\newcommand{\Err}{\textnormal{Err}}

\newcommand{\barcalC}{{\underline{\calC}}}

\newcommand{\pert}{{\mathrm{pert}}}

\newcommand{\rrho}{\mathfrak{z}}

\newcommand{\ringa}{{\mathring a}}
\newcommand{\ringb}{{\mathring b}}
\newcommand{\ringc}{{\mathring c}}
\newcommand{\uscE}{\underline{\scE}}
\newcommand{\uscB}{\underline{\scB}}
\newcommand{\usigma}{\underline{\sigma}}
\newcommand{\tsigma}{\tilde{\sigma}}

\newcommand{\ebar}{{\underline e}}

\renewcommand{\hbar}{{\underline h}}

\newcommand{\rbar}{{\underline r}}

\newcommand{\Lbar}{{\underline L}}

\newcommand{\Rbar}{{\underline R}}

\newcommand{\Xbar}{{\underline X}}

\newcommand{\Zbar}{{\underline Z}}

\newcommand{\thetabar}{{\underline{\theta}}}

\newcommand{\Omegabar}{\underline{\Omega}}
\newcommand{\bsUpsigma}{{\boldsymbol{\Upsigma}}}
\newcommand{\barbsUpsigma}{{\underline{\bsUpsigma}}}
\newcommand{\frakm}{\mathfrak m}
\newcommand{\frakmax}{\mathfrak{max}}
\newcommand{\temp}{\text{temp}}
\newcommand{\parallelsum}{\mathbin{\!/\mkern-5mu/\!}}
\newcommand{\hyp}{{\mathrm{hyp}}}
\newcommand{\flatt}{{\mathrm{flat}}}
\newcommand{\tran}{\mathrm{tran}}
\newcommand{\flatfar}{\mathrm{flatf}}
\newcommand{\flatnear}{\mathrm{flatn}}
\newcommand{\Vint}{V_{\mathrm{int}}}
\newcommand{\Vext}{V_{\mathrm{ext}}}
\newcommand{\Psiint}{\Psi_{\mathrm{int}}}
\newcommand{\Psiext}{\Psi_{\mathrm{ext}}}
\newcommand{\bfOmg}{\boldsymbol{\Omega}}
\newcommand{\Upomg}{\Upomega}
\newcommand{\vecbfOmg}{{\vec{\bfOmg}}}

\newcommand{\loc}{\mathrm{loc}}
\newcommand{\elliptic}{\mathrm{ell}}
\newcommand{\ringPhi}{\mathring{\Phi}}
\newcommand{\ringeps}{\mathring{\eps}}
\newcommand{\ringpsi}{\mathring{\psi}}
\newcommand{\ringSigma}{\mathring{\Sigma}}
\newcommand{\ringN}{\mathring{N}}
\newcommand{\ringcalF}{\mathring{\calF}}
\newcommand{\ringf}{\mathring{f}}
\newcommand{\ringcalC}{\mathring{\calC}}
\newcommand{\ringbsUpsigma}{\mathring{\bsUpsigma}}
\newcommand{\rings}{\mathring{s}}

\newcommand{\ringvarphi}{\mathring{\varphi}}

\newcommand{\stat}{{\mathrm{stat}}}
\newcommand{\bfUpomg}{{\boldsymbol{\Upomega}}}
\newcommand{\bfGmm}{{\mathbf{\Gmm}}}

\makeatletter
\newcommand{\rightarrowhead}{\mathrel{%
  \hbox{\let\f@size\sf@size\usefont{U}{lasy}{m}{n}\symbol{41}}}}

\newcommand{\leftarrowhead}{\mathrel{%
  \hbox{\let\f@size\sf@size\usefont{U}{lasy}{m}{n}\symbol{40}}}}
\makeatother

\newcommand\action{\mathrel{\ooalign{$\subset$\cr%
  \hidewidth\raise-0.440ex\hbox{$\rightarrowhead\mkern1mu$}}}}

\def\XXint#1#2#3{{\setbox0=\hbox{$#1{#2#3}{\int}$ }
\vcenter{\hbox{$#2#3$ }}\kern-.6\wd0}}

\title[Stability of catenoid in $4$ spatial dimensions]{Stability of the catenoid for the hyperbolic vanishing mean curvature equation in $4$ spatial dimensions}%: Title of the article
\author{Ning Tang}%
\address{Department of Mathematics, University of California,
Berkeley, CA 94720, USA}%
\email{ning\_tang@math.berkeley.edu}%


\usepackage{tabu}
\usepackage{titletoc}

\pgfplotsset{compat=1.18}
\allowdisplaybreaks

\usepackage{caption}
\usepackage{subcaption}
\usepackage{amsfonts}
\usepackage{float}
\usepackage{color}
\usepackage{slashed}
\usepackage[affil-it]{authblk}
\usepackage{mathrsfs}
\usepackage{upgreek}
\usepackage{pifont}
\usepackage{bbm}

\newcommand{\stoptoc}{%
  \addtocontents{toc}{\protect\setcounter{tocdepth}{-10}}%
}

\newcommand{\resumetoc}{%
  \addtocontents{toc}{\protect\setcounter{tocdepth}{2}}%
}

\begin{document}

\begin{abstract} 
We establish the asymptotic stability of the catenoid, as a nonflat stationary solution to the hyperbolic vanishing mean curvature (HVMC) equation in Minkowski space $\R^{1 + (n + 1)}$ for $n = 4$. Our main result is under a ``codimension-$1$'' assumption on initial perturbation, modulo suitable translation and boost (i.e. modulation), without any symmetry assumptions. In comparison to the $n \geq 5$ case addressed by L\"{u}hrmann--Oh--Shahshahani \cite{LOS22}, proving catenoid stability in $4$ dimensions shares additional difficulties with its $3$ dimensional analog, namely the slower spatial decay of the catenoid and slower temporal decay of waves. To overcome these difficulties in the $n = 3$ case, the strong Huygens principle, as well as a miraculous cancellation in the source term, plays an important role in \cite{OS24} to obtain strong late time tails. In $n = 4$ dimensions, without these special structural advantages, our novelty is to introduce an appropriate commutator vector field to derive a new hierarchy of estimates with higher $r^p$-weights so that an improved pointwise decay can be established. We expect this to be applicable for proving improved late time tails of other quasilinear wave equations in even dimensions or wave equations with inverse square potential.
\end{abstract} 

\maketitle
% ----------------------------------------------------------------
\setcounter{tocdepth}{2}
\titlecontents{section}
[0em]
{\vspace{0.4em}}% adding a vertical space before each section entry
{\contentslabel{2em}}
{\bfseries}
{\bfseries\titlerule*[0.5pc]{$\cdot$}\contentspage}
[\vspace{0.2em}]% adding a vertical space after each section entry
\titlecontents{subsection}
              [1.5em]
              {}%{\normalsize}%
              {\contentslabel{2em}}%
              {}
              {\titlerule*[0.5pc]{$\cdot$}\contentspage}%
\titlecontents{subsubsection}
              [2.5em]
              {}%{\normalsize}%
              {\contentslabel{2em}}%
              {}
              {\titlerule*[0.5pc]{$\cdot$}\contentspage}%

\tableofcontents
%\newpage

%\newpage
%\cleardoublepage
%\setcounter{page}{1}

%%Some substantial changes in Version2:
%%(1) We reformulated the rough statement of our main theorem to make it less dependent on future sections.
%%(2) We moved some decay estimates in Section~\ref{Section:mod_profile} to the new Section 5.4.1.

\section{Introduction}
We study the nonlinear asymptotic stability of the catenoid as a solution of the hyperbolic vanishing mean curvature equation, which characterizes extremal embedded hypersurfaces in Minkowski space $(\R^{1+(n+1)}, \eta)$ with $n = 4$. More precisely, we consider a $(1 + n)$ dimensional timelike orientable hypersurface $\calM$ with an embedding $\Phi: (\calM, \Phi^* \eta) \hookrightarrow (\R^{1+(n+1)}, \eta)$ with vanishing mean curvature $H \equiv 0$. Denoting its unit normal by $\hat{n}$, we recall that the mean curvature vector $\vec{H} = H \hat{n}$ can be represented in terms of the embedding as $n \vec{H} = \Box_{\Phi^*\eta} \Phi$ for a hypersurface in Minkowski space (see \cite[Section 2.5]{DHS10}). The following is then referred to as the hyperbolic vanishing mean curvature equation (HVMC equation):  
\begin{equation}\label{eqn:HVMC}
    \Box_{\Phi^*\eta} \Phi = 0.
\end{equation}
By a slight abuse of notation, we will often identify $\calM$ with its image $\Phi(\calM)$
and simply refer to $\calM$ as a hypersurface of $\R^{1+(n+1)}$.
This quasilinear wave equation admits a nontrivial asymptotically flat stationary solution $\calC = \R \times \barcalC$, where $\barcalC$ is the Riemannian catenoid. Here we identify $\calC$ with its standard embedding into Minkowski space and hence it can be interpreted as a solution to \eqref{eqn:HVMC}.
We refer to Section~\ref{Section:basic_properties_of_barcalC} for a detailed discussion of this special solution. 

Our main goal in this paper is to study the nonlinear asymptotic stability of $\calC$ as a solution to the HVMC equation \eqref{eqn:HVMC}, under suitable modulations regarding the translations and boosts in the ambient spacetime, with respect to a ``codimension-$1$'' set of initial data perturbations without any symmetry assumptions when $n = 4$. Our results complete the full picture of the asymptotic stability problem for the catenoid for $n \geq 3$ and our method is based on those developed for the cases when $n \geq 5$ and $n = 3$ in \cite{LOS22} and \cite{OS24}, respectively. In addition to these results, there is a vast literature on \eqref{eqn:HVMC} which we postpone discussing until Section~\ref{Section:prior_works}. Compared to other dimensions settled earlier, the major difficulty arises from the slow spatial decay of the catenoid (inverse polynomial, see \eqref{eqn:Q'_tilde_r}) due to the low dimensions and the absence of strong Huygens principle in even dimensions. We introduce a novel commutator vector field to obtain improved decay, which is applicable to quasilinear wave equations in even dimensions on non-stationary backgrounds and can be generalized to wave equations with inverse square potentials.
%%For Klein-Gordon type, soliton has spatial exponentially decay, so the foliation is not so important.

Before giving our first formulation of the main theorem, we describe the initial data we consider. Let $\Phi_0: \barcalC \to \{0\} \times \R^{4+1}$ be an embedding and $\Phi_1: \barcalC \to \R^{1 + (4+1)}$ be a family of future directed timelike vectors such that $(\Phi_0, \Phi_1)$ is the Cauchy data for a catenoid solution outside a compact set, that is, $\Phi_0 = \mathrm{id}$ and $\Phi_1 = (1, 0)$ outside a compact set. Given such data, the corresponding solution $\Phi: \calM \hookrightarrow \R^{1 + (4+1)}$ to \eqref{eqn:HVMC} is such that $\Phi(\calM) \cap \{X^0 = 0\} = \Phi_0(\barcalC)$ and $T_{\Phi_0(p)}(\Phi(\calM))$ is spanned by $\{(d\Phi_0)_p(T_p \barcalC), \Phi_1(p)\}$ for each $p \in \barcalC$. Now we are ready to state a rough version of our main theorem. 
\begin{theorem}[A rough version]\label{main_thm_rough}
    For any initial data that is a compactly supported perturbation of the catenoid data as mentioned above and belongs to an appropriate ``codimension-$1$'' subset in a suitable topology, there exists a solution $\Phi: \calM \hookrightarrow \R^{1+(4+1)}$ to \eqref{eqn:HVMC}, which

    \begin{enumerate}[leftmargin = *]
        \item can be parametrized by a ``Lorentzian catenoid'' $\calQ$ with a foliation $\calQ = \cup_{\uptau}\Sigma_\uptau$ (see \eqref{eqn:defn_calQ}), where $\Sigma_\uptau$ settles down, inverse polynomially, to a ``boosted and translated (Riemannian) catenoid''
        as $\uptau \to \infty$. Moreover, 
        \item some suitable ambient spacelike distance between $\Sigma_\uptau$ and its parametrization decays inverse polynomially to zero.
    \end{enumerate}
\end{theorem}
We defer the precise statement (Theorem~\ref{main_thm}) to future sections after all necessary notions are introduced.

\subsection{Overall scheme and main difficulties compared to other dimensions}

We describe our main ideas towards the nonlinear stability results, followed by the difficulties, especially those that arise in our special dimension $n = 4$.

\subsubsection{Overall scheme of the proof}
\begin{enumerate}[ wide = 0pt, align = left, label=\arabic*. ]
    \item \textit{Decomposition of the solution. } 
    The starting point of our analysis is to decompose the solution $\calM$ into the profile $\calQ$ plus the perturbation $\psi N$ (see Section~\ref{Section:profile_plus_gauge_choice}), where the addition can be fulfilled in the ambient spacetime. Here, $N$ denotes the choice of gauge, defined as a vector field transverse to $\calQ$, as specified in \eqref{eqn:defn_for_N}. The profile $\calQ$, will be the object we wish to prove the stability of, that is, we want to show a quantitative decay for $\psi$ while $\calQ$ shall manifest a boosted and translated catenoid.
    \item \textit{Foliations. }
    We foliate our profile $\calQ$ into constant-$\uptau$ leaves $\Sigma_\uptau$ and one would later introduce a global coordinate $(\uptau, \uprho, \uptheta)$. The variable $\uptau$ will become the suitable notion of time in Theorem~\ref{main_thm_rough} and each leaf is spacelike and asymptotically null (hyperboloidal), adapted for the purpose of applying the $r^p$-weighted vector field method. See Section~\ref{Section:profile_plus_gauge_choice} and \ref{Section:global_coord} for the detailed discussion. We remark that our foliation is slightly off from the one used in \cite{LOS22} and refer to Appendix~\ref{Section:different_foliation_LOS22}.
    \item \textit{Linearization around $\calQ$. } 
    To study the equation of $\psi$, we linearize \eqref{eqn:HVMC} around $\calQ$. As $\calQ$ is a catenoid-like object, we expect the spectral information of the linearized operator $-\p_t^2 + L$ (with $L$ defined in \eqref{eqn:stab_op_on_C}) around an actual catenoid would be similar to that of the linearization around $\calQ$. In Section~\ref{Section:spectral_analysis_of_stab_op_L}, we reveal that $L$ has $n$ zero modes and a positive eigenvalue. These are related to the non-decaying solutions to $(-\p_t^2 + L) \psi = 0$ and more specifically, a $2n$-dimensional family of solutions growing at most linearly in $t$ and a $1$-dimensional family of exponentially growing solutions, respectively. 
    \item \textit{The source term and the nonlinearity. }
    After linearization, the source term will capture how $\calQ$ differs from an actual catenoid and it is a manifestation of the choice $\dot \wp \not\equiv 0$ (refer to the following two items where this notation is introduced). The spatial decay is unfavorable for producing a twice integrable decay rate for $\dot\wp$ (this requirement will become clear in the following items) when $n = 4$ and one can see this via a heuristic verification for the flat wave equation. This prompts a decomposition $\psi = p + q$ so that the source term of the equation for $p$ captures the worst decay. Naturally, we would then expect better decay for $q$ compared to that of $p$. 
    The nonlinearity admits a null condition though we still need to take special care for 
    the quasilinearity in order to avoid any loss of derivatives that might appear in the following two places: top order energy estimates and modulation equations. 
    \item \textit{Modulation. } Due to the spectral properties, we need to ensure that our perturbation of $\calQ$ is transversal to those growing at most linearly in $t$ that comes from the zero modes of $L$. To this end, we impose $2n$ orthogonality conditions \eqref{eqn:ortho_cond_for_vec_bfOmg} on $q$. To compensate these extra restrictions, we introduce $2n$-dimensional parameters $\xi(\uptau), \ell(\uptau) \in \R^n$ for each $\uptau$, which represent the position and the boost,  respectively, within the plane perpendicular to the axis of symmetry of $\barcalC$. This is due to the presence of translation and boost symmetries in the ambient spacetime, which also appears in many other soliton stability problems (see for instance \cite{Wein85, Stu01}, recent survey articles \cite{KMM16, Ger24} and the references therein).
    These two parameters are encoded in the profile $\calQ$, which is a ``boosted and translated catenoid''. 

    \item \textit{Control of parameter derivatives. }
    One of the key objectives of our main theorem is to show the convergence of $(\xi, \ell)$ as $\uptau$ approaches the final time in the bootstrap arguments so that our solution indeed tends to a boosted and translated catenoid. Here, the convergence refers to the existence of $(\xi_\infty, \ell_\infty)$ such that $(\xi, \ell) \to (\xi_\infty + \uptau \ell_\infty, \ell_\infty)$. To ensure this convergence, twice integrability is required for $\dot\wp = (\dot\ell, \dot\xi - \ell)$. Therefore, we design the modulation equations for $\dot\wp$ so that $q$ enters linearly while $p$ only contributes in the form of $\p_t \vec p - M \vec p$ ($M$ corresponds roughly to the spatial part of the first order formulation) and hence the modulation parameters do not see the slow time decay part of $p$, where $\vec p$ is the corresponding variable in the first order formulation of the wave equation satisfied by $p$. Moreover,  note that the modulation equations for $\dot\wp$ depend on the perturbation $\psi$. We introduce parameter smoothing to avoid any loss of derivatives. 
    
    \item \textit{Shooting argument. }
    We further decompose $\vec q$, the first order formulation of $q$, into $\vec q = a_+ \vec Z_+ + a_- \vec Z_- + \vec\phi$ (see \eqref{eqn:q_decomp}), where $Z_\pm$ are the first order versions of the $1$-dimensional family of exponentially growing mode. By imposing two more orthogonality conditions (projecting $\vec q$ away from the eigenfunctions), we avoid the linearly unstable direction. By restricting to a ``codimension-$1$'' perturbation, we can prove Theorem~\ref{main_thm} by a shooting argument to ensure that $a_+$ satisfies \eqref{eqn:trapping_assumption} for some $C_{\trap}$.
    
    \item \textit{Integrated local energy decay and $r^p$-weighted vector field method. }
    Finally, we want to exploit the decay of $\phi$ under suitable bootstrap assumptions to handle the nonlinearity. A general strategy to conclude pointwise decay on hyperboloidal foliations arising in black hole stability problems (see \cite{DR10} and \cite{Mo15}) is to prove an integrated local energy decay estimate (known as ILED or Morawetz estimate, capturing a weak form of dispersive decay as an intermediate step) and a $r^p$-weighted hierarchy of estimates. The unstable trapping around the neck will cause degeneracy of local energy norms and this can be fixed by a Hardy-type argument. Running the standard $r^p$ hierarchy would give us $\tau^{-\frac{n}{2}+}$ pointwise decay when the source term has sufficient spatial and temporal decay. However, we demand a $\tau^{-2-}$ decay rate so that $\dot\wp$ is twice integrable since $q$ enters linearly into parameter derivatives as aforementioned. This issue is tackled by introducing a commutator vector field $K = r^{\frac32}\p_r$, which, heuristically, allows us to extend the standard range $p \in (0,2)$ to an improved range $p \in (0, \frac52)$. In short, the $p$-range corresponds to the decay rate of energy flux and hence an improved decay is expected to be extracted. We next present a brief schematic overview of the estimates and techniques.
\end{enumerate}

\subsubsection{The new commutator vector field adapted to $1 + 4$-dimensional wave equations}
We give a brief sketch of how our commutator vector field works to get improved decay in the case of the d'Alembertian. The notation used in this part is independent of all other places throughout this paper.

Suppose that $U$ is a solution to the linear wave equation $\Box_m U = f$ with $m$ representing the $1+4$ dimensional Minkowski metric and $f$ is a source term. By introducing the Bondi-type coordinates $(u, r, \theta)$ with $u = t - r$ and a new variable $\tilde U = r^{\frac32}U$, we could rewrite the equation as \[
    -2\p_u\p_r \tilde U + \p_r^2\tilde U  - \frac{3}{4r^2}\tilde U  + r^{-2}\sDelta_{\bbS^3}\tilde U  = \tilde f:= r^{\frac32}f.
\]
Accompanied with a suitable integrated local energy decay estimate, the decay of energy current can be exploited via the Dafermos-Rodnianski $r^p$-weighted vector field method. The core is to notice that the multiplier $r^p\p_r \tilde U$ gives rise to coercive energies when $p \in (0, 2]$: \[\begin{aligned}
    \calE^p[\tilde U](u_1) &= \int_{u = u_1, r \geq R} r^p (\p_r \tilde U)^2\, \ud r\, \ud \theta, \\
    \calB^p[\tilde U](u_1, u_2) &= \frac12\int_{u_1}^{u_2}\int_{r \geq R} p r^{p-1} (\p_r \tilde U)^2 + (2-p)r^{p-3}(|\snabla \tilde U|^2 + \frac34|\tilde U|^2) \, \ud r\, \ud \theta\, \ud u.
\end{aligned}
\]
They satisfy a hierarchy of estimates with acceptable localized (in $\{r \simeq R\}$) errors: \[
    \calE^p[\tilde U](u_2) + \calB^p[\tilde U](u_1, u_2) \lesssim \calE^p[\tilde U](u_1) + \int_{r \geq R} r^{p+1}|\tilde f|^2\, \ud r\, \ud \theta\, \ud u + \,\text{acceptable errors}.
\]
Heuristically, the range of $p$ will correspond to the decay rate of energy, i.e. $\calE^0[\tilde U](u) \lesssim u^{-2}$.

Motivated by \cite{AAG18}, a commutator vector field $K = r^{\frac32}\p_r$ is introduced to improve the range of $p$. Unlike the choice $r^2\p_r$ in \cite{AAG18}, we choose a different weight in $r$ so that the corresponding energy of the conjugated variable $Y:= K \tilde U$ is still coercive for some $p$. In contrast, the usage of the commutator $r^2\p_r \tilde U$ in \cite{AAG18} will cause the energy to be non-coercive unless the solution is projected to certain spherical harmonics, which is anticipated to be difficult to be applied to a non-spherically symmetric and nonstationary metric (see \eqref{eqn:calP_graph_via_h_metric}). In addition, the improved decay of zeroth spherical modes in \cite{AAG18} is obtained via an observation specific to $3$ dimensions due to the absence of inverse square potentials, which is not applicable to our setting. 

A commutation computation yields that \[
    -2\p_u\p_r Y + \p_r^2 Y - r^{-1}\p_u Y - r^{-1}\p_r Y + r^{-2}\sDelta_{\bbS^3} Y = (K + 2r^{\frac12})\tilde f:= f_1.
\]
We remark that there is no presence of inverse square potential anymore although we introduced some first order terms.
Now we derive the key identities \[\begin{aligned}
    \int_{\bbS^3} &-(\p_u - \frac12 \p_r)(r^p|\p_r Y|^2) - \frac{p}{2} r^{p-1}(\p_r Y)^2 - \frac12 \p_r(r^{p-2}|\snabla Y|^2) - \frac{2-p}{2}r^{p-3}|\snabla Y|^2 \\
    &- r^{p-1}(\p_u - \frac12 \p_r)Y \p_r Y - \frac32 r^{p-1} |\p_r Y|^2\, \ud \theta = \int_{\bbS^3} r^p \p_r Y f_1\, \ud \theta.
\end{aligned}
\]
Furthermore, \[
    \begin{aligned}
        &\ - \int_{\bbS^3} r^{p-1}(\p_u - \frac12 \p_r)Y \p_r Y\, \ud \theta \\
        &= \int_{\bbS^3} -\p_r \left(r^{p-1}(\p_u - \frac12 \p_r)Y\, Y\right)
        + (p-1)r^{p-2}(\p_u - \frac12 \p_r)Y\, Y \\
        &\qquad \qquad + \frac12 r^{p-1}(-r^{-1}\p_u Y - r^{-1}\p_r Y + r^{-2} \sDelta_{\bbS}Y - f_1)Y\, \ud \theta \\
        &=  \int_{\bbS^3} -\frac12\p_r \left(r^{p-1}(\p_u - \frac12 \p_r)Y^2\right) 
        + \frac{p-\frac32}{2}r^{p-2}(\p_u - \frac12 \p_r)Y^2 \\
        &\qquad\qquad - \frac12 r^{p-3} |\snabla Y|^2 - \frac12 r^{p-1} f_1 Y - \frac{3}{8} r^{p-2}\p_r Y^2 \, \ud \theta\\
        &= \int_{\bbS^3} -\frac12\p_r \left((\p_u - \frac12 \p_r)\big(r^{p-1}Y^2\big) + \frac{p-1}{2}r^{p-2} Y^2 + \frac34 r^{p-2} Y^2\right) 
        + \frac{p-\frac32}{2}(\p_u - \frac12 \p_r) (r^{p-2}Y^2) \\
        &\qquad\qquad - \frac12 r^{p-3} |\snabla Y|^2 - \frac12 r^{p-1} f_1 Y - \frac{3(2-p)}{8} r^{p-3} Y^2 + \frac{(p-2)(p-\frac32)}{4} r^{p-3}Y^2 \, \ud \theta \\
    \end{aligned}
\]
Combining the two identities, we obtain 
\[
\begin{aligned}
    \int_{\bbS^3} &-(\p_u - \frac12 \p_r) \left(r^p |\p_r Y|^2 + \frac{\frac32 - p}{2}r^{p-2} Y^2 \right) - \frac12 \p_r \left(r^{p-2} |\snabla Y|^2 + \frac{2p+1}4 r^{p-2}Y^2 \right) \\
    &- \frac12 \p_r (\p_u - \frac12 \p_r) (r^{p-1}Y^2) - \frac{p + 3}{2} r^{p-1} (\p_r Y)^2 - \frac{3-p}{2}r^{p-3}|\snabla Y|^2 - \frac{p(2-p)}{4} r^{p-3}Y^2\, \ud \theta \\
    & \qquad\qquad\qquad\qquad\qquad\qquad = \int_{\bbS^3} \frac12 r^{p-1} f_1 Y + r^p f_1 \p_r Y\, \ud \theta.
\end{aligned}
\]
Note that the boundary terms (on the leaves and at null infinity) all have good signs. In view of Hardy's inequality, we can afford to have error terms like $r^p(\p_r Y)^2$ in the boundary terms at null infinity as well when adapting our method to a more complicated setting.
Instead of exploiting positivity of boundary terms at null infinity, we argue like in \cite[Proposition 3.4]{AAG18} to justify that terms like \[
    \lim_{r \to \infty} \int_{\bbS^3} r^{p-1} Y^2\, \ud \theta = \lim_{r \to \infty} \int_{\bbS^3} r^{p-2} (r^2\p_r \tilde U)^2\, \ud \theta = 0
\]
for $p < 2$. Without going to the threshold $p = 2$, there is actually no need to exploit positivity as mentioned in \cite[Remark 4.2]{AAG18}. Note that the proof of this property requires suitable energy boundedness and due to our quasilinear error terms, this property does not hold for the top order derivatives. 
%%To push to the threshold p, we need to worry about the positivity. For vanishing properties, it is more like a taste problem whether we need to do this justification or not. Also, to prove that the limit exists, one can actually use constant $X^0$ energy boundedness instead of justifying the energy along constant $v$ can bound the term obtained via integrating along $u$ direction as one doesn't need positivity but only prefer control of all derivatives. It is also true that we can afford to lose derivatives when we argue the limit involving $Y$ is zero as we only need to get improved bound for non-top order derivatives.

With these considerations, integrating in the far-away region, we achieve that \[
    \tilde\calE^p[Y](u_2) + \tilde\calB^p[Y](u_1, u_2) \lesssim \tilde\calE^p[Y](u_1) + \int_{r \geq R} r^{p+1}|(K + 2r^{\frac12})\tilde f|^2\, \ud r\, \ud \theta\, \ud u + \,\text{acceptable errors},
\]
where \[\begin{aligned}
    \tilde\calE^p[Y](u_1) &= \int_{u = u_1, r \geq R} r^p (\p_r Y)^2 + \frac{3 - 2p}{4} r^{p-2} Y^2\, \ud r\, \ud \theta, \\
    \tilde\calB^p[Y](u_1, u_2) &= \int_{r \geq R} \frac{p+3}{2} r^{p-1} (\p_r Y)^2 + \frac{3-p}{2}r^{p-3}|\snabla Y|^2 + \frac{p(2-p)}{4}r^{p-3}|Y|^2 \, \ud r\, \ud \theta\, \ud u.
\end{aligned}
\]
It is then obvious that this is coercive when $0 \leq p \leq \frac{3}{2}$. Due to the $r$-weight $Y$ carries itself, the $p$-energy for $Y$ essentially corresponds to the $p+1$-energy for $\tilde U$. Therefore, the range of $p$ at the level of $\tilde U$ is improved to $[0, \frac{5}{2}]$ though it seems that the range of $p$ is decreased at the level of $Y$. 
Moreover, a careful analysis via Hardy's inequality shall relax the restriction on $p$, that is, \[
    \tilde\calE^p[Y](u_1) \simeq_p \calE^p[Y](u_1), \quad p \notin \{1, 2\}.
\]
This is due to the algebraic relation $\frac{(p-1)^2}{4} > \frac{2p - 3}{4}$ if and only if $(p-2)^2 > 0$. By extending the hierarchy at the level of $\tilde U$ to $p \in [0, 3)$, we could expect to obtain a pointwise decay rate $|U| \lesssim u^{-\frac52 + \kappa}$ when both the initial data and the source term have sufficient spatial decay.

To simplify our discussion, we only carry out part of this in our specific setting in Section~\ref{section_Y_vf_K}, \ref{Section:decay_for_Phi} to obtain a $u^{-\frac94 + \kappa}$ pointwise decay, a sufficient rate to complete our proof of stability. When implementing this new commutator vector field method in the context of a foliation with a moving center, we see some additional technicalities (namely the presence of $\tilde m_1^{\tau\tau}$, see Remark~\ref{rmk_calR_12_not_error_term}).

\subsubsection{Differences compared to \cite{LOS22}, \cite{OS24} and new challenges}
For those readers who are familiar with the work in dimensions $n \geq 5$, we record a list of differences in the setup and mention some extra expositions we provide and especially,  the new challenges in dimensions $n = 4$.

\begin{enumerate}[ wide = 0pt, align = left, label=\arabic*. ]
    \item The foliation $\cup_\tau \Sigma_\tau$ in this note is slightly off from those in \cite[Section 1.5]{LOS22} in the region (denoted by $\calC_\tran$) where the foliation transitions from flat to asymptotically null.
    One can refer to Remark~\ref{rmk:slight_different_foliation} and Appendix~\ref{Section:different_foliation_LOS22}. 
    Though the original construction is more direct, our parametrization provides an explicit way to write out the metric and verify 
    the form of the operators in the transition region. Moreover, we give a quantitative classification of regions in \eqref{eqn:defn_of_4_regions} so that the gauge $N$ can be chosen more carefully to make the graph formulation in Section~\ref{Section:parametrize_far_away} valid in the transition region as well.
    \item In Section~\ref{Section:spectral_analysis_of_stab_op_L}, we provide a simple justification for the characterization of the zero modes of $L$ in weighted-Sobolev norms, where $L$ denotes the stability operator of the Riemannian catenoid $\barcalC$. The idea is to use spherical harmonic decomposition and classical ODE theory. This provides a rigorous justification for this crucial fact, which ensures that all the possible instabilities of HVMC come from the shrinking of the neck of the catenoid, translations and Lorentz boosts, in turn motivating the choice of modulation parameters.
    \item We give a comprehensive derivation of the linearized operator around the profile in the whole region in various forms (see \cite[Remark 4.2]{LOS22}). This part is contained in Section~\ref{Section:2nd_order_formulation_whole_region} and is via an almost intrinsic computation. The computation is an extension of the proof of \cite[Lemma 3.5]{LOS22}.
    \item In Section~\ref{section_rp_methods_hyperboloidal}, we work in the hyperboloidal region and use coordinate derivatives in $(\tau, r, \theta)$-coordinates to develop the weighted $r^p$-estimates instead of the frame vector fields $T, L, \Lbar$ listed in \cite{LOS22}, which is much simpler when commutator computations are required. 
    \item Due to the slow decay of the source term for $n = 4$, we introduce a modified profile $p$ such that $\psi = p + q$, where $p$ only enters the modulation equations (see \eqref{eqn:modulation_eqn}) in the form of $\p_t \vec p - M \vec p$, which has better decay than $p$ itself. This idea is also introduced in \cite{OS24}.
    \item To conclude the pointwise decay of $\psi$, we analyze in weighted Sobolev spaces instead of standard Sobolev spaces. See Section~\ref{Section:decay_for_P} and \ref{Section:decay_for_Phi}.
\end{enumerate}

To compare with the work in $3$ dimensions, we make some comments. 
\begin{enumerate}[ wide = 0pt, align = left, label=\arabic*. ]
    \item The improved decay in \cite{OS24} is achieved via exploiting the structure of fundamental solution, i.e., sharp Huygens principle in $3$ dimensions. The method is adapted from \cite{MTT12} and \cite{LO24}. As opposed to this, our method relies on the commutator vector field $K = r^{\frac32}\p_r$ to get improved late time tails, which is generally applicable to quasilinear wave equations on non-stationary backgrounds in even dimensions and can be adapted to wave equations with inverse square potentials.
    \item A Darboux transform is introduced in \cite{OS24} to resolve the issue that the projection operator onto the zero modes is unbounded on $LE^*$ (see Section~\ref{Section:ILED}), which does not happen for $n = 4$. 
    \item The order of introducing the projection to unstable directions \eqref{eqn:q_decomp} is different. \cite{OS24} first defines $\phi$ by singling out these directions from $\psi$ and then considers the equation $\veps = \psi - p$, where $p$ is the modified profile. In what follows, we define $q = \psi - p$ and then enact the decomposition \eqref{eqn:q_decomp} for $q$. However, we mention that the different orders here do not cause any difference to the proof.
    \item On the technical level of bootstrap assumptions, \cite{OS24} does not pursue better decay for $\|\ddot\wp\|_{L^2_\uptau}$ compared to the trivial bound by integrating $L^\infty$ bound. Instead, both \cite{LOS22} and \cite{OS24} establish better $\uptau$-decay for $\p_\uptau^2 U$ than $\p_\uptau U$ (essentially, $U$ is $p$ or $q$) at the level of energy, which we choose not to pursue due to extra complications  coming from the coupling between $p$, $q$ in the quasilinearity and the slow decay in low dimensions in our case. Moreover, we do not experience the issue mentioned in \cite[Remark 6.1]{OS24}. 
\end{enumerate}

\subsection{Prior works}\label{Section:prior_works}
The HVMC equation has been studied for decades and appeared in physics literature \cite{Hop94} due to its relevance to string theory.
Its local existence was first established in \cite{AC79} in a geometric language. \cite{Ett13} and \cite{AIT24} proved low regularity local existence, which go below the sharp exponent for general quasilinear wave equations established in \cite{ST05} (sharpness in dimensions two and three in \cite{Lind98}). These two results are developed via a global graph formulation of timelike minimal hypersurfaces, i.e. \begin{equation}\label{eqn:HVMC_graph_formulation_global}
    -\frac{\p}{\p t}\left(\frac{u_t}{\sqrt{1-u_t^2+\left|\nabla_x u\right|^2}}\right)+\sum_{i=1}^n \frac{\p}{\p x^i}\left(\frac{u_{x_i}}{\sqrt{1-u_t^2+\left|\nabla_x u\right|^2}}\right) = 0 
\end{equation}
under the constraint $u_t^2 < 1 + |\nabla_x u|^2$. More recently, low regularity local wellposedness was  investigated using the novel balanced gauges in \cite{MosRod25}.

The study of the small data global existence of \eqref{eqn:HVMC_graph_formulation_global} was first studied in \cite{Bre02} \cite{Lind04}, where \cite{Lind04} addressed the case $n = 2$ by exploiting the null form of nonlinearity (see \cite{Chr86, Kl86} and \cite{Alin01-1, Alin01-2}). See \cite{AAI06}, \cite{Stef11}, \cite{Wong17} and \cite{WangZhou26} for further development as well. These results can be viewed as global nonlinear stability result for the hyperplane solution under the HVMC evolution. For the theory of large data, we refer to, for instance, \cite{WangWei22}.

The stability of simple planar traveling wave solutions to the HVMC equation was proved in \cite{AW20, Ab20, LZ24} and revisited in \cite[Section 13]{AZ23} as a special case.

The stability of catenoid, as a nonflat stationary solution to the HVMC equation, has been studied in \cite{KL12} for the case $n = 2$ in the graph formulation. The codimension-$1$ stability was proved in \cite{DKSW16} for $n = 2$ under proper symmetry assumptions via a geometric gauge and it is expected that their proof applies to higher dimensions as well. As previously mentioned, \cite{LOS22}, \cite{OS24} resolved the codimension-$1$ stability of catenoid without any symmetry assumption in dimensions $n \geq 5$ and $n = 3$, respectively. 

The stability of the helicoid solution was settled in \cite{Mar19} under radial perturbation. Some blowup scenarios and singularity formation were justified in \cite{NT13, EHHS15, JNO15, Wong18, BMP21}. We also would like to mention \cite{Wong11} and \cite{Wong14} on relevant topics.

Spacelike hypersurfaces with vanishing mean curvature in spacetimes are intensively studied as well, and they are named as maximal hypersurfaces. Constructions of examples for maximal surfaces date back to \cite{Kob83, Bart84}. The Bernstein problem for minimal surfaces \cite{CY76} in the Lorentzian setting is also of interest.

We refer to \cite{LOS22} for a comparison with black hole stability and other soliton stability problems.

\subsection{Outline of this paper and guide to the reader}

In Section~\ref{Section:preliminaries}, we describe our setup for the problem by introducing modulation parameters, foliations, profiles, parametrizations in different regions, and a global coordinate system.

We would like to note that Section~\ref{section_ell_est_on_Rm_catenoid}, combined with Section~\ref{Section:basic_properties_of_barcalC}, can be read as a self-contained part for those interested solely in the spectral properties and elliptic theory of the stability operator $L$ of the Riemannian catenoid $\barcalC$. On the other hand, the reader who is primarily interested in studying late time tails of wave equations in even dimensions like $n = 4$ can directly go to Section~\ref{section_rp_methods_hyperboloidal}, along with our bootstrap assumptions in Section~\ref{section:bootstrap_assumptions}, which can be treated as a highly self-contained part if one wishes to take the derivation of the equation, energy boundedness, ILED for granted. In this part, we introduce a new commutator $K = r^{\frac32}\p_r$ to extend the standard $p$-hierarchy of $r^p$-weighted estimates to $0 < p < \frac52$ (on the heuristic level), motivated by \cite{AAG18}.

In Section~\ref{Section:Derivation_of_op_modified_profile}, we present a detailed derivation of the equation for the perturbation $\psi$ in various forms. Tailored to the $n = 4$ case, we then design a modified profile $p$ as done in \cite{OS24}. In the last two subsections, the modulation equations for modulation parameters $\xi, \ell$ and parameters $a_\pm$ are derived by imposing suitable orthogonality conditions.

Section~\ref{section:bootstrap_assumptions} is devoted to the bootstrap assumptions and how we close the bootstrap arguments concerning parameter derivatives. The main nonlinear stability theorem and its proof are given in this part as well. 

The rest of the paper is to improve those bootstrap assumptions regarding the decay properties of our profiles. We shall give a further overview in Section~\ref{Section:overview_of_rp} once we are prepared with all relevant notation.

\stoptoc
\subsection*{Acknowledgments}
The author would like to express his gratitude to his advisor Sung\mbox{-}Jin Oh for suggesting this problem, for many insightful discussions, for the detailed comments on the early draft, and for his constant guidance and support. The author would also like to thank Sohrab Shahshahani for his important comments on preliminary versions of this paper. The author also thanks Yuchen Mao for pointing to useful references. The author’s research was partially supported by National Science Foundation CAREER Grant under NSF-DMS-1945615. The author also acknowledges G-Research's travel support, which enabled him to attend the British Isles Graduate Workshop V and present this work. The author is grateful to the anonymous referees for many valuable comments and suggestions that improved the presentation and readability of the paper.
\resumetoc

\section{Setup of the problem}\label{Section:preliminaries}
In this section, we review various notions introduced in \cite{LOS22} including the basic knowledge of a Riemannian catenoid, choice of foliations, profile and a global coordinate. Besides, we give a detailed derivation of the equation in the far-away region in Section~\ref{Section:eqn_derivation_for_varphi_in_Chyp}.

\subsection{Notation}\label{Section:notation}
To better prepare ourselves for the subsequent discussions, we introduce some universal notation that will be used repeatedly in this paper.

\begin{enumerate}[ wide = 0pt, align = left, label=\arabic*. ]
    \item \textit{The main parameters. } Two modulation parameters $\xi, \ell \in \R^n$ are kept track of throughout the paper with smallness assumptions $|\ell|, |\dot\xi| \ll 1$. Here, and below, the dot over a parameter denotes the time derivative. We will use the succinct notation $\dot\wp = (\dot\ell, \dot\xi - \ell)$ to denote the parameter derivatives. Sometimes, we abuse notation to denote $\wp = (\ell, \xi - \uptau\ell)$, where $\sigma$ is some proper time parameter. The notation $\dot\wp^{(k)}$ denotes a total of $k$ derivatives of the parameters, so for instance $\dot\wp^{(2)}$ could be any of $\ddot\ell$, $|\dot\xi-\ell|^2$, $\ddot\xi-\dot\ell$, etc.
    Moreover, there are two parameters $a_\pm$ capturing the projection of the perturbation to the unstable directions in the first order formulation. The difference is that without derivatives, we do expect decay for $a_\pm$ but not for $\wp$.

    \item \textit{Constants. } We use $R_f$ to denote a large spatial scale distinguishing different regions in our foliations. We refer to \eqref{eqn:defn_of_4_regions} for details. Another small parameter $\delta_1$ is used here as well. The scale of the cutoffs for approximate eigenfunctions $\vec Z_j, \vec Y_j, \vec W_j$'s also involves $R_f$. See \eqref{eqn:proxy_eftn} and \eqref{eqn:defn_of_Z_i_s}. The size of the initial perturbation is assumed to be less than $\eps$, which appears in Section~\ref{Section:BA} and Appendix~\ref{Section:LWP} especially. We assume $|\ell| \lesssim \eps$ and quantities like $R_f^j |\ell|$ are considered small up to some certain $j$ since we will finally take $R_f$ to be sufficiently large first and take $\eps$ small correspondingly. Among bootstrap assumptions, $\delta_\wp$ characterizes some extra smallness, which is of the form \eqref{eqn:delta_wp_defn}. Throughout the paper, we will specify the dependence of implicit constants on bootstrap constants by inserting $C_k$ explicitly into these estimates.

    \item \textit{Coordinates. }
    We will introduce and work in different coordinates depending on our demands. In the flat region, $(t,\rho,\omega)$ is an analog of the coordinate on the product Lorentzian manifold $\calC$, while $(\tau,r,\theta)$ in the far-away or hyperboloidal region is introduced such that $\tau$ is an analog of the retarded time.  In addition to these, a global non-geometric coordinates $(\uptau,\uprho,\uptheta)$ is defined in Section~\ref{Section:global_coord}. We mention that in the overlapping region $\uptau = t = \tau$ and we will be slightly sloppy on this sometimes. 
    In general, $\partial$ denotes arbitrary derivatives that have size of order one, i.e. $\{\p_\uptau, \p_\uprho, \brk{\uprho}^{-1}\p_\uptheta\}$.  $\p_\Sigma$ denotes the ones that are tangential to the leaves of the foliations. 
    \item \textit{Big-$\calO$ notation. }
    We will use the notation $\calO$ for some terms that appear in the computation in local coordinates.
    This notation, in contrast to the standard big-$O$ notation, captures the decay properties for all coordinate derivatives, which verifies the heuristics that $\p_r$ will obtain a gain of $r^{-1}$ (spatial decay only matters in the hyperboloidal region) and $\p_\uptau \calO(\wp r^{-k}) = \calO( \dot\wp r^{-k})$. In other words, an error term of the form $\calO(f)$ is still bounded by $|f|$ after any number of differentiations by $r\p_r$ or $\p_\uptheta$. An error that is denoted by $\calO(\dot\wp)$ will still be bounded by $\calO(\dot\wp)$ after application of $\p_\uptau$. More specifically, it becomes $\calO(\dot\wp^{(k+1)})$ after $k$ applications. We also point out that for some specific $\calO(1)$ terms, it will become $\calO(\dot\wp)$ after applying $\p_\uptau$.

    \item \textit{Norms and energies. }
    We will use $E$ and $LE$ to denote standard energy and integrated local energy respectively. These two pieces of notation will be introduced in Section~\ref{Section:ILED} but also appears occasionally in previous chapters due to the consideration of bootstrap assumptions, for instance, in \eqref{eqn:L2_est_of_c_im} and Section~\ref{Section:BA}. In Section~\ref{section_rp_methods_hyperboloidal}, we warn the reader that the following $r^p$-weighted energy notation $\calE^p, \calB^p, \scE^p, \scB^p, \tilde\calE^p, \uscE^p, \uscB^p$ need to be distinguished in this part.
\end{enumerate}

\subsection{Basic properties of a Riemannian catenoid}\label{Section:basic_properties_of_barcalC}
In this subsection, we recall some basic properties for a Riemannian catenoid which are listed in \cite[Section 1]{LOS22} and the references therein.
A Riemannian catenoid $\barcalC$ is a surface of revolution obtained by rotating the graph of \[
X^{n+1} \mapsto X' =(\frakf(X^{n+1}),0,\dots,0)
\]
about the $X^{n+1}$-axis, where $\frakf$ satisfies the ODE \begin{equation}\label{eqn:ODE_for_frakf_rrho}
    \frac{\frakf''}{(1+(\frakf')^2)^{\frac{3}{2}}}-\frac{n-1}{\frakf(1+(\frakf')^2)^{\frac{1}{2}}}=0 \iff \frac{\ud^2 \rrho}{\ud \frakf^2}+\frac{n-1}{\frakf}\frac{\ud\rrho}{\ud \frakf}+\frac{n-1}{\frakf}\big(\frac{\ud \rrho}{\ud \frakf}\big)^3=0
\end{equation}
with initial condition $\frakf(0)=1$, $\frakf'(0)=0$.
Hence, $\barcalC$ can be parameterized as \begin{equation}\label{eqn:defn_F_for_Rm_catenoid}
    F: I \times \bbS^{n-1}\to \R^{n+1}, \quad F(\rrho,\omg)=(\frakf(\rrho)\Theta(\omg),\rrho),
\end{equation}
where $\Theta:\bbS^{n-1}\to \bbR^{n}$ is the standard embedding of the unit sphere. On the other hand, $\barcalC$ can be viewed as a level set $\{G(X) = 0\}$ of the function $G(X) = |X'|^2 - \frakf(X^{n+1})^2 = 0$.
Note that $\frac12 \nabla_X G = (X', - \frakf(X^{n+1}) \frakf'(X^{n+1}))$ and $X' = \frakf(X^{n+1})\Theta(\omg)$, the unit outward normal is given by \begin{equation}\label{eqn:nu_outer_normal_Rm_catenoid}
    \nu = \frac{\nabla G}{|\nabla G|} = \frac{1}{(1 + \frakf'(\rrho)^2)^{\frac12}} (\Theta(\omg), - \frakf'(\rrho)).
\end{equation}
From this, one can derive that $I = \R$ when $n = 2$ and $I = (-S, S)$ when $n \geq 3$, where 
\begin{equation}\label{eqn:S_bound_X_n+1}
    S:=\int_{1}^{\infty}\frac{\ud \frakf}{\sqrt{\frakf^{2(n-1)}-1}} < \infty.
\end{equation}

Now we give the induced metric on $\barcalC$. Using cylindrical coordinates $\Xbar = (X', X^{n+1}) = (\Rbar\, \Omegabar, \Zbar)$,
the ambient Euclidean metric becomes
\[
\ud \Zbar^2 + \ud \Rbar^2 + \Rbar^2\, \ud  \Omegabar^2,
\]
where $\ud \Omegabar^2$ denotes the standard metric on $\bbS^{n-1}$.
On $\barcalC\cap\{\Zbar>0\}$ we introduce radial coordinates $(\rbar,\thetabar)\in[1,\infty)\times \bbS^{n-1}$ by 
\begin{equation}\label{eqn:dZbar_drbar}
    (\rbar,\thetabar)\mapsto (\Zbar= \Zbar(\rbar), \Rbar = \rbar,\Omegabar = \Theta(\thetabar)),\qquad \frac{\ud \Zbar}{\ud \rbar} = \frac{1}{\sqrt{\rbar^{2(n-1)}-1}},
\end{equation} 
where $\Zbar, \rbar$ correspond to $\rrho, \frakf$, respectively, in the preceding context.
The induced Riemannian metric on $\barcalC$ in these coordinates becomes
\begin{equation}\label{eqn:metric_C_in_r}
    \left(1+\frac{1}{\rbar^{2(n-1)}-1}\right)\, \ud  \rbar\otimes \ud \rbar + \rbar^2  \ringsg_{ab}\, \ud \thetabar^a \otimes \ud\thetabar^b,
\end{equation}
where $\ringsg_{ab}$ refers to the round metric.

Instead of the geometric radial coordinate function $\rbar$, it is more convenient to use the global coordinates $\rho \in(-\infty,\infty)$ with $\rbar(\rho) = \brk{\rho} = \sqrt{1+\rho^2}$, $\omg = \theta$ with metric 
\begin{equation}\label{eqn:metric_C_in_rho}
    \frac{\rho^2\brk{\rho}^{2(n-2)}}{\brk{\rho}^{2(n-1)}-1}\, \ud  \rho\otimes d \rho+\brk{\rho}^2\ringsg_{ab}\, \ud \omega^a \otimes \ud\omega^b.
\end{equation}
Note that in the variables $(\rho, \omg)$ we introduced, the expression of $F$ \eqref{eqn:defn_F_for_Rm_catenoid} becomes \begin{equation}\label{eqn:F_for_Rm_catenoid_rho_omg_coord}
    F(\rho, \omg) = (\brk{\rho}\Theta(\omg), \Zbar(\rho)).
\end{equation}

\subsection{Modulation parameters, foliations, profile and gauge choice}

\subsubsection{Modulation parameters $\xi$ and $\ell$}
Let $\xi = \xi(\sigma)$ and $\ell = \ell(\sigma)$ be two functions on some subinterval of $[0,\infty) \subset \R$ with values in $\R^n$ and $|\ell|, |\dot\ell|, |\dot\xi| \ll 1$. (This assumption will become clear in Section~\ref{section:bootstrap_assumptions} after we introduce the bootstrap assumptions.) We impose $\xi(0) = \ell(0) = 0$ as our initial conditions for the modulation parameters. (The vanishing of $\xi(0)$ is just for simplicity but not important.)
By a slight abuse of notation, we will view $\xi, \ell$ as vectors in $\R^{n+1}$ as well with the natural embedding $\R^n \hookrightarrow \R^{n+1}$, $X' \mapsto (X', 0)^T$, where the last component is zero. Let $P_\ell$ denote the Riemannian orthogonal projection in $\R^{n+1}$ to the direction $\ell$ and let $P_\ell^\perp$ denote the orthogonal projection, that is, \[
    P_\ell X = \frac{X \cdot \ell}{|\ell|} \frac{\ell}{|\ell|}, \quad P_\ell^\perp X = X - P_\ell X, \quad X \in \R^{n+1}.
\]
Then the linear boost in $\R^{1+(n+1)}$ in the direction $\ell$ is given by \begin{equation}\label{eqn:defn_Lambda_ell_and_gmm}
    \Lambda_\ell = \begin{pmatrix}
        \gmm & -\gmm\ell^T \\ 
        -\gmm\ell & A_\ell
    \end{pmatrix}, \quad 
    A_\ell = \gmm P_\ell + P_\ell^\perp, \quad \gmm = \frac{1}{\sqrt{1 - |\ell|^2}}.
\end{equation}

Recall that $\barcalC$ denotes the Riemannian catenoid in $\R^{n+1}$ and $\calC = \R \times \barcalC$ the product Lorentzian catenoid in $\R^{1+(n+1)}$. For any $\sigma$, one can define \begin{equation}\label{eqn:calC_sigma}
    \calC_\sigma:= \{\Lambda_{-\ell(\sigma)}X: X \in \calC\} + \begin{pmatrix}
        0 \\ \xi(\sigma) - \sigma\ell(\sigma)
    \end{pmatrix},
\end{equation}
obtained by applying a boost $\Lambda_{-\ell}$ on $\calC$ then translating the object by $\xi - \sigma\ell$ for any fixed $\sigma$. In other words, $\ell$ characterizes a boost and $\xi$ will denote the moving center of our foliation (see \eqref{eqn:hyperboloid_tilde_bsUpsigma_sigma}), which corresponds to the boost and translational symmetry of catenoid as a minimal surface, respectively. These finally will be related to the zero modes of the linearized operator as well and we refer to Remark~\ref{rmk_geometric_meaning_of_zero_modes}. One might also refer to \cite{Stu01} to compare our choice of modulation parameters with the ones therein.
We end by a remark that these parameters haven't been fully determined yet at this point and the governing equations for $\xi$ and $\ell$ will be introduced in Section~\ref{Section:modulation_equations}.

\subsubsection{Foliation $\calD$ of the ambient spacetime with a coordinate chart $(\tau, r, \theta, X^{n+1})$}\label{Section:foliations}

Under the same assumptions on $\xi, \ell$ as in preceding subsections, we are able to give a hyperboloidal foliation $\calD = \cup_{\tau \in [0, \infty)} \bsUpsigma_\tau$ 
of the region \begin{equation}\label{eqn:region_of_ambient_space}
    \calR:= \left\{X = (X^0, X', X^{n+1}) \in \R^{1 + (n+1)}: X^0 + \gmm(0) R_f \geq \sqrt{|X' - \eta(0)|^2 + 1}\right\}
\end{equation}
in the ambient spacetime, where $\eta(\sigma):= \xi(\sigma) - \gmm(\sigma)R_f\ell(\sigma)$.

In order to define such a foliation whose leaves are hyperboloidal away from the moving center $\xi(\sigma)$, we select a family of reference hyperboloids: \[
    \calH_\sigma:= \{y = (y^0, y', y^{n+1}) \in \R^{1+(n+1)}: y^0 - \gmm(\sigma)^{-1}\sigma + R_f = \sqrt{|y'|^2 + 1}\},
\] 
where $\gmm$ is defined in \eqref{eqn:defn_Lambda_ell_and_gmm} and $R_f$ is some large constant to be determined. Finally, $R_f$ would be taken sufficiently large and then the smallness magnitude $\eps$ of $\ell$ will be taken sufficiently small afterwards so that $\calO(|\ell| R_f^j)$ should be small for some $j$; see Section~\ref{Section:notation}.

For any fixed $\sigma$, we perform the same transformation on $\calH_\sigma$ as in \eqref{eqn:calC_sigma} to obtain \begin{equation}\label{eqn:hyperboloid_tilde_bsUpsigma_sigma}
    \tilde\bsUpsigma_\sigma:=  \Lambda_{-\ell(\sigma)}\calH_\sigma + \begin{pmatrix}
        0 \\ \xi(\sigma) - \sigma\ell(\sigma)
    \end{pmatrix} 
    = \left\{X: X^0 - \sigma + \gmm R_f = \sqrt{|X' - \xi + \gmm R_f\ell|^2 + 1}\right\},
\end{equation}
where the computation is easy thanks to the fact that $\calH_\sigma - (\gmm^{-1}\sigma - R_f, 0)^T$ is part of the hyperboloid $\{|y^0|^2 - |y'|^2 = 1\}$, which shall be preserved by Lorentz boost $\Lambda_{-\ell}$.

\begin{remark}
    We still need to justify $\cup_{\sigma \in [0, \infty)} \tilde\bsUpsigma_\sigma$ foliates $\calR$. Since all the submanifolds defined just now are translational invariant along $X^{n+1}$, it suffices to show that $\cup_\sigma \tilde\barbsUpsigma_\sigma$ forms a foliation of $\{X^{n+1} = c\}$ by setting \begin{equation}\label{eqn:barbsUpsigma_tilde}
        \tilde\barbsUpsigma_\sigma:= \tilde\bsUpsigma_\sigma \cap \{X^{n+1} = c\}.
    \end{equation}

    Suppose by contradiction that some $(X^0, X') \in \calR \cap \{X^{n+1} = c\}$ belongs to $\tilde\barbsUpsigma_{\sigma_1}$ and $\tilde\barbsUpsigma_{\sigma_2}$ with $\sigma_1 < \sigma_2$. Then \begin{equation}\label{eqn:foliation_proof}
        \begin{aligned}
        |\sigma_1 - \sigma_2| &\leq |\gmm(\sigma_1) - \gmm(\sigma_2)| + \left|\sqrt{|X' - \eta(\sigma_1)|^2 + 1} - \sqrt{|X' - \eta(\sigma_2)|^2 + 1}\right| \\ 
        &\lesssim |\gmm(\sigma_1) - \gmm(\sigma_2)| + |\eta(\sigma_1) - \eta(\sigma_2)| < |\sigma_1 - \sigma_2|,
    \end{aligned}
    \end{equation}
    where the last inequality follows from the smallness $|\ell| + |\dot\ell| + |\dot\xi - \ell| = O(\eps)$ for all $\sigma \geq 0$ (see \eqref{BA-dotwp}). 
    On the other hand, the surjectivity simply follows from the continuity of 
    $\sigma - \gmm(\sigma)R_f + \sqrt{|X' - \eta(\sigma)|^2 + 1}$ in $(\sigma, X')$.
\end{remark}

Now we can parametrize $\cup_\sigma \tilde\bsUpsigma_\sigma$ (defined in \eqref{eqn:hyperboloid_tilde_bsUpsigma_sigma}) in terms of new variables $(\tau, r, \theta, X^{n+1})$ as follows.
Set 
\begin{equation}\label{eqn:relation_tau_sigma}
    \tau = \sigma - \gmm(\sigma)R_f,  \quad X' - \eta(\sigma) = r \Theta(\theta).
\end{equation}
From the fact that \[
    \frac{\ud\tau}{\ud\sigma} = 1 - \dot\gmm R_f \simeq 1,
\] 
we could write $\tau:= \tau(\sigma)$ as there is a one-to-one correspondence between $\sigma$ and $\tau$. 
Therefore, one has the following parametrization for $\cup_\sigma \tilde\bsUpsigma_\sigma$: \begin{equation}\label{eqn:parametrize_tilde_bsUpsigma}
    \begin{cases}
        X^0 = \tau + \brk{r}, \\
        X' = \eta(\tau) + r\Theta(\theta), \\
        X^{n+1} = X^{n+1}.
    \end{cases}
\end{equation}
Going forward, we sometimes abuse notation between $\tilde\bsUpsigma_\sigma$ and $\tilde\bsUpsigma_\tau$.

The motivation is to set up a foliation on which the $r^p$-weighted vector field method can be carried out. This only requires our foliation to be spacelike and asymptotically null (more precisely, in our case, hyperboloidal) in the far-away region. This suggests that we could freely modify $\cup_\sigma \tilde\bsUpsigma_\sigma$ to become a foliation within a compact region such that it is easy to work on. The easiest option is definitely constant $X^0$-foliation. In the following, we smoothly patch the two different foliations together.

In order to obtain a $C^\infty$-foliation, we fix $\frakmax \in C^\infty(\R \times \R \to \R)$ to be a smoothed out version of the maximum function such that for some small $\delta_1 > 0$, \begin{equation}\label{eqn:maximum_ftn_w_delta_1}
    \frakmax(t_1, t_2) = \max(t_1, t_2), \quad \text{when } |t_1 - t_2| > \delta_1.
\end{equation}

Now we define $\cup_\tau \bsUpsigma_\tau$ by the following parametrization: 
\begin{equation}\label{eqn:parametrize_bsupsigma_in_tau_r_theta}
    \begin{cases}
        X^0 = \tau + \frakmax\{\gmm R_f, \brk{r}\}, \\
        X' = \eta(\tau) + r\Theta(\theta), \\
        X^{n+1} = X^{n+1},
    \end{cases}
\end{equation}
where the relation between $\tau$ and $\sigma$ is given by \eqref{eqn:relation_tau_sigma}.
Note that for each fixed $\tau$, the transition between the hyperboloidal region and the flat region happens when $|\brk{r} - \gmm R_f| < \delta_1$.
Note that $\cup_\tau \bsUpsigma_\tau$ is indeed a foliation of $\calR \cap \{X^0 \geq 0\}$ by examining like in \eqref{eqn:foliation_proof}: \[
    |\tau_1 - \tau_2| = \left|\left.\frakmax\{\gmm R_f, \sqrt{|X' - \eta|^2 + 1}\}\right|_{\tau_1} - \left.\frakmax\{\gmm R_f, \sqrt{|X' - \eta|^2 + 1}\}\right|_{\tau_2}\right| < |\tau_1 - \tau_2|
\]
if there were $X \in \calR \cap \{X^0 \geq 0\} \cap \bsUpsigma_{\tau_1} \cap \bsUpsigma_{\tau_2}$. 
Similarly, one could define $\cup_\tau \barbsUpsigma_\tau$ like in \eqref{eqn:barbsUpsigma_tilde} as \begin{equation}\label{eqn:barbsUpsigma_tau}
    \barbsUpsigma_\tau:= \bsUpsigma_\tau \cap \{X^{n+1} = S\}.
\end{equation}
Here, $S$ is the value in \eqref{eqn:S_bound_X_n+1} and the reason why to choose this specific constant if for the purpose of parametrizing our profile as a graph over this foliation in the far-away region. The other end shall be treated similarly; see Section~\ref{Section:parametrize_far_away}.

\begin{remark}
    We will see soon that our further analysis of the far-away region highly relies on $(\tau, r, \theta)$ coordinates. Though $(\tau, r, \theta)$ gives an explicit global coordinate for $\cup_\tau \barbsUpsigma_\tau$, we would not use this to parametrize the whole profile $\calQ = \cup_\tau \Sigma_\tau$ of $\cup_\tau \bsUpsigma_\tau$ (see  \eqref{eqn:defn_calQ} for definition). The underlying reason is that in the flat region (near the neck of the catenoid), one cannot simply parametrize $\calQ$ as a function over a hyperplane $\{X^{n+1} = S\}$. Instead, we will define another global charts $(\uptau, \uprho, \uptheta)$ on $Q$ in Section~\ref{Section:global_coord}, where it manifests $(\tau, r, \theta)$ in the far-away region while it is originated from the coordinate of catenoid $\calC$ in the near region.
\end{remark}

\begin{remark}\label{rmk:slight_different_foliation}
    Note that our definition for $\bsUpsigma_\tau$ is slightly different from the $\bsUpsigma_\sigma$ defined in \cite[Section 1.5]{LOS22}. These two are exactly identical (for corresponding $\tau = \sigma - \gmm R_f$) in the flat region and hyperboloidal region except in the transition region, as one can see from the parametrizations introduced in \cite[Section 4.1]{LOS22}. Using the functions $\sigma, \sigma_\temp$ and the smoothed out minimum function $\frakm$ defined in \cite[Section 1.5]{LOS22}, one can show that $\bsUpsigma_\sigma$ in \cite{LOS22} has the implicit parametrization \[
        X^0 = \tau + \frakmax\{\gmm(\sigma)R_f, \brk{r} + (\gmm(\sigma(X)) - \gmm(\sigma_\temp(X)))R_f\},
    \]
    where $\frakmax(t_1, t_2) = -\frakm(-t_1, -t_2)$. These two parametrizations are identical if $\gmm \equiv 1$.
\end{remark}

In view of the parametrization \eqref{eqn:parametrize_bsupsigma_in_tau_r_theta}, we decompose $\calD = \cup \bsUpsigma_\tau$ into four regions \[
    \bsUpsigma_\tau = \bsUpsigma_{\tau, \hyp} \cup \bsUpsigma_{\tau, \tran} \cup \bsUpsigma_{\tau, \flatfar} \cup \bsUpsigma_{\tau, \flatnear}, \quad \bsUpsigma_{\bullet}:= \cup_\tau \bsUpsigma_{\tau, \bullet} \text{ with } \bullet \in \{\hyp, \tran, \flatfar, \flatnear\}
\]
which are defined as follows \begin{equation}\label{eqn:defn_of_4_regions}
    \begin{cases}
    \bsUpsigma_{\tau, \hyp} & \text{ is the region on $\bsUpsigma_\tau$ where } \gmm R_f < \brk{r} - \delta_1, \\ 
    \bsUpsigma_{\tau, \tran} & \text{ is the region on $\bsUpsigma_\tau$ where } |\brk{r} - \gmm R_f| \leq \delta_1, \\ 
    \bsUpsigma_{\tau, \flatfar} & \text{ is the region on $\bsUpsigma_\tau$ where }  \frac12 R_f < \brk{r} < \gmm R_f - \frac12\delta_1, \\ 
    \bsUpsigma_{\tau, \flatnear} & \text{ is the region on $\bsUpsigma_\tau$ where }  \brk{r} < \frac12 R_f. \\ 
\end{cases}
\end{equation}
\begin{figure}[htbp]
\begin{center}
\input{figure_1.tex} 
\caption{An illustration of the four regions.}
\end{center}
\end{figure}
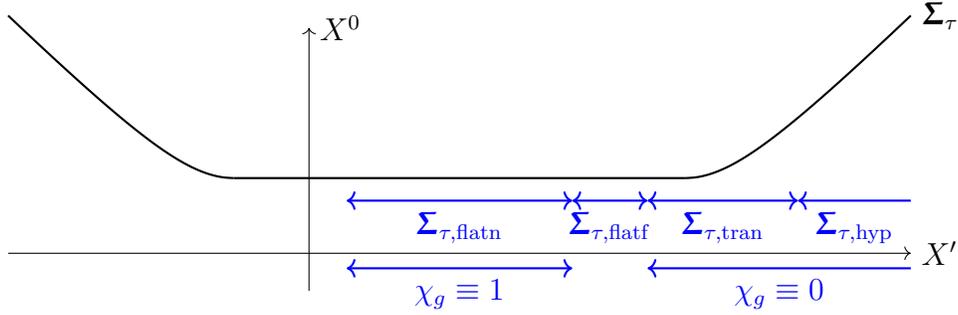
Recall that $\delta_1$ is a small constant given in \eqref{eqn:maximum_ftn_w_delta_1}. This classification is valid due to \eqref{BA-dotwp} ensuring $\gmm \sim 1$ globally. In particular, one has the inclusion $\bsUpsigma_{\tau, \flatfar} \cup \bsUpsigma_{\tau, \flatnear} \subset \bsUpsigma_{\tau, \flatt}$, where $\bsUpsigma_{\tau, \flatt}$ denotes the part on $\bsUpsigma_\tau$ such that $X^0 = \sigma(\tau) = \tau + \gmm(\tau) R_f$. In other words, constant $\tau$ is equivalent to constant $X^0$ on $\bsUpsigma_{\tau, \flatt}$.

Since $R_f$ will be taken to be large eventually, we usually refer the region $\bsUpsigma_{\tau, \tran} \cup \bsUpsigma_{\tau, \hyp}$ as the far-away region. Later, different regions will be treated separately, where we employ the $r^p$-weighted energy estimates on $\calC_\hyp$ (see Section~\ref{Section:profile_plus_gauge_choice} for its definition). Also, the gauge we shall choose will vary from one to the other and we refer to Remark~\ref{rmk_gauge_in_diff_region}.

One needs to note that $\bsUpsigma_{\tau, \flatfar}$ overlaps with $\bsUpsigma_{\tau, \tran}$ in the flat region $\bsUpsigma_{\tau, \flatt}$ and this will be important in Section~\ref{Section:global_coord} when constructing global coordinates. 

\subsubsection{Profile $\calQ$ and gauge choice $N$}\label{Section:profile_plus_gauge_choice}
The basic idea is to decompose a solution $\calM$ into a profile part and perturbation part, that is, $\calM = \calQ + \psi N$, where $\calQ$ is the profile to be defined in the following paragraph, $N$ is an almost normal vector field $N: \calQ \to \R^{1 + (n+1)}$ and $\psi$ is a scalar function $\psi: \calQ \to \R$ on the profile with the property that $p + \psi(p) N(p) \in \calM$ for all $p \in \calQ$. 

Now we define our profile $\calQ$ as the following \begin{equation}\label{eqn:defn_calQ}
    \calQ:= \cup_\sigma \Sigma_\sigma, \quad \Sigma_\sigma:= \calC_\sigma \cap \bsUpsigma_\sigma.
\end{equation}
Note that it can be seen from our parametrization later that $\calQ$ is a hypersurface.
Similar to the classification \eqref{eqn:defn_of_4_regions}, we define $\Sigma_{\tau, \bullet}:= \calC_\tau \cap \bsUpsigma_{\tau, \bullet}$ for $\bullet \in \{\hyp, \tran, \flatfar, \flatnear\}$.
Moreover, we define $\calC_\bullet:= \cup_\tau \Sigma_{\tau, \bullet}$ so that 
\[
    \calQ = \calC_\hyp \cup \calC_\tran \cup \calC_\flatfar \cup \calC_\flatnear.
\]

\begin{remark}
    It is easy to see that for each $\sigma$, $\Sigma_\sigma$ is homeomorphic to the Riemannian catenoid $\barcalC$ since for any fixed $\sigma$, it is homeomorphic to $\calC \cap \calH_\sigma$
    and in turn homeomorphic to $\barcalC$.
\end{remark}

Next, we fix a gauge $N$, which allows us to measure the deviation from the profile $\calQ$ to our solution $\calM$. In order to ease the computation of the equation which is satisfied by $\psi$, we decide to work with a less geometric $N$ defined as follows. 
First, we define a vector field $n_\wp$ on $\Sigma_\sigma$ by considering the normal vector of $\calC_\sigma$, that is, \begin{equation}\label{eqn:n_wp}
    n_\wp(p):= \Lambda_{-\ell(\sigma)}n(q), \quad \Sigma_\sigma \ni p = \Lambda_{-l} q + (0, \xi - \sigma\ell)^T, \quad q \in \calC,
\end{equation}
where $n(q)$ is the normal to $\calC$ at $q$.
In the interior of the flat region, one can define $\tilde N_{\mathrm{int}}$ to be the normal to $\Sigma_\sigma$ as a subspace of $\bsUpsigma_\sigma$, which amounts to ignoring the first component of $n_\wp$ since we are in the flat region.
Then we define an intermediate vector field $\tilde N$ given by interpolation:\begin{equation}\label{eqn:tilde_N_defn}
    \tilde N = \chi_g \tilde N_{\mathrm{int}} + (1-\chi_g) \p_{X^{n+1}},
\end{equation}
where $\chi_g \in C^\infty_0(\calQ)$ is chosen to be \begin{equation}\label{eqn:chi_g_used_in_global_coord}
    \chi_g = \begin{cases}
        1, \quad & \brk{r} < \frac58 R_f, \\
        0, \quad & \brk{r} > \frac78 R_f.
    \end{cases}
\end{equation} 
In particular, $\chi_g \equiv 1$ in $\calC_\flatnear$ and $\chi_g \equiv 0$ on 
$\calC_\tran \cup \calC_\hyp$.
This cutoff function $\chi_g$ will also assist in the definition of the global polar coordinates.

Finally, we renormalize $\tilde N$ to obtain $N$ by imposing \begin{equation}\label{eqn:defn_for_N}
    \eta(n_\wp, N) = 1, \quad \tilde N \parallelsum N,
\end{equation}
where ``$\parallelsum$'' denotes parallel vectors.
As a remark, the vector field $N$ is well-defined as $n_\wp$, $\tilde N_{\mathrm{int}}$, $\p_{X^{n+1}}$ are all spacelike in their domains, respectively.
\begin{remark}\label{rmk_gauge_in_diff_region}
    According to this construction process, $N \parallelsum \tilde N_{\mathrm{int}}$ in $\calC_\flatnear$ while $N \parallelsum \p_{X^{n+1}}$ in $\calC_\tran \cup \calC_\hyp$.
\end{remark}

Then we are ready to introduce the scalar perturbation $\psi$ via the decomposition (as vectors in $\R^{1+(4+1)}$) \begin{equation}\label{eqn:defn_psi}
    \Phi = \Psi_\wp + \psi N,
\end{equation}
where $\Phi$ and $\Psi_\wp$ denote $\calM$ and $\calQ$, respectively. Here, $\wp$ is the way we suppress the dependence on $\xi$ and $\ell$. In particular, $\dot\wp:= (\dot\xi - \ell, \dot\ell)$. To represent $\psi: \calQ \to \R$, we need to parametrize $\calQ$ and this is the main task in the following subsection.

\subsection{Parametrization of the profile \texorpdfstring{$\calQ$}{Q}}

\subsubsection{Parametrization in the flat region \texorpdfstring{$\calC_\flatt = \calC_\flatnear \cup \calC_\flatfar$}{}}
\label{Section:parametrize_interior}

We parametrize the profile $\calQ$ in the flat region $\calC_\flatt$, where $\calC_\flatt$ consists of two parts $\calC_\flatnear$ and $\calC_\flatfar$ with $N \parallelsum \tilde N_{\mathrm{int}}$ in the near region.
From \eqref{eqn:defn_calQ}, $\Sigma_\sigma \cap \calC_\flatt = \calC_\sigma \cap \{X^0 = \sigma\}$. Thanks to \eqref{eqn:calC_sigma}, we start to compute \[
    \left(\Lambda_{-\ell(\sigma)}\calC \cap \{X^0 = \sigma\}\right) + \begin{pmatrix}
        0 \\ \xi - \sigma \ell
    \end{pmatrix}
\]
As $\calC$ is parametrized by $(t, F(\rho, \omg))^T$ with $F$ given by \eqref{eqn:F_for_Rm_catenoid_rho_omg_coord}, we compute \[
    \Lambda_{-\ell(\sigma)}\begin{pmatrix}
        t \\ F(\rho, \omg)
    \end{pmatrix}
    = \begin{pmatrix}
        \gmm & \gmm\ell^T \\ 
        \gmm\ell & A_{-\ell}
    \end{pmatrix} \begin{pmatrix}
        t \\ F
    \end{pmatrix}
    = \begin{pmatrix}
        \gmm t + \gmm\ell^T F \\ 
        \gmm\ell t + A_{-\ell} F.
    \end{pmatrix}
\] 
In particular, $\ell \in \R^n \subset \R^{n+1}$ allows us to ignore the last entry of $F$. Hence, we might abuse notation later, viewing $F$ as a vector in $\R^n$ by taking the first $n$ entries.
We solve $t$ from $\sigma = \gmm t + \gmm\ell^T F$ and plug this into \[
    (\xi - \sigma\ell) + (\gmm\ell t + A_{-\ell} F) = \xi - \gmm (\ell\cdot F)\ell + A_{-\ell} F = \xi - \gmm |\ell|^2 P_\ell F + \gmm P_\ell F + P_\ell^\perp F
    = \xi + \gmm^{-1}P_\ell F + P_\ell^\perp F.
\]
In order to differentiate the parametrizations in different regions, replacing $\sigma$ by $t$, with a slight abuse of notation, $\calQ \cap \calC_\flatt$ can be parametrized by \begin{equation}\label{eqn:parametrize_calQ_in_flatt}
    \Psi_\wp(t, \rho, \omg) = (t, \xi + \gmm^{-1}P_\ell F + P_\ell^\perp F)
\end{equation}
in the flat region, where $\xi = \xi(t), \ell = \ell(t)$ are our time-dependent modulation parameters. We note that the image of $\Psi_\wp$ is represented using the natural coordinate in the ambient spacetime $\R^{1+(n+1)}$.
Then we compute the expression for $N$ in the flat region using the definition \eqref{eqn:defn_for_N}. First, we find \[
    n_\wp:= \Lambda_{-\ell} \begin{pmatrix}
        0 \\ \nu
    \end{pmatrix} = \begin{pmatrix}
        \gmm\ell \cdot \nu \\ A_{-\ell}\nu
    \end{pmatrix},
\]
where $\nu = \nu(\rho, \omg)$ is the geometric normal to the Riemannian catenoid $\barcalC$ given in \eqref{eqn:nu_outer_normal_Rm_catenoid}. 
Moreover, $\tilde N_{\mathrm{int}} = (0, A_{-\ell}\nu)^T$ and it follows from \eqref{eqn:tilde_N_defn} that \[
    \tilde N = \chi_g \begin{pmatrix}
        0 \\ A_{-\ell}\nu \\ 0
    \end{pmatrix}
    + (1 - \chi_g)\begin{pmatrix}
        0 \\ 0 \\ 1
    \end{pmatrix}.
\]
Finally, we could derive the exact formula of $N$ from \eqref{eqn:defn_for_N}. In particular, $N$ is given by \begin{equation}\label{eqn:N_in_C_flatn}
    N = \begin{pmatrix}
        0 \\ |A_{-\ell}\nu|^{-2}A_{-\ell}\nu
    \end{pmatrix} = \nu + \calO(|\ell|^2) \quad \text{ in } \calC_\flatnear.
\end{equation}

\subsubsection{Parametrization of the solution $\calM$ as a graph in the far-away region $\calC_\tran \cup \calC_\hyp$}\label{Section:parametrize_far_away}

Recall that $\barcalC$ has two ends and they asymptote to $\{X^{n+1} = \pm S\}$. Since $N \parallelsum \p_{X^{n+1}}$ in the far-away region $\calC_\tran \cup \calC_\hyp$, there is a natural projection along $N$ from $\calQ \cap (\calC_\tran \cup \calC_\hyp)$ to $\barbsUpsigma_\tran \cup \barbsUpsigma_\hyp \subset \{X^{n+1} = \pm S\}$, respectively, at each end. Therefore, a parametrization of  $\barbsUpsigma_\tran \cup \barbsUpsigma_\hyp$ would become one for $\calQ$ in the far-away region. In light of this, it is possible to parametrize both the solution and the profile as a graph over $\cup_\tau \barbsUpsigma_\tau$ (defined in \eqref{eqn:barbsUpsigma_tau}). The relation between such a graph formulation $u = Q_\wp + \varphi$ is linked with the geometric formulation $\calM = \calQ + \psi N$ directly due to the special form of $N$ in this region.

Recall that we have a nice coordinate $(\tau, r, \theta, X^{n+1})$ of the ambient spacetime in the far-away region. To define $\Sigma_\sigma$ as a graph over $\barbsUpsigma_\sigma$, we need to write $X^{n+1}$ in terms of $(\tau, r, \theta)$.
To obtain a parametrization, we recall the transformation we used in both \eqref{eqn:calC_sigma} and \eqref{eqn:hyperboloid_tilde_bsUpsigma_sigma}. We apply the inverse to find \begin{equation}\label{eqn:y0_y'_in_faraway_region}
\begin{aligned}
    \begin{pmatrix}
        y^0 \\ y'
    \end{pmatrix}
    &= \Lambda_\ell \begin{pmatrix}
        X^0 \\
        X' - \xi + \sigma\ell
    \end{pmatrix}
    = \Lambda_\ell \begin{pmatrix}
        \frakmax(\gmm R_f, \brk{r}) \\
        r \Theta(\theta)
    \end{pmatrix} - \Lambda_\ell \begin{pmatrix}
        -\sigma + \gmm R_f \\ \gmm R_f \ell - \sigma \ell
    \end{pmatrix}\\
    &= \begin{pmatrix}
        \gmm^{-1}(\sigma - \gmm R_f) + \gmm \frakmax(\gmm R_f, \brk{r}) - \gmm r\ell\cdot\Theta  \\
        -\gmm\frakmax(\gmm R_f, \brk{r}) \ell + r A_\ell \Theta 
    \end{pmatrix} 
\end{aligned}
\end{equation}
where we use $\gmm(-\sigma + \gmm R_f) - \gmm(\gmm R_f - \sigma)|\ell|^2 = -\gmm^{-1}(\sigma - \gmm R_f)$ in the last step.
Therefore, our solution can be written as a graph over the plane $\{X^{n+1} = 0\}$ as follows: \[
    \begin{cases}
        X^0 = \tau + \frakmax(\gmm R_f, \brk{r}), \\
        X' = \eta(\tau) + r \Theta(\theta), \\
        X^{n+1} = Q(r A_\ell \Theta -\gmm\frakmax(\gmm R_f, \brk{r}) \ell),
    \end{cases}
\]
where $Q$ satisfies an ODE derived from \eqref{eqn:ODE_for_frakf_rrho} \begin{equation}\label{eqn:Q_tilde_r_equation}
    Q''(\tilde r) + \frac{n-1}{\tilde r} Q'(\tilde r) + \frac{n-1}{\tilde r} (Q')^3 = 0 \iff Q''(\tilde r) + \frac{n-1}{\tilde r} Q'(\tilde r) - \frac{Q'(\tilde r)^2 Q''(\tilde r)}{1 + Q'(\tilde r)^2} = 0
\end{equation}
and it satisfies \begin{equation}\label{eqn:Q'_tilde_r}
     Q'(\tilde r) = \frac{1}{\sqrt{\tilde r^{2(n-1)} - 1}} 
\end{equation}
thanks to \eqref{eqn:dZbar_drbar}.
Going forward, when we try to work in the far-away region, we just purely discuss on one end as the discussion for both ends will be exactly the same.

Thanks to \eqref{eqn:y0_y'_in_faraway_region}, one can easily compute $n_\wp$ (defined in \eqref{eqn:n_wp}) \[
    n_\wp = \Lambda_{-\ell(\sigma)}\left(\frac1{\sqrt{1 + |Q'(r A_\ell \Theta -\gmm\frakmax(\gmm R_f, \brk{r}) \ell)|^2}}\big(0, Q'(r A_\ell \Theta -\gmm\frakmax(\gmm R_f, \brk{r}) \ell), 1\big)^T \right).
\]

Therefore, due to the renormalization in \eqref{eqn:defn_for_N} for the almost normal vector field $N$, it is given by \[
    N = s \p_{X^{n+1}}, \quad s = \left(1 + |Q'(r A_\ell \Theta -\gmm\frakmax(\gmm R_f, \brk{r}) \ell)|^2\right)^{\frac12}
\]
in $\calC_\hyp \cup \calC_\tran$. 
By representing our solution $\calM$ as a graph over $\cup_\tau \barbsUpsigma_\tau$, we obtain that 
the graph formulation $u = Q_\wp + \varphi$ is corresponding to $\calM = \calQ + \psi N$, where \[
    Q_\wp:= Q(r A_\ell \Theta -\gmm\frakmax(\gmm R_f, \brk{r}) \ell)
\]
and \begin{equation}\label{eqn:relation_varphi_psi}
    \varphi:= \brk{\psi N, \p_{X^{n+1}}} = s\psi.
\end{equation}

\begin{remark}
    It will become clear soon that there is no big difference between $\varphi$ and $\psi$ for our purpose since all the $L^2$-based energy norms in Section~\ref{section_rp_methods_hyperboloidal} of $\varphi$ are equivalent to those of $\psi$ provided the decay property of $Q'$ in \eqref{eqn:Q'_tilde_r} and bootstrap assumptions.
\end{remark}
\begin{remark}
    In Section~\ref{Section:2nd_order_formulation_whole_region}, the following alternative formula for $s$ will be used: \begin{equation}\label{eqn:alternative_formula_for_s}
        s = (1 + m^{\mu\nu}Q_\mu Q_\nu)^{\frac12},
    \end{equation}
    which is an easy consequence of
    \eqref{eqn:step_2_in_calculating_F_0}.
    It is important to note that $\mu, \nu$ are taken over $X^0, \cdots, X^n$ coordinates with respect to
    $m = -(\ud\,X^0)^2 + \sum_{j = 1}^n (\ud\,X^j)^2$. We remark that the notation \[
        Q_\mu = Q_{\wp, \mu}:= \p_\mu Q_\wp  \Big|_{\substack{\dot{\ell} = 0 \\\dot{\xi} = \ell}}
    \]
    is defined in the same spirit of \eqref{eqn:defn_h_mu_nu}.
\end{remark}

\subsection{Derivation of an equation for \texorpdfstring{$\varphi$}{varphi} in the far-away region with source term \texorpdfstring{$\calF_0$}{F\_0}}\label{Section:eqn_derivation_for_varphi_in_Chyp}
Now we present a derivation of the equation satisfied by $\varphi$.
By using the mean curvature formula for a manifold represented as a graph, we know that $u$ satisfies \[
    \nabla_\mu\left( \frac{\nabla^\mu u}{\sqrt{1 + \nabla^\alpha u \nabla_\alpha u}}\right) 
    = \frac1{\sqrt{|m|}} \p_\mu \left( \frac{\sqrt{|m|}m^{\mu\nu}\p_\nu u}{\sqrt{1 + m^{\alpha\beta}\p_\alpha u \p_\beta u}} \right) = 0,
\]
where $m = -(\ud X^0)^2 + (\ud X')^2$ denotes the standard Minkowski metric on $\cup_\tau \barbsUpsigma_\tau$ viewed as a subspace of $\{(X^0, X', X^{n+1}): X^{n+1} = 0\}$ and $\nabla$ is the corresponding covariant derivative.
The equation for $u$ can be expanded as \[
    \Box_m u - \frac{\nabla^\nu u \nabla^\mu u \nabla_{\mu\nu}u}{1 + \nabla^\alpha u \nabla_\alpha u} = 0.\]

Plugging in the decomposition $u = \calQ_\wp + \varphi$, we arrive at the following equation for $\varphi$: \[
    \calP_\graph \varphi = \calF_0 + \calF_2 + \calF_3,
\]
where the linear operator $\calP_\graph$ is given by \begin{align}\label{eqn:calP_graph}
\begin{split}
\calP_\graph &= \Box_m-(1+\nabla^\alpha Q \nabla_\alpha Q)^{-1}(\nabla^\mu Q)(\nabla^\nu Q)\nabla_{\mu}\nabla_\nu \\
&\quad-2(1+\nabla^\alpha Q \nabla_\alpha Q)^{-1}(\nabla^{\mu}\nabla^\nu Q)(\nabla_\mu Q)\nabla_\nu\\
&\quad+2(1+\nabla^\alpha Q\nabla_\alpha Q)^{-2}(\nabla^\nu Q)(\nabla^\mu Q)(\nabla^\lambda Q)(\nabla_\lambda \nabla_\mu Q)\nabla_\nu.
\end{split}
\end{align}

One should keep in mind that there is a slight abuse of notation $Q = Q_\wp$ (we might keep using this unless otherwise stated) in \eqref{eqn:calP_graph}, where \[
    Q_\wp:= Q(\tilde r), \quad \tilde r =|r A_\ell \Theta - \gmm \frakmax(\gmm R_f, \brk{r}) \ell|.
\]
When $\brk{r} > \gmm R_f + \delta_1$, one can reach the expansion of $\tilde r$ \begin{equation}\label{eqn:length_tilde_r}
    \tilde r = \gmm(1 - \Theta\cdot\ell) r + \calO(r^{-1}).
\end{equation}
Here $\calF_i$ ($i = 2, 3$) denote inhomogeneous terms of order $i$ (quadratic and cubic, respectively) in $\varphi$ with the following closed form \begin{align}\label{eq:calF2_1}
\begin{split}
\calF_2 &= -\frac{2\nabla^\mu Q\nabla^\nu \varphi\nabla^2_{\mu\nu}\varphi}{1+\nabla^\alpha u \nabla_\alpha u} -\frac{\nabla^2_{\mu\nu}Q\nabla^\mu\varphi\nabla^\nu\varphi}{1+\nabla_\alpha u \nabla^\alpha u} + \frac{\nabla^\mu Q \nabla^\nu Q \nabla^2_{\mu\nu}Q \nabla^\beta \varphi\nabla_\beta \varphi}{(1+\nabla^\alpha u \nabla_\alpha u)(1+\nabla^\alpha Q \nabla_\alpha Q)}\\
&\quad +\frac{2\nabla^\beta Q\nabla_\beta \varphi(\nabla^\mu Q\nabla^\nu Q \nabla^2_{\mu\nu} \varphi+2\nabla^2_{\mu\nu}Q\nabla^\mu Q \nabla^\nu \varphi)}{(1+\nabla^\alpha u \nabla_\alpha u)(1+\nabla^\alpha Q\nabla_\alpha Q)}-\frac{4(\nabla^\mu Q\nabla^\nu Q\nabla^2_{\mu\nu}Q)(\nabla^\beta Q\nabla_\beta \varphi)^2}{(1+\nabla^\alpha u\nabla_\alpha u)(1+\nabla^\alpha Q\nabla_\alpha Q)^2},
\end{split}
\end{align}
and 
\begin{align}\label{eq:calF3_1}
\begin{split}
\calF_3 &= -\frac{\nabla^\mu\varphi\nabla^\nu\varphi\nabla^2_{\mu\nu}\varphi}{1+\nabla^\alpha u \nabla_\alpha u}-\frac{\nabla^\beta\varphi\nabla_\beta\varphi(\nabla^\mu Q\nabla^\nu Q \nabla^2_{\mu\nu} \varphi+2\nabla^2_{\mu\nu}Q\nabla^\mu Q \nabla^\nu \varphi)}{(1+\nabla^\alpha u \nabla_\alpha u)(1+\nabla^\alpha Q\nabla_\alpha Q)}\\
&\quad+\frac{2(\nabla^\mu Q \nabla^\nu Q\nabla^2_{\mu\nu}Q)(\nabla^\beta Q \nabla_\beta\varphi)(\nabla^\gamma\varphi\nabla_\gamma\varphi)}{(1+\nabla^\alpha u \nabla_\alpha u)(1+\nabla^\alpha Q\nabla_\alpha Q)^2},
\end{split}
\end{align}
where the decay rate of $Q$ will be exploited in \eqref{eqn:derivatives_of_Q_decay} after computing the metric $m$ in $(\tau, r, \theta)$ to examine the behavior of the nonlinearities. 
Moreover, the source term $\calF_0$ (independent of the perturbation $\varphi$) is given by \begin{equation}\label{eqn:calF_0}
    \calF_0 = \Box_m Q - (1+\nabla^\alpha Q \nabla_\alpha Q)^{-1}(\nabla^\mu Q)(\nabla^\nu Q)\nabla_{\mu}\nabla_\nu Q,
\end{equation}
where we abuse notation $Q = Q_\wp$ throughout this subsection. Here, the source term is present because of $\dot\wp \not\equiv 0$.

\begin{remark}
    In particular, we shall be interested in writing $m$ in terms of $(y^0, y')$ when analyzing the source term \eqref{eqn:calF_0}. On the other hand, when we apply the $r^p$-weighted estimates in the hyperboloidal region $\calC_\hyp$, we would like to write $\calP_\graph$ in terms of $(\tau, r, \theta)$.
\end{remark}

Now we want to derive the asymptotics for $\calF_0$ via the exact cancellation for $Q$ (not $Q_\wp$) in \eqref{eqn:Q_tilde_r_equation}. 

\begin{lemma}
    The transformation between the two coordinate systems $(X^0, X')$ and $(\tau, r, \theta)$ \eqref{eqn:parametrize_bsupsigma_in_tau_r_theta} satisfies the property that $\frac{\p \tau}{\p X^\mu}$ is bounded for $\mu = 0, 1, \cdots, n$.
\end{lemma}
\begin{proof}
    From \eqref{eqn:parametrize_bsupsigma_in_tau_r_theta}, the Jacobian matrix $\left(\frac{\p(X^0, X')}{\p (\tau, r, \theta)}\right)$ is given by \[
    \left(\frac{\p(X^0, X')}{\p (\tau, r, \theta)}\right) = \begin{pmatrix}
        1 + \p_1\frakmax \cdot \gmm' R_f & \p_2\frakmax \cdot \frac{r}{\brk{r}} & 0 & \cdots & 0 \\
        \eta' & \Theta & r\p_{\theta^1}\Theta & \cdots & r\p_{\theta^{n-1}}\Theta
    \end{pmatrix}.
\]
    Thanks to the observation that \[
    r^{n-1}\det \begin{pmatrix}
        \Theta & \p_{\theta^1}\Theta & \cdots & \p_{\theta^{n-1}}\Theta
    \end{pmatrix} 
    \]
    is the usual Jacobian of the spherical coordinates in $n$-dimensional Cartesian coordinates, we know that $\frac{\p(X^0, X')}{\p (\tau, r, \theta)} = r^{n-1}(1 + \calO(\wp))$ provided the bootstrap assumptions.

    Therefore, we can find the inverse of the Jacobian by applying the adjoint matrix method, which allows us to conclude that \[
        \frac{\p\tau}{\p X^\mu} = \calO(1), \quad \mu = 0, 1, \cdots, n, 
    \]
    where the big-$\calO$ notation is introduced in Section~\ref{Section:notation}.

    Moreover, in the hyperboloidal region $\brk{r} > \gmm R$, it is easy to derive the following exact formula \begin{equation}\label{eqn:exact_formula_for_p_mu_tau}
        \frac{\p \tau}{\p X^0} = \frac{1}{1 - \frac{r}{\brk{r}}\Theta\cdot\eta'}, \quad \frac{\p \tau}{\p X^j} = -\frac{\frac{r}{\brk{r}}\Theta^j(\theta)}{1 - \frac{r}{\brk{r}}\Theta\cdot\eta'}.
    \end{equation}
    This can be also directly read off from \cite[(4.6)]{LOS22}.
\end{proof}

\begin{lemma}\label{lemma_F_0_decay}
    For $\calF_0$ defined in \eqref{eqn:calF_0}, we have that 
    $\calF_0 = \calO(\dot\wp r^{-3})$.
\end{lemma}
\begin{proof}
    Instead of plugging \eqref{eqn:y0_y'_in_faraway_region} into $Q_\wp = Q(y')$ to obtain a form with full dependence on $\tau, r, \theta$, i.e., $Q_\wp = Q(rA_\ell\Theta - \gmm\frakmax(\gmm R_f, \brk{r})\ell)$, we write in terms of an intermediate form via the relation \[
        \begin{pmatrix}
            X^0 \\ X'
        \end{pmatrix} = \Lambda_{-\ell} \begin{pmatrix}
            y^0 \\ y'
        \end{pmatrix} + \begin{pmatrix}
            0 \\ \xi - \sigma\ell
        \end{pmatrix}.
    \]
    For simplicity, we introduce the following notation: \[
        \tilde\xi = \xi - \sigma\ell, \quad T_{\tilde\xi} X = X - \begin{pmatrix}
            0 \\ \tilde\xi
        \end{pmatrix}.
    \]
    Moreover, we use $\p_\mu, \p^\mu$ to denote the corresponding derivatives with respect to $X^\mu$ and we keep track of the dependence of $\xi, \ell$ on $\tau$.
    We write \begin{equation}\label{eqn:y'_in_x_tildexi_ell}
    \begin{aligned}
        y' &= A_l(X' - (\xi - \sigma\ell)) - \gmm\ell X^0
    = \gmm P_\ell (X'-\tilde\xi) + P_\ell^\perp (X'-\tilde\xi) - \gmm\ell X^0\\
    &= (\gmm - 1) \frac{(X' - \tilde\xi)\cdot\ell}{|\ell|^2}\ell + (X' - \tilde\xi) - \gmm\ell X^0,
    \end{aligned}
    \end{equation}
    where $|y'| = \calO(r)$ in the hyperboloidal region thanks to \eqref{eqn:length_tilde_r}.  Besides, the first term in \eqref{eqn:calF_0} is \begin{equation}\label{eqn:Box_m_Q_in_y'}
        \Box_m Q_\wp 
        = \p^\mu \p_\mu Q(y') 
        = \p^\mu \left(Q'(y') \p_\mu |y'|\right)
        = Q''(y') \p^\mu |y'| \p_\mu |y'| + Q'(y') \p^\mu\p_\mu |y'| .
    \end{equation}

    \textit{Step 1: Compute \eqref{eqn:Box_m_Q_in_y'} with fixed $\ell, \tilde\xi$. }
    First, we assume $\ell, \tilde\xi$ is fixed and we verify the fact that $[\Box_m, \Lambda_\ell] = [\Box_m, T_{\tilde\xi}] = 0$ by an explicit computation, which will help us to understand the case when $\ell, \tilde\xi$ vary with $\tau$. By using \eqref{eqn:y'_in_x_tildexi_ell}, we compute \begin{equation}\label{eqn:p_j_|y'|_fixed_l}
        \p_j |y'| = |y'|^{-1} \sum_{k = 1}^n y^k \p_j y^k = |y'|^{-1} \sum_{k = 1}^n y^k \left((\gmm - 1) \frac{\ell^j\ell^k}{|\ell|^2} + \delta_j^k \right)
    \end{equation}
    for any $1 \leq j \leq n$.
    On the other hand, \begin{equation}\label{eqn:p_0_|y'|_fixed_l}
        \p_0 |y'| = -\gmm|y'|^{-1} \sum_{k=1}^n y^k\ell^k = -\gmm |y'|^{-1} y' \cdot \ell. 
    \end{equation}
    Therefore, $\p^\mu |y'| \p_\mu |y'|$ is equal to \begin{equation}\label{eqn:p_mu_|y'|_prod_w_l_fixed}
        \begin{aligned}
        &|y'|^{-2} \sum_{j,k,m = 1}^n y^k y^m \left((\gmm - 1)^2 \frac{(\ell^j)^2\ell^k\ell^m}{|\ell|^4} + \delta_j^k \delta_j^m y^k y^m + 2\delta^k_j (\gmm - 1) \frac{\ell^j\ell^m}{|\ell|^2} \right) - \gmm^2|y'|^{-2} (y' \cdot \ell)^2\\
        =& |y'|^{-2}\sum_{k,m = 1}^n \left((\gmm-1)^2 \frac{\ell^k\ell^m}{|\ell|^2} y^k y^m + 2(\gmm - 1) \frac{\ell^k\ell^m}{|\ell|^2} y^k y^m \right) + 1 - \gmm^2|y'|^{-2} (y' \cdot \ell)^2
        = 1.
    \end{aligned}
    \end{equation}
    Then we compute \[\begin{aligned}
        \p^\mu \p_\mu |y'| 
        =& \sum_{j=1}^n \p_j \left(|y'|^{-1} \sum_{k = 1}^n y^k \left((\gmm - 1) \frac{\ell^j\ell^k}{|\ell|^2} + \delta_j^k \right)\right)
        + \p_0 \left( \gmm|y'|^{-1} (y' \cdot \ell) \right) \\
        =& -|y'|^{-1} + |y'|^{-1} \sum_{j, k=1}^n (\p_j y^k) \left((\gmm - 1) \frac{\ell^j\ell^k}{|\ell|^2} + \delta_j^k \right) + \gmm |y'|^{-1} \p_0 (y'\cdot\ell) \\
        =& -|y'|^{-1} + |y'|^{-1} \sum_{j, k = 1}^n\left((\gmm - 1) \frac{\ell^j\ell^k}{|\ell|^2} + \delta_j^k \right)^2 - \gmm |y'|^{-1} (\gmm |\ell|^2) = (n-1)|y'|^{-1}.
    \end{aligned}
    \]

    \textit{Step 2: Compute the first term in \eqref{eqn:Box_m_Q_in_y'} with $\ell = \ell(\tau), \tilde\xi = \tilde\xi(\tau)$. }
    We start to consider the dependence of $\ell, \tilde\xi$ on $\tau$, and in turn on $(X^0, X')$. 
    We record some important relations first \begin{equation}\label{eqn:quantity_w_r_bdd}
        \dot{\tilde\xi} = (\dot\xi - \ell) - \sigma\dot\ell, \quad |X' - \tilde\xi| = |\tau\ell + r\Theta| \lesssim |r|, \quad |X^0 - \sigma| \lesssim \brk{r}.
    \end{equation}
    
    From \eqref{eqn:p_j_|y'|_fixed_l} and \eqref{eqn:p_0_|y'|_fixed_l}, we have \begin{equation}\label{eqn:p_mu_|y'|}
        \begin{aligned}
        \p_j |y'| &= |y'|^{-1} \sum_{k = 1}^n y^k \left((\gmm - 1) \frac{\ell^j\ell^k}{|\ell|^2} + \delta_j^k\right) + \frac{\p_j \sigma}{|y'|} \sum_{k=1}^n y^k \left(\dot{(A_{\ell})}(X' - \tilde\xi) - A_\ell \dot{\tilde \xi} - \dot{(\gmm\ell)}X^0\right)^k, \\
        \p_0 |y'| &= -\gmm |y'|^{-1} y' \cdot \ell + |y'|^{-1} \p_0 \sigma \sum_{k=1}^n y^k \left(\dot{(A_{\ell})}(X' - \tilde\xi) - A_\ell \dot{\tilde \xi} - \dot{(\gmm\ell)}X^0\right)^k.
    \end{aligned}
    \end{equation}
    Note that the summands in the last term are the same for both, we compute it separately \begin{equation}\label{eqn:last_term_in_p_mu_y'}
        \begin{aligned}
        &\left(\dot{(A_{\ell})}(X' - \tilde\xi) - A_\ell \dot{\tilde \xi} - \dot{(\gmm\ell)}X^0\right)^k 
        = \left(\dot{(A_{\ell})}(X' - \xi)
        + \dot{(A_{\ell})}(\sigma\ell) - A_\ell \dot\xi + A_\ell (\dot {\sigma\ell}) - \dot{(\gmm\ell)}X^0\right)^k \\
        =& \left(\dot{(A_{\ell})}(X' - \xi)\right)^k 
        + \left(\frac{\ud}{\ud\sigma}(A_\ell(\sigma\ell)) - A_\ell \dot\xi - \dot{(\gmm\ell)}X^0\right)^k\\ 
        =& \left(\dot{(A_{\ell})}(X' - \xi)\right)^k 
        + \left(\dot{(\gmm\ell)} (\sigma - X^0)\right)^k + (\ell - \dot\xi)^k + (\gmm - 1)  \frac{(\ell - \dot\xi)\cdot\ell}{|\ell|^2}l^k
        = \calO(r\dot\wp),
    \end{aligned}
    \end{equation}
    where we use $A_\ell(\sigma \ell) = \gmm\sigma\ell$ in the third step and \eqref{eqn:quantity_w_r_bdd} in the last step.
    Furthermore, in $r$-bounded region, we don't need to worry about $r$ growth while in $r$-large (hyperboloidal) region, we take advantage of the explicit formula for $\p_\mu \tau$ in \eqref{eqn:exact_formula_for_p_mu_tau}, which gives the extra cancellation $\sum_j(\Theta^j)^2 - 1 = 0$ for leading order terms in $r$. Therefore, combining with \eqref{eqn:p_mu_|y'|_prod_w_l_fixed}, we achieve \begin{equation}\label{eqn:step_2_in_calculating_F_0}
        Q''(y') \p_\mu |y'| \p^\mu |y'| = (1 + \calO(r\dot\wp)) Q''(y').
    \end{equation}
    
    \textit{Step 3: Compute the second term in \eqref{eqn:Box_m_Q_in_y'} with $\ell = \ell(\tau), \tilde\xi = \tilde\xi(\tau)$. }

    We claim \[
        Q'(y') \p^\mu \p_\mu |y'| = \left(\frac{n-1}{|y'|} + \calO(\dot\wp)\right) Q'(y').
    \]
    We rewrite \eqref{eqn:p_mu_|y'|} using \eqref{eqn:last_term_in_p_mu_y'} to make further computation neater \begin{equation}\label{eqn:p_mu_|y'|_neater}
    \begin{aligned}
        \p_j y^k &= \left((\gmm - 1) \frac{\ell^j\ell^k}{|\ell|^2} + \delta_j^k \right) + \p_j \sigma \left( (\dot{A_l})(X' - \xi) - (\dot{\gmm\ell})(X^0 - \sigma)\right)^k + \calO(\dot\wp), \\
        \p_0 y^k &= -\gmm \ell^k + \p_0 \sigma \left( (\dot{A_l})(X' - \xi) - (\dot{\gmm\ell})(X^0 - \sigma)\right)^k + \calO(\dot\wp), \quad 
        \p_\mu |y'| = |y'|^{-1} \sum_{k = 1}^n y^k \p_\mu y^k, 
    \end{aligned}
    \end{equation}
    where $1 \leq j \leq n$, $0 \leq \mu \leq n$. 
    In view of \eqref{eqn:p_mu_|y'|_neater}, we need to ensure the terms with $r$ growth in $\p^\mu\p_\mu$ exhibits cancellation to prove the claim. 
    From \eqref{eqn:p_mu_|y'|_neater}, 
    we compute \[\begin{aligned}
        &\p^\mu \p_\mu |y'|
        = \sum_j \p_j \left(|y'|^{-1} \sum_k y^k \p_j y^k \right)
        - \p_0 \left(|y'|^{-1} \sum_k y^k \p_0 y^k\right) \\
        =& - |y'|^{-3} \big(\sum_j |y' \cdot \p_j y'|^2 - |y' \cdot \p_0 y'|^2\big)+ |y'|^{-1} \big(\sum_j |\p_j y'|^2 - |\p_0 y'|^2\big) + |y'|^{-1} \big(\sum_j y' \cdot \p_j^2 y' - y' \cdot\p_0^2 y'\big).
    \end{aligned}
    \]
    The bad terms in this formula all come with factor \[
        \sum_j (\p_j \sigma)^2 - (\p_0 \sigma)^2,
    \]
    which is in fact of order $\calO(r^{-1})$ instead of $\calO(1)$ due to the same cancellation $\sum_j (\Theta^j)^2 = 1$.
    Therefore, we prove the claim.

    \textit{Step 4: Compute the last term in \eqref{eqn:calF_0}. }
    Thanks to the ODE \eqref{eqn:Q_tilde_r_equation} that governs $Q$, it suffices to establish \[
        \frac{\nabla^\mu Q \nabla^\nu Q \nabla_{\mu\nu}^2 Q}{1 + \nabla^\alpha Q \nabla_\alpha Q} = \frac{(Q'(y'))^2 Q''(y')}{1 + (Q'(y'))^2} + \calO(r^{-3}\dot\wp).
    \]
    It follows from \eqref{eqn:Box_m_Q_in_y'} and  \eqref{eqn:p_mu_|y'|_prod_w_l_fixed} that \[
        \frac{\nabla^\mu Q \nabla^\nu Q \nabla_{\mu\nu}^2 Q}{1 + \nabla^\alpha Q \nabla_\alpha Q}
        = \frac{(Q'(y'))^2 Q''(y') (\p^\mu |y'|\p_\mu |y'|)(\p^\nu |y'| \p_\nu |y'|) + (Q'(y'))^3 (\p^\mu |y'|)(\p^\nu |y'|) (\p^2_{\mu, \nu} |y'|)}{1 + (1 + \calO(r \dot\wp)) (Q'(y'))^2}.
    \]
    Due to Step 2 \eqref{eqn:step_2_in_calculating_F_0}, we know the first term is \[
        (Q'(y'))^2 Q''(y') (\p^\mu |y'|\p_\mu |y'|)(\p^\nu |y'| \p_\nu |y'|) = (Q'(y'))^2 Q''(y') (1 + \calO(r^2\dot\wp)).
    \]
    On the other hand, $\p_{\mu, \nu}^2 |y'|$ trivially contains $\dot\wp$ decay, which yields \[
        (Q'(y'))^3 (\p^\mu |y'|)(\p^\nu |y'|) (\p^2_{\mu, \nu} |y'|) = (Q'(y'))^3 \calO(r^3\dot\wp).
    \]
    Due to the fast decay of $Q'(\tilde r) \sim \tilde r^{-(n-1)}$ in $\tilde r$, we complete the proof of the claim.
\end{proof}

\begin{remark}
    It might be tempting to compute $\calF_0$ in $(y^0, y')$ coordinates purely, but it turns out that it is then hard to take advantage of the exact cancellation $[\Box_m, \Lambda_{\ell}] = [\Box_m, T_{\tilde\xi}] = 0$ in the case that $\ell, \tilde\xi$ are fixed. 
    Therefore, it is natural to conjecture that $\calF_0$ has nice $\tau$ decay since all the error terms when we introduce the dependence of $\ell, \tilde\xi$ on $\tau$ would come with $\p_\tau$ falling on these parameters.

    As a guide to the reader, one may want to view $A_\ell$ as $\gmm$ like in the one dimensional case, which will help a lot to make sense of the logic behind this computation.
\end{remark}

\begin{remark}
    One can anticipate a cancellation for terms with $r$ growth so that we obtain $r^{-3}$ decay when $n = 3$ using the following heuristics. To see this, we may write $\Box_m$ in the $(\tau, r, \theta)$ coordinates (in particular in the hyperboloidal region) (see \eqref{eqn:Box_m0_U_hyperboloidal}) and compute the leading order term in the asymptotic behavior with respect to $r$: \[\begin{aligned}
    &- \frac{2}{1 - \Theta \cdot \eta'} \p_\tau \p_r Q 
    - \frac{n-1}{r(1- \Theta \cdot \eta')} \p_\tau Q \\
    =& -\frac{2}{1 - \Theta \cdot \eta'}
    \p_r\left(\tilde r^{-(n-1)} \p_\tau(\gmm(1 - \Theta \cdot \ell))r\right)
    - \frac{n-1}{1- \Theta \cdot \eta'}\tilde r^{-(n-1)} \p_\tau(\gmm(1 - \Theta \cdot \ell)) \\
    =& - \frac{\p_\tau(\gmm(1 - \Theta \cdot \ell))}{1 - \Theta \cdot \eta'}   \left(-2(n-1) + 2 + (n-1) \right) \tilde r^{-(n-1)} = (n-3)\frac{\p_\tau(\gmm(1 - \Theta \cdot \ell))}{1 - \Theta \cdot \eta'} \tilde r^{-(n-1)}.
    \end{aligned}
    \] 
    One should note that, in general, one cannot directly infer from the heuristics above that $\calF_0 = \calO(\dot \wp r^{-3})$ for general $n$'s. The potential obstruction arises from the fact that the coefficients of $r^{-k}$ for $k \geq 4$ may not come with sufficient $\tau$-decay, as is apparent from \eqref{eqn:Box_m0_U_hyperboloidal}. Furthermore, as anticipated, it would be  difficult to systematically exploit potential cancellations at the level of each individual coefficient.
\end{remark}

\subsection{Global polar coordinates \texorpdfstring{$(\uptau, \uprho, \uptheta)$}{}} \label{Section:global_coord}

To define a global coordinate, we recall that $(t, \rho, \theta)$ coordinates (see Section~\ref{Section:parametrize_interior}) are well-defined on the whole flat region $\calC_\flatt$ even though the derivation of the linearized operator in Section~\ref{Section:1st_order_formulation} only works for $\calC_\flatnear$. On the other hand, the other coordinates $(\tau, r, \theta)$ (see Section~\ref{Section:parametrize_far_away}) are also accessible on $\calC_\flatfar \cap \calC_\tran \subset \calC_\flatfar \cap \{\chi_g = 0\} \subset \calC_\flatt$. The main goal in this subsection is to glue these two patches together along $\calC_\flatfar \cap \calC_\tran$ to define a global chart. The exposition here is a slight variant of \cite[Section 4.2.3]{LOS22} due to the modification of our foliation in the transition region. See Remark~\ref{rmk:slight_different_foliation}.

Let $\Vint:= \calC_\flatnear \cup \calC_\flatfar \setminus \calC_\tran$ be an open neighborhood of $\calC_\flatnear$ in the ambient $\R^{1+(n+1)}$ and $\Vext:= \mathring{(\calC_\flatnear^c)}$ be an open neighborhood of $\calC_\tran \cup \calC_\hyp$ in $\R^{1+(n+1)}$. From the definition, we know that $\Vint \cap \Vext \subset \calC_\flatfar$, where both coordinates $(t, \rho, \omg)$  and $(\tau, r, \theta)$ are well-defined.
We recall the parametrization \eqref{eqn:parametrize_calQ_in_flatt} so that a point $p = (t, \xi(t) + \gmm^{-1}P_{\ell(t)} F(\rho, \omg) + P_{\ell(t)}^\perp F(\rho, \omg))$ in $\Vint \cap \calQ$ can be represented in the following rectangular coordinate map \[
    \Psiint(p) = (X^0, \ldots, X^n), \quad X^0 = t, X^i = \rho\Theta^i(\omg), \quad i = 1, \ldots, n.
\]
The rectangular coordinate map $\Psiext$ in $\Vint \cap \Vext \cap \calQ$ can be defined analogously, following from \eqref{eqn:parametrize_bsupsigma_in_tau_r_theta}. For a point $p = (\sigma, \eta(\sigma) + r \Theta(\theta))$ in $\Vint \cap \Vext \cap \calQ$ with $\sigma = \tau + \gmm R_f$, \[
    \Psiext(p) = (X^0, \ldots, X^n), \quad X^0 = \sigma, X^i = r \Theta^i(\theta),\quad i = 1, \ldots, n.
\]
Now we recall the definition of $\chi_g$ in \eqref{eqn:chi_g_used_in_global_coord}. In particular, it is supported in $\Vint$ and equal to one on $\Vint \setminus \Vext$. We define the global rectangular coordinates by \[
    \Psi(p):= \chi_g(p) \Psiint(p) + (1 - \chi_g(p)) \Psiext(p), \quad p \in \Vint \cap \Vext \cap \calQ, 
\]
where \begin{equation}\label{eqn:relation_p_in_intersection}
    p = (t, \xi(t) + \gmm^{-1}P_{\ell(t)} F(\rho, \omg) + P_{\ell(t)}^\perp F(\rho, \omg)) = (\sigma, \eta(\sigma) + r \Theta(\theta)),
\end{equation}
where $F$, with a slight abuse of notation mentioned in Section~\ref{Section:parametrize_interior}, is given by $F = \brk{\rho}\Theta(\omg)$.
Moreover, the global coordinates are defined by expressing $(\Psi^0, \Psi^1, \ldots, \Psi^n)$ in polar coordinates, denoted by $(\uptau, \uprho, \uptheta)$.
From \eqref{eqn:relation_p_in_intersection}, one can read off the relation $\uptau = t = \sigma$ in the overlapping region. 
\begin{remark}
    It is worthnoting that $\uptau$ is not the same as $\tau$ (instead, same as $\sigma$ in the exterior as we feel this might be a point which would cause a confusion.
\end{remark}

To see that $\Psi$ indeed defines a coordinate map, it is left to check $d\Psi(p)$ is invertible for all $p \in \Vint \cap \Vext \cap \calQ$ and that $\Psi$ is one-to-one on $(\Vint \cup \Vext) \cap \calQ$. 

\subsubsection{Invertibility of $d\Psi$ in $\Vint \cap \Vext \cap \calQ$}
To check the invertibility of $\ud\Psi$, we consider \[\begin{aligned}
    \Psi \circ \Psiint^{-1}(x) &= \chi_g \circ \Psiint^{-1}(x) x + (1 - \chi_g \circ \Psiint^{-1}(x)) \Psiext \circ \Psiint^{-1}(x) \\ &= \Psiext \circ \Psiint^{-1}(x) + \chi_g \circ \Psiint^{-1}(x) (x - \Psiext \circ \Psiint^{-1}(x)).
\end{aligned}
\]
One could directly compute $\ud(\Psi \circ \Psiint^{-1})(x)$ and notice from the structure that it suffices to show that \begin{enumerate}[ wide = 0pt, align = left, label=\arabic*. ]
    \item $x - \Psiext \circ \Psiint^{-1}(x)$ is small for $x \in \Psiint(\Vint \cap \Vext \cap \calQ)$,
    \item $I - \ud(\Psiext \circ \Psiint^{-1})$ is also small for $x \in \Psiint(\Vint \cap \Vext \cap \calQ)$.
\end{enumerate}
To prove these two claims, we write $\Psiext \circ \Psiint^{-1}: x = (t, y) \mapsto (\sigma, z)$.
According to \eqref{eqn:relation_p_in_intersection}, \begin{equation}\label{eqn:relation_p_in_intersection_F_replaced_by_y}
    t = \sigma, \quad \xi(t) +  \gmm^{-1}(t) P_{\ell(t)} \left(\brk{y} \frac{y}{|y|} \right) + P_{\ell(t)}^\perp\left(\brk{y} \frac{y}{|y|}\right) = \eta(\sigma) + z.
\end{equation}
From the smallness of $\ell$, both $\eta - \xi$ and $\gmm^{-1} P_{\ell} \big(\brk{y} \frac{y}{|y|} \big) + P_{\ell}^\perp\big(\brk{y} \frac{y}{|y|}\big) - \brk{y} \frac{y}{|y|}$ are small. 
Therefore, we obtain smallness of $z - \brk{y} \frac{y}{|y|}$.
On the other hand, it follows from the definition \eqref{eqn:defn_of_4_regions} that $\brk{z} > \frac12 R_f$. Hence, \[
    |y| > R_f, \quad x = (t, y) \in \Psiint(\Vint \cap \Vext \cap \calQ).
\]
By noting that $R_f$ is large, combined with the preceding smallness properties, we know that $|z - y|$ is small, which proves the first claim.
Using the smallness of $\ell, \eta' - \xi', |y|^{-1}$, we can compute $\det(\frac{\p z^k}{\p y^j})_{1 \leq j, k \leq n}$ via implicit differentiation to find out that the leading term (when $|y|^{-1}$ and $|\ell|$ small enough) is $1$, which in turn concludes the proof of the second claim.

\subsubsection{Injectivity of $\Psi$ on $(\Vint \cup \Vext)\cap \calQ$}\label{Subsubsection:inj_of_Psi}
We then prove injectivity of $\Psi$ by contradiction. 
Since $\Psi|_{\{\chi_g = 1\}}$ and $\Psi|_{\{\chi_g = 0\}}$ are injective, respectively, it suffices to rule out the following four possibilities.

\textit{Case 1: There exist $p \neq q \in \Vint \cap \Vext \cap \calQ$ such that $\Psi(p) = \Psi(q)$. }
Assuming that there exists $\tilde\Psi =: \Psiext \circ \Psiint^{-1}: x_j \mapsto \tilde\Psi(x_j)$ with $j = 1, 2$, $x_1 \neq x_2$ sharing the same image under $\Psi \circ \Psiint^{-1}$, that is, \[
    \tilde\Psi(x_1) + \chi_g(\Psiint^{-1}(t_1, y_1)) \cdot (x_1 - \tilde\Psi(x_1)) = \tilde\Psi(x_2) + \chi_g(\Psiint^{-1}(t_2, y_2)) \cdot (x_2 - \tilde\Psi(x_2)).
\]
We compute \[\begin{aligned}
    |\tilde\Psi(x_1) - \tilde\Psi(x_2)| \leq |(I - \tilde\Psi)(x_2)| \cdot | (\chi_g\circ\Psiint^{-1})(x_1) - (\chi_g\circ\Psiint^{-1})(x_2)| \\ + |\chi_g(\Psiint^{-1}(t_1, y_1))| \cdot |(I - \tilde\Psi)(x_1) - (I - \tilde\Psi)(x_2)|.
\end{aligned}
\]
The two terms on the right hand side can be both bounded by a small multiple of $|x_1 - x_2|$ thanks to the two preceding claims, respectively. Therefore, by a similar consideration, \[
    |x_1 - x_2| \leq |\tilde\Psi(x_1) - \tilde\Psi(x_2)| + (\sup|I - \ud\tilde\Psi|) \cdot |x_1 - x_2| \leq \alpha |x_1 - x_2|,
\]
where $\alpha < 1$. This is a contradiction. 

\textit{Case 2: There exist $q \in (\Vint\setminus\Vext) \cap \calQ$ and $p \in \Vint \cap \Vext \cap \{\chi_g \neq 1\} \cap \calQ$ such that $\Psi(p) = \Psi(q)$. }

From
\[
    \Psiint(q) = \chi_g(p)\Psiint(p) + (1 - \chi_g(p))\Psiext(p), 
\]
we deduce that \[
    |\Psiint(q) - \Psiint(p)| = |(1 - \chi_g(p))(\Psiext(p) - \Psiint(p))|
    \leq \alpha |\Psiint(p)| \leq \alpha R_f,
\]
where $\alpha \ll 1$ could be chosen arbitrarily small, which follows from the second claim in Section~\ref{Subsubsection:inj_of_Psi}.
On the other hand, from the definition of $\chi_g$, we notice that $|\Psiint(q) - \Psiint(p)| \geq \dist(\{\chi_g < 1\}, \Vint\setminus\Vext) > \frac19 R_f$, which leads to a contradiction. We remark that though $\Psiint$ is not a coordinate in terms of $r$, we could obtain the preceding inequality due to the smallness of $|r - \rho|$ in the region $\Vint \cap \Vext \cap \calQ$ where these two are jointly defined.

\textit{Case 3: There exist $q \in (\Vext\setminus\Vint) \cap \calQ$ and $p \in \Vint \cap \Vext \cap \{\chi_g \neq 0\} \cap \calQ$ such that $\Psi(p) = \Psi(q)$.}
A contradiction can be derived similar to the previous case.

\textit{Case 4: There exist $p \in (\Vint\setminus\Vext) \cap \calQ$ and $q \in (\Vext\setminus\Vint) \cap \calQ$ such that $\Psi(p) = \Psi(q)$.}

Note that $\Psi(p) = \Psi(q)$ implies $\rho(p) = r(q)$. However, since $p$ and $q$ are separated from each other by $\Vint \cap \Vext$, where for any point $p_1 \in \Vint \cap \Vext \cap \calQ$, $|\Psiint(p_1) - \Psiext(p_1)| < \delta_2 \ll 1$ has uniform smallness. Thus, $\rho(p) \leq \frac12 R_f + \delta_2 < R_f - \delta_1 < r(q)$, which contradicts the assumption.

Thus, by a possible shrinking of $\Vint$ and $\Vext$, $\Psi$ defines an invertible map $\Psi$ from $\Vint \cap \Vext \cap \calQ$ onto some open set in $\R^{1+n}$ and hence defines a global coordinate chart.

\section{Analysis on the Riemannian Catenoid \texorpdfstring{$\barcalC$}{C}}\label{section_ell_est_on_Rm_catenoid}
This section is based on the notation introduced in Section~\ref{Section:basic_properties_of_barcalC} and might be of separate interest. Also, we keep the dimension $n$ to be general in this part. The spectral analysis of the linearized operator will be crucial as this will help to locate the  obstructions and provide a roadmap for designing orthogonality conditions in Section~\ref{Section:modulation_equations}. Some weighted elliptic estimates are derived in this part, which will serve as a building block for concluding pointwise decay in the end. 

Our stability proof requires that all the zero modes of the linearized operator $L$ on $\barcalC$ are generated from the symmetries of catenoid.
Recall that the stability operator (Jacobi operator) is defined by \begin{equation}\label{eqn:stab_op_on_C}
    L = \Delta_{\barcalC} + |\II|^2,
\end{equation}
which appears in the second variation formula \begin{equation}\label{eqn:2nd_variation_barcalC}
    -\int_{\barcalC}\brk{\cdot\,, L\,\cdot\,}    
\end{equation}
as $\barcalC$ is a minimal surface. See \cite[Chapter 1.8]{CM11}.

\subsection{Elliptic estimates (with lower order terms) for the stability operator \texorpdfstring{$L$}{L}}

Now we give weighted elliptic estimates for the stability operator $L$ defined in \eqref{eqn:stab_op_on_C}.
First, we give a definition of weighted spaces in the setting of Riemannian catenoid as follows.

\begin{definition}[Weighted $C^k$ class]
    $C^{k, \beta}(\barcalC)$ is a Banach space of $C^k$ functions such that the norm \[
        \|f\|_{C^{k,\beta}}^2 = \sup_{x \in \barcalC} \sum_{0\leq m\leq k} \brk{\rho}^{m + \beta} |\nabla^k f|
    \]
    is finite, where $\nabla$ is the covariant derivative.
\end{definition}

\begin{definition}[Weighted Sobolev spaces]\label{defn_weighted_Sobolev}
    $H^{s,\delta}(\barcalC)$ is a Hilbert space characterized by 
    \begin{equation}
    f \in H^{s,\delta}(\barcalC) \iff \|f\|_{H^{s,\delta}}^2:= \sum_{0 \leq k \leq s} \int_{\barcalC} \brk{\rho}^{2(\delta + k)} |\nabla^k f|^2\, \ud \Vol_{h_{\barcalC}} < \infty,
\end{equation}
    where $s \in \bbN$ and $h_{\barcalC}$ denotes the metric \eqref{eqn:metric_C_in_rho}.
    We remark that it is easy to extend the definition by duality and interpolation to $s \in \bbR$.
\end{definition}

The idea of weighted Sobolev spaces can be traced back to \cite{CBC81}, \cite{NW73}. Notions like $b$-Sobolev spaces \cite{Mel93}, though historically defined in a different context, are quite relevant as well.
In this subsection, we introduce some fundamental properties with proofs included. We refer the reader to \cite{CBC81}, \cite{LO24} and the reference therein, where ideas can easily be adapted to our setting although the assumptions might be a bit different.

\begin{lemma}\label{lemma_cpt_embedding_weighted_Sobolev}
    If $s > s_1, \delta > \delta_1$, then $H^{s,\delta} \subset\subset H^{s_1, \delta_1}$ is a compact inclusion. 
\end{lemma}

\begin{lemma}
    If $s > s_1, \delta > \delta_1$, then for every $\veps > 0$, there exists $C > 0$ such that for all $f \in H^{s,\delta}$, the following estimate holds: \[
        \|f\|_{H^{s_1, \delta_1}} \leq \veps \|f\|_{H^{s, \delta}} + C\|f\|_{H^{0, \delta}} .
    \]
\end{lemma}

\begin{lemma}\label{lemma_Sobolev_embedding}
    Suppose $s > \frac{n}{2} + s', 
    \delta' < \delta + \frac{n}{2}$, then we have the following continuous inclusion \[
        H^{s,\delta} \subset C^{s', \delta'}.
    \]
\end{lemma}
\begin{proof}
    The proof of this lemma relies on the property that $\barcalC$ has two ends (see \cite[Section 2, Definition]{Sch83} for a rigorous definition), which are asymptotically Euclidean.
\end{proof}

We now establish the elliptic regularity for $L$ (see \eqref{eqn:stab_op_on_C} for its definition) and defer the refined estimate until Section~\ref{Sec:ell-to-lowerorder}.
The weights can be read off from the failure case $p = 1$ for the Hardy's inequality \eqref{eqn:Hardy}, in Step 3 of the proof for Theorem~\ref{thm_ell_reg}. The estimates of this spirit can be found in \cite[Lemma 5.2, Theorem 6.2]{CBC81} but we provide a direct proof for a Riemannian catenoid $\barcalC$.
\begin{theorem}[Elliptic regularity]\label{thm_ell_reg}
    Suppose $f\in H^{0, \delta+2}$ with $-\frac{n}{2} < \delta < \frac{n}{2} - 2$, then every solution $u \in H^2_{loc} \cap H^{0,\delta}$ of $Lu = f$ is also in $H^{2,\delta}$ and satisfies \[
        \|u\|_{H^{2,\delta}} \lesssim \|f\|_{H^{0,\delta+2}} + \|u\|_{H^{1, \delta'}},
    \]
    for any $\delta' \in \R$.
\end{theorem}
\begin{proof}
    Let $\chi_{R_0}$ be a cutoff function being supported in $\{|\rho| > R_0\}$ and being one in $\{|\rho| > 2R_0\}$ region for some fixed $R_0$. (This $R_0$ is irrelevant to the $R_f$ chosen before.)

    \textit{Step 1: Introducing cutoff.}
    We write $u = (1 - \chi_{R_0}) u + \chi_{R_0} u:= u_1 + u_2$, then one can use the elliptic estimates in compact domain to obtain \[
        \|u_1\|_{H^2} \leq \|Lu_1\|_{L^2} + \|u_1\|_{L^2(B_{2R_0}\setminus B_{R_0})}.
    \]

    \textit{Step 2: Reduction using asymptotic flatness.}
    For $u_2$, we further split it into two parts $u_2 = u_{2,+} + u_{2, -}$ where $\supp u_{2,+} \subset \{\rho > 0\}$, $\supp u_{2,-} \subset \{\rho < 0\}$. 
    By writing this into \begin{equation}\label{eqn:u_2_+}
        L u_{2,+}
        = [L, \chi]u_{2,+} + \chi_+ Lu_2
        = [L, \chi]u_{2,+} + \chi_+ f,
    \end{equation}
    where $u_+ = u|_{\rho > 0}$ and $\chi_+ = \chi|_{\rho > 0}$.
    One notes that $u_{2,+}$, $[L, \chi]u_+$, $\chi f_+$ are all supported in $\rho > R_0$ region.
    For the purpose of setting up an estimate for $u_{2,+}$, we are allowed to switch back to $(r,\theta)$ coordinates \eqref{eqn:metric_C_in_r} from $(\rho, \omg)$ \eqref{eqn:metric_C_in_rho}.
    In the $\rho > R_0$ region, we write the right hand side of \eqref{eqn:u_2_+} as \[\begin{aligned}
        L u_{2,+}
        &= \frac{\sqrt{r^{2(n-1)}-1}}{r^{2(n-1)}} \p_r \left( \frac{\sqrt{r^{2(n-1)}-1}}{r^{n-1}}r^{n-1}\p_r u_{2,+}\right) + \frac1{r^2}\sDelta u_{2,+} + \frac{n(n-1)}{r^{2n}}u_{2,+} \\
        &= \frac1{r^{n-1}}\p_r(r^{n-1}\p_r u_{2,+}) + \frac1{r^2}\sDelta u_{2,+}  + O(r^{-1})\p_r^2 u_{2,+} + O(r^{-2}) \p_r u_{2,+} + O(r^{-3}) u_{2,+}.
    \end{aligned}
    \]
    Noting that \begin{equation}\label{eqn:equiv_volume_form_r_large}
        \ud\Vol_{h_\barcalC} = \frac{r^{2(n-1)}}{\sqrt{r^{2(n-1)}-1}}\, \ud r\, \ud \theta \simeq r^{n-1}\, \ud r\, \ud \theta,
    \end{equation}
    it suffices to show the following:  \begin{equation}\label{eqn:ell_est_after_ignoring_lot_reduced_to}
        \frac1{r^{n-1}}\p_r(r^{n-1}\p_r v) + \frac1{r^2}\sDelta v = g \Rightarrow \|v\|_{H^{2,\delta}} \lesssim  \|g\|_{H^{0, \delta+2}},
    \end{equation}
    where $\supp v \subset \{r \gtrsim R_0\}$ and $v \in H^2_{loc} \cap H^{0,\delta}$, $g \in H^{0,\delta+2}$ with the equivalent weighted Sobolev norms defined using the equivalent volume form \eqref{eqn:equiv_volume_form_r_large}.
    This is because we can treat \[
        O(r^{-1})\p_r^2 u_{2,+} + O(r^{-2}) \p_r u_{2,+} + O(r^{-3}) u_{2,+}
    \]
    as the source term and use \eqref{eqn:ell_est_after_ignoring_lot_reduced_to} to show \[
       \|u_{2,+}\|_{H^{2,\delta}} \lesssim \|f + [L,\chi] u_{2,+}\|_{H^{0,\delta+2}}
       \lesssim \|f\|_{H^{0,\delta+2}} + \|u_{2,+}\|_{H^{1, \delta'}}
    \]
    by choosing $R_0$ sufficiently large so that we can absorb error terms to the left hand side and choose $\delta'$ arbitrarily.

    \textit{Step 3: Proof of \eqref{eqn:ell_est_after_ignoring_lot_reduced_to} with $v \in C^\infty_c$.}

    We first prove \eqref{eqn:ell_est_after_ignoring_lot_reduced_to} for $v \in C^\infty_c$.
    We shall use the equivalent volume form \eqref{eqn:equiv_volume_form_r_large} in the following computation. 
    \[\begin{aligned}
        &\|g\|^2_{{H^{0,\delta+2}}}
        = \int r^{2(\delta + 2)} |g|^2\, r^{n-1}\, \ud r\, \ud \theta  
        = \int \left(\frac1{r^{n-1}}\p_r(r^{n-1}\p_r v) + \frac1{r^2}\sDelta v \right)^2 r^{n-1+2\delta+4}\, \ud r\, \ud \theta \\
        =& \int r^{-(n-2\delta-5)}\big(\p_r(r^{n-1}\p_r v)\big)^2 + r^{n+2\delta-1} |\sDelta v|^2 + 2r^{2\delta+2}\big(\p_r(r^{n-1}\p_r v)\big) \sDelta v\, \ud r\, \ud \theta.
    \end{aligned}
    \]
    For the first term, we write 
    \[
    \begin{aligned}
        r^{-(n-2\delta-5)}\big(\p_r(r^{n-1}\p_r v)\big)^2
        &= r^{-(n-2\delta-5)} \left(r^{n-1}\p_r^2 v + (n-1) r^{n-2}\p_r v\right)^2 \\
        &= r^{n+3+2\delta}(\p_r^2 v)^2 + (n-1)^2 r^{n+1+2\delta} (\p_r v)^2 + (n-1)r^{n+2+2\delta}\p_r(\p_r v)^2
    \end{aligned}
    \]
    while for the last term, we compute \[
        \begin{aligned}
            &\quad  \int 2r^{2\delta+2}\big(\p_r(r^{n-1}\p_r v)\big) \sDelta v\, \ud r\, \ud \theta \\
            % &= \int 2\p_r\big(r^{2\delta+2}r^{n-1}\p_r v \sDelta v\big) - 2r^{n+2\delta+1}\p_r v \sDelta \p_r v - 2(2\delta+2) r^{n+2\delta} \p_r v \sDelta v\, \ud r\, \ud \theta \\ 
            &= \int 2\p_r\big(r^{2\delta+2}r^{n-1}\p_r v \sDelta v\big) + 2r^{n+2\delta+1}|\p_r \snabla v|^2 + (2\delta+2) r^{n+2\delta} \p_r |\snabla v|^2 \, \ud r\, \ud \theta \\ 
            &= \int  2r^{n+2\delta+1}|\p_r \snabla v|^2 - (2\delta+2)(n+2\delta) r^{n-1+2\delta} |\snabla v|^2 \, \ud r\, \ud \theta.
        \end{aligned}
    \]

    Now we combine these results together \begin{equation}\label{eqn:ell_est_for_C_after_IBP_before_Hardy}
        \begin{aligned}
            \|g\|_{H^{0,\delta+2}}^2 =\ &\int r^{n+3+2\delta}(\p_r^2 v)^2 + (n-1)^2 r^{n+1+2\delta} (\p_r v)^2 + (n-1)r^{n+2+2\delta}\p_r(\p_r v)^2 + r^{n+2\delta-1} |\sDelta v|^2 \\ &\quad + 2r^{n+2\delta+1}|\p_r \snabla v|^2 - (2\delta+2)(n+2\delta) r^{n-1+2\delta} |\snabla v|^2\, \ud r\, \ud \theta 
            \\
            \geq\ & \int r^{n+3+2\delta}(\p_r^2 v)^2  - (n-1)(3+2\delta)r^{n+1+2\delta}(\p_r v)^2 + r^{n+2\delta-1} |\sDelta v|^2 \\ &\quad + 2r^{n+2\delta+1}|\p_r \snabla v|^2 - 2(\delta+1)(n+2\delta) r^{n-1+2\delta} |\snabla v|^2\, \ud r\, \ud \theta. \\
        \end{aligned}
    \end{equation}
    Furthermore, by applying Hardy's inequality \eqref{eqn:Hardy} with the exact constant, we know that \[\begin{aligned}
        &\int r^{n+3+2\delta}(\p_r^2 v)^2  - (n-1)(3+2\delta)r^{n+1+2\delta}(\p_r v)^2
        \gtrsim_{\delta} \int r^{n+3+2\delta}(\p_r^2 v)^2 \, \ud r\, \ud \theta \\
    \end{aligned}
    \]
    thanks to the fact that $(n-1)(3+2\delta) < \frac{(n+2+2\delta)^2}{4}$ holds with strict inequality sign as long as $\frac{n}{2} \neq \delta + 2$, which is the forbidden endpoint in the assumption of our theorem.
    A similar consideration of the applicability of Hardy's inequality in  \[
        \int r^{n+2\delta+1}|\p_r \snabla v|^2 - (\delta+1)(n+2\delta) r^{n-1+2\delta} |\snabla v|^2\, \ud r\, \ud \theta \gtrsim_\delta \int r^{n+2\delta+1}|\p_r \snabla v|^2\, \ud r\, \ud \theta
    \]
    gives $(\delta + 1)(n + 2\delta) < \frac{(n+2\delta)^2}{4}$, which gives $\delta < -2 + \frac{n}{2}$ and $\delta > -\frac{n}{2}$, verifying the parameter assumptions again.
    By combining these two estimates together with \eqref{eqn:ell_est_for_C_after_IBP_before_Hardy} and applying Hardy inequality \eqref{eqn:Hardy} once again, we finally reached \[
        \|g\|_{H^{0,\delta+2}}^2 \gtrsim \int r^{n-1+2\delta} \left(r^4|\p_r^2 v|^2 + r^2|\p_r v|^2 + r^2 |\p_r \snabla v|^2 + |v|^2 + |\snabla v|^2 + |\snabla^2 v|^2\right) \, \ud r\, \ud \theta,\\
    \]
    where to produce the term with integrand $r^{n-1+2\delta} |v|^2$, we implicitly used the fact that $n + 2\delta \neq 0$ corresponding to the other endpoint for the range of $\delta$. Thus we complete the proof of \eqref{eqn:ell_est_after_ignoring_lot_reduced_to} with $v \in C^\infty_c$.

    \textit{Step 4: Proof of \eqref{eqn:ell_est_after_ignoring_lot_reduced_to} under the general assumption $v \in H^2_{loc} \cap H^{0, \dlt}$.}
        We use a standard regularization argument, for $v$ satisfies \eqref{eqn:ell_est_after_ignoring_lot_reduced_to}, we introduce a new cut-off function $\chi_{R_1} = \tilde\chi(r/R_1)$, where $\tilde\chi \in C^\infty_c$, $\supp \tilde\chi \subset (\frac12, 2)$.
        Then $v_{R_1}:= \chi_{R_1}v$ satisfies \eqref{eqn:ell_est_after_ignoring_lot_reduced_to} with $g$ replaced by $g + 2\chi_{R_1}'\p_r v + \frac1{r^{n-1}}\p_r(r^{n-1}\chi_{R_1}')v$. By using the result in Step 3, we obtain \[
            \|v_{R_1}\|_{H^{2,\delta}} \lesssim  \|g\|_{H^{0, \delta+2}} + \|v_{R_1}\|_{H^{1, \delta}},
        \]
        where the implicit constant is independent of $R_1$. Therefore, by setting $R_1 \to \infty$ (technically, we take limsup as we don't know $v \in H^{1,\delta}$ a priori), it follows that \[
            \|v\|_{H^{2,\delta}} \lesssim  \|g\|_{H^{0, \delta+2}} + \|v\|_{H^{1, \delta}},
        \]
        where $v \in H^2_{loc} \cap H^{0,\delta}$.

        Moreover, by multiplying both sides of \eqref{eqn:ell_est_after_ignoring_lot_reduced_to} by $r^{2(\delta+1)}v$, one can prove that \[
            \|v\|_{H^{1,\delta}} \leq \|g\|_{H^{0,\delta+2}}.
        \]
        Combining these two equations allow us to conclude the proof of \eqref{eqn:ell_est_after_ignoring_lot_reduced_to} with $v \in H^2_{loc} \cap H^{0,\delta}$. 
\end{proof}

\begin{theorem}[Higher regularity]\label{thm_ell_high_reg}
    Suppose $f\in H^{s-2, \delta+2}$ with $-\frac{n}{2} < \delta < \frac{n}{2} - 2$, then every solution $u \in H^s_{loc} \cap H^{0,\delta}$ of $Lu = f$ is also in $H^{s,\delta}$ and satisfies \[
        \|u\|_{H^{s,\delta}} \lesssim \|f\|_{H^{s-2,\delta+2}} + \|u\|_{H^{s-1, \delta'}},
    \]
    for any $s \geq 2$ and for any $\delta' \in \bbR$.
\end{theorem}
\begin{proof}
    We mimic the proof of Theorem~\ref{thm_ell_reg} and one notes that it could be easily reduced to prove the following result regarding the Laplacian operator $\Delta_{\mathfrak{e}}$ in Euclidean background: \begin{equation}\label{eqn:higher_order_ell_reduced_to}
        \Delta_{\mathfrak{e}} v:= \frac1{r^{n-1}}\p_r(r^{n-1} \p_r v) + \frac1{r^2} \sDelta v = g \Rightarrow \|v\|_{H^{s, \delta}} \lesssim \|g\|_{H^{s-2, \delta + 2}}
    \end{equation}
    for $\supp v \subset \{r \gtrsim R_0\}$.

    Then we make an induction on $s$ and it is achieved thanks to the following commutation relation for $\Delta_{\mathfrak{e}}$:
    \[
        [x^j \p_k - x^k \p_j, \Delta_{\mathfrak{e}}] = 0, \quad [r\p_r, \Delta_{\mathfrak{e}}] = -2 \Delta_{\mathfrak{e}}.
    \]
    For example, if one commutes with $r\p_r$, then one might want to estimate \[
        \Delta_{\mathfrak{e}}(r\p_r u) = r\p_r g + 2 \Delta_{\mathfrak{e}} u
        = r\p_r g + 2 g
    \]
    via Theorem~\ref{thm_ell_reg}.
\end{proof}
\begin{remark}
    \cite[Section 7]{LO24} manages to include the endpoint cases $\delta = -\frac{n}2$ and $\delta = \frac{n}2 - 2$ by a dyadic decomposition modification in the definition of weighted Sobolev spaces. 
\end{remark}

\subsection{Linear spectral analysis}\label{Section:spectral_analysis_of_stab_op_L}
In this subsection, we recall some results about the spectrum of $L$ and provide an explicit calculation of all the zero modes, which is not contained in other literature as far as we know. In the following sections, the modulation argument and the decay mechanism of the profile will heavily rely on the geometric meaning (see Remark~\ref{rmk_geometric_meaning_of_zero_modes}) and the quantitative decay properties of all zero modes.

In \cite[Theorem 2.1]{TZ09}, they prove that the Morse index of $\barcalC$ is $1$. By translating from the language for the Hessian \eqref{eqn:2nd_variation_barcalC} to the level of the operator $L$, this amounts to say that $L$ has a unique positive eigenvalue $\mu^2$. We denote the corresponding eigenfunction by $\varphi_\mu$. This instability corresponds to the shrinking of the neck of the catenoid. 
%%if one perturbs the catenoid along the normal direction with $\varphi_\mu n$, then the new object formed will have smaller area (since it is more curved along the neck instead of along the curve, which generates the surface by rotation.
As the catenoid is asymptotically flat, the eigenfunction $\varphi_\mu$ decays exponentially. This can be deduced from \cite[Theorem 4.1]{Ag82} (integrated version of decay) and an adaptation of \cite[Proof of Theorem 5.1, 5.3]{Ag82}.
%%In fact, this only requires an ODE analysis since after knowing the zero eigenfunctions, we could use a comparison argument to see that only $l = 0$ can contributes to positive eigenvalues. 

Now we start to characterize elements in the null space of $L$ as an operator acting on $H^{0,\delta}$ with $-\frac{n}{2} < \delta < \frac{n}{2} - 2$. The weighted Sobolev spaces are defined in Definition~\ref{defn_weighted_Sobolev} and the weights are the same as those in the elliptic estimates we prove in the following subsection. The $n = 2$ case is more delicate and one will see a resonance; see \cite[(2.12)]{DKSW16}.

Thanks to the isometry of tensor products of Hilbert spaces (see \cite{RS80}), one has 
\begin{equation}\label{eqn:isometry_H_0_delta_as_Hilbert_space}
H^{0,\delta}(\barcalC) \simeq L^2(\R \times \bbS^{n-1}, \brk{\rho}^{n-1+2\delta}\, \ud \rho\times \ud\theta) \simeq L^2(\R, \brk{\rho}^{n-1+2\delta}\, \ud \rho)\otimes L^2(\bbS^{n-1}, \ud\theta)    
\end{equation}
with $\delta \in (-\frac{n}2, \frac{n}{2}-2)$.
Suppose $u \in H^{0,\delta}$ solves $Lu = 0$, then by elliptic regularity, $u \in C^\infty$. Moreover, by spherical harmonic decomposition, $u$ can be written into the form \[
    u = \sum_{l = 0}^{\infty} \sum_{m} u_l(\rho) Y_{lm}(\theta),
\] 
where $u_l \in L^2(\R, \brk{\rho}^{n-1+2\delta}\, \ud \rho)$ solves $L_l u_l = 0$ with \begin{equation}\label{eqn:L_l_formula}
    L_l = \frac{1}{\brk{\rho}^{n-1}|F_{\rho}|}\p_\rho(\brk{\rho}^{n-1}|F_\rho|^{-1}\p_\rho)- \frac{l(l+n-2)}{\brk{\rho}^2} + \frac{n(n-1)}{\brk{\rho}^{2n}}, \quad 
    |F_\rho| = \frac{\rho \brk{\rho}^{n-2}}{\sqrt{(\rho^2 + 1)^{n-1} - 1}}.
\end{equation}
Conversely, if $u_l \in L^2(\R, \brk{\rho}^{n-1+2\delta}\, \ud \rho)$ solves $L_l u_l = 0$, then $u_l(\rho) Y_{lm}(\theta)$ is a zero mode of $L$ in $H^{0,\delta}$ due to \eqref{eqn:isometry_H_0_delta_as_Hilbert_space}.
Hence, it suffices to find zero modes $u_l$ of $L_l$ in $L^2(\R, \brk{\rho}^{n-1+2\delta}\, \ud \rho)$.
By standard elliptic regularity, $u_l \in C^\infty$.
Set $v = u_l(\rho)Y_{lm}(\theta)$, then $v \in H^{0,\delta}$ solves $\Delta_{\barcalC} v + V v = 0$.
Thanks to Theorem~\ref{thm_ell_high_reg}, Lemma~\ref{lemma_Sobolev_embedding}, we know $v \in H^{\infty,\delta} \subset C^{\infty,\veps}$ for $0 < \veps < \delta + \frac{n}{2}$.
Thus, $|v(\rho)|, |v'(\rho)| \to 0$ as $\rho \to \infty$. 
Therefore, $u_l$ solves the following ODE \begin{equation}\label{eqn:ODE_for_u_l}
    L_l u_l = 0, \quad u_l(\infty) = u_l'(\infty) = 0.
\end{equation}
By making a change of variable $r = \brk{\rho}$ in the region $\rho > 0$, \eqref{eqn:ODE_for_u_l} becomes \[
    \p_r^2 u + \frac{n-1}{r}\p_r u - \frac{l(l+n-2)}{r^2}u - \frac{1}{r^{2(n-1)}}\p_r^2 u + \frac{n(n-1)}{r^{2n}}u = 0.
\]
\begin{remark}
    This equation in $r$ only works in the region $\{ r > 1\}$. If one examines $\p_\rho u = \frac{\rho}{\brk{\rho}}\p_r u$, we know that $\p_r u$ may have nontrivial boundary behavior at $r = 1$ even for $u$ smooth in $\rho$ variable, causing singularity at the boundary in the distributional sense.
\end{remark}

It is well-known that a solution to an elliptic equation with analytic coefficients is analytic. See \cite{Jo81, MN57}. Therefore, by putting \[
    u_l(r) = C_k r^{-k} + \sum_{m = k+1}^{\infty}  C_m r^{-m}, \quad r > 1
\]
into the equation above and matching coefficients, one would derive $k = n - 2 + l$ from the top order and inductively,
\[
    C_m\big(m(m-n+2) - l(l+n-2)\big) = C_j\big(j(j+1) - n(n-1)\big)
    , \quad m = j + 2(n-1), \quad j \in k + 2(n-1)\bbZ,
\]
while all other coefficients are zero.
The recursion formula can be rewritten as \begin{equation}\label{eqn:recursion_l_geq_2}\begin{aligned}
    C_m = C_{j + 2(n-1)} = \frac{(j-n+1)(j+n)}{(m-l- n + 2)(m + l)} C_j = \frac{j^2 + j - n(n-1)}{j^2 + (3n-2)j +(n-l)(2n+l-2)} C_j,
\end{aligned}
\end{equation}
for $j \in k + 2(n-1)\bbZ$ and $l \geq 2$.
If $l = 1$, the recursion stops at the very first step due to the vanishing of the multiplying factor and hence, $u_1 = \brk{\rho}^{-(n-1)}$, and one can check that this satisfies the equation $L_1 u_1 = 0$ (see \eqref{eqn:L_l_formula}) for all $\rho \in \R$. 
From \eqref{eqn:recursion_l_geq_2} and Raabe's test, the series is uniformly convergent on $\{r \geq 1\}$ while its term by term derivative is uniformly convergent on $\{r \geq 1 + \delta\}$ for any $\delta > 0$. Suppose that there exists $u_l(\rho)$ that is killed by \eqref{eqn:L_l_formula}. Then by considering \[
    0 = \brk{L_l u_l, u_1}_{L^2(\barcalC)}
    = \left\langle\big(L_1 - \frac{l(l+n-2) - (n-1)}{\brk{\rho}^2}\big)u_l, u_1\right\rangle_{L^2(\barcalC)} = C_l \brk{\brk{\rho}^{-2} u_l, u_1}_{L^2(\barcalC)}, 
\]
this implies that $u_l$ at least has one zero. However, we know from the formula of $u_l$ for $\rho \neq 0$ and the continuity that it doesn't have any zeroes, which is a contradiction. 
Finally, combining this with the knowledge of spherical harmonics, we know that all the zero modes of $L$ are given by \begin{equation}\label{eqn:zero_modes_of_L_ej}
    \nu^j = \frac{\Theta^j(\theta)}{\brk{\rho}^{n-1}} , \quad j = 1, \cdots, n
\end{equation}
where $x^j = r\Theta^j$'s are the standard Euclidean coordinates for ambient space.
We normalize these in $H^{2, \delta}$ norm and denote them by $\ebar_j$ for $1 \leq j \leq n$.

Moreover, \cite{OS24} proves that the operator $L$ does not have any threshold resonance, which is crucial for the proof of ILED. Indeed, one just need to rule out the potential candidates from scaling in the ambient spacetime (scaling of $S$) and translation along the axis of rotation.

\begin{remark}\label{rmk_geometric_meaning_of_zero_modes}
    Note that $\ebar_j \in L^2$ as well for $n \geq 3$. These will correspond to translation symmetries in $X^j$-directions of the Riemannian catenoid $\barcalC$ and can be obtained by formally differentiating the translation parameter. In the Lorentzian case, $t\ebar_j$ are generalized zero modes of $\Box_{\calC} + |\II|^2$, corresponding to boosts. These geometric considerations are crucial in the proof of Lemma~\ref{lemma_zero_modes}.
\end{remark}

\subsection{Elliptic estimates (without lower order terms) for \texorpdfstring{$L$}{L}}\label{Sec:ell-to-lowerorder}

Now we use the spectral properties of $L$ we developed just now to prove an elliptic estimate for $L$ in weighted Sobolev spaces without lower order terms. 

\begin{theorem}[Elliptic estimate without lower order terms]\label{thm_ell_est_for_L_w_lot}
    Suppose $f \in H^{0, \delta+2}$ with $-\frac{n}{2} < \delta < \frac{n}{2} - 2$, then for each $u \in H^{2,\delta}$ such that $Lu = f$ \begin{equation}
        \|u\|_{H^{2,\delta}} \lesssim \|Lu\|_{H^{0,\delta+2}} + \sum_j |\brk{u, \ebar_j}|,
    \end{equation}
    where $\{\ebar_j, 1 \leq j \leq n\}$, the zero modes of $L$, are normalized in $H^{2,\delta}$.
\end{theorem}
\begin{proof}
    We denote the spectral projection to the continuous spectrum by $\bbP_c$ (with respect to standard $L^2$), that is, \[
        \bbP_c u = u - \brk{u, \varphi_\mu}_{L^2}\varphi_\mu - \sum_{j = 1, 2, 3, 4} \brk{u, \ebar_j}_{L^2}\ebar_j.
    \]
    Note that $\bbP_c$ is also bounded on $H^{2, \delta}$ thanks to the strong decay of eigenfunctions.
    %%See Reed-Simon's functional analysis Chapter 7 for a discussion of different convention of continuous spectrum.
    
    \textit{Step 1: Showing an estimate for $\bbP_c u$ without lower order terms.}
    We prove that for any $u \in H^2_{loc} \cap H^{0,\delta}$, \begin{equation}\label{eqn:ell_est_with_spectral_proj}
        \|\bbP_c u\|_{H^{2,\delta}} \lesssim \|L \bbP_c u \|_{H^{0,\delta+2}}.
    \end{equation}
    Assuming \eqref{eqn:ell_est_with_spectral_proj} fails, then there exists $u_k \in H^2_{loc} \cap H^{0,\delta}$, $\|\bbP_c u_k\|_{H^{2,\delta}} = 1$, \[
        \|L \bbP_c u_k\|_{H^{0,\delta+2}} \leq \frac1{k}.
    \]
    Thanks to the compactness lemma (Lemma~\ref{lemma_cpt_embedding_weighted_Sobolev}), by choosing $\delta' < \delta$ and passing to a subsequence, $\{\bbP_c u_k\}_k$ is convergent to some function $v$ in $H^{1, \delta'}$. Thanks to  Theorem~\ref{thm_ell_reg}, we know that \[
        \|\bbP_c (u_k - u_m)\|_{H^{2,\delta}} \lesssim \frac1{m} + \frac1{k} + \|\bbP_c (u_k - u_m)\|_{H^{1,\delta'}} \to 0
    \]
    as $k, m \to \infty$.
    Thanks to the uniqueness of limit in $\calD'$, we know $\bbP_c u_k \to v$ in $H^{2,\delta}$.
    Since $\|\bbP_c u_k\|_{H^{2,\delta}} = 1$, we know $\|v\|_{H^{2,\delta}} = 1$ and $Lv = 0$. However, this is a contradiction since $v \notin \ker L$. Indeed, for any $1 \leq j \leq n$, \[
    \brk{v, \ebar_j} = \lim_k \brk{\bbP_c u_k, \ebar_j} = 0,
    \]
    which implies $v = \bbP_c v$. 

    \textit{Step 2: Using spectral decomposition and decay properties of zero modes to conclude the proof.}
    
    For any $u\in H^{2,\delta}$, we write \[
        u = \bbP_c u + \sum_j \brk{u, \ebar_j}\ebar_j,
    \]
    then we compute \[
        \|u\|_{H^{2,\delta}} \leq
        \|\bbP_c u\|_{H^{2,\delta}} + \sum_j |\brk{u, \ebar_j}| 
        \leq  \|Lu\|_{H^{0,\delta+2}} + \sum_j |\brk{u, \ebar_j}|,
    \]
    which completes the proof.
\end{proof}

\begin{remark}
    In the preceding proofs, we make use of the knowledge of eigenfunctions in $H^{2,\delta}$. In fact, one can get around this by performing a cut-off trick (see \cite[Section 8.4]{LOS22}) but one still needs to know all zero eigenfunctions in $L^2$.
\end{remark}

One uses the commutation argument as in the proof of Theorem~\ref{thm_ell_high_reg}, the higher regularity elliptic estimate without lower order terms follows.
\begin{corollary}
    Suppose $f\in H^{s-2, \delta+2}$ with $-\frac{n}{2} < \delta < \frac{n}{2} - 2$, then for each $u \in H^{s,\delta}$ with $s \geq 2$ such that $Lu = f$, \begin{equation}
        \|u\|_{H^{s,\delta}} \lesssim \|Lu\|_{H^{s-2,\delta+2}} + \sum_j |\brk{u, \ebar_j}|.
    \end{equation}
\end{corollary}
\begin{remark}
    For the purpose of elliptic estimates without lower order terms, we only need to project away from zero eigenfunctions. For the boundedness of energy type of estimates later, it is crucial to project away from $\varphi_\mu$ as well to ensure the coercivity of energy flux. 
\end{remark}

\section{Derivation of equations for \texorpdfstring{$\psi$, $\wp$}{psi, wp} and a decomposition for \texorpdfstring{$\wp$}{wp} }\label{Section:Derivation_of_op_modified_profile}

In this section, we will record the first order formulation derived in \cite{LOS22}, which works for $\calC_\flatnear$. Note that the first order formulation would be crucial for the derivation of the modulation equations. Then we will use a relative intrinsic formulation to derive a second order formulation in the whole space. Next, we list different forms of linear operator $\calP$ that will be used in the following chapters. Finally, we introduce the new idea of modified profile (\cite{OS24}), which in turn define the modulation parameters.

\subsection{First order formulation in \texorpdfstring{$\calC_\flatnear$}{C\_flatnear}}\label{Section:1st_order_formulation}
In this subsection, we record some relevant results of the first order formulation valid on $\calC_\flatnear$ for $\psi$ (see \eqref{eqn:defn_psi}). We refer to \cite[Section 3.1-3.4]{LOS22} for details.  
The derivation is based on the Euler-Lagrange equation \[
    \p_t\left(\frac{\delta \calL}{\delta (\p_t \psi)}\right) + \p_j\left( \frac{\delta \calL}{\delta (\p_j \psi)}\right) = \frac{\delta \calL}{\delta \psi},
\]
where $\calL = \sqrt{|g|} = \sqrt{-\det(g)}$. To facilitate the derivation, we are in need of three different types of metric components $h_{\mu\nu}, k_{\mu\nu}$ and $g_{\mu\nu}$: \begin{equation}\label{eqn:defn_g_k_h_in_Cflat}
    g_{\mu\nu}:= (\Phi^*\eta)_{\mu\nu} = \eta(\p_\mu \Phi, \p_\nu\Phi), \quad k_{\mu\nu}:= \eta(\p_\mu \Psi_\wp, \p_\nu \Psi_\wp), \quad h_{\mu\nu}:= \eta(\Psi_{\mu; \wp}, \Psi_{\nu; \wp}), 
\end{equation} 
where $\Psi_{\mu; \wp}:= \p_\mu \Psi_\wp  \Big|_{\substack{\dot{\ell} = 0 \\\dot{\xi} = \ell}}$ and $\Phi$ is the embedding given in \eqref{eqn:HVMC}. Here $0 \leq \mu, \nu \leq n$ are denoting $(t, \rho, \omg)$ coordinates. We will see in the following subsection that this notation could be generalized to the whole region. In $\calC_\flatnear$, we read off from the parametrization \eqref{eqn:parametrize_calQ_in_flatt} that \begin{equation}\label{eqn:Psi_mu_wp_in_flat_region}
    \Psi_{0; \wp} = \begin{pmatrix}
        1 \\ \ell 
    \end{pmatrix},
    \Psi_{j; \wp} = \begin{pmatrix}
        0 \\ \gmm^{-1} \p_j P_\ell F + \p_j P_\ell^\perp F
    \end{pmatrix}, \quad 1 \leq j \leq n.
\end{equation}
The metric components of $h$ can be then computed explicitly \begin{equation}
    \label{eqn:metric_h_in_Cflat}
    h_{00} = -1 + |\ell|^2, \quad h_{0j} = \calO(|\ell|), \quad h_{ij} = \p_i F \cdot \p_j F + \calO(|\ell|^2), 
\end{equation}
where $\cdot$ is the usual dot product in $\R^{n+1}$. Given \eqref{eqn:metric_C_in_rho} and \eqref{eqn:F_for_Rm_catenoid_rho_omg_coord}, the leading term of its determinant and inverse can be calculated easily: \begin{equation}\label{eqn:1st_order_formulation_h_det_inv}
    \sqrt{|h|} = \frac{\rho \brk{\rho}^{2n-3}}{\sqrt{\brk{\rho}^{2(n-1)}-1}} + \calO(|\ell|), \quad h^{00} = -1 + \calO(|\ell|), \quad h^{0j} = \calO(|\ell|).
\end{equation}

\begin{remark}\label{rmk:convention_fixing_parameters}
    We clarify that the parameters in \eqref{eqn:Psi_mu_wp_in_flat_region} are still time dependent. The only difference between $\p_\mu \Psi_\wp$ and $\Psi_{\mu; \wp}$ is that we impose two relations $\dot\ell = 0$ and $\dot\xi = \ell$ when taking $\p_\mu$, that is, we get rid of the terms involving $\dot\ell$ and replace $\dot\xi$ by $\ell$. However, there is no constancy of $\ell$ assumed in the quantity $\Psi_{\mu; \wp}$ and $h_{\mu\nu}$ as well. See Remark~\ref{rmk:explanation_between_our_form_of_the_discrepancy_of_decay_-gmm-1} for a comparison.
\end{remark}

Introducing a linear operator (with an abuse of notation, where $L_1$ is not the same as in Section~\ref{Section:spectral_analysis_of_stab_op_L}) \[
    L_1:= \frac1{\sqrt{|h|}}\p_j\left( \sqrt{|h|}\, \widetilde h^{jk} \p_k\right) + |\II|^2, \quad \widetilde h^{jk}:= h^{jk} - \frac{h^{j0}h^{0k}}{h^{00}},
\]
we define the matrix operator $M$ as \begin{equation}\label{eqn:matrix_operator_M}
    M \vec u:= \begin{pmatrix}
        - \frac{h^{0j}}{h^{00}}\p_j u_1 + \frac1{\sqrt{|h|}h^{00}} u_2 \\ 
        -\sqrt{|h|} L_1 u_1 -\p_j\big(\frac{h^{j0}}{h^{00}} u_2\big)
    \end{pmatrix},
    \quad \vec u = \begin{pmatrix}
            u_1 \\ u_2
        \end{pmatrix}.
\end{equation}
We set \begin{equation}\label{eqn:vec_K_in_C_flatn}
    \vec K:= \begin{pmatrix}
             K \\ \dot K
            \end{pmatrix}
          = \begin{pmatrix}
             - \eta \Bigl( \bigl( \dot{\ell} \cdot \nabla_\ell + (\dot{\xi}-\ell) \cdot \nabla_\xi \bigr) \Psi_\wp, N \Bigr) \\ - \eta \bigl( \sqrt{|h|} (h^{-1})^{00} (0, \dot{\ell}), N \bigr)
            \end{pmatrix}, 
  \qquad 
  \vec f:= \begin{pmatrix}
              f \\ \dot f 
             \end{pmatrix},
\end{equation}
where \begin{equation}\label{eqn:expression_of_f_in_1st_order_form}
        \begin{aligned}
f:= & -\frac{1}{\sqrt{|k|}k^{00}}\left(\mathcal{E}_1+\ldots+\mathcal{E}_4\right)+\left(\frac{1}{\sqrt{|k|}k^{00}}-\frac{1}{\sqrt{|h|}h^{00}}\right) \dot{\psi}
 -\left(\frac{\sqrt{|k|}k^{0j}}{\sqrt{|k|}k^{00}}-\frac{\sqrt{|h|}h^{0j}}{\sqrt{|h|}h^{00}}\right) \partial_j \psi \\
& +\left(-\frac{\sqrt{|h|}h^{00}}{\sqrt{|k|}k^{00}}+1\right) \eta\left(\left(\dot{\ell} \cdot \nabla_{\ell}+(\dot{\xi}-\ell) \cdot \nabla_{\xi}\right) \Psi_{\wp}, N\right), \\
&\calE_1 = \calO(\left(\psi, \partial_\mu \psi\right)^2), \quad \calE_2 = \calO(\dot\wp^2, \dot\wp\ell), \quad \calE_3 = \calO(\dot\wp \p\psi), \quad \calE_4 = \calO(\dot\wp\psi), \\
\dot f:= &\ \partial_j\left(\sqrt{|h|} h^{j 0} \,\eta\left( \big(\dot{\ell} \cdot \nabla_{\ell}+(\dot{\xi}-\ell) \cdot \nabla_{\xi}\big) \Psi_{\wp}, N\right)\right)-\partial_j\left(\sqrt{|h|}h^{j 0} f\right)+ \calE_5 + \ldots + \calE_{11} , \\
&\calE_5 + \ldots + \calE_{11} = \calO(\dot\wp \psi + \dot\wp\p\psi + \dot\wp \p_\Sigma\p\psi + \dot\wp^{2} + \dot\wp\ell + (\psi + \p\psi + \p_\Sigma\p\psi)^2).
\end{aligned}
\end{equation}
Here $\p_\Sigma$ denotes $\p_j$ with $1 \leq j \leq n$ while $\p$ denotes all possible derivatives $\p_\mu$ with $0 \leq \mu \leq n$.
Let $\vec\psi:= \begin{pmatrix}
    \psi & \dot\psi 
\end{pmatrix}^T$ with \[\begin{aligned}
    \dot\psi:= \eta\left(\sqrt{|g|}g^{0 \nu} \partial_\nu \Phi - \sqrt{|h|}h^{00}(1, \ell)^T, N\right) - B \psi, \quad
    B \psi:= \eta\left(\left.\frac{\delta (\sqrt{|g|} g^{0\nu} \partial_\nu \Phi )}{\delta \psi} \right|_{\substack{\psi=0 \\ k=h}} \psi, N\right),
\end{aligned}
\]
then a first order formulation of HVMC equation is given by \begin{equation}\label{eqn:1st_order_HVMC}
    (\p_t - M) \vec\psi = \vec K + \vec f.
\end{equation}
In particular, from the first entry of \eqref{eqn:1st_order_HVMC}, we can solve $\dot\psi$: \begin{equation}\label{eqn:1st_entry_of_HVMC_1st_order_eqn}
    \dot\psi = \sqrt{|h|} h^{0\nu}\p_\nu \psi - \sqrt{|h|}h^{00}(K + f).
\end{equation}
Plugging this into the second entry, we arrive at \begin{equation}\label{eqn:2nd_order_HVMC_in_C_flatn}
    \frac{1}{\sqrt{|h|}} \p_\mu \big( \sqrt{|h|} h^{\mu\nu} \partial_\nu \psi \big) + |\II|^2 \psi = \frac{1}{\sqrt{|h|}} (\dot{K} + \dot{f}) + \frac{1}{\sqrt{|h|}} \p_\nu \big( \sqrt{|h|} h^{0\nu} (K+f) \big).
\end{equation}
We end this subsection by a few remarks. 
\begin{enumerate}[ wide = 0pt, align = left, label=\arabic*. ]
    \item The error terms $\calE_1$ through $\calE_{11}$ in the first order formulation \eqref{eqn:1st_order_HVMC} still contain $\p_t \psi$, but all come with extra smallness.
    \item The derivation for \eqref{eqn:1st_order_HVMC} in \cite{LOS22} relies on the precise structure of parametrization and the vector $N$ in $\calC_\flatnear$.
\end{enumerate}

\subsection{Second order formulation in the whole region}\label{Section:2nd_order_formulation_whole_region}
Recall that \eqref{eqn:2nd_order_HVMC_in_C_flatn} is derived from the first order formulation \eqref{eqn:1st_order_HVMC} in $\calC_\flatnear$, which in turn relies on the fact that $N \parallelsum \tilde N_{\mathrm{int}}$ in $\calC_\flatnear$ and $\mu, \nu$ vary in the specific coordinate $(t, \rho, \omg)$ in $\calC_\flatt$. 

In what follows, the computation works in the whole region with respect to the global chart $(\uptau, \uprho, \uptheta)$ and use Greek letters, varying from $0$ to $n$, to denote the corresponding coordinates.
The aim is to derive a second order equation for $\psi$ in the whole region in a relatively intrinsic way. 

Note that the metric components defined in \eqref{eqn:defn_g_k_h_in_Cflat} in the flat region can be generalized to all regions as the decomposition $\calM = \calQ + \psi N$ is not exclusive for the flat region. We still use $\Psi_\wp$ to denote the parametrization of $\calQ$, where $\Phi = \Psi_\wp + \psi N$, then the metric components are given by \begin{equation}\label{eqn:defn_g_k_metric}
    g_{\mu\nu}:= (\Phi^*\eta)_{\mu\nu} = \eta(\p_\mu \Phi, \p_\nu\Phi), \quad k_{\mu\nu}:= \eta(\p_\mu \Psi_\wp, \p_\nu \Psi_\wp) .
\end{equation}
We also introduce \begin{equation}\label{eqn:defn_h_mu_nu}
    \Psi_{\mu; \wp}:= \p_\mu \Psi_\wp  \Big|_{\substack{\dot{\ell} = 0 \\\dot{\xi} = \ell}}, \quad h_{\mu\nu}:= \eta(\Psi_{\mu; \wp}, \Psi_{\nu; \wp}),
\end{equation}
where we use the convention mentioned in Remark~\ref{rmk:convention_fixing_parameters}. We shall use $\nabla$ to denote the covariant derivative with respect to the embedding $\Phi$.

\subsubsection{Expansions of \texorpdfstring{$\Box_g \Phi$}{Box\_g Phi} \texorpdfstring{in terms of $(\psi, \p\psi)$}{}}
Since the mean curvature tensor of $\Phi(\calM)$ is zero, we derive \begin{equation}\label{eqn:Box_g_Phi_expansion}
    0 = \Box_g \Phi = g^{\mu\nu} \nabla_\mu\nabla_\nu \Psi_\wp + g^{\mu\nu} \nabla_\mu \nabla_\nu (\psi N)
    = g^{\mu\nu} \nabla_\mu\nabla_\nu \Psi_\wp + g^{\mu\nu}\p_\mu\p_\nu(\psi N) - g^{\mu\nu} \Gmm_{\mu\nu}^\lambda \p_\lambda(\psi N),
\end{equation}
where the metric component $g$ is given by \[
    g_{\mu\nu} = k_{\mu\nu} + \eta(\p_\mu \Psi_\wp, N) \p_\nu \psi + \eta(\p_\mu \Psi_\wp, \p_\nu N) \psi + \eta(\p_\mu N, \p_\nu \Psi_\wp) \psi + \eta(N, \p_\nu \Psi_\wp) \p_\mu \psi + \calO((\psi, \p\psi)^2).
\]
For simplicity, we define \begin{equation}\label{eqn:veps_munu}
    \veps_{\mu\nu}:= \eta(\p_\mu \Psi_\wp, N) \p_\nu \psi + \eta(\p_\mu \Psi_\wp, \p_\nu N) \psi + \eta(\p_\mu N, \p_\nu \Psi_\wp) \psi + \eta(N, \p_\nu \Psi_\wp) \p_\mu \psi.
\end{equation}
Using the expansion $(A + \delta A)^{-1} = A^{-1} - A^{-1} (\delta A) A^{-1} + O((\delta A)^2)$ valid for any matrix $A$, we obtain the expansion of $g^{-1}$ \[
    g^{\mu\nu} = k^{\mu\nu} - k^{\mu\alpha} \veps_{\alpha\beta} k^{\beta\nu} + \calO((\psi, \p\psi)^2). 
\]
Combining these with \eqref{eqn:Box_g_Phi_expansion}, we arrive at \begin{equation}\label{eqn:Box_g_Phi_preliminary_exp}
\begin{aligned}
    0 = k^{\mu\nu}\nabla_\mu\nabla_\nu \Psi_\wp - k^{\mu\alpha} \veps_{\alpha\beta} k^{\beta\nu} \nabla_\mu\nabla_\nu \Psi_\wp &+ 
    (\Box_k \psi) N + \psi (\Box_k N) \\&+ k^{\mu\nu}(\p_\mu \psi \p_\nu N + \p_\nu \psi \p_\mu N)
    + \calO((\psi, \p\psi)^2).
\end{aligned}
\end{equation}
Using the short hand notation $\dot\Gmm_{\mu\nu}^\lambda:= \Gmm_{\mu\nu}^\lambda - \tilde \Gmm_{\mu\nu}^\lambda$, we further expand the first term as \begin{equation}\label{eqn:kmunu_nabla_mu_nabla_nu_exp}
    k^{\mu\nu}\nabla_\mu\nabla_\nu \Psi_\wp
    = k^{\mu\nu}(\p_\mu \p_\nu \Psi_\wp - \Gmm_{\mu\nu}^\lambda \p_\lambda \Psi_\wp) 
    = \Box_k \Psi_\wp - k^{\mu\nu} \dot\Gmm_{\mu\nu}^\lambda \p_\lambda \Psi_\wp,
\end{equation}
where \[
    \Box_k = k^{\mu\nu} \tilde\nabla_\mu \tilde\nabla_\nu, \quad \tilde\nabla_\mu \tilde\nabla_\nu \Psi_\wp:= \p_\mu\p_\nu \Psi_\wp - \tilde\Gmm^\lambda_{\mu\nu}\p_\lambda\Psi_\wp, \quad \tilde\Gmm^\lambda_{\mu\nu}:= \frac12 k^{\lambda\kappa} (\p_\mu k_{\kappa\nu} + \p_\nu k_{\mu\kappa} - \p_\kappa k_{\mu\nu}).
\]
Thanks to \eqref{eqn:kmunu_nabla_mu_nabla_nu_exp}, \eqref{eqn:Box_g_Phi_preliminary_exp} becomes \begin{equation}\label{eqn:Box_g_Phi_preliminary2}
    \begin{aligned}
    0 =&\ \Box_k \Psi_\wp - k^{\mu\alpha} \veps_{\alpha\beta} k^{\beta\nu} \nabla_\mu\nabla_\nu \Psi_\wp - k^{\mu\nu} \dot\Gmm_{\mu\nu}^\lambda \p_\lambda \Psi_\wp \\ &+ 
    (\Box_k \psi) N + \psi (\Box_k N) + k^{\mu\nu}(\p_\mu \psi \p_\nu N + \p_\nu \psi \p_\mu N)
    + \calO((\psi, \p\psi)^2) .
\end{aligned}
\end{equation}
Finally, we compute $\dot\Gmm^\lambda_{\mu\nu}$ by expanding \[\begin{aligned}
    \Gmm_{\mu\nu}^\lambda
    &= \frac12 g^{\lambda\delta} (\p_\mu g_{\delta \nu} + \p_\nu g_{\mu\delta} - \p_\delta g_{\mu\nu}) \\
    &= \tilde\Gmm^\lambda_{\mu\nu} + \frac12 k^{\lambda\delta}(\p_\mu \veps_{\delta \nu} + \p_\nu \veps_{\mu\delta} - \p_\delta \veps_{\mu\nu}) - \frac12 k^{\lambda\alpha}\veps_{\alpha\beta}k^{\beta\delta}(\p_\mu k_{\delta \nu} + \p_\nu k_{\mu\delta} - \p_\delta k_{\mu\nu}) + \calO((\psi, \p\psi)^2) \\
    &= \tilde\Gmm^\lambda_{\mu\nu} + \frac12 k^{\lambda\delta}(\p_\mu \veps_{\delta \nu} + \p_\nu \veps_{\mu\delta} - \p_\delta \veps_{\mu\nu}) -  k^{\lambda\alpha}\veps_{\alpha\beta}\tilde\Gmm_{\mu\nu}^\beta + \calO((\psi, \p\psi)^2).
\end{aligned}
\]
In view of \[\begin{aligned}
    &\p_\mu \veps_{\delta \nu} + \p_\nu \veps_{\mu\delta} - \p_\delta \veps_{\mu\nu} \\
    =&\  \p_\mu \left( \eta(\p_\delta \Psi_\wp, N) \p_\nu \psi + \eta(\p_\delta \Psi_\wp, \p_\nu N) \psi + \eta(\p_\delta N, \p_\nu \Psi_\wp) \psi + \eta(N, \p_\nu \Psi_\wp) \p_\delta \psi \right) \\
    &\ + \p_\nu \left( \eta(\p_\mu \Psi_\wp, N) \p_\delta \psi + \eta(\p_\mu \Psi_\wp, \p_\delta N) \psi + \eta(\p_\mu N, \p_\delta \Psi_\wp) \psi + \eta(N, \p_\delta \Psi_\wp) \p_\mu \psi \right) \\
    &\ - \p_\delta \left( \eta(\p_\mu \Psi_\wp, N) \p_\nu \psi + \eta(\p_\mu \Psi_\wp, \p_\nu N) \psi + \eta(\p_\mu N, \p_\nu \Psi_\wp) \psi + \eta(N, \p_\nu \Psi_\wp) \p_\mu \psi \right)\\
    =& 2 \big( \eta(\p_\delta\Psi_\wp, \p_\mu N) \p_\nu \psi + \eta(\p_\delta\Psi_\wp, N) \p_\mu\p_\nu \psi + \eta(\p_\delta\Psi_\wp, \p_\mu\p_\nu N) \psi  \\
    & + \eta(\p_\delta \Psi_\wp, \p_\nu N) \p_\mu\psi + \eta(\p_\delta N, \p_\mu\p_\nu \Psi_\wp)\psi + \eta(N, \p_\mu\p_\nu \Psi_\wp) \p_\delta \psi \big),
\end{aligned}
\]
we obtain \begin{equation}\label{eqn:dot_Gmm_Christoff}
    \begin{aligned}
    \dot\Gmm^\lambda_{\mu\nu} 
    = &\  \eta(\tilde\nabla^\lambda\Psi_\wp, \p_\mu N) \p_\nu \psi + \eta(\tilde\nabla^\lambda\Psi_\wp, N) \p_\mu\p_\nu \psi + \eta(\tilde\nabla^\lambda\Psi_\wp, \p_\mu\p_\nu N) \psi  \\
    & + \eta(\tilde\nabla^\lambda \Psi_\wp, \p_\nu N) \p_\mu\psi + \eta(\tilde\nabla^\lambda N, \p_\mu\p_\nu \Psi_\wp)\psi + \eta(N, \p_\mu\p_\nu \Psi_\wp) \tilde\nabla^\lambda \psi + \calO((\psi, \p\psi)^2)\\
    & - \eta(\tilde\nabla^\lambda\Psi_\wp, N)\tilde\Gmm^\beta_{\mu\nu}\p_\beta\psi - \eta(\tilde\nabla^\lambda \Psi_\wp, \tilde\Gmm^\beta_{\mu\nu}\p_\beta N)\psi - \eta(\tilde\nabla^\lambda N, \tilde\Gmm^\beta_{\mu\nu}\p_\beta \Psi_\wp)\psi - \eta(N, \tilde\Gmm^\beta_{\mu\nu}\p_\beta\Psi_\wp) \tilde\nabla^\lambda \psi \\
    = &\ \eta(\tilde\nabla^\lambda\Psi_\wp, \p_\mu N) \p_\nu \psi + \eta(\tilde\nabla^\lambda\Psi_\wp, N) \tilde\nabla_\mu\tilde\nabla_\nu \psi + \eta(\tilde\nabla^\lambda\Psi_\wp, \tilde\nabla_\mu\tilde\nabla_\nu N) \psi  \\
    & + \eta(\tilde\nabla^\lambda \Psi_\wp, \p_\nu N) \p_\mu\psi + \eta(\tilde\nabla^\lambda N, \tilde\nabla_\mu\tilde\nabla_\nu \Psi_\wp)\psi + \eta(N, \tilde\nabla_\mu\tilde\nabla_\nu \Psi_\wp) \tilde\nabla^\lambda \psi + \calO((\psi, \p\psi)^2). \\
\end{aligned}
\end{equation}
With \eqref{eqn:dot_Gmm_Christoff}, we write \eqref{eqn:Box_g_Phi_preliminary2} as \begin{equation}\label{eqn:Box_g_Phi_preliminary3}
    \begin{aligned}
    0 =&\ \Box_k \Psi_\wp - \veps_{\mu\nu}  \tilde\nabla^\mu\tilde\nabla^\nu \Psi_\wp - 2\eta(\tilde\nabla^\lambda\Psi_\wp, \p_\mu N) \tilde\nabla^\mu \psi \p_\lambda\Psi_\wp
    - \eta(\tilde\nabla^\lambda\Psi_\wp, N) 
    \Box_k\psi \p_\lambda\Psi_\wp \\ 
    & - \eta(\tilde\nabla^\lambda\Psi_\wp, \Box_k N) \psi \p_\lambda\Psi_\wp -  \eta(\tilde\nabla^\lambda N, \Box_k \Psi_\wp)\psi \p_\lambda\Psi_\wp - \eta(N, \Box_k \Psi_\wp) \tilde\nabla^\lambda \psi 
    + (\Box_k \psi) N + \psi (\Box_k N) \\ 
    & + k^{\mu\nu}(\p_\mu \psi \p_\nu N + \p_\nu \psi \p_\mu N)
    + \calO((\psi, \p\psi)^2) .
\end{aligned} 
\end{equation} 

\subsubsection{Computations with fixed \texorpdfstring{$\ell$, $\xi - \uptau \ell$}{parameters}}
We first assume $\ell(\uptau) = \ell$ for some fixed $\ell \in \R^{n+1}$ and $\xi(\uptau) = a + \uptau \ell$ for some fixed $a \in \R^{n+1}$. In other words, we abuse notation between $h_{\mu\nu}$ and $k_{\mu\nu}$ in the following computation.
Under this assumption, $\Box_h \Psi_\wp = 0$ since $\Psi_\wp$ is obtained by a boost and translation of a catenoid so that the mean curvature tensor vanishes. Besides, $\eta(\p_\lambda \Psi_\wp, n_\wp) = 0$ follows from the definition of $n_\wp$. We also recall the relation $\eta(N, n_\wp) = 1$.
Therefore, by testing
\eqref{eqn:Box_g_Phi_preliminary3} with $n_\wp$, it reduces to \begin{equation}\label{eqn:Box_g_Phi_preliminary4}
    \begin{aligned}
    0 = \eta(\Box_g \Phi, n_\wp) = - \veps_{\mu\nu}  \eta(\tilde\nabla^\mu\tilde\nabla^\nu \Psi_\wp, n_\wp)    
    + \Box_k \psi  + \eta(\Box_k N, n_\wp)\psi + 2\eta(\tilde\nabla^\mu N, n_\wp)\p_\mu \psi 
    + \calO((\psi, \p\psi)^2) .    
\end{aligned}
\end{equation}
We single out the linear parts of $\psi$ and $\p \psi$ and compute separately:  \begin{equation}\label{eqn:2_lin_parts_of_Box_g_Phi}
    \begin{aligned}
    \eta\left(\frac{\delta(\Box_g \Phi)}{\delta \psi}, n_\wp\right)
    &= - \eta(\tilde\nabla^\mu\tilde\nabla^\nu \Psi_\wp, n_\wp) \eta(\p_\mu \Psi_\wp, \p_\nu N) - \eta(\tilde\nabla^\mu\tilde\nabla^\nu \Psi_\wp, n_\wp) \eta(\p_\mu N, \p_\nu \Psi_\wp) + \eta(\Box_k N, n_\wp), \\
    \eta\left(\frac{\delta(\Box_g\Phi)}{\delta (\p_\mu\psi)}, n_\wp\right)
    &= -\eta(\tilde\nabla^\mu\tilde\nabla^\nu \Psi_\wp, n_\wp) \eta(N, \p_\nu \Psi_\wp) - \eta(\tilde\nabla^\nu\tilde\nabla^\mu \Psi_\wp, n_\wp)\eta(\p_\nu \Psi_\wp, N) + 2 \eta(\tilde\nabla^\mu N, n_\wp).
\end{aligned}
\end{equation}

We introduce an auxiliary vector $W:= N - n_\wp$. One observes that $\eta(W, n_\wp) = 0$ and hence $W$ 
denotes the tangential part of $N$.
In preparation, we record the following expansion \begin{equation}\label{eqn:W_expansion}
     W = \eta(W, \p_\mu \Psi_\wp) \tilde\nabla^\mu \Psi_\wp
\end{equation}
thanks to the fact that $\{\p_\nu \Psi_\wp\}$ spans the tangent space.

Then we start to show the second quantity in \eqref{eqn:2_lin_parts_of_Box_g_Phi} is vanishing: \[\begin{aligned}
    \eta\left(\frac{\delta(\Box_g\Phi)}{\delta (\p_\mu\psi)}, n_\wp\right)
&= 2\eta(\tilde\nabla^\nu \Psi_\wp, \tilde\nabla^\mu n_\wp) \eta(N, \p_\nu \Psi_\wp) + 2 \eta(\tilde\nabla^\mu N, n_\wp) \\
&= 2\eta(\tilde\nabla^\nu \Psi_\wp, \tilde\nabla^\mu n_\wp) \eta(W, \p_\nu \Psi_\wp) + 2 \eta(\tilde\nabla^\mu W, n_\wp) \\
&= 2\eta(\tilde\nabla^\nu \Psi_\wp, \tilde\nabla^\mu n_\wp) \eta(W, \p_\nu \Psi_\wp) - 2 \eta(W, \tilde\nabla^\mu n_\wp) = 0,
\end{aligned}
\]
where the last equality follows from  \eqref{eqn:W_expansion}.

To simplify the first line in \eqref{eqn:2_lin_parts_of_Box_g_Phi}, a result $\Box_h n_\wp = -|\II|^2 n_\wp$ by \cite[Corollary II]{RV70} is needed. We first note that \[
    \eta(\tilde\nabla_\mu\tilde\nabla_\nu \Psi_\wp, \p_\sigma \Psi_\wp) = \eta(\p_\mu\p_\nu \Psi_\wp, \p_\sigma \Psi_\wp) - \frac12 k^{\lambda\kappa}(\p_\mu k_{\kappa\nu} + \p_\nu k_{\mu\kappa} - \p_\kappa k_{\mu\nu}) k_{\lambda\sigma}
    = 0
\]
and hence $\eta(\tilde\nabla_\mu\tilde\nabla_\nu \Psi_\wp, W) = 0$.
We next compute $\Box_h W$ using \eqref{eqn:W_expansion}: \[\begin{aligned}
    \Box_h W = &\eta(W, \p_\mu \Psi_\wp) \tilde\nabla^\mu \Box_h \Psi_\wp + \eta(W, \p_\mu \Psi_\wp) \tilde R^{\mu\lambda}\p_\lambda \Psi_\wp + (\Box_h \eta(W, \p_\mu \Psi_\wp)) \tilde\nabla^\mu \Psi_\wp \\ &+ 2\eta(\p_\nu W, \p_\mu \Psi_\wp) \tilde\nabla^\nu \tilde\nabla^\mu \Psi_\wp + 2\eta(W, \tilde\nabla_\mu\tilde\nabla_\nu \Psi_\wp) \tilde\nabla_\mu\tilde\nabla_\nu \Psi_\wp,
\end{aligned}
\]
where the first and last terms are both zero. By testing against $n_\wp$, one obtains \[
    \eta(\Box_h W, n_\wp) = 2\eta(\p_\nu W, \p_\mu \Psi_\wp) \eta(\tilde\nabla^\nu \tilde\nabla^\mu \Psi_\wp, n_\wp).
\]
We may thus compute \[\begin{aligned}
    \eta\left(\frac{\delta(\Box_g \Phi)}{\delta \psi}, n_\wp\right)
    &= - 2\eta(\tilde\nabla^\nu \tilde\nabla^\mu \Psi_\wp, n_\wp) \eta(\p_\mu \Psi_\wp, \p_\nu N) + \eta(\Box_h n_\wp, n_\wp) + \eta(\Box_h W, n_\wp) \\
    &= -2 \eta(\tilde\nabla^\nu \tilde\nabla^\mu \Psi_\wp, n_\wp) \eta(\p_\mu \Psi_\wp, \p_\nu n_\wp) - |\II|^2 \\
    &= 2\eta(\tilde\nabla^\mu \Psi_\wp, \tilde\nabla^\nu n_\wp) \eta(\p_\mu \Psi_\wp, \p_\nu n_\wp) - |\II|^2 = |\II|^2.
\end{aligned}
\]
Thus, \eqref{eqn:Box_g_Phi_preliminary4} can be rewritten as  \begin{equation}\label{eqn:Box_g_Phi_preliminary5}
    (\Box_h + |\II|^2)\psi + \calO((\psi, \p\psi)^2) = 0.  
\end{equation}

\begin{remark}\label{rmk_lin_op_not_affected}
    As a side product, this computation tells us that though the gauge choice is not the geometric normal, it does not affect the linearized operator if the parameters are fixed, i.e., $\xi - \upsigma \ell$ and $\ell$ are constants. Heuristically, the linearized operator is obtained via pairing \eqref{eqn:Box_g_Phi_expansion} with $n_\wp$, where only the inner product $\eta(N, n_\wp) = 1$ matters instead of the gauge itself. 
    %%This just tells us that the gauge choice is simple. At the beginning, we just want the vector to lie in each slice, so we cannot choose the standard normal. A second variation formula with non-standard normal vector might also suggest this. 
\end{remark}

\subsubsection{Computations with time-dependent $\ell, \xi$}
Now we let the parameters to be time-dependent as usual. If we use $(\uptau, \uprho, \uptheta)$ coordinates, then all the extra terms coming out are due to a derivative falling on parameters, which produces $\calO(\dot\wp)$ coefficients. More precisely, the discrepancy in the coefficients we care about is of size $O(k_{\mu\nu} - h_{\mu\nu}) = \calO(\dot\wp)$, which one can easily verify in $\calC_\flatt$. A more involved computation for the nonlinearity has already been done in Lemma~\ref{lemma_F_0_decay}.

Therefore, by examining how we obtain \eqref{eqn:Box_g_Phi_preliminary5} from \eqref{eqn:Box_g_Phi_preliminary3} through \eqref{eqn:Box_g_Phi_preliminary4}, we know that $\psi$ satisfies \begin{equation}\label{eqn:P_psi_eqn}
    \calP \psi = f,
\end{equation}
where $\calP = \calP_h + \calP_\pert$ with \begin{equation}\label{eqn:form_of_P_P_0_P_pert}
    \calP_h = \frac1{\sqrt{|h|}}\p_\mu( \sqrt{|h|}h^{\mu\nu}\p_\nu) + V, \quad \calP_\pert = \calO(\dot\wp) + \calO(\dot\wp)\p_\mu + \calO(\dot\wp)\p_\mu\p_\nu, \quad f = \calO(\dot\wp + (\psi, \p\psi)^2).
\end{equation}
As a note, $V$ is the squared norm of the second fundamental form of the embedding via $\Psi_\wp$ up to terms involving $\dot\ell, \dot\xi - \ell$.

Note that \eqref{eqn:P_psi_eqn} together with \eqref{eqn:form_of_P_P_0_P_pert} is sufficient for our use in regions with bounded $\uprho$, i.e. the complement of $\calC_\hyp$. On the other hand, we would like to know the precise $\uprho$ decay as well as the $\uptau$ decay of the coefficients of $\calP_\pert$ and in $\calF$. The equation derived in terms of the graph formulation introduced in Section~\ref{Section:eqn_derivation_for_varphi_in_Chyp} makes this accessible. We fix the parameters first as before.

\subsubsection{Computation via graph formulation in $\calC_\hyp$}\label{Section:identify_P_and_P_graph}
In this subsubsection, we establish the relation between the operator $\calP_\graph$ (see \eqref{eqn:calP_graph}) derived in the graph formulation and the operator $\calP$ derived in this section (Section~\ref{Section:2nd_order_formulation_whole_region}) in terms of a semi-intrinsic computation (see \eqref{eqn:P_psi_eqn} and \eqref{eqn:form_of_P_P_0_P_pert}).

We first fix the parameters as usual and then view them as time dependent later.
Recall that the parametrization $\Psi_\wp$ is via $(X^0, X^1, \ldots, X^n, Q_\wp)$ in $\calC_\hyp$, we calculate \[
    h_{\mu\nu} = \eta((\p_\mu, Q_\mu), (\p_\nu, Q_\nu)) =  m_{\mu\nu} + Q_\mu Q_\nu,
\]
where the computation is in the $(X^0, \cdots, X^n)$ coordinates.
A direct computation shows that \begin{equation}\label{eqn:h_inverse_in_C_hyp}
    h^{\mu\nu} = m^{\mu\nu} - \frac{Q^\mu Q^\nu}{1 + Q^\alpha Q_\alpha} = m^{\mu\nu} - s^{-2} Q^\mu Q^\nu, 
\end{equation}
where $Q^\mu = m^{\mu\nu} Q_\nu$ and the last equality follows from \eqref{eqn:alternative_formula_for_s}.
Rewriting the operator $\calP_\graph$ \eqref{eqn:calP_graph}, we obtain \begin{equation}\label{eqn:calP_graph_via_h_metric}
    \begin{aligned}
    \calP_\graph &= h^{\mu\nu}\p_\mu\p_\nu - 2s^{-2}\left(Q^{\mu\nu} Q_\mu \p_\nu - s^{-2} Q^\nu Q^\mu Q^\lambda Q_{\mu\lambda}\p_\nu\right) \\
    &=  h^{\mu\nu}\p_\mu\p_\nu - 2s^{-2}\left(m^{\mu\nu} (s\p_\mu s) \p_\nu - s^{-2} Q^\nu Q^\mu (s\p_\mu s)\p_\nu\right) = h^{\mu\nu} \p_\mu\p_\nu - 2s^{-1}(\p_\mu s)h^{\mu\nu}\p_\nu
\end{aligned}
\end{equation}
thanks to the identity $s\p_\mu s = Q_{\mu\nu} Q^\nu$.

Recall that $\varphi$ (see \eqref{eqn:relation_varphi_psi}) satisfies $\calP_\graph \varphi = \calF$. We then could read off the equation satisfied by $\psi$: \begin{equation}\label{eqn:relation_calP_and_calP_graph}
    \calP \psi = s^{-1}\calP_\graph \varphi = s^{-1}\calF, \quad \calP:= \calP_\graph + s^{-1}[\calP_\graph, s].
\end{equation}
Though it seems that $\calP$ is a slight abuse of notation compared to \eqref{eqn:P_psi_eqn}, it would be clear that the operator $\calP$ here shares the same decomposition  \eqref{eqn:form_of_P_P_0_P_pert} with more precise $r$-decay properties.
Using \eqref{eqn:calP_graph_via_h_metric}, we compute \begin{equation}\label{eqn:calP_obtained_from_calP_graph_preliminary}
    \calP \psi = s^{-1}h^{\mu\nu}\p_\mu\p_\nu(s\psi) - 2s^{-2}(\p_\mu s) h^{\mu\nu}\p_\nu(s\psi)
    = h^{\mu\nu}\p_\mu\p_\nu \psi + s^{-1}h^{\mu\nu}\left(\p_\mu\p_\nu s - 2s^{-1}(\p_\mu s)(\p_\nu s)\right) \psi.
\end{equation}
On the other hand, the Christoffel symbol relative to metric $h$ is given by \[
    \Gmm_{\mu\nu}^\lambda = \frac12 h^{\lambda\kappa}(\p_\mu h_{\nu\kappa} + \p_\nu h_{\mu\kappa} - \p_\kappa h_{\mu\nu}) = h^{\lambda\kappa} Q_{\mu\nu}Q_\kappa.
\]
We note that \[
    h^{\mu\nu} Q_{\mu\nu}
    = m^{\mu\nu} \p_\nu \p_\mu Q - s^{-2}Q^\mu Q^\nu Q_{\mu\nu}
    = \calF_0
\]
exhibits $\calO(\dot\wp r^{-3})$ decay thanks to Lemma~\ref{lemma_F_0_decay}. Therefore, \[
    \Box_h = h^{\mu\nu}\p_\mu\p_\nu + h^{\mu\nu} \Gmm_{\mu\nu}^{\lambda}\p_\lambda = h^{\mu\nu}\p_\mu\p_\nu + \calO(\dot\wp r^{-4})\p_\lambda.
\]
Combining with \eqref{eqn:calP_obtained_from_calP_graph_preliminary}, we obtain \[
    \calP = \Box_h + s^{-1}(\Box_h s - 2s^{-1}h^{\mu\nu} (\p_\mu s)(\p_\nu s)) + \calO(\dot\wp r^{-4})\p^{\leq 2}_\lambda.
\]
Furthermore, we observe \[
    -|\II|^2 = \eta(\Box_h n_\wp, N) = s\Box_h s^{-1} = s^{-1}(\Box_h s - 2s^{-1}h^{\mu\nu} (\p_\mu s)(\p_\nu s)),
\]
where first two equalities follow from $\Box_h n_\wp = -|\II|^2 n_\wp$ and $N = s(0, 1)^T$, respectively.

Finally, we let those parameters start to vary with respect to $\uptau$. We would like to write out using $\p_\uptau, \p_\uprho, \p_\uptheta$ instead of $\p_{X^\mu}$'s. As a note, \eqref{eqn:h_inverse_in_C_hyp} remains true even if we change the coordinates as long as parameters are fixed. However, when parameters start to vary, we need to add extra caution in terms of the $r$-decay since the discrepancy caused by $m^{\mu\nu}$ might not come with $r$-decay. To read off the precise structure in the following formula \eqref{eqn:calP_pert_in_C_hyp}, the $r$-decay of those coefficients in \eqref{eqn:Box_m0_U_hyperboloidal} is needed.
We record the following formula, which works in the hyperboloidal region $\calC_\hyp$: \begin{equation}\label{eqn:calP_pert_in_C_hyp}
\begin{aligned}
    \calP = \calP_h + \calP_\pert, \quad 
    \calP_\pert = \ringa \p_\tau(\p_r + \frac{n-1}{2r}) + \ringa^{\mu\nu}\p_{\mu\nu} + \ringb^\mu \p_\mu + \ringc, \\
    |\ringa|, |\ringa^{rr}| \lesssim \dot\wp, \quad |\ringa^{\tau y}|, |\ringb^y| \lesssim \dot\wp r^{-1}, \quad |\ringa^{\tau\tau}|, |\ringb^\tau| \lesssim \dot\wp r^{-2}, \quad |\ringc| \lesssim \dot\wp r^{-4}
\end{aligned}
\end{equation}
with $y$ denoting the spatial variables $(r, \theta)$. 
\begin{remark}\label{rmk:explanation_between_our_form_of_the_discrepancy_of_decay_-gmm-1}
    \eqref{eqn:calP_pert_in_C_hyp} is in line with \cite[Remark 4.2 (2)]{LOS22}. The only difference is that they further restrict the metric component $h_{\mu\nu}$ at the final time $\uptau = \tau_f$ of the bootstrap assumption (see Section~\ref{Section:BA}) which amounts to another error caused by Taylor's expansion \[
        |\ell(\uptau) - \ell(t_2)| \leq \uptau \sup_{\uptau \leq s \leq t_2} \dot\ell(s) \lesssim \uptau^{-\gmm}
    \] 
    if $|\dot\wp| \lesssim \uptau^{-\gmm-1}$.
    We will also record the form with restriction to final time later, in preparation for the ILED estimates and the modified profiles. For example, this decomposition was used in \cite[Proof of Lemma 7.8]{LOS22}. Moreover, we refer to Remark~\ref{rmk:convention_fixing_parameters} for a comparison.
\end{remark}

\begin{remark}
    By making a separate Taylor expansion of the coefficients in \eqref{eqn:form_of_P_P_0_P_pert}, \eqref{eqn:calP_pert_in_C_hyp} around $\dot\xi = \ell$, $\ell = 0$ (i.e., expanding terms with respect to the original catenoid without any shift and boost), in view of \eqref{eqn:1st_order_formulation_h_det_inv} , we obtain \begin{equation}\label{eqn:form_of_calP_suit_for_ell_est}
        \calP = \calP_\elliptic + \calP_\uptau,
    \end{equation}
    where \begin{equation}\label{eqn:calP_uptau_op}
        \calP_\uptau = \calO(1)\p_\uprho\p_\uptau + \calO(\brk{\uprho}^{-1})\p_\uptheta \p_\uptau + \calO(\brk{\uprho}^{-1}) \p_\uptau^2 + \calO(\brk{\uprho}^{-1})\p_\uptau 
    \end{equation}
    and \begin{equation}\label{eqn:op_P_elliptic}
        \calP_\elliptic = L + \calP_\elliptic^\pert, \quad  \calP_\elliptic^\pert:= o_{\wp, R_f}(1)\left(\p_\Sigma^2 + \brk{\uprho}^{-1} \p_\Sigma + \brk{\uprho}^{-2} \right) + \calO(\dot\wp) \p_\Sigma.
    \end{equation}
    Both operators are represented in $(\uptau, \uprho, \uptheta)$ coordinates and $\calP_\elliptic$ is an elliptic operator which does not contain any $\p_\uptau$ derivatives. Moreover, we recall the definition of the linearized operator $L$ (see \eqref{eqn:stab_op_on_C} and \eqref{eqn:L_l_formula}): \[
        L = \Delta_{\barcalC} + |\II|^2 = \frac{1}{\brk{\uprho}^{n-1}|F_\uprho|}\p_\uprho \left(\brk{\uprho}^{n-1}|F_\uprho|^{-1} \p_\uprho \right) + \frac1{\brk{\uprho}^2} \sDelta_{\bbS^{n-1}} + \frac{12}{\brk{\uprho}^{2n}}.
    \]
    This remark is in line with \cite[Remark 4.2 (1)]{LOS22}.
\end{remark}

We end this part by analyzing the source term. In view of \eqref{eqn:P_psi_eqn}, \eqref{eqn:relation_varphi_psi}, \eqref{eqn:relation_calP_and_calP_graph} and Lemma~\ref{lemma_F_0_decay}, the source term in $C_\hyp$ is of the form \begin{equation}\label{eqn:source_term_f_in_C_hyp}
    f = s^{-1}\calF_0 + s^{-1}(\calF_2 + \calF_3) = \calO(\dot\wp r^{-3} + (\psi, \p\psi)^2).
\end{equation}
Though $\calP$ and \eqref{eqn:relation_calP_and_calP_graph} might be different up to something non-decaying in $\uptau$ and $\uprho$, we abuse notation to identify those as in \eqref{eqn:relation_calP_and_calP_graph} going forward since this won't affect the decay analysis.

\subsection{Modified profile}\label{Section:mod_profile}
From \eqref{eqn:source_term_f_in_C_hyp}, we know that the perturbation $\psi$ satisfies an equation with forcing term $\calO(\dot\wp r^{-3})$ plus some nonlinearity. One fundamental idea in proving time decay (late time tails) of wave type equation is to trade $r$-decay for $\tau$-decay. In order to achieve twice integrable $\tau$-decay, we need the source term to have more spatial decay, say $\calO(\dot\wp r^{-4})$. 

Our goal in this subsection is to construct a modified profile $p$ such that $\psi = p + q$ and $p$ captures the slow decay of the source term of $\psi$. On the other hand, we expect to establish an improved (more specifically, a twice integrable rate) $\tau$-decay of $q$. This idea is based on how \cite{OS24} resolves a similar issue in the $n = 3$ case.

Naively, one might wish to choose $p$ such that $\calP p = \chi s^{-1}\calF_0$ with suitable cutoff $\chi$. However, due to the spectral property of the linearized operator, we need to ensure appropriate orthogonality conditions to let $p$ stay away from the unstable directions. To compensate, we construct $p$ to be a solution with zero initial data (on $\Sigma_{\uptau = 0}$) to \begin{equation}\label{eqn:calP_h_eqn_for_mod_profile}
    \calP_h p = f_c + g_1 + \p_\tau g_2, \quad f_c:= \chi_{|\uprho| > \uprho_0}(\rho) s^{-1}\calF_0, 
\end{equation}
where $\chi_{|\uprho| > \uprho_0}(\rho)$ is a smooth bump function with support in $|\uprho| > \uprho_0 \gg R_f \gg 1$ and $g_1 + \p_\tau g_2$ (to be determined) features a $10$ dimensional flexibility. Note that the change from $\calP$ to $\calP_h$ is just for simplicity as the perturbation part $\calP_\pert$ can be treated as acceptable errors. (See \eqref{eqn:form_of_P_P_0_P_pert} and \eqref{eqn:calP_pert_in_C_hyp}.)The existence of such $p$ shall follow from a similar style of proof for implicit function theorem, namely iteration.
Assuming the existence of such $p$ for now, then $q = \psi - p$ would satisfy \begin{equation}\label{eqn:calP_q_source_term}
    \calP q = (f - f_c) - (g_1 + \p_\tau g_2) - \calP_\pert p = \chi_{\leq \uprho_0} (s^{-1} \calF_0) + s^{-1}(\calF_2 + \calF_3) - (g_1 + \p_\tau g_2) - \calP_\pert p,
\end{equation}
where a slight abuse of notation appears here that we use $s^{-1}(\calF_2 + \calF_3)$ to denote the nonlinearity for the the whole region instead of only in $\calC_\hyp$.

Note that $\calP_h|_{\ell \equiv \ell(\tau_f)}$ is the exact linearized operator corresponding to $\barcalC_{\ell(\tau_f), a}$ (see Remark~\ref{rmk_lin_op_not_affected}),
where $a:= \lim_{\uptau \to \tau_f} \xi(\uptau) - \uptau \ell(\uptau)$ is a fixed constant. Therefore, setting that pullback of $\ebar_m$ as $\varphi_m^\stat$ ($m = \mu, 1, 2, 3, 4$), to be all the admissible eigenfunctions of $\calP_h|_{\ell \equiv \ell(\tau_f)}$. Therefore, $\varphi_m^\stat$ is a nice approximation of the eigenfunctions of $\calP_h$ in the sense that the error has $\uptau$-decay: \begin{equation}\label{eqn:calP_0_e_ftn_nice_decay}
    \calP_h(\varphi_j^\stat) = \calO(\uptau \dot\wp \brk{\uprho}^{-3}), \quad j = 1, 2, 3, 4, \qquad \calP_h(\varphi_\mu^\stat) = \mu^2 \varphi_\mu^\stat +  \calO(\uptau \dot\wp \brk{\uprho}^{-3}).
\end{equation}
See Remark~\ref{rmk:reason_chosen_cutoff_and_eftn} why this selection is necessary and we could not simply use $\ebar_j$'s.

In the following, we follow the convention that a subscript $m$ (or $l$) and $j$ shall vary in $\{\mu, 1, 2, 3, 4\}$ and $\{1, 2, 3, 4\}$, respectively.
Let \begin{equation}\label{eqn:proxy_eftn}
    W_j:= \chi_{< \frac12 R_f} \varphi_j^\stat, \quad 
    Y_j:= \chi_{< \tau + R_f} \varphi_j^\stat, \quad 
    W_\mu = Y_\mu:= \chi_{< \tau + R_f} \varphi_\mu^\stat.
\end{equation}
Here we mention that the cutoff for $W_j$ is chosen to be the same as the $\chi$ in \eqref{eqn:defn_of_Z_i_s}.
See Remark~\ref{rmk:reason_chosen_cutoff_and_eftn} for a detailed discussion for the motivation of these choices.

One may expect the natural orthogonality condition $\brk{p, W_m} = 0$ to hold. To this end, we start with \[
    \left.\int_{\Sigma_\tau} p W_m \sqrt{|h|} \right|_{\tau_1}^{\tau_2}
    = \int_{\tau_1}^{\tau_2} \int_{\Sigma_\tau} \p_\tau (p W_m \sqrt{|h|}) \, \ud \uptheta\, \ud \uprho\, \ud \uptau
\]
and we use divergence theorem to write this in a more symmetric way \begin{equation}\label{eqn:W_m_int_eqn}
    \begin{aligned}
    0 = &\int_{\Sigma_\tau} \p_\tau (p W_m \sqrt{|h|}) \, \ud \uptheta\, \ud \uprho
    = \int_{\Sigma_\tau} \p_\nu \left(\frac{h^{0\nu}}{h^{00}}p W_m \sqrt{|h|} \right)\, \ud \uptheta\, \ud \uprho \\
    =& \int_{\Sigma_\tau} \frac1{\sqrt{|h|}} \p_\nu\left(\frac{h^{0\nu}\sqrt{|h|}}{h^{00}}\right) p W_m \sqrt{|h|} \, \ud \uptheta\, \ud \uprho
    + \int_{\Sigma_\tau}  \frac{h^{0\nu}}{h^{00}} (W_m \p_\nu p + p \p_\nu W_m)  \sqrt{|h|} \, \ud \uptheta\, \ud \uprho.
\end{aligned}
\end{equation}
Next, pairing \eqref{eqn:calP_h_eqn_for_mod_profile} with $Y_m$: \begin{equation}\label{eqn:mod_profile_test_w_Ym}
    \int_{\tau_1}^{\tau_2} \int_{\Sigma_\tau} Y_m \calP_h p \sqrt{|h|}\, \ud \uptheta\, \ud \uprho\, \ud \uptau 
    = \int_{\tau_1}^{\tau_2} \int_{\Sigma_\tau} Y_m (f_c + g_1 + \p_\tau g_2) \sqrt{|h|}\, \ud \uptheta\, \ud \uprho\, \ud \uptau,
\end{equation}
where the right hand side can be written as \[
    \int_{\tau_1}^{\tau_2} \int_{\Sigma_\tau} (Y_m g_c + Y_m g_1 - \frac{\p_\tau(Y_m \sqrt{|h|})}{\sqrt{|h|}} g_2) \sqrt{|h|}\, \ud \uptheta\, \ud \uprho\, \ud \uptau 
    + \left.\int_{\Sigma_\tau} Y_m g_2 \sqrt{|h|}\, \ud \uptheta\, \ud \uprho \right|_{\tau_1}^{\tau_2}.
\]
On the other hand, integrating by parts on the left hand side of \eqref{eqn:mod_profile_test_w_Ym} and then equating them lead to 

\begin{equation}\label{eqn:Y_m_int_eqn}
    \begin{aligned}
        &\left.\int_{\Sigma_\tau} 
    \left(Y_m g_2 + h^{0\nu} \left(p \p_\nu Y_m - Y_m \p_\nu p\right)\right) \sqrt{|h|} \, \ud \uptheta\, \ud \uprho \right|_{\tau_1}^{\tau_2} \\
    =& \int_{\tau_1}^{\tau_2} \int_{\Sigma_\tau} \left(p \calP_h Y_m - Y_m g_c - Y_m g_1 + \frac{\p_\tau(Y_m \sqrt{|h|})}{\sqrt{|h|}}g_2\right) \sqrt{|h|} \, \ud \uptheta\, \ud \uprho\, \ud \uptau.
    \end{aligned}
\end{equation}

Motivated by \eqref{eqn:W_m_int_eqn} and \eqref{eqn:Y_m_int_eqn}, it is natural to think about the following identities as constraint equations for $g_1$ and $g_2$:  
\begin{align}
&\int_{\Sigma_\tau} \left(p \calP_h Y_m - Y_m g_c - Y_m g_1 + \frac{\p_\tau(Y_m \sqrt{|h|})}{\sqrt{|h|}}g_2\right) \sqrt{|h|} \, \ud \uptheta\, \ud \uprho = 0, \label{eqn:constraint_g12_1}\\
&\int_{\Sigma_\tau} \left(\frac1{\sqrt{|h|}} \p_\nu\left(\frac{h^{0\nu}\sqrt{|h|}}{h^{00}}\right) p W_m + h^{0\nu} \left((\frac{W_m}{h^{00}} + Y_m) \p_\nu p + p (\frac{\p_\nu W_m}{h^{00}} - \p_\nu Y_m)\right)   \right)\sqrt{|h|} \, \ud \uptheta\, \ud \uprho \nonumber \\& \qquad \qquad \qquad \qquad \qquad \qquad \qquad \qquad \qquad \qquad \qquad \qquad \qquad \qquad = \int_{\Sigma_\tau} Y_m g_2 \sqrt{|h|} \, \ud \uptheta\, \ud \uprho     \label{eqn:constraint_g12_2}
\end{align}
for $m = \mu, 1, 2, 3, 4$. Here, we remark that $h^{00} < 0$.
To avoid loss of derivatives, we modify \eqref{eqn:constraint_g12_2} as \begin{equation}\label{eqn:constraint_g12_2_Sp}
    \begin{aligned}
        &S_p\int_{\Sigma_\tau} \left(\frac1{\sqrt{|h|}} \p_\nu\left(\frac{h^{0\nu}\sqrt{|h|}}{h^{00}}\right) p W_m + h^{0\nu} \left((\frac{W_m}{h^{00}} + Y_m) \p_\nu p + p (\frac{\p_\nu W_m}{h^{00}} - \p_\nu Y_m)\right)   \right)\sqrt{|h|} \, \ud \uptheta\, \ud \uprho \\& \qquad \qquad \qquad \qquad \qquad \qquad \qquad \qquad \qquad \qquad \qquad \qquad \qquad \qquad \qquad  = \int_{\Sigma_\tau} Y_m g_2 \sqrt{|h|} \, \ud \uptheta\, \ud \uprho,
    \end{aligned}
\end{equation}
where $S_p$ is a smoothing operator with kernel $k_p$ \[
    (S_p h)(t):= \int_\R \chi_{[-1, \infty)}(s) h(s) k_p(t-s)\, \ud s, \quad t \geq -1, \quad \forall h\in L^1_\loc([-1, \infty))
\]
satisfying the moment condition $\int_0^1 s k_p(s)\, \ud s = 0$. See \eqref{eqn:defn_of_op_S} for comparison and we define $\tilde S_p$ correspondingly. The reason why this $S_p$ is necessary is explained in Remark~\ref{rmk_smoothing_S_p}.
By requiring \eqref{eqn:constraint_g12_1} and \eqref{eqn:constraint_g12_2_Sp}, we obtain a modified orthogonality condition for $p$: \begin{equation}\label{eqn:mod_ortho_for_p}
\begin{aligned}
    \int_{\Sigma_\tau} p W_m \sqrt{|h|}\, \ud \uptheta\, \ud \uprho = -&\tilde{\tilde S}_p \p_\tau \int_{\Sigma_\tau} \left(\frac1{\sqrt{|h|}} \p_\nu\left(\frac{h^{0\nu}\sqrt{|h|}}{h^{00}}\right) p W_m \right.\\
    & \qquad \qquad \left. + h^{0\nu} \left((\frac{W_m}{h^{00}} + Y_m) \p_\nu p + p (\frac{\p_\nu W_m}{h^{00}} - \p_\nu Y_m)\right)   \right)\sqrt{|h|} \, \ud \uptheta\, \ud \uprho,
\end{aligned}
\end{equation}
where $\tilde{\tilde S}_p$ is a convolution operator with kernel $\tilde{\tilde k}_p:= \int_0^r \tilde k_p(s)$. It then satisfies the relation $\frac{\ud}{\ud t}\tilde{\tilde S}_p = \tilde S_p$. The moment condition ensures the appearance of $\p_\tau$ on the right hand side of \eqref{eqn:mod_ortho_for_p}.

\begin{remark}
    In contrast with the orthogonality conditions in \eqref{eqn:ortho_cond_for_vec_bfOmg} and \eqref{eqn:ortho_cond_on_phi_for_Z_mu}, the orthogonality conditions for $p$ is derived in the second order formulation. This is because our choices in \eqref{eqn:proxy_eftn} are not compactly supported.
\end{remark}

We further require that $g_1, g_2$ have the following decomposition \begin{equation}\label{eqn:decomp_of_g_i}
    g_i = \sum_{m = \mu,1,2,3,4} c_{im}(\uptau) X_m, \quad X_m:= \chi(\uprho) \varphi_m^{\stat}(\uprho), \quad i = 1, 2,
\end{equation}
where $\chi$ is a cutoff to the region $\{ |\uprho| \leq R_f/3 \}$.
Set $A_{ml} =: \brk{X_m, Y_l}$, then the entry of the inverse matrix $|(A^{-1})^{ml}|$ is uniformly bounded from above in view of $\brk{\nu^i, \nu^j} = \delta_{ij}$ and the smallness assumptions of $\wp$ and $R_f^{-1}$. 
Therefore, the coefficients can be represented by \begin{equation}\label{eqn:coeff_c_im}
    c_{im} = \sum_{l = \mu, 1, 2,3,4,} (A^{-1})^{lm} \brk{g_i, Y_l}, \quad i = 1, 2.
\end{equation}
These coefficients, including those obtained after applying certain number of $\p_\uptau$ to them, satisfy decay estimates that follow from the decay properties of $p$, which will be introduced in \eqref{BA-p}, \eqref{BA-dtau-p} and \eqref{BA-p-LE}. The derivation of the decay estimates will be done in Section~\ref{Section:parameter_control}.

\begin{remark}
    The existence of $p$ with zero Cauchy data satisfying \eqref{eqn:calP_h_eqn_for_mod_profile} with the existence of $g_1$ and $g_2$ satisfying \eqref{eqn:decomp_of_g_i} was shown in \cite[Proposition 4.1]{OS24}.
\end{remark}

\subsection{Modulation equations} \label{Section:modulation_equations}
In this subsection, we aim to introduce the governing equations of $\xi$ and $\ell$ by imposing orthogonality condition \eqref{eqn:ortho_cond_for_vec_bfOmg}. 

\subsubsection{Introducing the symplectic form $\bfOmg$ and related quantities}
We define the symplectic form $\bfOmg$ \[
    \bfOmg(\vec u, \vec v):= \brk{\vec u, J\vec v}, \quad \vec u = \begin{pmatrix}
        u_1 \\ u_2
    \end{pmatrix}, \vec v = \begin{pmatrix}
        v_1 \\ v_2 
    \end{pmatrix}, \quad J = \begin{pmatrix}
        0 & 1 \\ -1 & 0
    \end{pmatrix}
\]
via an inner product \begin{equation}\label{eqn:inner_product_u_v}
    \brk{\vec u, \vec v}:= \int u_1 v_1 + u_2 v_2 \, \ud \omg\, \ud \rho.
\end{equation}
We note that the reason why we do not use the global coordinates and not include the volume form in this inner product is that at least one of $\vec u$, $\vec v$ is supported in the flat region in our applications later.

For any matrix operator $P = J$ or $M$, we shall use $P^*$ to denote the operator adjoint of $P$ with respect to the inner product \eqref{eqn:inner_product_u_v} (instead of its adjoint matrix). Then the adjoint $M^*$ of $M$ (see \eqref{eqn:matrix_operator_M}) is given by \[
    M^* \vec v = \begin{pmatrix}
        - \p_j\big(\frac{h^{0j}}{h^{00}} v_1\big) & -\sqrt{|h|} L_1 v_2 \\
        \frac1{\sqrt{|h|}h^{00}} v_1 & \frac{h^{j0}}{h^{00}} \p_j v_2
    \end{pmatrix}.
\]
It follows from $J^* = - J$ and $JM + M^* J = 0$ that $JM = (JM)^*$ and hence \begin{equation}\label{eqn:s.a._M_wrt_bfOmg}
    \bfOmg(\vec u, M\vec v) = \bfOmg(M\vec u, \vec v).
\end{equation}

Now we define a truncated version of the (generalized) eigenfunctions of $M$. To this end, we record the results in \cite[Section 3.5, 3.6]{LOS22} as follows. 
\begin{lemma}\label{lemma_zero_modes}
    For $\ell = \ell(t) \in \R^n$ with $|\ell| < 1$, we define the kernel, resp. generalized kernel of $M$: \[
        \vec\varphi_i = \begin{pmatrix}
            \varphi_i \\ \dot\varphi_i
        \end{pmatrix}, \quad \vec\varphi_{n+i}:= \begin{pmatrix}
            \varphi_{n+i} \\ \dot\varphi_{n+i}
        \end{pmatrix}, \quad 1 \leq i \leq n,
    \]
    in the sense that $M\vec\varphi_i = 0$, $M \vec\varphi_{n+i} = \vec \varphi_{i}$ for $1 \leq i \leq n$. The explicit forms of $\vec\varphi_i, \vec\varphi_{n+i}$ are given by \[
        \begin{aligned}
    \varphi_i &:= |\ell|^{-2}(\gamma-1)(\ell \cdot \nu) \ell^i + \nu^i = \nu^i + \calO(|\ell|^2), \\
    \dot{\varphi}_i &:= \sqrt{|h|}h^{0 j} \partial_j \varphi_i = \calO(|\ell|), \\
    \varphi_{n+i} &:= -\gamma(\ell \cdot F) \nu^i-|\ell|^{-2} \gamma(\gamma-1)(\ell \cdot F)(\ell \cdot \nu) \ell^i = \calO(|\ell|), \\
    \dot{\varphi}_{n+i} &:= \sqrt{|h|}h^{0 j} \partial_j \varphi_{n+i} + \sqrt{|h|}h^{00} \varphi_i = \sqrt{|h|}h^{00} \varphi_i + \calO(|\ell|).
\end{aligned}
    \]
\end{lemma}

We next introduce truncated versions \begin{equation}\label{eqn:defn_of_Z_i_s}
    \vec Z_i = \chi \vec\varphi_i, \quad \vec Z_{n+i} = \chi \vec\varphi_{n+i}, \quad i = 1, \ldots, n,
\end{equation}
where $\chi$ is a smooth cut-off with support in $|\rho| < \frac12 R_f$ and being $1$ in $|\rho| < \frac14 R_f$. This is chosen to be the same as $\chi_{< \frac12 R_f}$ for $W_j$'s in \eqref{eqn:proxy_eftn}. See \eqref{eqn:defn_of_4_regions} for the geometric meaning of $R_f$. 

Using the first order formulation of HVMC equation \eqref{eqn:1st_order_HVMC} and taking advantage of \eqref{eqn:s.a._M_wrt_bfOmg}, we find that \begin{equation}\label{eqn:p_t_bfOmg}
    \p_t \big( \bfOmg(\vec\psi, \vec Z_j)\big) - \bfOmg(\vec\psi, M \vec Z_j)= \bfOmg(\vec K, \vec Z_j) + \bfOmg(\vec f, \vec Z_j) + \bfOmg(\vec\psi, \p_t \vec Z_j), \quad j = i \text{ or } n+i, \quad 1 \leq i \leq n.
\end{equation}

We discuss the structure of the three terms on the right hand side of \eqref{eqn:p_t_bfOmg} separately.
In view of \eqref{eqn:parametrize_calQ_in_flatt}, \eqref{eqn:N_in_C_flatn} and \eqref{eqn:vec_K_in_C_flatn}, we find that \begin{equation}\label{eqn:leading_term_of_K}
    K = -(\dot\xi - \ell) \cdot (\nu + \calO(|\ell|^2)) + \dot\ell \cdot \calO(|\ell|), \quad 
    \dot K = - \sqrt{|h|}h^{00}(\dot\ell \cdot \nu + \calO(|\ell|^2)),
\end{equation}
and thus \[\begin{aligned}
    \bfOmg(\vec K, \vec Z_i) &= \sum_j \dot{\ell}_j ( d_{ij} + r_{ij} ) + \sum_j (\dot{\xi}_j - \ell_j) b_{ij}, \\
    \bfOmg(\vec K, \vec Z_{n+i}) &= \sum_j \dot{\ell}_j \tilde b_{ij} - \sum_j (\dot{\xi}_j - \ell_j) ( d_{ij} + \tilde r_{ij})
\end{aligned}
\]
with $r_{ij}$, $\tilde r_{ij}$, $b_{ij}$,  $\tilde b_{ij}$ all of order $\calO(|\ell|)$ and (see \eqref{eqn:1st_order_formulation_h_det_inv}) \[
    d_{ij} = \int \chi \nu^i \nu^j \sqrt{|h|}h^{00}\, \ud \rho\, \ud \omg = C \delta^{ij} + o_{R_f, \ell}(1).
\]
Here we use the fact $\brk{\Theta^i, \Theta^j}_{L^2(\bbS^{n-1})} = \delta^{ij}$, namely the orthogonality of first spherical harmonics.
Therefore, if we use $D$ to denote the $n \times n$ matrix with entries $d_{ij}$, we know that $D$ is invertible provided $R_f^{-1}$ and $|\ell|$ both sufficiently small. 
On the other hand, the second term on the right hand side of \eqref{eqn:p_t_bfOmg} can be expressed as \[
    \calO(\dot\wp, \ell, \psi, \p\psi, \p_\Sigma\p\psi) \dot\wp + \calO((\psi, \p\psi, \p_\Sigma\p\psi)^2).
\] 
while the last term also comes with $\calO(\dot\wp)$ smallness coefficients.
Therefore, we find that \begin{equation}\label{eqn:p_t_vec_bfOmg_prelim}
    \p_t\vec\bfOmg - \vec N = \begin{pmatrix}
        D + R & R \\
        R & -D + R \\
    \end{pmatrix}\dot\wp + \vec H,
\end{equation}
where \begin{equation}\label{eqn:defn_of_vec_bfOmg_N_R_H}
    \begin{aligned}
    \vec \bfOmg:= \big( \bfOmg(\vec\psi, \vec Z_1), \ldots, \bfOmg(\vec\psi, \vec Z_{2n})\big), \quad& \vec N:= \big( \bfOmg(\vec\psi, M\vec Z_1), \ldots, \bfOmg(\vec\psi, M\vec Z_{2n})\big), \\
    R = \calO(\psi, \p\psi, \p_\Sigma\p\psi, \ell, \dot\wp), \quad& \vec H = \calO((\psi, \p\psi, \p_\Sigma\p\psi)^2).
\end{aligned}
\end{equation}

One can solve for $\p_t\psi$ from the first entry of \eqref{eqn:1st_order_HVMC} and an application of implicit function theorem. Thus, by inserting $\p_t \psi$ to the right hand side of \eqref{eqn:p_t_vec_bfOmg_prelim}, we reach the form \begin{equation}\label{eqn:p_t_vec_bfOmg}
    \p_t\vec\bfOmg - \vec N = \begin{pmatrix}
        D + \tilde R & \tilde R \\
        \tilde R & -D + \tilde R \\
    \end{pmatrix}\dot\wp + \vec {\tilde H}:= \vec F,
\end{equation}
where \[
    \tilde R = \calO(\p_{\Sigma}^{\leq 2}\vec\psi, \ell, \dot\wp), \quad \vec{\tilde H} = \calO\big((\p_\Sigma^{\leq 2}\vec \psi)^2, (\ell\p_{\Sigma}^{\leq 2}\vec \psi)^2\big).
\]

With a slight abuse of notation, we denote \[
    \vec F = \vec F(\p_\Sigma^{\leq 2}\vec\psi, \ell, \dot\wp)
\]
by identifying the function $\vec F$ and its upper bound which only depends on tangential derivatives up to the second order $\p_\Sigma^{\leq 2}\vec\psi$ and parameters $\ell$, $\dot\wp$.
We write \[
    \dot\wp = \begin{pmatrix}
        D & 0 \\
        0 & -D \\
    \end{pmatrix}^{-1} \left(\vec F - \vec{\tilde H} - \begin{pmatrix}
        \tilde R & \tilde R \\
        \tilde R & \tilde R \\
    \end{pmatrix} \dot\wp \right).
\]
and a further application of implicit function theorem due to smallness dependence on $\dot\wp$ on the right hand side would give arise to the relation $\dot \wp = \vec G$. With an abuse of notation as well, \begin{equation}\label{eqn:vec_G_eqn_derived_from_1st_order_HVMC_1}
    \vec G = \vec G(\p_{\Sigma}^{\leq 2} \vec\psi, \ell, \vec F), \quad |\vec G| \lesssim |\p_{\Sigma}^{\leq 2}\vec\psi|^2 + |\vec F|
\end{equation}
provided the smallness in the bootstrap assumptions. We should keep in mind that $\vec F = \p_t \vec\bfOmg - \vec N$ thanks to \eqref{eqn:p_t_vec_bfOmg} in the derivation above and hence we arrive at \begin{equation}\label{eqn:vec_G_eqn_derived_from_1st_order_HVMC_2}
    \dot \wp = \vec G(\p_{\Sigma}^{\leq 2} \vec\psi, \ell, \p_t \vec\bfOmg - \vec N).
\end{equation}
To summarize, for small $x, y, w$ and $z$, we have $z = \vec F(x, y, w)$ if and only if $w = \vec G(x, y, z)$ and $\vec G$ satisfies the estimate $|\vec G(x, y, z)| \lesssim |x|^2 + |z|$ uniformly for small $x, y$ and $z$.

Note that all the equations in this section so far are derived from the first order formulation of HVMC equation \eqref{eqn:1st_order_HVMC} satisfied by $\vec \psi$. Now we are prepared to state the governing equations of $\dot\wp$, i.e. modulation equations. 

\subsubsection{A first order formulation for the modified profile}
In order to define the modulation equations, we need to introduce a first order formulation for the modified profile $p$ given in \eqref{eqn:calP_h_eqn_for_mod_profile}.
For $p$ defined above, we set \begin{equation}\label{eqn:defn_vec_p}
    \vec p = \begin{pmatrix}
        p \\ \sqrt{|h|}h^{0\nu}\p_\nu p
    \end{pmatrix}.
\end{equation}
Note that although $p$ is well-defined in the whole region, we will only use the information in the flat region like in Section~\ref{Section:1st_order_formulation}.
A direct computation using the matrix $M$ (see \eqref{eqn:matrix_operator_M}) reveals that \begin{equation}\label{eqn:1st_order_eqn_for_p}
    (\p_t - M) \vec p = \begin{pmatrix}
        0 \\ \p_\nu \big( \sqrt{|h|} h^{\mu\nu} \p_\mu p\big) + \sqrt{|h|}|\II|^2 p
    \end{pmatrix}
    = \begin{pmatrix}
        0 \\ \sqrt{|h|}(f_c + g_1 + \p_\tau g_2)
    \end{pmatrix}.
\end{equation}
Like in \eqref{eqn:p_t_bfOmg}, we write \begin{equation}\label{eqn:p_t_Omg_p_minus_N_p}
    \p_t \big(\bfOmg(\vec p, \vec Z_j)\big) - \bfOmg(\vec p, M \vec Z_j) = \bfOmg((\p_t - M)\vec p, \vec Z_j) + \bfOmg(\vec p, \p_t \vec Z_j).
\end{equation}
Naturally, $\vec \bfOmg_p$, $\vec N_p$, $\vec \bfOmg_q$ and $\vec N_q$ could be defined analogously to $\vec N$ in \eqref{eqn:defn_of_vec_bfOmg_N_R_H} with $\vec \psi$ replaced by $\vec p$ and $\vec q:= \vec \psi - \vec p$, respectively. 
To avoid ambiguities, we also record the trivial relations $\vec N = \vec N_p + \vec N_q$ and $\vec \bfOmg = \vec \bfOmg_p + \vec \bfOmg_q$.
The following formula of the second component of $\vec q:= \vec \psi - \vec p$ is taken down for future use as well. From \eqref{eqn:defn_vec_p} and \eqref{eqn:1st_entry_of_HVMC_1st_order_eqn}, we reach that \begin{equation}\label{eqn:2nd_entry_of_vec_q}
    \dot q = \sqrt{|h|}h^{0\nu} \p_\nu q - \sqrt{|h|}h^{00}(K+f).
\end{equation}

\begin{remark}\label{rmk_on_construction_of_p}
    At this stage, we leave a heuristic comment about closing the bootstrap. Note that $\p_t \vec \bfOmg_p - \vec N_p$ would have the same (or better) $\tau$-decay as in $\dot\wp$ as well as the extra $\delta_\wp$-smallness thanks to \eqref{eqn:all_decay_of_c_im} and  the cutoff introduced in $f_c$. Momentarily, we will see in \eqref{eqn:modulation_eqn} that $\p_t \vec \bfOmg_p - \vec N_p$ enters the parameter derivatives. Therefore, both decay and extra smallness here are equally important in our construction for this modified profile $p$.
\end{remark}

\subsubsection{Modulations equations and orthogonality conditions}
We aim to impose orthogonality conditions so that the following modulation equation \eqref{eqn:modulation_eqn} is satisfied by the modulated parameters $\xi$ and $\ell$: \begin{equation}\label{eqn:modulation_eqn}
    \dot\wp = \vec G\big(S \p_{\Sigma}^{\leq 2}\vec\psi, \ell, S(\p_t \vec \bfOmg_p - \vec N_p - \vec N_q) - \beta \vec \omg\big),
\end{equation}
where the smoothing operator $S$ (see \eqref{eqn:defn_of_op_S}) and the damping term $-\beta\vec\omg$ (see \eqref{eqn:defn_of_vec_Upomg_omg_2}) will be defined shortly. The motivation could be explained as follows.
Compared with \eqref{eqn:vec_G_eqn_derived_from_1st_order_HVMC_2}, $S$ is introduced to avoid the potential loss of derivatives caused by the presence of $\p^2_{\Sigma}\psi$ in \eqref{eqn:vec_G_eqn_derived_from_1st_order_HVMC_2}. Moreover, compared with the modulation equations in \cite{LOS22}, we keep the form $\p_t \vecbfOmg_p - \vec N_p$ in the third entry but only change the part relative to $\vec q$ instead of modifying for the whole quantity $\vec \psi = \vec p + \vec q$. This is due to the slow decaying rate of $\vec p$ itself while $\p_t \vecbfOmg_p - \vec N_p$ as a whole satisfies nice decay property in $\tau$. If one ignores the damping term and the operator $S$, it is easy to see that the orthogonality condition prompting \eqref{eqn:modulation_eqn} is simply $\vec\bfOmg_q = 0$, appearing to be virtually the same as those in the classical modulation theory (see for instance \cite{Wein85, Stu01}). 

Now we derive an equivalent form of \eqref{eqn:modulation_eqn}. First, plugging \eqref{eqn:modulation_eqn} into \eqref{eqn:p_t_vec_bfOmg}, we obtain \[
    \p_t \vec\bfOmg - \vec N = \vec F\left(\p_{\Sigma}^2 \vec\psi, \ell, \vec G\big(S \p_{\Sigma}^{\leq 2}\vec\psi, \ell, S(\p_t \vec \bfOmg_p - \vec N_p - \vec N_q) - \beta \vec \omg\big)\right)
\]
Using that $z = \vec F(0, \ell, \vec G(0, \ell, z))$, it is further equivalent to \begin{equation}\label{eqn:equiv_form_1_of_mod_eqn}
    \p_t \vec\bfOmg - \vec N = S(\p_t \vec \bfOmg_p - \vec N_p - \vec N_q) - \beta \vec \omg - \vec F_\omg, 
\end{equation}
where \[
    \vec F_\omg:= \vec F\left(0, \ell, \vec G\big(0, \ell, S(\p_t \vec \bfOmg_p - \vec N) - \beta \vec \omg\big) \right) - \vec F\left(\p_{\Sigma}^2 \vec\psi, \ell, \vec G\big(S \p_{\Sigma}^{\leq 2}\vec\psi, \ell, S(\p_t \vec \bfOmg_p - \vec N) - \beta \vec \omg\big)\right).
\]
Rewriting \eqref{eqn:equiv_form_1_of_mod_eqn}, we achieve \[
    \p_t \vec\bfOmg_q = (S - I)(\p_t \vec \bfOmg_p - \vec N_p - \vec N_q) - \beta \vec \omg - \vec F_\omg.
\]

We impose the following decomposition \begin{equation}\label{eqn:ortho_cond_for_vec_bfOmg}
    \vec\bfOmg_q = \vec\Upomg + \vec\omg,
\end{equation}
where \begin{align}
        \p_t \vec\Upomg &= (S - I)(\p_t \vec \bfOmg_p - \vec N_p - \vec N_q + \vec F_\omg), \label{eqn:defn_of_vec_Upomg_omg_1}\\
        \p_t\vec\omg &= -\beta\vec\omg -S\vec F_\omg. \label{eqn:defn_of_vec_Upomg_omg_2}
    \end{align}
We first view \eqref{eqn:defn_of_vec_Upomg_omg_2} as the definition of $\vec\omg$, then viewing $\vec\omg$ as known, \eqref{eqn:defn_of_vec_Upomg_omg_1} can be viewed as the definition of $\vec\Upomg$.
To be specific, we write \eqref{eqn:defn_of_vec_Upomg_omg_2} into its integral form \begin{equation}\label{eqn:defn_of_vec_Upomg_omg_2_int_form}
    \vec\omg(t) = \int_0^t e^{-\beta(t - s)} (S\vec F_\omg)(s)\, \ud s.
\end{equation}
Then the existence of $\vec\omg \in C^\infty([0,\infty))$ follows from a standard Picard iteration argument once (by choosing $\beta < 1/ (2 + C)$) we notice that \[\begin{aligned}
    |\vec F_{\omg_j} - \vec F_{\omg_{j-1}}|
    \leq &\beta|\vec\omg_j - \vec\omg_{j-1}| + |\vec H_{\omg_j} - \vec H_{\omg_{j-1}}| \leq \beta|\vec\omg_j - \vec\omg_{j-1}| + C\beta |\vec \omg_j - \vec \omg_{j-1}|, \\
    &\vec H_{\omg}:= \vec F\left(\p_{\Sigma}^2 \vec\psi, \ell, \vec G\big(S \p_{\Sigma}^{\leq 2}\vec\psi, \ell, S(\p_t \vec \bfOmg_p - \vec N) - \beta \vec \omg\big)\right).
\end{aligned}
\] 

Then the decomposition \eqref{eqn:ortho_cond_for_vec_bfOmg} is the so-called orthogonality condition so that the modulation equation \eqref{eqn:modulation_eqn} holds. The terminology orthogonality condition shall be clear once we prove Lemma~\ref{lemma_decay_of_bfOmg_phi}, indicating the smallness of $\vec\bfOmg_q$.

Now we introduce the smoothing operator (in time) $S$. Let $k \in C^\infty_c(\bbR)$ be a nonnegative function supported in $[0, 1]$ such that $\int_0^1 k(s)\, \ud s = 1$. We define the operator $S$ as \begin{equation}\label{eqn:defn_of_op_S}
    (Sh)(t):= \int_\R \chi_{[-1, \infty)}(s) h(s) k(t-s)\, \ud s, \quad t \geq -1, \quad \forall h \in L^1_\loc([-1, \infty)),
\end{equation}
where $\chi_{[-1, \infty)}$ is a cutoff to the interval $[-1, \infty)$ taking constant value $1$ in $[0, \infty)$. In lieu of choosing $k$ to be the convolution kernel directly, the cutoff $\chi_{[-1, \infty)}$ is inserted as we only care about forward solutions. This gives rise to the property that $(Sh)(t)$ is smooth for all $t \geq 0$. 
We also define an auxiliary operator $\widetilde S$ such that \begin{equation}\label{eqn:relation_widetilde_S_w_S}
    (S - I)h = \frac{\ud}{\ud t}(\widetilde S h),
\end{equation} 
which will provide a succinct way to express \eqref{eqn:defn_of_vec_Upomg_omg_1}.
To this end, we set \begin{equation}\label{eqn:defn_of_widetil_S_op}
    \tilde k(r):= \begin{cases}
        0 & r < 0 \\
        -\int_r^\infty k(s)\, \ud s & r \geq 0
    \end{cases}, \quad 
    (\widetilde S h)(t):= \int_\R \chi_{[-1, \infty)}(s) h(s) \tilde k(t-s)\, \ud s, \quad t \geq -1.
\end{equation} 
Note that $\tilde k(r)$ is also supported in the interval $[0, 1]$.

We then rewrite \eqref{eqn:defn_of_vec_Upomg_omg_1} as \begin{equation}\label{eqn:defn_of_vec_Upomg_omg_1_int_form}
    \vec\Upomg:= \widetilde S(\p_t \vec \bfOmg_p - \vec N_p - \vec N_q + \vec F_\omg).
\end{equation}
Later in Section~\ref{Section:parameter_control}, we will work with \eqref{eqn:defn_of_vec_Upomg_omg_2_int_form} and \eqref{eqn:defn_of_vec_Upomg_omg_1_int_form} to perform estimates for both $\vec\omg$ and $\vec\Upomg$.

\begin{remark}
    Note that though $S$ and $\widetilde S$ are nonlocal operators, both $S$ and $\widetilde S$ are chosen to be almost local in time. This ensures that $Sh$ and $h$ have comparable decay properties. Specifically, if $|h| \lesssim \brk{t}^{-\gmm}$, then $|Sh|, |\widetilde S h| \lesssim \brk{t}^{-\gmm}$.

    Moreover, one might have noticed that we sometimes abuse notation by applying $S$ or $\widetilde S$ onto some functions only defined on $[0, \infty)$. For any $h$ with domain in $[0, \infty)$, this makes sense in the following way: we first extend it trivially such that $h \equiv 0$ on $[-1, 0)$ and then apply the operator $S$ or $\widetilde S$.
\end{remark}

\subsubsection{The existence of the decomposition \eqref{eqn:ortho_cond_for_vec_bfOmg}}\label{Section:Existence_of_decomp_Omg}
We define a map \[
    \vec\Upsilon(\Phi, \xi, \ell):= \vec\bfOmg_q - \vec\Upomg - \vec\omg.
\]
Thanks to \eqref{eqn:defn_of_vec_Upomg_omg_1_int_form}, this is equivalent to \begin{equation}\label{eqn:vec_Upsilon_defn}
    \vec\Upsilon(\Phi, \xi, \ell) = \vec\bfOmg - \vec\bfOmg_p - (S - I)\vec\bfOmg_p - \widetilde{S}(- \vec N + \vec F_\omg) - \vec\omg
    = \vec\bfOmg + \widetilde{S}(\vec N - \vec F_\omg) - \vec \omg - S\vec\bfOmg_p.
\end{equation}
Since $\vec\Upsilon(\Psi_0,0,0) = 0$ and $\vec\Upsilon: C^5 \times C^2_\uptau \times C^2_\uptau \to C^2_\uptau$ (see \cite[Section 5.1]{OS24}), the Fr\'{e}chet derivative $D_{\xi, \ell}(\Psi_0, 0, 0) \in L(C^2 \to C^2)$ can be viewed a matrix, where $\Psi_0 = (t, F)$. The invertibility of this matrix will imply the existence thanks to the  implicit function theorem (see for instance \cite{Pat19}). Note that the Fr\'{e}chet derivative of the first four terms in \eqref{eqn:vec_Upsilon_defn} can be computed identically as in \cite[Section 3.6]{LOS22}. Here, the time interval of interest is $[0, \tau_0]$ for $\tau_0 \leq \tau_f$ but we need to mention that we still need to make an extension of all the functions to $[-1, 0]$ in order to make nonlocal operators well-defined. See Remark~\ref{rmk_on_extension_in_t} for a further discussion.

We write \[
    \frac{\delta \bfOmg(\vec p, \vec Z_j)}{\delta \wp} = \bfOmg\big(\frac{\delta \vec p}{\delta \wp}, \vec Z_j\big) + \bfOmg\big(\vec p, \frac{\delta \vec Z_j}{\delta \wp}\big), \quad j = 1, 2, \cdots, 2n, \quad \wp = \xi \text{ or } \ell.
\]
Recall that $p|_{\Phi = \Psi_0} = 0$, therefore, after restricting to $(\Phi, \xi, \ell) = (\Psi_0, 0, 0)$, we obtain \begin{equation}\label{eqn:Frechet_der_of_p}
    \left.\frac{\delta \bfOmg(\vec p, \vec Z_j)}{\delta \wp}\right|_{(\Phi, \xi, \ell) = (\Psi_0, 0, 0)} = 
    \begin{cases}
        \int_{\Sigma_\uptau} \chi\nu^i \big(\p_0\dfrac{\delta p}{\delta \wp}\big) \sqrt{|h|}|_{\ell = 0}, \quad 1 \leq j \leq n, \quad j = i, \\ 
        -\int_{\Sigma_\uptau} \chi\nu^i \dfrac{\delta p}{\delta \wp} \sqrt{|h|}|_{\ell = 0}, \quad n+1 \leq j \leq 2n, \quad j = i + n, \\ 
    \end{cases}
\end{equation}
where we used \eqref{eqn:metric_h_in_Cflat}, \eqref{eqn:defn_vec_p} and Lemma~\ref{lemma_zero_modes}.
We claim that the right hand side of \eqref{eqn:Frechet_der_of_p} is zero.
To see this, one first notices that \[
W_i|_{(\Phi, \xi, \ell) = (\Psi_0, 0, 0)} = Z_i|_{(\Phi, \xi, \ell) = (\Psi_0, 0, 0)} = \dot Z_{4+i}|_{(\Phi, \xi, \ell) = (\Psi_0, 0, 0)} = \chi \nu^i,\]
by taking $\frac{\delta}{\delta\wp}$ on both sides of \eqref{eqn:mod_ortho_for_p}, we obtain \begin{equation}\label{eqn:taking_Fre_der_in_ortho_cond}
    \left|\int_{\Sigma_\uptau} \chi\nu^i \dfrac{\delta p}{\delta \wp} \sqrt{|h|}|_{\ell = 0}\right| = O(R_f^{-1}) \left|\tilde S_p\int_{\Sigma_\uptau} \chi\nu^i \dfrac{\delta (p + \p_0 p)}{\delta \wp} \sqrt{|h|}|_{\ell = 0}\right|.
\end{equation}
We further notice that $\tilde S_p \p_0 p = (S - I)p$ and hence it shows that the left hand side of \eqref{eqn:taking_Fre_der_in_ortho_cond} is zero. By taking $\p_0$ first in \eqref{eqn:mod_ortho_for_p} and then applying $\frac{\delta}{\delta\wp}$, combining with the preceding conclusion, it proves the claim. This concludes the proof of existence for $\xi$ and $\ell$.

\subsection{Controlling the unstable mode}\label{subsec:unstable-mode}
In the preceding subsection, we have settled the issue caused by the presence of zero modes of the stability operator $L$. Now we are ready to take care of the non-decaying solution produced by the positive eigenvalue of $L$.

Starting from the first order formulation \eqref{eqn:1st_order_HVMC}, combining with \eqref{eqn:leading_term_of_K}, \eqref{eqn:expression_of_f_in_1st_order_form}, together with a similar application of implicit function theorem in the derivation of \eqref{eqn:p_t_vec_bfOmg}, we obtain \[
    (\p_t - M) \vec \psi = \vec F_1(\p_{\Sigma}^{\leq 2}\vec\psi, \ell, \dot\wp),
\]
where the coefficients of $\p_{\Sigma}^{\leq 2}\vec\psi$ come with $\calO(\dot\wp, \p_{\Sigma}^{\leq 2}\vec\psi)$ factor. 
In view of \eqref{eqn:1st_order_eqn_for_p}, we obtain \begin{equation}\label{eqn:vec_F_2_defn}
    (\p_t - M) \vec q = \vec F_2(\p_{\Sigma}^{\leq 2}\vec \psi, \ell, \dot\wp, f_c + g_1 + \p_\tau g_2):= \vec F_1 - (0, f_c + g_1 + \p_\uptau g_2)^T,
\end{equation}
where the smallness in $\vec F_1$ mentioned above still holds for $\vec F_2$.

Recall that for the stability operator $L$ (defined in \eqref{eqn:stab_op_on_C}) has only $1$ positive eigenvalue $\mu^2$ (we use the convention $\mu > 0$). Denote the unique corresponding eigenfunction by $\varphi_
\mu$, then \[
    L \varphi_\mu = \mu^2 \varphi_\mu.
\]
This leads to the matrix form \begin{equation}\label{eqn:eftn_wrt_mu}
    \begin{pmatrix}
        0 & -\frac1{\sqrt{|h|}}|_{\ell = 0} \\
        - \sqrt{|h|}|_{\ell = 0} L & 0 \\
    \end{pmatrix}
    \begin{pmatrix}
        \varphi_\mu \\ \pm \mu \sqrt{|h|}|_{\ell = 0}\, \varphi_\mu
    \end{pmatrix}
    = \mp \mu \begin{pmatrix}
        \varphi_\mu \\ \pm \mu \sqrt{|h|}|_{\ell = 0}\, \varphi_\mu
    \end{pmatrix},
\end{equation}
where we know from \eqref{eqn:metric_h_in_Cflat} that $h|_{\ell = 0}$ is simply the metric on $\calC$.  
\begin{remark}
    If one would like to compare this with almost zero modes $\vec \varphi_1, \ldots, \vec \varphi_n$ in Lemma~\ref{lemma_zero_modes}, one could view the eigenfunctions \eqref{eqn:eftn_wrt_mu} here as \[
    \begin{pmatrix}
        e^{\pm\mu t} \varphi_\mu \\
        \p_t(e^{\pm\mu t} \varphi_\mu)
    \end{pmatrix}
    \]
    with modifying factor $\sqrt{|h|}|_{\ell = 0}$ due to the factor in the matrix operator.
\end{remark}

Motivated by \eqref{eqn:eftn_wrt_mu}, we define the time-independent almost eigenfunctions of the operator $M$ as \begin{equation}\label{eqn:defn_Z_pm}
    \vec Z_{\pm}:= c_\pm \begin{pmatrix}
        \chi \varphi_\mu \\ \mp \mu \sqrt{|h|}|_{\ell = 0} \chi \varphi_\mu 
    \end{pmatrix}
    = c_\pm \begin{pmatrix}
        \chi \varphi_\mu \\ \pm \mu (\sqrt{|h|} h^{00})|_{\ell = 0} \chi \varphi_\mu 
    \end{pmatrix},
\end{equation}
where $\chi$ is the same cutoff as the one used in \eqref{eqn:defn_of_Z_i_s} and $c_\pm$ are normalization constants such that $\bfOmg(\vec Z_+, \vec Z_-) = 1$. In light of \eqref{eqn:eftn_wrt_mu}, \begin{equation}\label{eqn:M_Z_pm}
    M \vec Z_\pm = \pm \mu \vec Z_\pm + \vec\calE_\pm,
\end{equation}
where $\vec\calE_\pm$ consist of terms that are supported in $|\uprho| \simeq R_f$ or that have additional smallness in terms of $\ell$.

\subsubsection{Orthogonality conditions \eqref{eqn:ortho_cond_on_phi_for_Z_mu} and some corollaries}
Now we aim to enact a decomposition for $\vec q$ into the form \begin{equation}\label{eqn:q_decomp}
    \vec q = \vec \phi + a_+(t) \vec Z_+ + a_-(t) \vec Z_-,
\end{equation}
where the time-dependent parameters $a_+(t)$ and $a_-(t)$ will be defined shortly by imposing suitable orthogonality conditions. More precisely, we will first motivate and give the orthogonality condition \eqref{eqn:ortho_cond_on_phi_for_Z_mu} assuming a general decomposition \eqref{eqn:q_decomp}. Then assuming \eqref{eqn:ortho_cond_on_phi_for_Z_mu} , it implies from the implicit function theorem that $a_\pm(t)$ can be uniquely determined such that $\vec \phi := \vec q - a_+\vec Z_+ - a_- \vec Z_-$, which will be done in  Section~\ref{subsubsec:define-apm}.

A direct computation shows that \[
    (a_+' - \mu a_+) \vec Z_+ + (a_-' + \mu a_-)\vec Z_- + (\p_t - M)\vec \phi = a_+\vec\calE_+ + a_-\vec\calE_- + \vec F_2,
\]
where $\vec F_2$ is defined in \eqref{eqn:vec_F_2_defn} and slightly different from the $\vec K + \vec f$ in \eqref{eqn:vec_K_in_C_flatn}.
Furthermore, it yields that \[
    \frac{\ud}{\ud t}(e^{-\mu t}a_+) = - \frac{\ud}{\ud t}\left( e^{-\mu t} \bfOmg(\vec\phi, \vec Z_-)\right) + e^{-\mu t}F_+, \quad
    -\frac{\ud}{\ud t}(e^{\mu t} a_-) = - \frac{\ud}{\ud t}\left( e^{\mu t} \bfOmg(\vec\phi, \vec Z_+)\right) + e^{\mu t}F_-
\]
where \begin{equation}\label{eqn:defn_of_F_pm}
    F_\pm:= \bfOmg(a_+ \vec\calE_+ + a_- \vec\calE_-, \vec Z_\mp) + \bfOmg(\vec\phi, \vec\calE_\mp) + \bfOmg(\vec F_2, \vec Z_\mp).
\end{equation}
It is then natural to impose the following orthogonality conditions on $\vec q$ that \begin{equation}\label{eqn:ortho_cond_on_phi_for_Z_mu}
    \bfOmg(\vec \phi, \vec Z_-) = -e^{\mu t} \widetilde S( e^{-\mu t} F_+), \quad \bfOmg(\vec \phi, \vec Z_+) = -e^{-\mu t} \widetilde S( e^{\mu t} F_-),
\end{equation}
where the operator $\widetilde S$ is defined in \eqref{eqn:defn_of_widetil_S_op}. Using the fact $\frac{\ud}{\ud t} \widetilde S = S - I$, we obtain that \begin{equation}\label{eqn:diff_eqn_for_a_pm}
    \frac{\ud}{\ud t}(e^{-\mu t} a_+) = S(e^{-\mu t} F_+), \quad \frac{\ud}{\ud t}(e^{\mu t} a_-) = -S(e^{\mu t} F_-).
\end{equation}
Besides, the commutability of $\widetilde S$ and $\frac{\ud}{\ud t}$ on the time interval of interest in $[0, \infty)$ implies that \[
    \big(\frac{\ud}{\ud t}\big)^j \bfOmg(\vec \phi, \vec Z_-) = -e^{\mu t} \widetilde S\left( e^{-\mu t} \big(\frac{\ud}{\ud t}\big)^j F_+\right), \quad 
    \big(\frac{\ud}{\ud t}\big)^j \bfOmg(\vec \phi, \vec Z_+) = -e^{-\mu t} \widetilde S\left( e^{\mu t} \big(\frac{\ud}{\ud t}\big)^j F_-\right)
\]
for any $j \geq 0$.

\subsubsection{Existence of the decomposition \eqref{eqn:q_decomp} obeying \eqref{eqn:ortho_cond_on_phi_for_Z_mu}}\label{subsubsec:define-apm}
We first summarize what we have done. First, we start with a given $\psi$ and show the existence of decomposition $\vec\psi = \vec p + \vec q$ for any arbitrary $\xi, \ell$ satisfying suitable smallness assumptions. Simultaneously, $c_{im}$'s ($i = 1, 2$, $m = 1,2,3,4,\mu$) are determined. In particular, if the perturbation $\psi \equiv 0$, then $f_c \equiv 0$ and finally $p$, $q$ are vanishing as well. Afterwards, the modulation parameters are determined via \eqref{eqn:modulation_eqn} by imposing suitable orthogonality conditions. Finally, we want to show the existence decomposition \eqref{eqn:q_decomp} for $\vec q$.

Given this consideration, $\vec p$ is known given $\vec \psi$ so we could view $\vec q$ as given now. Also, we need to recall that vanishing of $\vec \psi$ means vanishing of $\vec q$. Then it suffices to show that there exists $a_+(t), a_-(t)$ so that the functional \[
    \Upsilon_\mu(\vec q, a_+, a_-):= \begin{pmatrix} 
    \bfOmg\big( \vec q - a_+ \vec Z_+ - a_- \vec Z_-, \vec Z_- \big) + e^{\mu t}\widetilde S(e^{-\mu t}F_+) \\
    \bfOmg\big( \vec q - a_+ \vec Z_+ - a_- \vec Z_-, \vec Z_+ \big) + e^{-\mu t}\widetilde S(e^{\mu t}F_-) 
    \end{pmatrix}
\]
vanishes provided that $\vec q$ is given. We remark that we actually replace $\vec\phi$ in $\vec F_\pm$ by $\vec q - a_+ \vec Z_+ - a_- \vec Z_-$ and $\vec F_2$ by $(\p_t - M)\vec q$ in this definition of $\Upsilon_\mu$ so that $\Upsilon_\mu$ indeed depends solely on $a_+, a_-$ and $\vec q$. 
Now we first note that $\Upsilon_\mu(0, 0, 0) = 0$. Then we compute the Fr\'{e}chet derivative \[
    D_{a_+, a_-}\Upsilon_\mu(0, 0, 0) = \begin{pmatrix}
        -1 & 0 \\ 0 & 1 \\
    \end{pmatrix}
    + o_{R_f, \ell}(1),
\]
where the presence of $\calE_\pm$ in $\calF_\pm$ contributes to the error terms.
Thus, for $\vec \psi$ under bootstrap assumptions, the existence of \eqref{eqn:q_decomp} with \eqref{eqn:ortho_cond_on_phi_for_Z_mu} follows from the implicit function theorem.
We remark that $C^2([0, \tau_0])$ is the functional space we choose to work in for $a_+, a_-$ when applying implicit function theorem like in Section~\ref{Section:Existence_of_decomp_Omg}. Here, $\tau_0$ is any arbitrary time no later the final time $\tau_f$ in the bootstrap argument. 
Therefore, it follows from \eqref{eqn:q_decomp} and \eqref{eqn:ortho_cond_on_phi_for_Z_mu} that \begin{equation}\label{eqn:eqn_for_a_pm}
     \bfOmg(\vec q, \vec Z_-) = -e^{\mu t} \widetilde S(e^{-\mu t} F_+) + a_+(t), \quad \bfOmg(\vec q, \vec Z_+) = -e^{-\mu t} \widetilde S(e^{\mu t} F_-) - a_-(t), 
\end{equation}
for $t \geq 0$.
In particular, putting $t = 0$ in \eqref{eqn:eqn_for_a_pm}, in view of $p$ having zero Cauchy data (i.e. $\vec q(0) = \vec \psi(0)$), we know that  \begin{equation}\label{eqn:size_eps_of_a_-_0}
    a_-(0) = -\bfOmg(\vec \psi(0), \vec Z_+(0)) - \widetilde S(e^{\mu \cdot}F_-)(0), \quad
    a_+(0) = \bfOmg(\vec\psi(0), \vec Z_-(0)) + \widetilde S(e^{-\mu \cdot} F_+)(0)
\end{equation}
are both of size $\eps$. 
\begin{remark}\label{rmk_on_extension_in_t}
    Note that $\widetilde S(e^{\pm\mu \cdot}F_\mp)(0)$ is purely determined via $F_\pm|_{[-1,0]}$. In other words, this term purely depends on the initial data and how we extend the solutions. One could recall that we extend all parameters trivially and $\psi$ is extended in a linear way $\psi = \psi_0 + t \psi_1$, which is explicitly mentioned in the proof of Theorem~\ref{main_thm}. These extensions depend on nothing other than the initial data and the forms of which do not matter too much as long as they inherit $\veps$-smallness from initial data.
\end{remark}

To summarize, all conditions including   \eqref{eqn:mod_ortho_for_p}, \eqref{eqn:ortho_cond_for_vec_bfOmg} and \eqref{eqn:ortho_cond_on_phi_for_Z_mu} are imposed in this section for $p$ and $q$ to deal with the obstructions on the linear level.

\section{Bootstrap Assumptions}\label{section:bootstrap_assumptions}

We first briefly summarize what we have accomplished so far. First, we assume the existence of $\psi$, which solves 
\eqref{eqn:P_psi_eqn}. Then we provide a scheme to determine a so-called modified profile $p$, which solves a similar equation as $\psi$ with a different source term. Then we show that we could find $\xi, \ell$ such that the decomposition \eqref{eqn:ortho_cond_for_vec_bfOmg} is valid. Under proper smallness assumptions of modulation parameters, we show the existence of decomposition \eqref{eqn:ortho_cond_for_vec_bfOmg}, and \eqref{eqn:q_decomp}, leading to the governing equations of parameters \eqref{eqn:modulation_eqn} and \eqref{eqn:diff_eqn_for_a_pm}. The modulation equation \eqref{eqn:modulation_eqn} serves as our heuristics to see that we would at least expect $\p_t\vec\bfOmg_p - \vec N_p$ and $\vec N_q$ both having $\tau^{-2-}$ decay in order to ensure the desired $\tau^{-2-}$ decay for $\dot\wp$.
In what follows, we will use $U_k$ to denote $(\uprho\p_\uprho)^{k_3}(\p_\uptheta)^{k_2}(\p_\uptau)^{k_1} U$ for $k = (k_1, k_2, k_3)$. The order of $k$ is given by $|k| = k_1 + k_2 + k_3$. 

\subsection{Bootstrap assumption list}\label{Section:BA}

We assume that the following trapping assumption holds on $\uptau \in [0, \tau_f]$ for some small $\kpp > 0$: \begin{equation}\label{eqn:trapping_assumption}\tag{trapping-$a_+$}
    |a_+(\uptau)| \leq C_\trap \eps \brk{\uptau}^{-\frac94 + \kappa}.
\end{equation}
\eqref{eqn:trapping_assumption} will not be improved like other assumptions in Section~\ref{Section:BA}. Instead, we will close the bootstrap under this trapping assumption and show that we could select codimension-$1$ initial data set such that this trapping assumption holds in all time. Note that $a_+(0) = O(\eps)$ from \eqref{eqn:size_eps_of_a_-_0} and hence this assumption is at least valid locally in time.

Now, we state our bootstrap assumptions and we refer to Remark~\ref{rmk_on_extra_smallness_delta_wp} for an explanation of the parameter $\delta_\wp$. We assume that the following estimates are satisfied for $\tau \in [0, \tau_f]$:  
\begin{align}
    |\frac{\ud^k}{\ud\uptau^k}\dot\wp(\uptau)| &\leq 2C_k\delta_\wp \eps \brk{\uptau}^{-\frac{9}{4}+\kappa},\quad k \geq 0, \tag{BA-$\dot\wp$}\label{BA-dotwp} \\
    % |\frac{d^k}{d\uptau^k}\ddot\wp(\uptau)| &\leq 2C_k\delta_\wp \eps \brk{\uptau}^{-\frac{9}{4}+\kappa},\quad k \geq 0, \tag{BA-$\ddot\wp$}\label{BA-ddotwp} \\
    \|\frac{\ud^k}{\ud\uptau^k}\ddot\wp\|_{L^2_{\uptau}[\tau_1, \tau_2]} &\leq 2C_k \delta_\wp \eps \brk{\tau_1}^{-\frac{9}{4} + \kappa}, \quad k \geq 0, \tag{BA-$\ddot\wp$-$L^2$}\label{BA-ddotwp-L2} \\
    |a_+^{(k)}(\uptau)| &\leq C_\trap \eps \brk{\uptau}^{-\frac{9}{4}+\kappa} + 2C_k\delta_\wp^{\frac12} \eps \brk{\uptau}^{-\frac{9}{4}+\kappa}, \quad k \geq 0, \tag{BA-$a_+$}\label{BA-a_+}\\
    |a_-^{(k)}(\uptau) - (-\mu)^k a_-(0)e^{-\mu t}| &\leq 2C_k\delta_\wp^{\frac12} \eps \brk{\uptau}^{-\frac{9}{4}+\kappa}, \quad k \geq 0, \tag{BA-$a_-$}\label{BA-a_-} \\
    |p_k| &\leq 2C_k\eps \brk{\uptau}^{-2+\kappa}, \quad 0 \leq |k| \leq M - 7, \tag{BA-$p$}\label{BA-p} \\
    |\p_\uptau p_k| &\leq 2C_k\eps \brk{\uptau}^{-\frac94+\kappa}, \quad 0 \leq |k| \leq M - 8, \tag{BA-$\p_\tau p$}\label{BA-dtau-p} \\
    \chi_{\uprho \geq R} |p_k| &\leq 2C_k\eps \brk{\uprho}^{-1} \brk{\uptau}^{-1 + \frac{\kappa}2},\quad 0 \leq |k| \leq M - 2, \tag{BA-$p$-ext1}\label{BA-p-ext1} \\
    \chi_{\uprho \geq R} |p_k| &\leq 2C_k\eps \brk{\uprho}^{-\frac32}\brk{\uptau}^{-\frac{1}{2}+\frac{\kappa}{4}},\quad 0 \leq |k| \leq M - 2, \tag{BA-$p$-ext2}\label{BA-p-ext2} \\
    \|\p_\tau^2 \bbP_k\|_{LE(\calD^{\tau'}_\tau)} &\leq 2 C_k \eps \brk{\tau}^{-\frac94 + \kappa}, \quad |k| \leq M - 7, \tag{BA-$p$-LE}\label{BA-p-LE} \\
    % \|\brk{\uprho}^{\frac{\kappa}{2}-2}\p_\tau p_k\|_{L^2(\Sigma_\tau, \ud V)} &\leq 2C_k\eps \brk{\uptau}^{-\frac94+\kappa}, \quad 0 \leq |k| \leq M - 8, \tag{BA-$\p_\tau p$-$L^2$}\label{BA-dtau-p-L2} \\ %this estimate has been closed in the subsequent sections, but we do not need this assumptions
    |\phi_k| &\leq 2C_k\eps \brk{\uptau}^{-\frac94 + \kappa}, \quad 0 \leq |k| \leq M - 29, \tag{BA-$\phi$}\label{BA-q} \\
    % |\p_\uptau \phi_k| &\leq 2C_k\eps \brk{\uptau}^{-\frac{13}4 + \kappa}, \quad 0 \leq k \leq M - 43, \tag{BA-$\p_\uptau \phi$} \label{BA-dtau-q} \\ % this should be the thing we expect in the best case, but this cannot be achieved due to coupling with $p$
    \chi_{\uprho \geq R} |\phi_k| &\leq 2C_k\eps \brk{\uprho}^{-1}\brk{\uptau}^{-\frac{5}{4}+\kappa}, \quad 0 \leq |k| \leq M - 26, \tag{BA-$\phi$-ext1}\label{BA-q-ext1}\\
    \chi_{\uprho \geq R} |\phi_k| &\leq 2C_k\eps \brk{\uprho}^{-\frac32}\brk{\uptau}^{-\frac{5}{8}+\frac{\kappa}2}, \quad 0 \leq |k| \leq M - 26, \tag{BA-$\phi$-ext2}\label{BA-q-ext2} \\
    \|\p_\tau \mathbb\Phi_k\|_{LE(\calD^{\tau'}_\tau)} &\leq 2 C_k \eps \brk{\tau}^{-\frac94 + \kappa}, \quad |k| \leq M - 29, \tag{BA-$\phi$-LE}\label{BA-q-LE}
\end{align}
where we remark that all $C_k$'s could be viewed sufficiently large as it needs to beat the linear effects in short time scales. Here, the notation $\bbP:= sp$, $\mathbb{\Phi}:= s\phi$ would be clear in Section~\ref{Section:overview_of_rp} and the $LE$ norms are defined in Section~\ref{Section:ILED}.
These assumptions are assumed to be true throughout the paper including previous sections.

\begin{remark}\label{rmk_on_extra_smallness_delta_wp}
    In our bootstrap argument, the energy of the perturbation $\psi$ linearly enters when estimating the parameter derivatives, while the parameter derivatives enter linearly in the energy estimates in turn. What breaks the circularity is that the linear appearance of the energy of the perturbation in the estimates for the parameter derivatives is always accompanied by a small (but non-decaying) constant, denoted by $\delta_\wp$. Indeed, in view of the equation for parameter derivatives \eqref{eqn:modulation_eqn} and the property of $\vec G$ in \eqref{eqn:vec_G_eqn_derived_from_1st_order_HVMC_1}, we notice that the linear contribution of $\vec q$ comes from $\vec N_q$, which will possess extra smallness in terms of $\calO(\ell)$ or inverse power of $R_f$ (see Step 2 and 3 in the proof of Lemma~\ref{lemma_decay_of_bfOmg_phi} and in particular \eqref{eqn:inverse_R_1_decay_contribution_for_delta_wp}). This small constant is denoted by $\delta_\wp$ whose precise form is given in \eqref{eqn:delta_wp_defn}. Moreover, the linear contribution of $\vec p$ in \eqref{eqn:modulation_eqn} is of the form $\p_t \vec p - M\vec p$, having extra smallness as well. See Remark~\ref{rmk_on_construction_of_p}.
    Shortly, we shall see that the smallness of $\delta_\wp$ is indeed in terms of $O(\ell)^{\frac12}$, inverse powers of $R_f$.
\end{remark}

\begin{remark}
    To compare our trapping assumptions with the one in \cite{OS24}, one needs to note that even though the quantity therein seems to be on $C^2$-level according to the orthogonality conditions, it's still understood on the $C^0$ level of $a_\pm$. The reason why we do not need this modification in this paper is that we do not attempt to get better decay when taking derivatives in $\uptau$ for $\phi$. In contrast, better decay of $\ddot a_+$ is required (even if this might not be stated explicitly in bootstrap assumptions for $a_+^{(k)}$) to close bootstrap for quantities with at least one $\mathbf{T}$ derivative (see for instance \cite[Equation (6.7), Lemma 7.1]{OS24}). 
\end{remark}

\subsection{Improving bootstrap assumptions}
\begin{proposition}\label{prop_closing_BA_1}
    Suppose the bootstrap assumptions in Section~\ref{Section:BA} are satisfied. Moreover, the orthogonality conditions \eqref{eqn:mod_ortho_for_p}, \eqref{eqn:ortho_cond_for_vec_bfOmg} and \eqref{eqn:ortho_cond_on_phi_for_Z_mu} hold true. If the size of the initial perturbation $\eps$ is sufficiently small, the constants $C_k$ appearing in bootstrap assumptions are sufficiently large (compared to $\frac{a_-(0)}{\eps}$ and the linear effect), then the following improved bootstrap estimates hold: 
    \begin{align}
        |\frac{\ud^k}{\ud\uptau^k}\dot\wp(\uptau)| &\leq C_k\delta_\wp \eps \brk{\uptau}^{-\frac94+\kappa},\quad k \geq 0, \tag{IBA-$\dot\wp$}\label{IBA-dotwp}\\
        \|\frac{\ud^k}{\ud\uptau^k}\ddot\wp\|_{L^2_{\uptau}[\tau_1, \tau_2]} &\leq C_k \delta_\wp \eps \brk{\tau_1}^{-\frac{9}{4} + \kappa}, \quad k \geq 0, \tag{IBA-$\ddot\wp$-$L^2$}\label{IBA-ddotwp-L2} \\
        |a_+^{(k)}(\uptau)| &\leq C_\trap \eps \brk{\uptau}^{-\frac{9}{4}+\kappa} +  C_k\delta_\wp^{\frac12} \eps \brk{\uptau}^{-\frac94+\kappa}, \quad k \geq 0, \tag{IBA-$a_+$}\label{IBA-a_+}\\
        |a_-^{(k)}(\uptau) - (-\mu)^k a_-(0)e^{-\mu t}| &\leq C_k\delta_\wp^{\frac12} \eps \brk{\uptau}^{-\frac{9}{4}+\kappa}, \quad k \geq 0. \tag{IBA-$a_-$}\label{IBA-a_-}  
    \end{align}
\end{proposition}

\begin{proposition}\label{prop_closing_BA_2}
    Suppose the bootstrap assumptions in Section~\ref{Section:BA} are satisfied. Moreover, the orthogonality conditions \eqref{eqn:mod_ortho_for_p}, \eqref{eqn:ortho_cond_for_vec_bfOmg} and \eqref{eqn:ortho_cond_on_phi_for_Z_mu} hold true. If the size of the initial perturbation $\eps$ is sufficiently small, the constants $C_k$ appearing in bootstrap assumptions are sufficiently large (compared to $\frac{a_-(0)}{\eps}$, $C_\trap$ and the linear effect), then the following improved bootstrap estimates hold: 
    \begin{align}
    |p_k| &\leq C_k\eps \brk{\uptau}^{-2+\kappa}, \quad 0 \leq |k| \leq M - 7, \tag{IBA-$p$}\label{IBA-p} \\
    |\p_\uptau p_k| &\leq C_k\eps \brk{\uptau}^{-\frac94+\kappa}, \quad 0 \leq |k| \leq M - 8, \tag{IBA-$\p_\tau p$}\label{IBA-dtau-p} \\
    \chi_{\uprho \geq R} |p_k| &\leq C_k\eps \brk{\uprho}^{-1} \brk{\uptau}^{-1 + \frac{\kappa}2},\quad 0 \leq |k| \leq M - 2, \tag{IBA-$p$-ext1}\label{IBA-p-ext1} \\
    \chi_{\uprho \geq R} |p_k| &\leq C_k\eps \brk{\uprho}^{-\frac32}\brk{\uptau}^{-\frac{1}{2}+\frac{\kappa}{4}},\quad 0 \leq |k| \leq M - 2, \tag{IBA-$p$-ext2}\label{IBA-p-ext2} \\
    \|\p_\tau^2 \bbP_k\|_{LE(\calD^{\tau'}_\tau)} &\leq C_k \eps \brk{\tau}^{-\frac94 + \kappa}, \quad |k| \leq M - 7, \tag{IBA-$p$-LE}\label{IBA-p-LE} \\
    |\phi_k| &\leq C_k\eps \brk{\uptau}^{-\frac94 + \kappa}, \quad 0 \leq |k| \leq M - 29, \tag{IBA-$\phi$}\label{IBA-q} \\
    \chi_{\uprho \geq R} |\phi_k| &\leq C_k\eps \brk{\uprho}^{-1}\brk{\uptau}^{-\frac{5}{4}+\kappa}, \quad 0 \leq |k| \leq M - 26, \tag{IBA-$\phi$-ext1}\label{IBA-q-ext1}\\
    \chi_{\uprho \geq R} |\phi_k| &\leq C_k\eps \brk{\uprho}^{-\frac32}\brk{\uptau}^{-\frac{5}{8}+\frac{\kappa}2}, \quad 0 \leq |k| \leq M - 26, \tag{IBA-$\phi$-ext2}\label{IBA-q-ext2} \\
    \|\p_\tau \mathbb\Phi_k\|_{LE(\calD^{\tau'}_\tau)} &\leq C_k \eps \brk{\tau}^{-\frac94 + \kappa}, \quad |k| \leq M - 29, \tag{IBA-$\phi$-LE}\label{IBA-q-LE}
    \end{align}
\end{proposition}

We end this subsection with several heuristic remarks on the bootstrap argument.
\begin{remark}
    The proof of the main theorem (Theorem~\ref{main_thm}) relies on a continuity argument which is based on an adequate version of local existence theorem (see Appendix~\ref{Section:LWP}) and bootstrapping. Although in our setup for the problem, there are two nonlocal constructions, namely the parameter smoothing and the shooting argument to show the existence of $a_+$ satisfying \eqref{eqn:trapping_assumption}. However, the procedure will not be affected by the nonlocality as these constructions still preserve causality, i.e., they do not depend on future times. This is in contrast to some global arguments in other works, for instance, the foliation chosen in \cite{CK90} is constructed from the last slice (i.e., the leaf parametrized the final time in the bootstrap). This process is described in \cite[Step 4 in Proof 10.2.1]{CK90} in detail. The bootstrap and modulation arguments that we adopt here are also different from the dyadic approach developed in \cite{DHRT24}. Modulation on dyadic time-slabs has also been performed in \cite{Kadar24a}, \cite{Kadar24b} and \cite{AKU24}. 
\end{remark}

\begin{remark}\label{rmk_on_a_-_0_in_BA}
    For the bootstrap regarding $a_-(t)$, there is a certain part that does not contain $\delta_\wp$-smallness. Nevertheless, it does not cause a problem to improve the bootstrap for estimates regarding $p$ and $\phi$ as $a_-(0)$ is an absolute constant determined by the initial data of $\psi$. See \eqref{eqn:size_eps_of_a_-_0}. We could simply choose $C_k$'s in the bootstrap assumptions in Proposition~\ref{prop_closing_BA_2} (not required for the ones in Proposition~\ref{prop_closing_BA_1}) large enough to beat this constant. The appearance of $C_\trap$ in the bootstrap assumptions of $a_+$ does not cause any trouble due to the same consideration.
\end{remark}

\subsection{Precise statement of main theorem and its proof}
We shall defer the proof of Proposition~\ref{prop_closing_BA_1} to the following subsection and that of Proposition~\ref{prop_closing_BA_2} to Section~\ref{section_rp_methods_hyperboloidal}. Instead, we state and prove the main theorem assuming Proposition~\ref{prop_closing_BA_1} and \ref{prop_closing_BA_2}.

To state the main theorem, we first recall the coordinates on $\barcalC$ in \eqref{eqn:metric_C_in_rho}. 
We first recall the $\chi \varphi_\mu$ defined to be the first entry of $\vec Z_\pm$ (see \eqref{eqn:defn_Z_pm}), where $\varphi_\mu$ is introduced in Section~\ref{Section:spectral_analysis_of_stab_op_L}. 
\begin{theorem}\label{main_thm}
    Given $\psi_0, \psi_1 \in C^\infty_0(\barcalC)$ with support in $\{|\Xbar| \leq R_f/2\} \subset \barcalC$. Suppose \[
        \sum_{k = 0}^M \|(\brk{\rho} \p_\Sigma)^k (\psi_0 \circ F)\|_{L^2(\barcalC)} + \|(\brk{\rho} \p_\Sigma)^{k-1} (\psi_1 \circ F)\|_{L^2(\barcalC)} \leq \eps
    \]
    for some sufficiently large fixed constant $M$ and sufficiently small $\eps$. (See Remark~\ref{rmk_decay_assumption_on_id} to see why we choose this norm.) 
    Here, $\p_\Sigma$ indicates size $1$ derivative $\p_\rho$ or $\brk{\rho}^{-1}\p_\omg$. Additionally, we require $\psi_0, \psi_1$ satisfying the codimension-$1$ assumption \eqref{eqn:codim_1_cond}.
    Then there exists $b \in \R$ with $|b| \lesssim 1$ and $\Phi$ being an embedding $\Phi: \calM \to \R^{1+(4+1)}$ satisfying \eqref{eqn:HVMC} such that \[
        \Phi|_{X^0 = 0} = \Phi_0[\eps(\psi_0 + b \chi\varphi_\mu)], \quad \p_{X^0}\Phi|_{X^0 = 0} = \Phi_1[\eps(\psi_1 - \mu b \chi\varphi_\mu)],
    \]
    where we use the notation in local theory \eqref{eqn:data_on_const_t_slice} to define $\Phi_0, \Phi_1$. 
    Moreover, the region $\Phi(\calM) \cap \calR$ in the ambient spacetime (see \eqref{eqn:region_of_ambient_space}) can be parametrized as follows: \begin{equation}\label{eqn:parametrization_in_statement_thm}
        \cup_\uptau \Sigma_\uptau \ni p \mapsto p + \psi(p) N(p) \in \R^{1 + (4+1)},
    \end{equation}
    where $\Sigma_\uptau$ given in \eqref{eqn:defn_calQ} with $\xi, \ell$ satisfying $\xi(0) = \ell(0) = 0$ (see Remark~\ref{rmk_decay_assumption_on_id}) and $|\dot\xi|, |\ell| \lesssim \eps$. Along with the perturbation $\psi$, the following quantitative decay is shown \[
        |\dot\ell|, |\dot\xi - \ell| \lesssim \eps \brk{\uptau}^{-\frac94 + \kappa}, \quad \|\psi\|_{L^\infty(\Sigma_\uptau)} \lesssim \eps \brk{\uptau}^{-2 + \kappa}.
    \]
\end{theorem}
\begin{proof}
    Given $(\psi_0, \psi_1)$, we artificially define $\psi$ on $[-1, 0]$ as $\psi = \psi_0 + t \psi_1$ to stay compatible with first order formulation so that our nonlocal smoothing operators are well-defined on the positive real line.
    For each $b$, we define $\tau_f(b)$ as the maximal time so that there is a solution parameterized as \eqref{eqn:parametrization_in_statement_thm} with two parameters $\xi, \ell$ on $[0, \tau_f(b)]$ ($[0, \infty)$ when $\tau_f(b) = \infty$), 
    where the trapping assumptions and bootstrap assumptions in Section~\ref{Section:BA} and orthogonality conditions \eqref{eqn:mod_ortho_for_p}, \eqref{eqn:decomp_of_g_i}, \eqref{eqn:ortho_cond_for_vec_bfOmg}, \eqref{eqn:ortho_cond_on_phi_for_Z_mu} are satisfied. Thanks to the local theory in Appendix~\ref{Section:LWP}, we know that $\tau_f(b) > 0$ for any choice of $b \in \R$. Thanks to Proposition~\ref{prop_closing_BA_1} and \ref{prop_closing_BA_2}, it suffices to show the existence of some $b$ so that $\tau_f(b)$ is infinite. We suppose by contradiction that $\tau_f(b) < \infty$ for each $b$.

    \textit{Step 1: \eqref{eqn:trapping_assumption} must be saturated at $\tau_f(b)$, i.e., the inequality becomes an equality. } 
    We suppose \eqref{eqn:trapping_assumption} is an strict inequality on $[0, \tau^*]$ and it suffices to show that $\tau_f(b) > \tau^*$.
    By applying Proposition~\ref{prop_closing_BA_1} and \ref{prop_closing_BA_2} (with $\tau_f$ replaced by $\tau^*$), we can improve the bootstrap assumptions except for \eqref{eqn:trapping_assumption}. Then by applying local theory in Appendix~\ref{Section:LWP} with $\ell_0 = \ell(\tau^*)$ and $\xi_0 = \xi(\tau^*)$, we could extend our solution beyond $\tau^*$, say on interval $[0, \tau^* + \tau']$. According to the improvement of bootstrap assumptions on $[0, \tau^*]$, all bootstrap assumptions are strict on $[0, \tau^*]$ and hence can still hold on $[0, \tau^* + \tau']$. Therefore, the parametrization and orthogonality conditions are still satisfied with the same implicit function theorem argument to demonstrate the existence in Section~\ref{Section:Derivation_of_op_modified_profile}. Thus, we obtain a larger interval than $[0, \tau^*]$ on which bootstrap assumptions (as well as parametrization, orthogonality conditions) are satisfied. This implies that $\tau_f(b) > \tau_*$ by contradiction.

    \textit{Step 2: Select $b$ given $(\psi_0, \psi_1)$ satisfying the codimension-$1$ condition \eqref{eqn:codim_1_cond}. }
    We define \[
        \beta(\uptau):= a_+(\uptau), \quad \lambda(\uptau):= C_\trap \eps \brk{\uptau}^{-\frac94+\kappa}, \quad \lambda_0:= C_\trap\eps.
    \]
    We now define a map  $\calZ: C^\infty(\barcalC) \times C^\infty(\barcalC) \times I \to \bbR$, where $I \subset \R$ is a neighborhood of zero, by \[
        \calZ(\psi_0, \psi_1, b) = \beta(0).
    \]
    This map is well-defined, that is,  $\beta(0)$ is indeed determined purely by the initial data. One can see this by restricting  \eqref{eqn:eqn_for_a_pm} to $t = 0$.

    Now we restrict attention to $(\psi_0, \psi_1)$ satisfying codimension-$1$ condition \begin{equation}\label{eqn:codim_1_cond}
        \calZ(\psi_0, \psi_1, 0) = 0.
    \end{equation}
    One can see that $\dfrac{\delta \calZ}{\delta b} \sim c_\pm \neq 0$(with $c_\pm$ being the normalization constant in \eqref{eqn:defn_Z_pm}),
    then it follows from implicit function theorem that for each $\beta_0$ such that $|\beta_0| \leq \lambda_0$ ($\eps$ is chosen small enough so that this is in the local range where implicit function theorem holds), there is a choice $b = b(\beta_0)$ in a neighborhood of $0$ for which $\beta(0) = \beta_0$.

    \textit{Step 3: A shooting argument to derive a contradiction. }
    It follows from Step 1 and the contradiction assumption that for any $\beta_0$ with $|\beta_0| \leq \lambda_0$, there exists $\tau_\trap(\beta_0):= \tau_f(b(\beta_0))$ and a solution with $\beta(0) = \beta_0$ such that \begin{equation}\label{eqn:beta_property}
        |\beta(\uptau)| < |\lambda(\uptau)|, \ \forall \uptau \in [0, \tau_\trap(\beta_0)), \qquad |\beta(\tau_\trap(\beta_0))| = |\lambda(\tau_\trap(\beta_0))|,
    \end{equation}
    that is, it gets saturated at $\tau_\trap(\beta_0)$.
    We then use a standard shooting argument (see for instance \cite[Lemma 3.9]{OP24}) to falsify this claim.
    Define $\Lambda: [-\lambda_0, \lambda_0] \to \{\pm \lambda_0\}$ by \[
        \Lambda(\beta_0) = \beta(\tau_\trap(\beta_0)) \brk{\tau_\trap(\beta_0)}^{\frac{9}{4} - \kappa}.
    \]
    Rewriting the equation for $a_+$ in \eqref{eqn:diff_eqn_for_a_pm} as \begin{equation}\label{eqn:DE_of_q}
        \dot \beta(\uptau) = \mu \beta(\uptau) + e^{\mu\uptau} S(e^{-\mu\uptau} F_+),
    \end{equation}
    it follows that \begin{equation}\label{eqn:q_square_increasing}
        \frac{\ud}{\ud\uptau} \beta^2(\uptau) \geq \mu \beta^2(\uptau)\ \text{for $\uptau$ such that } \frac12 \lambda(\uptau) < |\beta(\uptau)| < \lambda(\uptau).
    \end{equation}
    Then it follows from \eqref{eqn:beta_property} and \eqref{eqn:q_square_increasing} that \[
        \Lambda(\beta_0) = \begin{cases}
            -\lambda_0, \quad &-\lambda_0 < \beta_0 < -\frac34\lambda_0, \\
            \lambda_0, \quad &\frac34\lambda_0 < \beta_0 < \lambda_0. \\
        \end{cases}
    \]
    Now it suffices to show the continuity of $\Lambda$ to derive a contradiction.

    \textit{Step 4: Continuity of $\Lambda$. }
    We now specify the dependence of $\beta$ on $\beta(0) = \beta_0$ by $\beta_{\beta_0}$. By \eqref{eqn:q_square_increasing}, given any $\tilde\veps > 0$, there exists $\tilde\delta \ll 1$ such that if $(1 - \tilde\delta)\lambda(\uptau) < |\beta_{\beta_0}(\uptau)| < \lambda(\uptau)$ for some $\uptau < \tau_\trap(\beta_0)$, then $|\uptau - \tau_\trap(\beta_0)| < \tilde\veps$. Given $|\beta_1| < \lambda_0$, there exists $\uptau_1 < \tau_\trap(\beta_1)$ such that $(1 - \tilde\delta^2)\lambda(\uptau_1) < |\beta_{\beta_1}(\uptau_1)| < (1 - \tilde\delta^3)\lambda(\uptau_1)$. For $\beta_2$ sufficiently close to $\beta_1$, by the continuous dependence, bootstrap assumptions (in the range where trapping assumption holds) and condition $|\beta_{\beta_1}(\uptau_1)| < (1 - \tilde\delta^3)\lambda(\uptau_1)$, we know that $\uptau_1 < \tau_\trap(\beta_2)$ and $(1 - \tilde\delta)\lambda(\uptau_1) < |\beta_{\beta_2}(\uptau_1)| < \lambda(\uptau_1)$. Hence, $|\tau_\trap(\beta_1) - \tau_\trap(\beta_2)| < |\tau_\trap(\beta_1) - \uptau_1| + |\tau_\trap(\beta_2) - \uptau_1| < 2\tilde\veps$, which completes the proof.
\end{proof}

\subsection{Proof of Proposition~\ref{prop_closing_BA_1}}\label{Section:parameter_control}
In this subsection, we aim to prove Proposition~\ref{prop_closing_BA_1}. The purpose of the smoothing operator $S$ will become clear in this proof. Taking \eqref{IBA-dotwp} as an example, we would absorb all the excessive derivatives in $\uptau$ via the smoothing operator $S$ in \eqref{eqn:modulation_eqn} so that the finiteness of the number of derivatives stated in the bootstrap assumptions of those estimates regarding $p$ and $\phi$ won't restrict us putting as many derivatives as we want on $\dot\wp$. The argument would be quite similar to the ones in \cite{LOS22} or \cite{OS24}. Nevertheless, we inspect in detail to get a precise form of $\delta_\wp$.

\subsubsection{Decay properties of $c_{1m}$ and $c_{2m}$}
Recall that these coefficients were introduced in Section~\ref{Section:mod_profile}. 
From \eqref{eqn:constraint_g12_2_Sp} and \eqref{BA-p}, we obtain $c_{2m}$ satisfies the same pointwise decay in $\tau$ as in \eqref{BA-p}. In addition, we argue that it contains $\delta_\wp \eps$ smallness. First, in $\{\uprho \gtrsim R_f\}$ region, it is obvious that the strong spatial decay of eigenfunctions gives the desired extra smallness. In the flat region, we notice that for $\nu \in \{1, 2, \cdots, n\}$, we could use the smallness of $h^{0j}$ in \eqref{eqn:1st_order_formulation_h_det_inv}. Lastly, for $\nu = 0$, we argue the three integrands in \eqref{eqn:constraint_g12_2_Sp} separately. For the first integrand, even though $\sqrt{|h|}$ itself does not have any smallness, $\p_0 \sqrt{|h|}$ indeed captures $\dot\ell$ smallness. The term $\frac{W_m}{h^{00}} + Y_m$ is of $\calO(\ell + R_f^{-1})$ due to cancellation (recall the form of $h^{00}$ in \eqref{eqn:1st_order_formulation_h_det_inv}). For $\frac{\p_0 W_m}{h^{00}} - \p_0 Y_m$, we know that $\p_0$ produces $R_f^{-1}$ smallness if it falls on the cutoff while it produces $\calO(\dot\wp)$ smallness if it falls on the eigenfunctions. This discussion can be summarized as the decay \begin{equation}\label{eqn:decay_of_c_2m}
    |c_{2m}(\uptau)| \lesssim C_k\delta_\wp \eps \brk{\uptau}^{-2 + \kappa}.
\end{equation}
One can refer to Section~\ref{Section:BA} for the definition of $\delta_\wp$.
Taking this updated information \eqref{eqn:decay_of_c_2m} for $c_{2m}$ into consideration, it follows from \eqref{eqn:constraint_g12_1} that $c_{1m}$ satisfies \[
    |c_{1m}(\uptau)| \lesssim C_k\delta_\wp \eps \brk{\uptau}^{-\frac94 + \kappa},
\]
where $C_k$ is our bootstrap constant and the extra smallness $\delta_\wp$ is obtained like in \eqref{eqn:decay_of_c_2m}.
\begin{remark}\label{rmk:reason_chosen_cutoff_and_eftn}
    In the estimate above, the term $p \calP_h Y_m$ deserves a detailed discussion. Even though $p$ itself does not have sufficient decay $\brk{\uptau}^{-\frac94 + \kappa}$, thanks to the choice of the cutoff, for any $\p_\uptau$ derivative falling on the cutoff, one gains one order decay in $\uptau$. Moreover, it is also crucial that $\varphi_m^\stat$ is a nice proxy of eigenfunctions of $\calP_h$. This leads to a sufficient gain in $\tau$ decay as well when $\calP_h$ falls on $\varphi_j^\stat$ ($j = 1,2,3,4$). On the other hand, when $\calP_h$ falls on $\varphi_\mu^\stat$,  the major contribution besides the nice error is of the form $\mu^2\brk{p, Y_m}$, which in turn can be bounded via \eqref{eqn:mod_ortho_for_p} due to our choice $Y_m = W_m$. 
\end{remark}

Applying $\p_\uptau$ on the coefficients via \eqref{eqn:coeff_c_im},  \eqref{eqn:constraint_g12_1} and \eqref{eqn:constraint_g12_2_Sp}, we obtain \begin{equation}\label{eqn:all_decay_of_c_im}
    |\p_\uptau^j c_{1m}(\uptau)| \lesssim C_j \delta_\wp \eps \brk{\uptau}^{-\frac94 + \kappa}, \quad 0 \leq j \leq 2, \quad 
    |\p_\uptau^j c_{2m}(\uptau)| \lesssim C_j \delta_\wp \eps \brk{\uptau}^{-\frac94 + \kappa}, \quad 1 \leq j \leq 3
\end{equation}
thanks to \eqref{BA-dtau-p}, where $C_j$'s here denote the dependence on bootstrap constants. Note that no improvement of decay is achieved after applying $\p_{\uptau}$. This feature will be discussed at the end of Section~\ref{Section:overview_of_rp}. For the purpose of proving Lemma~\ref{lemma_control_higher_der_of_parameters}, we also record the following \begin{equation}\label{eqn:L2_est_of_c_im}
    \|\p_\uptau^j c_{1m}(\uptau)\|_{L^2_\uptau} \lesssim C_j \delta_\wp \eps \brk{\uptau}^{-\frac94 + \kappa}, \quad 1 \leq j \leq 2, \quad
    \|\p_\uptau^2 c_{2m}(\uptau)\|_{L^2_\uptau} \lesssim C_j \delta_\wp \eps \brk{\uptau}^{-\frac94 + \kappa},
\end{equation}
where the first one follows from \eqref{BA-p}, \eqref{BA-dtau-p} and the support condition of $Y_m$ while the second one follows from \eqref{BA-p}, \eqref{BA-dtau-p} and \eqref{BA-p-LE}. 

\begin{remark}
    The importance of the extra smallness and decay properties of $c_{im}$'s are discussed in Remark~\ref{rmk_on_construction_of_p}.
\end{remark}

\subsubsection{Bounds for quantities from Sections~\ref{Section:modulation_equations} and \ref{subsec:unstable-mode}, and proof of Proposition~\ref{prop_closing_BA_1}}
\begin{lemma}\label{lemma_on_omg}
    Under bootstrap assumptions in Section~\ref{Section:BA}, when $R_f$ is sufficiently large and $\eps$ is sufficiently small, we have \[
        \big(\frac{\ud}{\ud t}\big)^k \vec\omg \lesssim \eps^{\frac32} \brk{\uptau}^{-4 + 2\kappa}, \quad \forall k \geq 0,
    \]
    where the constant would not depend on the bootstrap constants $C_k$'s.
\end{lemma}
\begin{proof} 
    First, we generalize the integral form \eqref{eqn:defn_of_vec_Upomg_omg_2_int_form} of the differential equation \eqref{eqn:defn_of_vec_Upomg_omg_2} to arbitrary order derivatives. By taking $\p_t^k$ in \eqref{eqn:defn_of_vec_Upomg_omg_2} and considering its integral form again, we obtain \[
        \p_t^k \vec\omg(t) = \int_0^t e^{-\beta(t-t')} (\p_{t'}^k (S \vec F_\omg))(t')\, \ud t'.
    \]
    Due to the structure of the smoothing operator $S$, we could absorb the $\p_{t'}^k$ derivatives into the smoothing kernel and hence reduces the proof to the $k = 0$ case. 
    
    For $k = 0$, we notice that \eqref{eqn:vec_G_eqn_derived_from_1st_order_HVMC_1} and an application of Taylor's inequality help us to bound \[
        |\vec F_\omg| \leq |\p^2_\Sigma \vec\psi| \left(|\p^2_\Sigma \vec\psi| + |S(\p_t \vec \bfOmg_p - \vec N_p - \vec N_q) - \beta \vec \omg|\right).
    \]
    Then by rewriting the first entry of \eqref{eqn:1st_order_HVMC}, we obtain that \[
        |\p_\Sigma^l \dot\psi| \lesssim |\p_{\Sigma}^l \p \psi| + |\dot\wp| + |\p_{\Sigma}^l \psi|^2,
    \]
    where we only need $l \leq 2$ cases for this proof. Here, $\dot\psi$ denotes the second entry of $\vec\psi$.
    Moreover, $\dot p$ is directly related to $p$ itself according to the definition of $\vec p$ in \eqref{eqn:defn_vec_p}. For $q$, we recall the decomposition $q = \phi + a_+ Z_+ + a_- Z_-$ as well. Hence, putting all the preceding considerations together, we can estimate \[
        |\vec F_\omg| \lesssim C_k^2 \eps^2 \brk{\uptau}^{-4 + 2\kappa} + C_k \eps \brk{\uptau}^{-2 + \kappa} |\vec\omg|.
    \]

    Then \[
        |\vec\omg(t)| \lesssim C_k^2 \eps^2 \int_0^t e^{-\beta(t-s)} \brk{s}^{-4 + 2\kappa}\, \ud s + C_k \eps \int_0^t e^{-\beta(t-s)} \brk{s}^{-2 + \kappa} |\vec\omg(s)|\, \ud s.
    \]
    The first term can be estimated by separating the integral bounds \begin{equation}\label{eqn:separating_integral_est}
        \int_0^t e^{-\beta(t-s)} \brk{s}^{-4 + 2\kappa}\, \ud s = \int_0^{\frac{t}2} e^{-\beta(t-s)} \brk{s}^{-4 + 2\kappa}\, \ud s + \int_{\frac{t}{2}}^t e^{-\beta(t-s)} \brk{s}^{-4 + 2\kappa}\, \ud s 
        \lesssim e^{-\beta t} + \brk{t}^{-4 + 2\kappa}.
    \end{equation}
    Then an application of Gronwall's inequality on $e^{\beta s}|\vec\omg(s)|$ gives \[
        |\vec\omg(\tau)| \lesssim C_k^2 \eps^2 \brk{\tau}^{-4 + 2\kappa}.
    \]
    By taking $\eps$ sufficiently small, we obtain the desired estimate.
\end{proof}

\begin{lemma}\label{lemma_decay_of_bfOmg_phi}
    Under bootstrap assumptions in Section~\ref{Section:BA}, when $R_f$ is sufficiently large and $\eps$ is sufficiently small, we have \[
        |\vec\bfOmg_q|, |\vec\bfOmg_\phi| \lesssim (\eps^{\frac12} + R_f^{-1})
        \eps \brk{t}^{-\frac94 + \kappa} + \eps^{\frac32} \brk{t}^{-4 + 2\kappa},
    \]
    where the constant would not depend on the bootstrap constants $C_k$'s.
\end{lemma}
\begin{proof}
    \textit{Step 1: Reduction to estimating $\vec\Upomg$. } 
    In view of the decomposition \eqref{eqn:q_decomp} of $q$, the extra smallness we could obtain in $\bfOmg(\vec Z_i, \vec Z_\pm) = O(R_f^{-5})$ thanks to the orthogonality $\brk{\varphi_i, \varphi_\mu} = 0$ and the spatial decay of eigenfunctions. Therefore, by \eqref{BA-a_-} and \eqref{BA-a_+}, the difference $|\vec\bfOmg_q - \vec\bfOmg_\phi|$ satisfies $R_f^{-1} \eps \brk{t}^{-\frac94 + \kappa}$ bound with $C_k$-independent implicit constant. Hence, it suffices to prove the stated decay for $\vec\bfOmg_q$.
    Moreover, by Lemma~\ref{lemma_on_omg} and the orthogonality condition \eqref{eqn:ortho_cond_for_vec_bfOmg}, it suffices to prove the estimate for $\vec\Upomg$ instead.

    \textit{Step 2: Reduction to estimating $\vec N_q$.}
    We first note that even if $\widetilde S$ is not a (infinitely) smoothing operator, it is still able to absorb excessive $\p_t$ derivative thanks to the relation \eqref{eqn:relation_widetilde_S_w_S} so that it reduces estimating $\widetilde Sh$ to estimate $Sh$ and $h$.
    
    We now estimate $\vec\Upomg$ in terms of \eqref{eqn:defn_of_vec_Upomg_omg_1}. Recall the expression of $\p_t \vecbfOmg_p - \vec N_p$ in \eqref{eqn:p_t_Omg_p_minus_N_p}.
    
    In view of \eqref{eqn:1st_order_eqn_for_p}, the decay of $\bfOmg((\p_t - M)\vec p, \vec Z_j)$ with extra smallness follows from $\supp f_c \cap \supp Z_j = \varnothing$ and \eqref{eqn:all_decay_of_c_im}. For the part $\bfOmg(\vec p, \p_t\vec Z_j)$ in \eqref{eqn:p_t_Omg_p_minus_N_p}, it generates extra $\dot\wp$ smallness and decay.
    On the other hand, $\vec F_\omg$ is estimated again by Lemma~\ref{lemma_on_omg}.
    Now it suffices to estimate $\vec N_q$, which is the only term left to be controlled in \eqref{eqn:defn_of_vec_Upomg_omg_1_int_form}. 

    \textit{Step 3: Estimating the first $n$ components in $\vec N_q$ to conclude the proof of this lemma for the first $n$ components. }
    For $1 \leq i \leq n$, recall from Lemma~\ref{lemma_zero_modes}, $h^{0j} = \calO(|\ell|)$ and $M\vec\varphi_i = 0$ that \[
        M \vec Z_i = M(\chi \vec\varphi_i) 
        = \begin{pmatrix}
            \calO(|\ell|)\p\chi \\ 
            -\sqrt{|h|} \big((\p\chi) (\p\varphi_i) + (\p^2 \chi) \varphi_i\big) + \calO(|\ell|)\p\chi 
        \end{pmatrix}.
    \]
    The terms involving $\calO(|\ell|)$ gains extra smallness obviously. For other terms, thanks to the antisymmetry of $\vec\bfOmg$, \eqref{eqn:M_Z_pm}, \eqref{eqn:2nd_entry_of_vec_q}, \eqref{eqn:vec_K_in_C_flatn} and $\p^l \chi \simeq R_f^{-l} \chi_{\uprho \sim R_f}$, the desired control \[
        |\vec\bfOmg(\phi, Z_i)| \lesssim \calO(\ell + R_f^{-1}) \eps \brk{t}^{-\frac94 + \kappa}, \quad 1 \leq i \leq n
    \]
    can be reached. One typical term is given by \begin{equation}\label{eqn:inverse_R_1_decay_contribution_for_delta_wp}
        R_f^{-2} \brk{\chi_{\uprho \sim R_f} \phi, \calO(r^{-n+1})} \lesssim C_k R_f^{-2} \eps \brk{\tau}^{-\frac94 + \kappa} \int_{\uprho \sim R_f} r^{-3}\, \ud r \lesssim \eps R_f^{-4} \brk{\uptau}^{-\frac94 + \kappa},
    \end{equation}
    where \eqref{BA-q} is applied.
    Note that we switch from $q$ to $\phi$ with no cost as $M \vec Z_i$ contains either good support condition or extra smallness.
    This concludes the proof that \begin{equation}\label{eqn:bfOmg_i_phi_est_1st_part}
        |\vec\bfOmg(\phi, Z_i)| \lesssim \eps \brk{t}^{-\frac94 + \kappa}, \quad 1 \leq i \leq n.
    \end{equation}
    
    \textit{Step 4: Concluding the proof of the lemma for the last $n$ components by estimating the last $n$ components of $\vec N_q$. }
    Now we estimate $|\vec\bfOmg(\phi, MZ_{n+i})|$.
    Note that \[
    \vec\bfOmg(\phi, MZ_{n+i}) = \vec\bfOmg(\phi, Z_i) + \calO(|\ell|)\big( (\p\chi) (\p \phi) + (\p^2\chi)\phi \big)
    \]
    thanks to Lemma~\ref{lemma_zero_modes} and $M\varphi_{n+i} = \varphi_i$.
    The first term is exactly what we have provided estimates for in  \eqref{eqn:bfOmg_i_phi_est_1st_part}. The error terms can be estimated via bootstrap assumption in \eqref{BA-q} with extra smallness in $\calO(|\ell|)$. 
\end{proof}

Now we are ready to prove the estimate for the parameter derivatives in Proposition~\ref{prop_closing_BA_1}.
\begin{proof}[Proof of \eqref{IBA-dotwp} in Proposition~\ref{prop_closing_BA_1}]
    The estimate for $\dot\wp$ simply follows from \eqref{eqn:modulation_eqn} and \eqref{eqn:vec_G_eqn_derived_from_1st_order_HVMC_1} with the following consideration in mind: 
    \begin{enumerate}
        \item For quadratic terms in $\p_{\Sigma}^{\leq 2} \vec\psi$, we use \eqref{BA-p}, \eqref{BA-q}. 
        \item For the linear contribution in $\bfOmg(\p_t \vec p - M \vec p, Z_i)$, see Step 2 of Lemma~\ref{lemma_decay_of_bfOmg_phi}.
        \item The linear contribution in $\vec N_q$ could be estimated in the same way as in Step 3 and 4 in the proof of Lemma~\ref{lemma_decay_of_bfOmg_phi}.
        \item The linear contribution of $\vec \omg$ is estimated using Lemma~\ref{lemma_on_omg}.
    \end{enumerate}
    Then we arrive at \[
        |\dot\wp| \lesssim \eps^{\frac32} \brk{\uptau}^{-4 + 2\kappa} + (\eps^{\frac12} + R_f^{-1})\eps \brk{\uptau}^{-\frac94 + \kappa}
    \]
    with $C_k$-independent implicit constant.
    We furthermore denote this constant by \begin{equation}\label{eqn:delta_wp_defn}
        \delta_\wp:= \eps^{\frac12} + R_f^{-1},
    \end{equation}
    which is the one we mentioned in Remark~\ref{rmk_on_extra_smallness_delta_wp} and used in Section~\ref{Section:BA}.
    For higher derivatives, we absorb the derivatives by $S$ and the proof follows from the same argument.
\end{proof}

\begin{remark}
    In view of \eqref{eqn:delta_wp_defn}, the estimate in Lemma~\ref{lemma_decay_of_bfOmg_phi} could be rewritten as \[
        |\vec\bfOmg_\phi| \lesssim \delta_\wp \eps \brk{t}^{-\frac94 + \kappa}.
    \]
\end{remark}

Similar to Lemma~\ref{lemma_decay_of_bfOmg_phi}, we establish the following for $F_\pm$. 

\begin{lemma}\label{lemma_decay_of_F_pm}
    Under bootstrap assumptions in Section~\ref{Section:BA}, when $R_f$ is sufficiently large and $\eps$ is sufficiently small, we have (see \eqref{eqn:defn_of_F_pm} for the formula of $F_\pm$) \[
        |F_\pm| \lesssim (\eps + R_f^{-1} + C_k\delta_\wp)
        \eps \brk{t}^{-\frac94 + \kappa} \lesssim \delta_\wp^{\frac12} \eps \brk{t}^{-\frac94 + \kappa},
    \]
    where the implicit constants do not depend on the bootstrap constants $C_k$'s.
    As a corollary, $|\bfOmg(\phi, Z_\pm)|$ satisfies the same decay rate estimate.
\end{lemma}
\begin{proof}
     Recall that $F_\pm$ is given by \eqref{eqn:defn_of_F_pm}, where we note that the first two terms on the right hand side of which has nice decay property with extra smallness $\calO(\ell)$ or $R_f^{-1}$ coming from $\calE_\pm$, which is in turn due to the support condition of the error $\calE_\pm$ in \eqref{eqn:M_Z_pm}. 
     
     It suffices to establish desired control on the linear term $\bfOmg(\vec F_2, \vec Z_\mp)$ with $\vec F_2$ given in \eqref{eqn:vec_F_2_defn}.
    As mentioned in the remark after \eqref{eqn:vec_F_2_defn}, the dependence on $\p_{\Sigma}^{\leq 2} \vec\psi$ comes with coefficients with extra smallness and decay (in other words, they are nonlinear in view of the decay). This ensures that the presence of $\vec p$ is accompanied by other terms with $\tau$-decay. The analysis of $f_c + g_1 + \p_\tau g_2$ is directly from the decay \eqref{eqn:all_decay_of_c_im} with the support condition of $Z_\mp$. Finally, for the linear dependence on $\dot\wp$, we read from the definition of $\vec K$ in \eqref{eqn:vec_K_in_C_flatn} that it is of the form $\calO(\dot\wp)$ and we estimate it via \eqref{BA-dotwp}. See \eqref{rmk_on_extra_smallness_delta_wp} for definition of $\delta_\wp$.
    Thus, we can conclude the desired estimate for $F_\pm$ stated in this lemma. 

    In addition, examining \eqref{eqn:ortho_cond_on_phi_for_Z_mu} gives the same decay rate for $|\bfOmg(\phi, Z_\pm)|$. 
\end{proof}

Finally, we prove the last two estimates in Proposition~\ref{prop_closing_BA_1}. 
\begin{proof}[Proof of \eqref{IBA-a_-}, \eqref{IBA-a_+} in Proposition~\ref{prop_closing_BA_1}]
    For $a_-$, we write \[
        a_-(t) = a_-(0) e^{-\mu t} + \int_0^t e^{-\mu t} S(e^{\mu s} F_-(s))\, \ud s.
    \]
    Note that the second term could be estimated directly by Lemma~\ref{lemma_decay_of_F_pm} as in \eqref{eqn:separating_integral_est} and it yields that \begin{equation}\label{eqn:est_of_a_-}
        |a_-(t) - a_-(0)e^{-\mu t}| \lesssim \delta_\wp^{\frac12} \eps \brk{t}^{-\frac94 + \kappa}.
    \end{equation}

    For $a_+$, the case for $k = 0$ is directly included in \eqref{eqn:trapping_assumption}. For higher $k$, we need to use expand \eqref{eqn:diff_eqn_for_a_pm}. Taking $k = 1$ as an example, \[
        \dot a_+ = \mu a_+ + e^{\mu t} S(e^{-\mu t F_+})
    \]
    can be estimated via \eqref{eqn:trapping_assumption} and Lemma~\ref{lemma_decay_of_F_pm}. Inductively, this proves \eqref{IBA-a_+}.
\end{proof}

Even though we have established some estimates indicating the decay rate for $\dot\wp$ and its arbitrary order derivatives, we would naturally expect to show better decay for higher derivatives, say $\ddot\wp$. If one recalls the proof, all the excessive derivatives are absorbed by the smoothing operator and therefore, no improved decay is exploited for higher derivatives before. In the following, we commute one derivative with $S$ and absorb all the extra derivatives in $S$ as before. Moreover, we estimate at the level of $L^2_t$ with the advantage of sufficient available energy bounds.
\begin{remark}
    To get pointwise bound via the $r^p$-weighted vector field method, one needs to go to one $\p_\uptau$-derivative higher in energy space. This is why we discuss $L^2_t$-based estimates instead for higher derivatives.
\end{remark}

To start with, one could derive the following equation for $\ddot\wp$ from \eqref{eqn:p_t_bfOmg}, \eqref{eqn:vec_G_eqn_derived_from_1st_order_HVMC_1} and \eqref{eqn:modulation_eqn} by differentiating in $t$: \[
    \ddot\wp = \vec G_{\mathrm{prime}}(S\p_{\Sigma}^{\leq 2}\vec\psi, \dot\wp, \vec F),
\]
where $\vec G_{\mathrm{prime}}$ satisfies \[
    |\vec G_{\mathrm{prime}}| \lesssim |\dot\wp|^2 + |S\p_{\Sigma}^{\leq 2}\vec\psi| \left|\frac{\ud}{\ud t}\big(S\p_{\Sigma}^{\leq 2}\vec\psi\big)\right| + |\dot\wp| |S\p_{\Sigma}^{\leq 2}\vec\psi|^2 + \left|\p_t \vec F\right|
\]
with \begin{equation}\label{eqn:vec_F_in_Gprime}
    \vec F = S(\p_t \vec\bfOmg_p - \vec N_p - \vec N_q) - \beta\vec\omg - \vec F_\omg
\end{equation}
The nonlinear terms can be simply estimated by pointwise bounds. We summarize the proof of \eqref{IBA-ddotwp-L2} in the following lemma with an even stronger conclusion.

\begin{lemma}\label{lemma_control_higher_der_of_parameters}
    Under bootstrap assumptions in Section~\ref{Section:BA}, when $R_f$ is sufficiently large and $\eps$ is sufficiently small, we have 
    \[
    \|\bfOmg(\p_\uptau \vec\phi, \vec Z_j)\|_{L^2_\uptau[\tau_1, \tau_2]} + \|\dot a_+\|_{L^2_\uptau [\tau_1, \tau_2]} + \|\dot a_- + \mu a_-(0) e^{-\mu t}\|_{L^2_\uptau [\tau_1, \tau_2]} + \|\ddot\wp\|_{L^2_\uptau [\tau_1, \tau_2]} \lesssim (C_\trap + \delta_\wp) \eps \brk{\uptau}^{-\frac{9}{4} + \kappa},
    \]
    where the constant would not depend on the bootstrap constants $C_k$'s.
\end{lemma}
\begin{proof}
    From the discussion above, it suffices to control $\|\p_t \vec F\|_{L^2_\uptau}$ with $F$ given above in \eqref{eqn:vec_F_in_Gprime}. The proof is similar to the ones before. 

    \textit{Step 0: Estimating $\p_t(\beta\vec\omg + \vec F_\omg)$. }
    This is trivial due to the  $L^\infty$ bound in Lemma~\ref{lemma_on_omg}.

    \textit{Step 1: Estimating $\p_t(\p_t \vecbfOmg_p - \vec N_p)$. }
    For $\p_t f_c$, we estimate by \eqref{BA-ddotwp-L2}, which possess an extra small coefficient that allows us to absorb to the left hand side. For $\p_t(g_1 + \p_\uptau g_2)$, we estimate via \eqref{eqn:L2_est_of_c_im}, which in turn by the LE norm of $\p^2_\uptau p$. The other part related to $\bfOmg(\vec p, \p_t\vec Z_j)$ in \eqref{eqn:p_t_Omg_p_minus_N_p} will be bounded via \eqref{BA-dotwp} and \eqref{BA-p} since it contains $\dot\wp$.

    \textit{Step 2: Estimating $\p_t \vec N_q$. }
    First, we need to reduce to $\p_t \vec N_\phi$ by allowing small constant times $\|\dot a_\pm\|_{L^2_\uptau}$ error, which can be absorbed to the left hand side. Next, we expand $M\vec Z_i$ like in Step 3 of the proof for Lemma~\ref{lemma_decay_of_bfOmg_phi}, where the typical term in \eqref{eqn:inverse_R_1_decay_contribution_for_delta_wp} is instead estimated by \[
        R_f^{-2} \|\brk{\chi_{\uprho \sim R_f} \p_\uptau \phi, \calO(r^{-n+1})}_{L^2_x}\|_{L^2_\uptau} \lesssim C_k R_f^{-2} \| \chi_{\uprho \sim R_f} \uprho^{-2}\|_{L^\infty L^2} \|\p_\uptau \phi\|_{LE} \lesssim \eps \delta_\wp \brk{\uptau}^{-\frac94 + \kappa},
    \]
    where the last step follows from \eqref{BA-q-LE}.
    As a side product, we also get the estimate for $\bfOmg(\p_\uptau \vec\phi, \vec Z_j)$ itself and we also include this in the statement as well.

    \textit{Step 3: Estimating $a_\pm$. }
    Finally, we need to estimate $L^2_\uptau$ norm of $\dot a_\pm$ on the left hand side. Note that $\vec F_2$ can be simply estimated by pointwise bootstrap assumptions except the contribution from $\p_t (f_c + g_1 + \p_\uptau g_2)$, which is bounded via \eqref{eqn:L2_est_of_c_im}. The estimate for $a_-$ is standard.
    For $\dot a_+$, we differentiate \eqref{eqn:diff_eqn_for_a_pm} in $t$ and then integrate along $[t, \tau_f]$: \[
        \dot a_+(t) = \dot a_+(\tau_f) e^{-\mu(\tau_f - t)} -  e^{\mu t} \int_t^{\tau_f} S(e^{-\mu s} \dot F_+)\, \ud s.
    \]
    Note that we already closed the bootstrap of \eqref{IBA-a_+}. By directly estimating $\dot a_+(\tau_f)$ in the first term pointwisely, we obtain \[
        \|\dot a_+(\tau_f) e^{-\mu(\tau_f - t)}\|_{L^2_\uptau[t, \tau_f]} \lesssim (C_\trap \eps + \delta_\wp \eps) \brk{\tau_f}^{-\frac{9}{4} + \kappa} \lesssim (C_\trap \eps + \delta_\wp \eps) \brk{t}^{-\frac{9}{4} + \kappa}.
    \]
    For the second term, we rewrite it as \[
        \int_t^{\tau_f} e^{-\mu(s - t)} \left(e^{\mu s}S(e^{-\mu s} \dot F_+)\right)\, \ud s
    \]
    and then invoke Schur's test with kernel $e^{-\mu(s-t)} \chi_{s \geq t}$. This allows us to conclude that \[
        \|\dot a_+\|_{L^2_\uptau[\tau_1, \tau_2]} \lesssim (C_\trap \eps + \delta_\wp \eps) \brk{\tau_1}^{-\frac{9}{4} + \kappa} + \|e^{\mu s}S(e^{-\mu s} \dot F_+)\|_{L^2_\uptau[\tau_1, \tau_2]}.
    \]
    Finally, estimating $F_+$ in $L^2$ and absorbing terms with extra smallness to the left hand side, it completes the proof.
\end{proof}

\section{Late time tails and proof of Proposition~\ref{prop_closing_BA_2}}\label{section_rp_methods_hyperboloidal}

\subsection{Derivation of the operator \texorpdfstring{$\calP_\graph$}{P\_graph}}\label{Section:derivation_of_calP_graph_op}

\subsubsection{Representation \texorpdfstring{$\Box_m$}{Box\_m} in \texorpdfstring{$(\tau, r, \theta)$}{(tau, r, theta)} coordinates}
Since in the asymptotically hyperboloidal region $\calC_{hyp}$, we have \[
    \begin{cases}
        x^0 = \tau + \brk{r} \\
        x' = \eta(\tau) + r\Theta(\theta)
    \end{cases}
    ,
\]
one can write \[\begin{aligned}
    \ud s^2 &= -(\ud x^0)^2 + (\ud x')^2 
    = -(\ud\tau + \frac{r}{\brk{r}}\, \ud r)^2 + (\eta'\, \ud \tau + \Theta\, \ud r + r\Theta_a \ud \theta^a)^2 \\
    &= - (1 - (\eta')^2)\, \ud \tau^2 + \brk{r}^{-2}\, \ud r^2
    + r^2 \ringsg_{ab}\, \ud \theta^a \ud \theta^b + 2(\eta' \cdot \Theta - \frac{r}{\brk{r}})\, \ud \tau \ud r + 2 r \eta' \cdot \Theta_a \, \ud \tau \ud\theta^a,
\end{aligned}
\]
where we denote $\Theta_a:= \p_{\theta^a} \Theta$.

Therefore, we could write the metric in a matrix form in $(\tau, r, \theta)$ coordinates as 
$m = m_0 + m_1$, where
\[\begin{aligned}
    m_0 &= \begin{pmatrix}
       - (1 - (\eta')^2) &  -(1 - \Theta \cdot \eta') & r\eta' \cdot \Theta_a \\ 
       -(1 - \Theta \cdot \eta') & 0 & 0 \\
       r\eta' \cdot \Theta_a & 0 & r^2 \ringsg_{ab} \\
    \end{pmatrix}, \\
    m_1 &= \begin{pmatrix}
        0 & 1 - \frac{r}{\brk{r}} & 0 \\
        1 - \frac{r}{\brk{r}} & \brk{r}^{-2} & 0 \\
        0 & 0 & 0 \\
    \end{pmatrix}
    = \brk{r}^{-2}\begin{pmatrix}
        0 & \frac12 & 0 \\
        \frac12 & 1 & 0 \\
        0 & 0 & 0 \\
    \end{pmatrix} + \begin{pmatrix}
        0 & \calO(r^{-4}) & 0 \\
        \calO(r^{-4}) & 0 & 0 \\
        0 & 0 & 0 \\
    \end{pmatrix}
\end{aligned}
\]
where $m_1 = \calO(\brk{r}^{-2})$.

Note that $\ringsg$ is a diagonal matrix, it follows that $
\displaystyle\{\Theta, \Theta_a/\sqrt{\ringsg_{aa}}, \ a = 1, \cdots, n-1 \}
$
form an orthonormal basis. 
Thus \[
    (\Theta_b \cdot \eta') (\Theta^b \cdot \eta') = |\eta'|^2 - (\Theta \cdot \eta')^2.
\]
Using this identity, it is easy to compute the inverse of $m_0$ in the $(\tau, r, \theta)$ coordinates  \begin{equation}\label{eqn:m_0_inv_formula}
    m_0^{-1} = \begin{pmatrix}
    0 & -\frac{1}{1 - \Theta \cdot \eta'} & 0 \\ 
    -\frac{1}{1 - \Theta \cdot \eta'} & \frac{1 + \Theta \cdot \eta'}{1 - \Theta \cdot \eta'} & \frac{\Theta^b \cdot \eta'}{r(1 - \Theta \cdot \eta')} \\
    0 & \frac{\Theta^b \cdot \eta'}{r(1 - \Theta \cdot \eta')} & r^{-2} \ringsg^{-1}, \\ 
\end{pmatrix}
\end{equation}
where we use the notation $\Theta^b = \ringsg^{ba}\Theta_a$.
Moreover, if we write $m^{-1} = m_0^{-1} + \tilde m_1$, then 
\[
    I = mm^{-1} = (m_0 + m_1) (m_0^{-1} + \tilde m_1) = I + m_0 \tilde m_1 + m_1 m_0^{-1},
\]
which yields $\tilde m_1 = -m_0^{-1} m_1 m_0^{-1}$. Expanding this gives \begin{equation}\label{eqn:tilde_m_1_formula}
    \begin{aligned}
    \tilde m_1 
    % = -m_0^{-1} m_1 m_0^{-1} \\
%     = & - \begin{pmatrix}
%     0 & -\frac{1}{1 - \Theta \cdot \eta'} & 0 \\ 
%     -\frac{1}{1 - \Theta \cdot \eta'} & \frac{1 + \Theta \cdot \eta'}{1 - \Theta \cdot \eta'} & \frac{\Theta^b \cdot \eta'}{r(1 - \Theta \cdot \eta')} \\
%     0 & \frac{\Theta^b \cdot \eta'}{r(1 - \Theta \cdot \eta')} & r^{-2} \ringsg^{-1}, \\ 
% \end{pmatrix}
% \begin{pmatrix}
%         0 & 1 - \frac{r}{\brk{r}} & 0 \\
%         1 - \frac{r}{\brk{r}} & \brk{r}^{-2} & 0 \\
%         0 & 0 & 0 \\
%     \end{pmatrix}
%     \begin{pmatrix}
%     0 & -\frac{1}{1 - \Theta \cdot \eta'} & 0 \\ 
%     -\frac{1}{1 - \Theta \cdot \eta'} & \frac{1 + \Theta \cdot \eta'}{1 - \Theta \cdot \eta'} & \frac{\Theta^b \cdot \eta'}{r(1 - \Theta \cdot \eta')} \\
%     0 & \frac{\Theta^b \cdot \eta'}{r(1 - \Theta \cdot \eta')} & r^{-2} \ringsg^{-1}, \\ 
% \end{pmatrix} \\
= & -\begin{pmatrix}
    - \frac{1 - \frac{r}{\brk{r}}}{1 - \Theta \cdot \eta'} & -\frac{\brk{r}^{-2}}{1 - \Theta \cdot \eta'} & 0 \\
    \frac{1 + \Theta \cdot \eta'}{1 - \Theta \cdot \eta'}(1 - \frac{r}{\brk{r}}) & -(1 - \frac{r}{\brk{r}})\frac{1}{1 - \Theta \cdot \eta'} + \frac{1 + \Theta \cdot\eta'}{1 - \Theta\cdot\eta'}\brk{r}^{-2} & 0 \\ 
    \frac{\Theta^b \cdot \eta'}{r(1 - \Theta \cdot \eta')}(1 - \frac{r}{\brk{r}}) & \frac{\Theta^b \cdot \eta'}{r\brk{r}^2(1 - \Theta \cdot \eta')} & 0 \\
\end{pmatrix}
\begin{pmatrix}
    0 & -\frac{1}{1 - \Theta \cdot \eta'} & 0 \\ 
    -\frac{1}{1 - \Theta \cdot \eta'} & \frac{\Theta \cdot \eta'}{1 - \Theta \cdot \eta'} & \frac{\Theta^b \cdot \eta'}{r(1 - \Theta \cdot \eta')} \\
    0 & \frac{\Theta^b \cdot \eta'}{r(1 - \Theta \cdot \eta')} & r^{-2} \ringsg^{-1}, \\ 
\end{pmatrix} \\
= &\ \brk{r}^{-2}\begin{pmatrix}
        \calO_{\theta,\tau}(1) & \calO_{\theta,\tau}(1) & 0 \\
        \calO_{\theta,\tau}(1) & \calO_{\theta,\tau}(1) & 0 \\ 
        0 & 0 & 0 \\ 
    \end{pmatrix} + \calO(\brk{r}^{-3}).
\end{aligned}
\end{equation}
The notation $\calO_{\theta,\tau}(1)$ denotes a term that only depends on $\theta, \tau$ and is independent of $r$. 
Also, recall that for any term $\calO(\cdot)$, by taking $\p_\tau$, it will produce extra order of $\dot\wp$.
In particular, $\p_\theta\tilde m_1^{\mu\nu} = \calO(\wp r^{-2} + r^{-3})$ and $\p_\tau\tilde m_1^{\mu\nu} = \calO(\dot\wp r^{-2} + r^{-3})$.

For determinants, one can compute \begin{equation}\label{eqn:determinant_relation_m0_m}
\begin{aligned}
    \det m &= \det m_0 \det (I + m_0^{-1}m_1)
    = \det m_0 + \calO(r^{-2}), \\
    |m_0| &= |\det m_0|
    = (1 - \Theta \cdot \eta')^2 r^{2(n-1)} |\ringsg|, \quad 
    \sqrt{|m|} - \sqrt{|m_0|} = \sqrt{|m_0|}\calO(r^{-2}),
\end{aligned}
\end{equation}
where we use the formula $\det(I + \veps A) = 1 + \veps \tr A + O(\veps^2).$
Now we derive an expression for the wave operator $\Box_m$: \begin{equation}\label{eqn:Box_m_U_hyperboloidal}
    \begin{aligned}
    \Box_m U &= \frac1{\sqrt{|m|}} \p_\mu (\sqrt{|m|} m^{\mu\nu} \p_\nu U) \\
    % &= \frac1{\sqrt{|m_0|}} \p_\mu (\sqrt{|m_0|} m_0^{\mu\nu} \p_\nu U)
    % + \frac1{\sqrt{|m|}} \p_\mu \left((\sqrt{|m|} m^{\mu\nu} - \sqrt{|m_0|} m_0^{\mu\nu}) \p_\nu U\right) \\
    % & \quad + \left(\frac1{\sqrt{|m|}} - \frac1{\sqrt{|m_0|}}\right)\p_\mu (\sqrt{|m_0|} m_0^{\mu\nu} \p_\nu U)\\
    &= \frac1{\sqrt{|m_0|}} \p_\mu (\sqrt{|m_0|} m_0^{\mu\nu} \p_\nu U)
    + \frac1{\sqrt{|m|}} \p_\mu \left((\sqrt{|m|} (m^{\mu\nu} - m_0^{\mu\nu}) + (\sqrt{|m|} - \sqrt{|m_0|}) m_0^{\mu\nu})\right) \p_\nu U \\
    & \quad + \left(\frac1{\sqrt{|m|}} - \frac1{\sqrt{|m_0|}}\right)\p_\mu (\sqrt{|m_0|} m_0^{\mu\nu}) \p_\nu U 
    + (m^{\mu\nu} - m_0^{\mu\nu})\p_\mu\p_\nu U \\
\end{aligned}
\end{equation}
Using the closed form of $m_0^{-1}$, 
we write \begin{equation}\label{eqn:Box_m0_U_hyperboloidal}
    \begin{aligned}
    &\frac1{\sqrt{|m_0|}} \p_\mu (\sqrt{|m_0|} m_0^{\mu\nu} \p_\nu U)
    = m_0^{\mu\nu} \p_\mu\p_\nu U + \frac1{\sqrt{|m_0|}} \p_\mu(\sqrt{|m_0|} m_0^{\mu\nu}) \p_\nu U \\
    &= - \frac{2}{1 - \Theta \cdot \eta'} \p_\tau \p_r U 
    + \frac{1 + \Theta \cdot \eta'}{1 - \Theta \cdot \eta'} \p_r^2 U + \frac{2\Theta^\theta \cdot \eta'}{r(1 - \Theta \cdot \eta')} \p_r \p_\theta U + r^{-2} \sDelta_{\bbS^{n-1}} U 
    - \frac{n-1}{r(1- \Theta \cdot \eta')} \p_\tau U \\
    & \quad  + \frac{(n-1)(1 + \Theta \cdot \eta')}{r(1- \Theta \cdot \eta')} \p_r U
    + \frac{n-3}{r^2(1 - \Theta \cdot \eta')} (\Theta^a \cdot \eta') \p_a U + 
    \frac1{\sqrt{|m_0|}} \p_a(\sqrt{|m_0|} m_0^{ra}) \p_r U. \\
\end{aligned}
\end{equation}

\subsubsection{Equations used to derive $r^p$-weighted estimates}
In contrast with \cite{LOS22}, where the geometrically adapted frames $L, \Lbar, T, \Omg$ are chosen to facilitate the computation, we recall the usual incoming, outgoing vector fields for the linear wave equation. Comparing their expansions, one would notice that \[
    T = \calO(1) \p_\tau, \quad L = \calO(1) \p_r
\]
up to a factor of $\calO(r^{-3} + \dot\wp)$ in the asymptotically hyperboloidal region.
Therefore, it is natural to compute things in local coordinates $(\tau, r, \theta)$ instead of using frames.

Let $U = r^{-\frac{n-1}{2}} \tilde U$,  we write 
\[\begin{aligned}
    \p_r U = r^{-\frac{n-1}{2}}\left(-\frac{n-1}{2r} \tilde U + \p_r \tilde U\right),  \quad  
    \p_\tau U &= r^{-\frac{n-1}{2}}\p_\tau\tilde U, \\
    \p_r\p_\tau U = r^{-\frac{n-1}{2}} \left(-\frac{n-1}{2r} \p_\tau\tilde U + \p_r \p_\tau\tilde U\right), \quad \p_r^2 U &= r^{-\frac{n-1}{2}}\left( \frac{(n-1)(n+1)}{4r^2} \tilde U - \frac{n-1}{r} \p_r \tilde U + \p_r^2 \tilde U\right).
\end{aligned}
\]

Therefore, the main contribution in \eqref{eqn:Box_m0_U_hyperboloidal} is 
\[\begin{aligned}
    &- \frac{2}{1 - \Theta \cdot \eta'} \p_\tau \p_r U 
    + \frac{1 + \Theta \cdot \eta'}{1 - \Theta \cdot \eta'} \p_r^2 U
    - \frac{n-1}{r(1- \Theta \cdot \eta')} \p_\tau U
    + \frac{(n-1)(1 + \Theta\cdot\eta')}{r(1- \Theta \cdot \eta')} \p_r U
    + r^{-2} \sDelta_{\bbS^{n-1}} U \\
    =& \frac1{1 - \Theta \cdot \eta'} \left(-2\p_\tau \p_r U 
    + \p_r^2 U
    + \frac{n-1}{r} (\p_r U - \p_\tau U)\right)
    + r^{-2} \sDelta_{\bbS^{n-1}} U + \calO_{\theta,\tau}(\wp) (\p_r^2 U + r^{-1}\p_r U)  \\
    =& \frac{r^{-\frac{n-1}{2}}}{1 - \Theta \cdot \eta'} \left( -\frac{(n-1)(n-3)}{4r^2} \tilde U - 2\p_r\p_\tau\tilde U + \p_r^2 \tilde U\right) + r^{-2} \sDelta_{\bbS^{n-1}} \tilde U + r^{-\frac{n-1}{2}}\calO_{\theta,\tau}(\wp) (\p_r^2 \tilde U + r^{-1}\p_r \tilde U + r^{-2}\tilde U).
\end{aligned}
\]

\begin{remark}\label{rmk:exact_cancel_p_tau}
    Note that there is an exact cancellation for the $\p_\tau \tilde U$ produced from the following two terms: \[
        - \frac{2}{1 - \Theta \cdot \eta'} \p_\tau \p_r U  - \frac{n-1}{r(1- \Theta \cdot \eta')} \p_\tau U
    \]
    in the computation above, which is crucial when we perform the $r^p$-weighted estimates. 
\end{remark}

Combining this with \eqref{eqn:Box_m0_U_hyperboloidal}, we could write \begin{equation}\label{eqn:r_conjugated_Box_m0_U}
    \begin{aligned}
        r^{\frac{n-1}{2}}(1 - \Theta \cdot \eta') \Box_{m_0} U 
    &= -\frac{(n-1)(n-3)}{4r^2} \tilde U - 2\p_r\p_\tau \tilde U + \p_r^2 \tilde U + r^{-2} \sDelta_{\bbS^{n-1}} \tilde U \\
    &+ \calO_{\theta, \tau}(\wp) \p_r^2 \tilde U + \calO_{\theta, \tau}(\wp)r^{-1} (\p_r\p_\theta \tilde U + \p_r \tilde U)
    + \calO_{\theta, \tau}(\wp)r^{-2} (\tilde U + \p_a \tilde U + \sDelta \tilde U) \\
    &+ \calO(r^{-2}) \p_r \tilde U + \calO(r^{-3}) \p_\theta \tilde U,
    \end{aligned}
\end{equation}
where $\calO_{\theta, \tau}(\wp)$ denotes that the constant here is independent of $r$ and contains $\wp$ smallness. Furthermore, in view of \eqref{eqn:determinant_relation_m0_m} and \eqref{eqn:Box_m_U_hyperboloidal}, we obtain extra terms of which some could not be treated as errors. We record \[
    \begin{aligned}
        r^{\frac{n-1}{2}}(1 - \Theta \cdot \eta') \Box_m U 
    &= -\frac{(n-1)(n-3)}{4r^2} \tilde U - 2\p_r\p_\tau \tilde U + \p_r^2 \tilde U + r^{-2} \sDelta_{\bbS^{n-1}} \tilde U \\
    &+ \calO_{\theta, \tau}(\wp) \p_r^2 \tilde U + \calO_{\theta, \tau}(\wp)r^{-1} (\p_r\p_\theta \tilde U + \p_r \tilde U)
    + \calO_{\theta, \tau}(\wp)r^{-2} (\tilde U + \p_a \tilde U + \sDelta \tilde U) \\
    &+ \calO(r^{-2}) \p_r \tilde U + \calO(r^{-3}) \p_\theta \tilde U + \calO(\dot\wp r^{-2} + r^{-3}) (r^{-1}\tilde U + \p_\tau \tilde U) + \tilde m_1^{\mu\nu}\p_\mu\p_\nu \tilde U,
    \end{aligned}
\]
where $\tilde m_1$ is given in \eqref{eqn:tilde_m_1_formula}.

Now we specialize to the $4$-dimensional case. In $(\tau, r, \theta)$ coordinates, and notice from the form of $m^{-1} = m_0^{-1} + \tilde m_1$, \begin{equation}\label{eqn:derivatives_of_Q_decay}
    \nabla^\tau Q = \calO(r^{-5}), \quad \nabla^r Q = \calO(r^{-3}), \quad \nabla^\theta Q = \calO(r^{-5}).
\end{equation}
Putting other errors terms in \eqref{eqn:calP_graph} into consideration, 
we finally derive \begin{equation}\label{eqn:eqn_for_rp_for_U}
    \begin{aligned}
    r^{\frac{3}{2}}\tilde\calP_\graph U 
    &= -\frac{3}{4r^2} \tilde U - 2\p_r\p_\tau \tilde U + \p_r^2 \tilde U + r^{-2} \sDelta_{\bbS^{n-1}} \tilde U \\
    &\ + \calO(\wp r^{-2} + r^{-4}) \tilde U + \calO(r^{-2}) \p_r\p_\tau \tilde U + \calO(\wp + r^{-2}) \p_r^2 \tilde U
    + \calO(\wp r^{-2} + r^{-3}) \p_a\p_b \tilde U \\
    &\ + \tilde m_1^{\tau\tau} \p_\tau^2 \tilde U +
     \calO(\wp r^{-1} + r^{-3})\p_r\p_\theta \tilde U 
     + \calO(r^{-3})\p_\tau\p_\theta \tilde U \\
    &\ + \calO(\wp r^{-1} + r^{-2}) \p_r \tilde U
    + \calO(\wp r^{-2} + r^{-3}) \p_a \tilde U + \calO(\dot\wp r^{-2} + r^{-3}) \p_\tau \tilde U,
\end{aligned}
\end{equation}
where $\tilde \calP_\graph = (1-\Theta\cdot\eta')\calP_\graph$ and \begin{equation}\label{eqn:defn_m1_tt}
    \tilde m_1^{\tau\tau} = - \frac{\brk{r}^{-2}}{(1 - \Theta \cdot \eta')^2} = \calO(r^{-2}).
\end{equation}

\begin{remark}\label{rmk_on_eqn_of_P_graph_coeff}
    The coefficients of the error terms turn out to be extremely delicate and some of which require extra care. 
    \begin{enumerate}[ wide = 0pt, align = left, label=\arabic*. ]
        \item The decay rate of the coefficient in front of $\p_\tau \tilde U$ in \eqref{eqn:eqn_for_rp_for_U} is at least $r^{-2}$ due to the exact cancellation mentioned in Remark~\ref{rmk:exact_cancel_p_tau}. 
        \item We group the terms such that the second line in \eqref{eqn:eqn_for_rp_for_U} is a small perturbation of the first line and hence no threat to the $r^p$-weighted estimates of $\tilde U$.
        \item Despite the preceding bullet point, we claim that multiplying $(1 - \Theta \cdot \eta')$ on both sides is crucial later for deriving $r^p$ estimates for $Y$; see Section~\ref{section_Y_vf_K} for its definition. This leads to the absence of $\calO(\wp)\p_r\p_\tau \tilde U$ term on the right hand side, which further corresponds to the absence of $\calO(\wp) r^{-1}\p_\tau$ in \eqref{eqn:operator_Rc_defn}. 
        This term is harmless at the level of estimates for $\tilde U$ but will be problematic when deriving $r^p$ estimates for $Y$. See Remark~\ref{rmk_calR_12_not_error_term}.
        \item In the $r^p$ estimates for $\tilde U$ (see Proposition~\ref{prop_rp_for_U}), the exact formula for $\tilde m_1^{\tau\tau}$ does not affect the proof. However, we need to keep track of this in the estimate for $Y$. See Proposition~\ref{prop_rp_est_for_Y_w_G12}.
        \item All other error terms are harmless and can be handled by integration by parts and Cauchy-Schwarz inequality. See the proof of Proposition~\ref{prop_rp_for_U}, Lemma~\ref{lemma_computation_for_R_c}) and Proposition~\ref{prop_rp_est_for_Y_w_G12}.
    \end{enumerate}
\end{remark}

We denote the main part of the operator on the right hand side by $Q_0$ and we can write out the equation in a succinct way: \begin{equation}\label{eqn:Q_0_defn}
    Q_0:= -\frac{3}{4r^2} - 2\p_r\p_\tau + \p_r^2 + r^{-2} \sDelta_{\bbS^{n-1}}, \quad r^{\frac32}\calP_\graph U = (Q_0 + ErrQ_0) \tilde U.
\end{equation}

\subsection{General \texorpdfstring{$r^p$}{rp}-weighted estimates for \texorpdfstring{$\tilde U$}{U} in dimension \texorpdfstring{$n = 4$}{n=4}}
In this subsection, we develop the Dafermos–Rodnianski \cite{DR10} $r^p$-hierarchy for $\tilde U = r^{\frac32} U$, where $U$ solves the equation \[
    \calP_\graph U = \calF.
\]
We set $\tilde\calF = r^{\frac32} (1 - \Theta \cdot \eta')\calF$.
Let $\chi = \chi_R = \chi_{\geq R}$ be a smooth cut-off function with support in the hyperboloidal region ($R \gg R_f$) and it also satisfies $\chi \equiv 1$ when $r \gtrsim R$, $\supp\chi \subset \{r \gtrsim R\}$.

Before proving the estimates, we define some short hand notation: 
\begin{equation}\label{eqn:shorthand_notation_Ep_Bp}
    \begin{aligned}
    \calE^p[\tilde U](\tau):= \int_{\Sigma_\tau} \chi_R r^p (\p_r \tilde U)^2\, \ud r\, \ud \theta, \quad 
    \calB^p[\tilde U](\tau_1, \tau_2):= \int_{\calD_{\tau_1}^{\tau_2}} \chi_R r^{p-1}\left((\p_r\tilde U)^2 + \frac1{r^2}(\tilde U^2 + |\snabla\tilde U|^2) \right)\, \ud \tau\, \ud r\, \ud \theta.
\end{aligned}
\end{equation}
Also, we recall that the standard energy is given by \begin{equation}\label{eqn:standard_eng_defn}
    E[U](\tau):= \int_{\Sigma_\tau} \chi_{\leq R}\left(|\p U|^2 + \brk{\rho}^{-2}|U|^2 \right)\, \ud V + 
    \int_{\Sigma_\tau} \chi_R \left(|\p_r U|^2 + r^{-2}(|\p_\tau U|^2 + |U|^2 + |\p_\theta U|^2) \right)\, \ud V,
\end{equation}
where \begin{equation}\label{eqn:non_induced_volume_form}
    \ud V = \brk{\uprho}^{n-1}\, \ud \uprho\, \ud \uptheta
    \simeq \begin{cases}
        \, \ud \uprho\, \ud \uptheta, \text{ in the finite } \uprho \text{ region}, \\
        \, r^{n-1}\, \ud r\, \ud \theta, \text{ in the hyperboloidal region}.\\
    \end{cases}
\end{equation}
One should keep in mind that this is not the same as the volume form obtained from the induced metric on $\Sigma_\tau$. See \cite[Appendix A]{Mo15} for more discussion. Note that the zeroth order term is involved in the standard energy due to the presence of $\brk{U, Z_\mu}$ and $\brk{U, Z_j}$'s and the consideration of Hardy's inequality.

We also mention that we usually stick to the volume form $\ud\tau\, \ud r\, \ud \theta$ instead of the induced volume form from the metric $m$. The volume form $\ud\tau\, \ud r\, \ud \theta$ is induced from metric $-\ud\tau\otimes \ud r - \ud r\otimes \ud\tau + \ud r\otimes \ud r + \ringsg_{\bbS^{n-1}}$. and hence the divergence theorem is applied corresponding to this metric. 

\subsubsection{Preliminary $r^p$-estimates for $\tilde U$}
\begin{proposition}\label{prop_rp_for_U}
    Let $U$ be a solution to $\calP_\graph U = \calF$. We assume that the bootstrap assumption \eqref{BA-dotwp} and the smallness assumption of $\calO(\wp)$ are satisfied.
    For any $0 < \alpha \ll 1$, there exists $\tilde R, \tilde\eps$ such that for all $R > \tilde R$, $\eps < \tilde\eps$, for all $\tau_1 < \tau_2$, all $\delta < p < 2 - 2\alpha$,
    \begin{equation}\label{eqn:main_rp_est}
        \begin{aligned}
            \sup_{\tau_1 \leq \tau \leq \tau_2} \calE^p[\tilde U](\tau) + \calB^p[\tilde U](\tau_1, \tau_2) 
            \lesssim_{\alpha, \delta} &\ \calE^p[\tilde U](\tau_1) + \int_{\calD_{\tau_1}^{\tau_2}} \left(\chi r^p \left|\tilde\calF \left(\p_r \tilde U + \calO(r^{-3}) \p_\theta\tilde U\right)\right| + |\Err_\chi[\tilde U]|\right)\, \ud \tau\, \ud r\, \ud\theta \\
            &\ + \delta \int_{\calD_{\tau_1}^{\tau_2}} \chi r^{-1-\alpha} |\p_\tau \tilde U|^2\, \ud \tau\, \ud r\, \ud \theta + \delta \sup_{\tau_1 \leq \tau \leq \tau_2} E[U](\tau),
        \end{aligned}
    \end{equation}
    where $\delta$ is a small constant with $|\delta| \lesssim \tilde\eps + \tilde R^{-\alpha}$, $\delta < 2\alpha$ and $\Err_\chi[\tilde U]$ denotes all terms multiplied by $\chi'$ and squares of derivatives of $\tilde U$ up to order $1$. 
\end{proposition}
\begin{remark}
    Later we will fix $\alpha$ when we apply this estimate so that $\kappa > 2\alpha$. Here, $\kappa$ is the small loss of the decay rate that we assume in the bootstrap assumptions. 

    When invoking integration by parts below involving error terms, one should keep in mind that \[
        \p_r \calO(\wp r^{-k}) = \calO(\wp r^{-k-1}), \quad 
        \p_\theta \calO(\wp r^{-k}) = \calO(\wp r^{-k}), \quad 
        \p_\tau \calO(r^{-k}) = \calO(\dot\wp r^{-k}).
    \]
\end{remark}

\begin{proof}
    We multiply \eqref{eqn:eqn_for_rp_for_U} by $\chi r^p \p_r \tilde U$ and obtain \begin{equation}\label{eqn:preliminary_rp_IBP}
        \begin{aligned}
    \int_{\calD_{\tau_1}^{\tau_2}} \chi\left(- r^p \p_\tau (\p_r \tilde U)^2 + \frac12 r^p \p_r (\p_r\tilde U)^2
    - \frac38 r^{p-2} \p_r \tilde U^2 - \frac12 r^{p-2} \p_r|\snabla \tilde U|^2\right)\, \ud \tau\, \ud r\, \ud \theta
    + \calR = \int_{\calD_{\tau_1}^{\tau_2}} \chi r^p \tilde\calF \p_r \tilde U, 
\end{aligned}
    \end{equation}
where the remainder $\calR = \sum_{j=1}^6 \calR_j$ is given by 
\[
    \begin{aligned}
        &\calR_1 = \int_{\calD_{\tau_1}^{\tau_2}} \chi  r^p \left(\calO(\wp r^{-2} + r^{-4}) \tilde U + \calO(r^{-2}) \p_r\p_\tau \tilde U + \calO(\wp + r^{-2}) \p_r^2 \tilde U \right) \p_r \tilde U \, \ud \tau\, \ud r\, \ud \theta, \\
        &\calR_2 = \int_{\calD_{\tau_1}^{\tau_2}} \chi r^p \calO(\wp r^{-2} + r^{-3}) \p_a\p_b \tilde U \p_r \tilde U\, \ud \tau\, \ud r\, \ud \theta, 
        \quad \calR_3 = \int_{\calD_{\tau_1}^{\tau_2}} \chi r^p \calO(r^{-2}) \p_\tau^2 \tilde U \p_r \tilde U\, \ud \tau\, \ud r\, \ud \theta, \\
        &\calR_4 = \int_{\calD_{\tau_1}^{\tau_2}} \chi r^p \calO(\wp r^{-1} + r^{-3})\p_r\p_\theta \tilde U \p_r \tilde U\, \ud \tau\, \ud r\, \ud \theta, 
        \quad \calR_5 = \int_{\calD_{\tau_1}^{\tau_2}} \chi r^p \calO(r^{-3})\p_\tau\p_\theta \tilde U \p_r \tilde U\, \ud \tau\, \ud r\, \ud \theta, \\
        &\calR_6 = \int_{\calD_{\tau_1}^{\tau_2}} \chi r^p \left( \calO(\wp r^{-1} + r^{-2}) \p_r \tilde U + \calO(\wp r^{-2} + r^{-3}) \p_a \tilde U + \calO(\dot\wp r^{-2} + r^{-3}) \p_\tau \tilde U \right) \p_r \tilde U \, \ud \tau\, \ud r\, \ud \theta, \\
    \end{aligned}  
\]
    From \eqref{eqn:preliminary_rp_IBP}, we can obtain a weighted estimates as follows: \begin{equation}\label{eqn:preliminary_rp_est}
        \begin{aligned}
            \sup_{\tau_1 \leq \tau \leq \tau_2}\int_{\Sigma_\tau} \chi r^p (\p_r \tilde U)^2\, \ud r\, \ud \theta + \frac12 \int_{\calD_{\tau_1}^{\tau_2}} \chi \left( pr^{p-1}(\p_r\tilde U)^2 + (2-p)\left(\frac34 r^{p-3} \tilde U^2 + r^{p-3}|\snabla\tilde U|^2\right) \right)\, \ud \tau\, \ud r\, \ud \theta \\
            \leq \int_{\Sigma_{\tau_1}} \chi r^p (\p_r \tilde U)^2\, \ud r\, \ud \theta + \left|\int_{\calD_{\tau_1}^{\tau_2}} \chi r^p \tilde\calF \p_r \tilde U\, \ud \tau\, \ud r\, \ud\theta\right|+ \int_{\calD_{\tau_1}^{\tau_2}}|\Err_\chi[\tilde U]|\, \ud \tau\, \ud r\, \ud\theta + |\calR| .
        \end{aligned}
    \end{equation}
    
    Now our task is to estimate $\calR$.
    We shall use $\delta$ throughout the proof to denote a small constant satisfying \[
        |\delta| \leq C(\eps + R^{-\alpha}),
    \]
    where $\alpha$ is a small constant $0 < \alpha \ll 1$.
    For $\calR_1$, we handle it using divergence theorem and using the smallness of $\wp$ and $R^{-1}$ to get it absorbed to the left hand side of \eqref{eqn:preliminary_rp_est}.

    Before discussing $\calR_2$, we record the following integration by parts identity \begin{equation}\label{eqn:IBP_for_pab_pmu}
        \p_{ab}\tilde U \p_\mu \tilde U
        = \frac12\left(\p_a(\p_b\tilde U \p_\mu \tilde U) + \p_b(\p_a\tilde U \p_\mu \tilde U) - \p_\mu(\p_a\tilde U \p_b \tilde U)\right).
    \end{equation}
    By using \eqref{eqn:IBP_for_pab_pmu} with $\mu = r$, the first two terms can be bounded in terms of the bulk terms after a further integration by parts and both absorbed to the left hand side due to the smallness of $\wp$, $r^{-1}$ in front of $r^{-2}$ type error: \begin{equation}\label{eqn:est_error_theta_r_bulk_in_R_2}
        \begin{aligned}
            \left|\int_{\calD_{\tau_1}^{\tau_2}} \chi r^p \calO(\wp r^{-2} + r^{-3}) \p_\theta \tilde U \p_r \tilde U\, \ud \tau\, \ud r\, \ud \theta\right| \lesssim \delta \int \chi r^{p-1} |\p_r \tilde U|^2 + \chi r^{p-3} |\p_\theta \tilde U|^2\, \ud \tau\, \ud r\, \ud \theta
        \end{aligned}
    \end{equation}
    while the last term of $\calR_2$ after applying \eqref{eqn:IBP_for_pab_pmu} is handled by \[
        \left|\int_{\calD_{\tau_1}^{\tau_2}} \chi r^p \calO(\wp r^{-3} + r^{-4}) \p_a \tilde U \p_b \tilde U\, \ud \tau\, \ud r\, \ud \theta\right| \lesssim \delta \int \chi r^{p-3} |\p_\theta \tilde U|^2 \, \ud \tau\, \ud r\, \ud \theta
    \]
    plus errors $\Err_\chi[\tilde U]$ in the integrand. The computation above could be expected since $\calR_2$ is a small perturbation of the last term in $Q_0 \tilde U$ especially when $a = b$.

    For $\calR_3$, we use integration by parts in $\tau$. Then the flux term is estimated by Cauchy-Schwarz \begin{equation}\label{eqn:estimate_flux_term_for_R3}
    \begin{aligned}
        \left|\int_{\Sigma_\tau} \chi r^p \calO(r^{-2})\p_\tau \tilde U \p_r \tilde U\, \ud r\, \ud \theta\right|
        &\leq \sup_{\tau_1 \leq \tau \leq \tau_2} \left(\veps 
        \int_{\Sigma_\tau} \chi r^p|\p_r \tilde U|^2 \, \ud r\, \ud \theta + \veps^{-1} \int_{\Sigma_\tau} \chi r^{p-4} |\p_\tau \tilde U|^2 \, \ud r\, \ud \theta\right)\\
        &\lesssim \veps \sup_{\tau_1 \leq \tau \leq \tau_2} \calE^p[\tilde U](\tau) + \delta \sup_{\tau_1 \leq \tau \leq \tau_2} E[U](\tau),
    \end{aligned}
    \end{equation}
    where the first term can be absorbed to the left hand side of the final estimate and we gain $\delta$ smallness in the second inequality thanks to the fact $p < 2 - 2\alpha$.
    The bulk term after integration by parts on $\calR_3$ can be estimated by \begin{equation}\label{eqn:bound_3rd_term_in_R3}
        \int_{\calD_{\tau_1}^{\tau_2}} \chi r^{p-3} |\p_\tau \tilde U|^2 \, \ud \tau\, \ud r\, \ud \theta \lesssim 
        R^{-\alpha}\int_{\calD_{\tau_1}^{\tau_2}} \chi r^{-1-\alpha} |\p_\tau \tilde U|^2 \, \ud \tau\, \ud r\, \ud \theta
    \end{equation}
    when $p \leq 2 - 2\alpha$ after a further integration by parts in $r$.

    The next error term $\calR_4$ is simply estimated by integration by parts in $\theta$ and the smallness of $\calO(\wp r^{-1} + r^{-3})$ to get absorbed.
    
    For $\calR_5$, we use the identity \[
        \calO(r^{-3})\p_\tau\p_\theta\tilde U \p_r \tilde U
        = \p_\tau(\calO(r^{-3})\p_\theta\tilde U \p_r \tilde U) - \calO(r^{-3})\p_\theta\tilde U \p_r\p_\tau \tilde U - \calO(\dot\wp r^{-3})\p_\theta\tilde U \p_r \tilde U,
    \]
    where the first term will be turned into a flux term and estimated by Cauchy-Schwarz inequality like in \eqref{eqn:estimate_flux_term_for_R3}. 
    The last term is handled like in \eqref{eqn:est_error_theta_r_bulk_in_R_2}.
    For the second term in the identity above, we replace $\p_r\p_\tau \tilde U$ using the equation $r^{\frac32}\calP_\graph U = \tilde\calF$ with the expression \eqref{eqn:calP_graph}.
    Each term produced by this process can be easily estimated as an error using divergence theorem except the term \[
        \int \chi \calO(r^{p-3}) \p_\theta\tilde U \tilde \calF\, \ud \tau\, \ud r\, \ud \theta,  
    \]
    which involves $\tilde\calF$. There is one more term among the errors which might need some attention, which is of the form $\p_\theta \tilde U \p_a\p_b \tilde U$. This can be estimated using \eqref{eqn:IBP_for_pab_pmu} with $\mu = \theta$. 
    
    For $\calR_6$, we estimate it using Cauchy-Schwarz inequality and the smallness of $\wp$ and $R^{-1}$. Thus \[
    |\calR_6| \lesssim \delta\int_{\calD_{\tau_1}^{\tau_2}} \chi \left(r^{p-1} |\p_r \tilde U|^2 + r^{p-3} |\snabla \tilde U|^2\right) + \chi r^{p-3}|\p_\tau \tilde U|^2\, \ud \tau\, \ud r\, \ud \theta,
    \]
    where the first two terms can be absorbed to the left hand side and the last one can be bounded by \eqref{eqn:bound_3rd_term_in_R3}
    when $p \leq 2 - 2\alpha$.

    Therefore, by combining all these estimates for $\calR_i$ together with \eqref{eqn:preliminary_rp_est}, we obtain \[
    \begin{aligned}
            &\sup_{\tau_1 \leq \tau \leq \tau_2}\int_{\Sigma_\tau} \chi r^p (\p_r \tilde U)^2\, \ud r\, \ud \theta + \int_{\calD_{\tau_1}^{\tau_2}} \chi \left( (p-\delta)r^{p-1}(\p_r\tilde U)^2 + (2-p-\delta) r^{p-3}\left(\tilde U^2 + |\snabla\tilde U|^2\right) \right)\, \ud \tau\, \ud r\, \ud \theta \\
            \lesssim& \int_{\Sigma_{\tau_1}} \chi r^p (\p_r \tilde U)^2\, \ud r\, \ud \theta + \left|\int_{\calD_{\tau_1}^{\tau_2}} \chi r^p \tilde\calF \left(\p_r \tilde U + \calO(r^{-3}) \p_\theta\tilde U\right)\, \ud \tau\, \ud r\, \ud\theta\right|+ \int_{\calD_{\tau_1}^{\tau_2}}|\Err_\chi|\, \ud \tau\, \ud r\, \ud\theta \\
            &+ \delta \int_{\calD_{\tau_1}^{\tau_2}} \chi r^{-1-\alpha} |\p_\tau \tilde U|^2\, \ud \tau\, \ud r\, \ud \theta + \delta \sup_{\tau_1 \leq \tau \leq \tau_2} E[U](\tau).
        \end{aligned}
    \]  
\end{proof}

\subsubsection{Higher order $r^p$-weighted estimates}
Then we follow the exposition in \cite[Section 7]{G23} to show $r^p$-estimates for higher order derivatives of $\tilde U$. In what follows, we sometimes use $\sDelta_S$ to denote the spherical Laplacian.
\begin{lemma}
    Suppose $Q_0 \tilde U = \tilde f$, then $\tilde U_N:= (r\p_r)^N \tilde U$ satisfies \begin{equation}\label{eqn:tilde_U_N_boxm}
        \begin{aligned}
        -2\p_\tau\p_r \tilde U_N + \p_r^2 \tilde U_N - \frac{3}{4r^2} \tilde U_N + \frac1{r^2}\sDelta_S \tilde U_N
        + \frac{(N+1)N}{2r^2} \tilde U_N - N r^{-1}\p_r \tilde U_N - \frac{N}{r^2} \sDelta_S \tilde U_{N-1} &\\
        + N \sum_{k=0}^{N-1} O_\infty(r^{-2}) \tilde U_k + N(N-1) \sum_{k=0}^{N-2}O_\infty(r^{-2}) \sDelta_S \tilde U_k
        = \p_r(r\p_r)^{N-1}r\tilde f = (r\p_r + 1)^N\tilde f&.
    \end{aligned}
    \end{equation}
\end{lemma}
\begin{proof}
    First, thanks to 
        $\p_r (r\tilde f) = (r\p_r + 1)\tilde f$,
    it follows immediately that \[
        \p_r(r\p_r)^{N-1}r\tilde f = (r\p_r + 1)^N\tilde f.
    \]
    Now we prove by induction on $N$.
    When $N = 0, 1$, it is easy to verify that \eqref{eqn:tilde_U_N_boxm} holds true.
    We just make a remark that in the second line of \eqref{eqn:tilde_U_N_boxm}, $N O_\infty(r^{-2}) = O_\infty(r^{-2})$, but we just take $N$ out to emphasize that this constant vanishes when $N = 0$. Also, the same remark applies to $ N(N-1)O_\infty(r^{-2})$.
    Now we suppose that \eqref{eqn:tilde_U_N_boxm} is true for $N\geq 1$. 
    Thanks to \[
        r\p_r^2 \tilde U_N = \p_r(r\p_r \tilde U_N) - \p_r \tilde U_N
        = \p_r \tilde U_{N+1} - r^{-1} \tilde U_{N+1},
    \] 
    we multiply both sides by $r$: \[\begin{aligned}
        -2\p_\tau \tilde U_{N+1} + \p_r \tilde U_{N+1} - r^{-1} \tilde U_{N+1} - \frac{3}{4r} \tilde U_N + \frac1{r}\sDelta_S \tilde U_N 
        + \frac{(N+1)N}{2r} \tilde U_N -N \p_r \tilde U_N - \frac{N}{r}\sDelta_S \tilde U_{N-1}& \\
        + N \sum_{k=0}^{N-1} O_\infty(r^{-1}) \tilde U_k + N(N-1) \sum_{k=0}^{N-2}O_\infty(r^{-1}) \sDelta_S \tilde U_k
        = (r\p_r)^{N}r\tilde f &
    \end{aligned}
    \]
    Now we apply $\p_r$ on both sides and notice that $\p_r \tilde U_k = r^{-1} \tilde U_{k+1}$, the left hand side become \[\begin{aligned}
        &-2\p_\tau\p_r \tilde U_{N+1} + \p_r^2 \tilde U_{N+1} + \frac{1}{r^2} \tilde U_{N+1} - r^{-1}\p_r\tilde U_{N+1} 
        - \frac{3}{4r^2} \tilde U_N + \frac1{r^2}\sDelta_S \tilde U_{N+1}  - \frac1{r^2} \sDelta_S \tilde U_N \\ 
        &+ \frac{(N+1)N}{2r^2} \tilde U_{N+1} 
        -N \p_r^2 \tilde U_N - \frac{N}{r^2}\sDelta_S \tilde U_{N}
        + \sum_{k=0}^{N} O_\infty(r^{-2}) \tilde U_k + \sum_{k=0}^{N-1}O_\infty(r^{-2}) \sDelta_S \tilde U_k.
    \end{aligned}
    \]
    In this expression, we further compute \[\begin{aligned}
        &\frac{1}{r^2} \tilde U_{N+1} - r^{-1}\p_r\tilde U_{N+1} 
        - \frac1{r^2} \sDelta_S \tilde U_N  
        + \frac{(N+1)N}{2r^2} \tilde U_{N+1} 
        -N \p_r^2 \tilde U_N - \frac{N}{r^2}\sDelta_S \tilde U_N \\
        =& \frac{N^2 + N + 2}{2r^2} \tilde U_{N+1} - \frac1{r}\p_r \tilde U_{N+1} - \frac{N+1}{r^2}\sDelta_S \tilde U_N - \frac{N}{r}(\p_r \tilde U_{N+1} - r^{-1}\tilde U_{N+1}) \\
        =& \frac{(N+2)(N+1)}{2r^2}\tilde U_{N+1} - \frac{N+1}{r}\p_r \tilde U_{N+1} - \frac{N+1}{r^2} \sDelta_S \tilde U_N,
    \end{aligned}
    \]
    which completes the induction.
\end{proof}

\begin{lemma}
    For any triple $N = (n_1, n_2, n_3)$, $\tilde U$, $\tilde f$ satisfying $Q_0 \tilde U = \tilde f$, we introduce the notation \begin{equation}\label{eqn:notation_tilde_U_N_triple}
        \tilde U_N:= \tilde U_{n_1, n_2, n_3}:= (r\p_r)^{n_3}(\p_\theta)^{n_2}(\p_\tau)^{n_1} \tilde U, \quad 
        \tilde f_N:= \tilde f_{n_1, n_2, n_3}:= (r\p_r + 1)^{n_3}(\p_\theta)^{n_2}(\p_\tau)^{n_1} \tilde f.
    \end{equation}
    Then the following equation is satisfied:
    \begin{equation}
        \label{eqn:tilde_U_n123_boxm}
        \begin{aligned}
            &-2\p_\tau\p_r \tilde U_N + \p_r^2 \tilde U_N - \frac{3}{4r^2} \tilde U_N + \frac1{r^2}\sDelta_S \tilde U_N
            + \frac{(n_3+1)n_3}{2r^2} \tilde U_N - n_3 r^{-1}\p_r \tilde U_N - \frac{n_3}{r^2} \sDelta_S \tilde U_{n_1, n_2, n_3 -1} \\
            &+ n_3 \sum_{|k|=0}^{|N|-1} O_\infty(r^{-2}) \tilde U_k + n_3(n_3-1) \sum_{|k|=0}^{|N|-2}O_\infty(r^{-2}) \sDelta_S \tilde U_k
            = \tilde f_N.
        \end{aligned}
    \end{equation}
\end{lemma}
\begin{proof}
    Note that $[Q_0, \p_\theta] = [Q_0, \p_\tau] = 0$ and hence 
    $Q_0 \tilde U_{n_1, n_2, 0} = \tilde f_{n_1, n_2, 0}$.
    Then the proof follows from the preceding lemma \eqref{eqn:tilde_U_N_boxm}.
\end{proof}

\begin{corollary}
    Suppose $\calP_\graph U = \calF$, then \begin{equation}\label{eqn:tilde_U_N_calP}
        \begin{aligned}
        &-2\p_\tau\p_r \tilde U_N + \p_r^2 \tilde U_N - \frac{3}{4r^2} \tilde U_N + \frac1{r^2}\sDelta_S \tilde U_N
            + \frac{(n_3+1)n_3}{2r^2} \tilde U_N - n_3 r^{-1}\p_r \tilde U_N - \frac{n_3}{r^2} \sDelta_S \tilde U_{n_1, n_2, n_3 -1} \\
        &+ n_3 \sum_{|k|=0}^{|N|-1} O_\infty(r^{-2}) \tilde U_k + n_3(n_3-1) \sum_{|k|=0}^{|N|-2}O_\infty(r^{-2}) \sDelta_S \tilde U_k \\
        &+ \sum_{|k|=0}^{|N|} \left(\calO(\wp r^{-2} + r^{-4}) \tilde U_k + \calO(r^{-2}) \p_r\p_\tau \tilde U_k + \calO(\wp + r^{-2}) \p_r^2 \tilde U_k
        + \calO(\wp r^{-2} + r^{-3}) \p_a\p_b \tilde U_k \right.\\
        & \qquad\qquad + \calO(r^{-2}) \p_\tau^2 \tilde U_k +
         \calO(\wp r^{-1} + r^{-3})\p_r\p_\theta \tilde U_k 
         + \calO(r^{-3})\p_\tau\p_\theta \tilde U_k \\
        & \qquad\qquad \left. + \calO(\wp r^{-1} + r^{-2}) \p_r \tilde U_k
        + \calO(\wp r^{-2} + r^{-3}) \p_a \tilde U_k + \calO(\dot\wp r^{-2} + r^{-3}) \p_\tau \tilde U_k \right)   
        = \tilde\calF_N,
    \end{aligned}
    \end{equation}
    where we refer to \eqref{eqn:notation_tilde_U_N_triple} for the notation.
\end{corollary}
\begin{proof}
    This follows directly from \eqref{eqn:Q_0_defn} by writing \[
        Q_0 \tilde U = \tilde \calF - \Err Q_0 \tilde U   
    \] 
    and applying the preceding lemma. The form of $\Err Q_0$ is given in \eqref{eqn:eqn_for_rp_for_U}.
\end{proof}

Now we use \eqref{eqn:tilde_U_N_calP} to derive higher order $r^p$-weighted estimates. The smallness assumptions on $\delta$ and $\alpha$ shall be the same as in Proposition~\ref{prop_rp_for_U}.
\begin{proposition}\label{prop_higher_order_rp_for_U}
    Let $U$ be a solution to $\calP_\graph U = \calF$. We assume that the bootstrap assumption \eqref{BA-dotwp} and the smallness assumption of $\calO(\wp)$ are satisfied.
    For any $\tau_1 < \tau_2$, for $R$ sufficiently large and $\eps$ small enough, 
    there exists a small constant $\delta$ with $|\delta| \lesssim \eps + R^{-\alpha}$ for some $\alpha \ll 1$
    such that for $\delta < p < 2 - 2\alpha$, 
    for any $N \in \bbN$, \begin{equation}\label{eqn:main_higher_order_rp_est}
        \begin{aligned}
            &\sum_{0 \leq |k| \leq N} \sup_{\tau_1 \leq \tau \leq \tau_2} \calE^p[\tilde U_k](\tau) + \calB^p[\tilde U_k](\tau_1, \tau_2) \\
            \lesssim& \ \sum_{0 \leq |k| \leq N} \left(\calE^p[\tilde U_k](\tau_1) + \int_{\calD_{\tau_1}^{\tau_2}} \chi r^p \left|\tilde\calF_k \left(\p_r \tilde U_k + \calO(r^{-3}) \p_\theta\tilde U_k\right)\right|\, \ud \tau\, \ud r\, \ud\theta \right)\\
            &\ + \int_{\calD_{\tau_1}^{\tau_2}} |\Err_{\chi, N}[\tilde U]|\, \ud \tau\, \ud r\, \ud \theta + \sum_{|k| = 0}^{N} \left(\delta \int_{\calD_{\tau_1}^{\tau_2}} \chi r^{-1-\alpha} |\p_\tau \tilde U_k|^2\, \ud \tau\, \ud r\, \ud \theta + \delta \sup_{\tau_1 \leq \tau \leq \tau_2} E[U_k](\tau)\right).
        \end{aligned}
    \end{equation}
    where $\Err_{\chi, N}[\tilde U]$ denotes all terms multiplied by $\chi'$ and squares of derivatives of $\tilde U$ up to order $N+1$. 
\end{proposition}
\begin{proof}
    We prove by induction on $N \in \bbN$. In the following, we specify a triple $N$ with order $N$. To conclude the proof, one just need to sum up all the estimates with order $N$ and use the induction hypothesis. (We abuse notation here with $N$ referring to both a triple $N$ and an integer $N$, alternatively.) 
    
    We multiply \eqref{eqn:tilde_U_N_calP} by $\chi r^p \p_r \tilde U_N$ and integrate over $\calD_{\tau_1}^{\tau_2}$.
    Note that the first four terms can be handled using the same method as the $|N| = 0$ case.
    Among all the other terms in \eqref{eqn:tilde_U_N_calP}, there are two terms
    \begin{equation}\label{eqn:2_terms_in_higher_order_rp_consideration}
        \frac{(n_3+1)n_3}{2r^2} \tilde U_N, \quad - \frac{n_3}{r^2}\sDelta_S \tilde U_{n_1, n_2, n_3-1}
    \end{equation}
    that need some special attention, where they are nonvanishing only when $n_3 \geq 1$.
    Though these two have the ``wrong'' sign, they can be estimated by the induction hypothesis.
    By multiplying by $\chi r^p \p_r \tilde U_N$, we estimate the first one in \eqref{eqn:2_terms_in_higher_order_rp_consideration} as follows: \begin{equation}\label{eqn:est_using_ind_hyp_higher_order_rp}
        \begin{aligned}
        \left|r^{p-2}\tilde U_N \p_r \tilde U_N\right| \lesssim \delta r^{p-1} |\p_r \tilde U_N|^2 + r^{p-3} |\tilde U_N|^2
        \lesssim \delta r^{p-1} |\p_r \tilde U_N|^2 + r^{p-1} |\p_r\tilde U_{n_1, n_2, n_3-1}|^2,
        \end{aligned}
    \end{equation}
    where the first term can be absorbed by the left hand side of the estimates while the other can be estimated using the induction hypothesis.
    The second one in \eqref{eqn:2_terms_in_higher_order_rp_consideration} can be handled by the following differential by parts formula: \[\begin{aligned}
        -r^{p-2} \p_r \tilde U_N \sDelta_S \tilde U_{N-1}
        &= \p_r \left( r^{p-2} \snabla\tilde U_N \snabla \tilde U_{N-1} \right) - r^{p-2} \snabla\tilde U_N \snabla \p_r \tilde U_{N-1} - (p-2) r^{p-3} \snabla\tilde U_N \snabla \tilde U_{N-1} \\
        &= \p_r \left( r^{p-2} \snabla\tilde U_N \snabla \tilde U_{N-1} \right) - r^{p-3} |\snabla\tilde U_N|^2 - (p-2) r^{p-3} \snabla\tilde U_N \snabla \tilde U_{N-1},
    \end{aligned}
    \]
    where the second term has a good sign and the third one can be estimated like in \eqref{eqn:est_using_ind_hyp_higher_order_rp}. Here we slightly abuse a notation $\tilde U_{N - 1} = \tilde U_{n_1, n_2, n_3 - 1}$ for simplicity.
    All the other terms can be estimated using either the induction hypothesis (for those with subscript $< N$) or treated as in the proof of Proposition~\ref{prop_rp_for_U} for error terms $\calR_j$'s (for those with subscript $N$). Therefore, one can conclude the proof.
\end{proof}

\subsubsection{Higher order $r^p$-weighted estimates for $\p_\tau \tilde U_k$ and $\p_\tau^2 \tilde U_k$}
In preceding discussions, we treat all $\p_\tau$, $\p_\theta$, $r\p_r$ in the same way. For future use, we derive separately for those with at least one or two $\p_\tau$ falling on the function. The following proposition would be useful to derive improved decay for quantities with $\p_\tau$ falling on the solution.

\begin{proposition}
    Let $U$ be a solution to $\calP_\graph U = \calF$. With the same assumptions in Proposition~\ref{prop_higher_order_rp_for_U}, we have \begin{equation}\label{eqn:main_higher_order_rp_est_p_tau_U}
        \begin{aligned}
            &\sum_{0 \leq |k| \leq N} \sup_{\tau_1 \leq \tau \leq \tau_2} \calE^p[\p_\tau \tilde U_k](\tau) + \calB^p[\p_\tau \tilde U_k](\tau_1, \tau_2) \\
            \lesssim& \ \sum_{0 \leq |k| \leq N} \left(\calE^p[\p_\tau \tilde U_k](\tau_1) + \int_{\calD_{\tau_1}^{\tau_2}} \chi r^p \left|\p_\tau\tilde\calF_k \left(\p_r \p_\tau\tilde U_k + \calO(r^{-3}) \p_\theta\p_\tau\tilde U_k\right)\right|\, \ud \tau\, \ud r\, \ud\theta \right)\\
            &\ + \sum_{|k| = 0}^{N} \left(\delta \int_{\calD_{\tau_1}^{\tau_2}} \chi r^{-1-\alpha} |\p_\tau^2 \tilde U_k|^2\, \ud \tau\, \ud r\, \ud \theta + \delta \sup_{\tau_1 \leq \tau \leq \tau_2} E[\p_\tau U_k](\tau)\right) \\
            &\ + \int_{\calD_{\tau_1}^{\tau_2}} |\Err_{\chi, N}[\p_\tau\tilde U]|\, \ud \tau\, \ud r\, \ud \theta + \eps^2 \brk{\tau_1}^{-\frac72 + 2\kappa}\sum_{|k| = 0}^{N+1} \sup_{\tau_1 \leq \tau \leq \tau_2} E[\tilde U_k](\tau),
        \end{aligned}
    \end{equation}
    where $\Err_{\chi, N}[\p_\tau\tilde U]$ denotes all terms multiplied by $\chi'$ and squares of derivatives of $\p_\tau\tilde U$ up to order $N+1$. 
\end{proposition}
\begin{proof}
    By differentiating \eqref{eqn:tilde_U_N_calP} by $\p_\tau$, the main contribution stays in the same form with $\tilde U_N$ replaced by $\p_\tau \tilde U_N$. Then we multiply the equation by $\chi r^p \p_r \p_\tau \tilde U_N$ and perform the same kind of analysis. The only slight difference happens when $\p_\tau$ falls on the error terms and we present the treatment for a typical term among the errors as an example \[
        \p_\tau \left(\calO(\wp r^{-2} + r^{-3}) \p_a\p_b\tilde U_N\right)
        = \calO(\wp r^{-2} + r^{-3}) \p_a\p_b (\p_\tau \tilde U_N) + \calO(\dot\wp r^{-2}) \p_a\p_b\tilde U_N.
    \]
    The first summand can be taken care in the exact same way as before. For the second summand, by observing that \begin{equation}\label{eqn:typical_term_after_take_derivative_in_tau}
        \begin{aligned}
        |r^{p-2}\dot\wp \p_a\p_b \tilde U_N \p_r\p_\tau \tilde U_N| &\lesssim \veps r^{p-1} |\p_r \p_\tau \tilde U_N|^2 + r^{p-3}|\dot\wp|^2 |\p_a\p_b\tilde U_N|^2 \\
        &\lesssim \veps r^{p-1} |\p_r \p_\tau \tilde U_N|^2 + \eps^2 \brk{\uptau}^{-\frac{9}{2} + 2\kappa} \calO(r^{-6}) r^{p-3}|\p_a\p_b\tilde U_N|^2,
    \end{aligned}
    \end{equation} 
    we integrate over $\calD_{\tau_1}^{\tau_2} \cap \supp \chi$ and then absorb the first term to the left hand side while the last term can be bounded by $\eps^2 \brk{\tau_1}^{-\frac{7}{2} + 2\kappa} \sup_{\tau_1 \leq \tau \leq \tau_2} \sum_{|k| \leq N+1}E[\tilde U_k](\tau)$. This estimate is very loose and one can easily upgrade this to adapt to their needs. 
\end{proof}

\subsection{Ingredients for extracting pointwise decay and an overview of proof}

\subsubsection{Boundedness of energy and integrated local energy decay estimates}\label{Section:ILED}
In this part, we record the integrated local energy decay estimates that would be used to conclude the pointwise decay in the next section. 
More specifically, these will be applied to deal with the error terms present in \eqref{eqn:main_rp_est} and \eqref{eqn:main_higher_order_rp_est}.
We only highlight some aspects in the proof and refer the details to \cite{LOS22}. 

We work in the global coordinate $(\uptau, \uprho, \uptheta)$ and denote the spacetime region $\calD_{\uptau_1}^{\uptau_2}:= \cup_\uptau \Sigma_\uptau$.
Recall the definition of the standard energy in \eqref{eqn:standard_eng_defn} and the notation therein, then the local energy norm and its dual on any spacetime region $\calD$ is given by \[\begin{aligned}
    \|U\|_{LE(\calD)}^2 &:= \int_{\calD} \left(\chi_{\leq R}\left(|\uprho \p U|^2 + |\uprho U|^2 + |\p_\uprho U|^2 \right) + \chi_R \left(\uprho^{-3-\alpha} |U|^2 + \uprho^{-1-\alpha}|\p U|^2\right)\right) \sqrt{|h|}\, \ud \uprho\, \ud \uptheta\, \ud \uptau,\\
    \|g\|_{LE^*(\calD)}^2 &:= \int_{\calD} \left(\chi_{\leq R} |f|^2 + \chi_R \uprho^{1+\alpha} |f|^2 \right)\sqrt{|h|}\, \ud \uprho\, \ud \uptheta\, \ud \uptau,\\
\end{aligned}
\]
where one can refer to Section~\ref{Section:notation} for the meaning of $\p$ and here $0 < \alpha \ll 1$ is a small constant. We also introduce \[
\|U\|_{L^pL^q(\calD_{\uptau_1}^{\uptau_2})}:= \left(\int_{\uptau_1}^{\uptau_2} \|U\|_{L^q(\Sigma_\uptau)}^p\, \ud \uptau \right)^{\frac1{p}}.
\]

For the purpose of linear estimates, we introduce linear proxies of $\bfOmg_k$: 
\begin{align*}
\begin{split}
\bfUpomg_k(U)(c) &= -\int_{\{\uptau= c\}}Z_k \p_\alpha U n^\alpha \sqrt{|h|}\, \ud \uprho\, \ud \uptheta, \quad k = 1,2,3,4, \\
\bfUpomg_{4+k}(U)(c) &= \int_{\{\uptau= c\}} U Z_k (\p_\alpha \uptau) n^\alpha \sqrt{|h|}\, \ud \uprho\, \ud \uptheta, \quad k=1,2,3,4,\\
\bfUpomg_{\pm\mu}(U)(c) &= \int_{\{\uptau= c\}}(\pm\mu U Z_\mu \p_\alpha \uptau - Z_\mu \p_\alpha U)n^\alpha \sqrt{|h|}\, \ud \uprho\, \ud \uptheta.
\end{split}
\end{align*}
Here $n = (d\uptau)^\sharp$ denotes the spacetime normal to $\Sigma_c$ with respect to $h$.
Note that the difference caused by introducing proxies of $\bfOmg_k$'s are bounded by a small multiple of standard energy norms. These linear proxies are used to bound $\brk{U, Z_k}_{L^2(\Sigma_\tau)}$'s.

We now record the integrated local energy decay estimates (ILED) obtained in \cite{LOS22}.

\begin{proposition}\label{prop_ILED}
    Suppose $U$ satisfies $\calP U = g$, where $\calP$ is given in \eqref{eqn:form_of_P_P_0_P_pert} and \eqref{eqn:calP_pert_in_C_hyp}. For any $\uptau_1 < \uptau_2$, $U$ satisfies the estimates \[\begin{aligned}
        \sup_{\uptau_1 \leq \uptau \leq \uptau_2}\|U\|_{E(\Sigma_\uptau)} + \|U\|_{LE(\calD_{\uptau_1}^{\uptau_2})} \lesssim \|U\|_{E(\Sigma_{\uptau_1})} &+ \|g\|_{L^1L^2(\calD_{\uptau_1}^{\uptau_2})} + \sum_{k \in \{\pm\mu, 1, \cdots, 8\}} \|\bfUpomg_k(U)\|_{L^2([\uptau_1, \uptau_2])}, \\
        \sup_{\uptau_1 \leq \uptau \leq \uptau_2}\|U\|_{E(\Sigma_\uptau)} + \|U\|_{LE(\calD_{\uptau_1}^{\uptau_2})} \lesssim \|U\|_{E(\Sigma_{\uptau_1})} &+ \|g\|_{LE^*(\calD_{\uptau_1}^{\uptau_2})} + \|\p_\uptau g\|_{LE^*(\calD_{\uptau_1}^{\uptau_2})} \\ &+ \|g\|_{L^\infty L^2(\calD_{\uptau_1}^{\uptau_2})} + \sum_{k \in \{\pm\mu, 1, \cdots, 8\}} \|\bfUpomg_k(U)\|_{L^2([\uptau_1, \uptau_2])}. \\
    \end{aligned}
    \]
\end{proposition}
%%If one examines the proof, $L^\infty L^2$ term (appearing in LOS22 but not in OS24) on the RHS can be eliminated. However, whether this term is present or not won't affect the decay argument later. 

\begin{remark}\label{rmk_smoothing_S_p}
    Note that the appearance of $\p_\uptau g$ on the right hand side is the reason why we need to make a smoothing $S_p$ for $g_2$ as there is a loss of derivative if the smoothing were not involved. 
\end{remark}
\begin{remark}
    For the energy boundedness statement, we can apply the product manifold case \cite[Section 2.3]{LOS22} where the coercivity of energy is justified and it is no harm to treat other terms as perturbations. However, to show the top order energy boundedness, we do not need to worry about coercivity of energy anymore as the part that might cause trouble is of lower order. For top order estimates, we need to use multiplier argument to ensure there is no loss of derivatives as we always need one estimate without loss to estimate the nonlinearity later for quasilinear equations.
\end{remark}

Specifying ILED to the case $U = \bbP$ and $U = \mathbb\Phi$ (defined in Section~\ref{Section:overview_of_rp}), respectively, we could obtain the following ILED estimates.

\begin{proposition}
    Assuming the bootstrap assumptions in Section~\ref{Section:BA}, for all $\uptau_1 < \uptau_2$, \begin{align}
         \sup_{\uptau_1 \leq \uptau \leq \uptau_2}\|\bbP_k\|^2_{E(\Sigma_\uptau)} + \|\bbP_k\|^2_{LE(\calD_{\uptau_1}^{\uptau_2})} &\lesssim \sum_{j \leq k} \|\bbP_j\|^2_{E(\Sigma_{\uptau_1})} + \eps^2 \brk{\uptau_1}^{-\frac{5}{2} + 2\kappa},  \quad |k| \leq M, \tag{ILED-$\bbP$}\label{ILED-P}\\
         \sup_{\uptau_1 \leq \uptau \leq \uptau_2}\|\p_\uptau\bbP_k\|^2_{E(\Sigma_\uptau)} + \|\p_\uptau\bbP_k\|^2_{LE(\calD_{\uptau_1}^{\uptau_2})} &\lesssim \sum_{j \leq k} \|\p_\uptau\bbP_j\|^2_{E(\Sigma_{\uptau_1})} + \eps^2 \brk{\uptau_1}^{-\frac{9}{2} + 2\kappa}, \quad |k| \leq M - 1. \tag{ILED-$\p_\uptau\bbP$} \label{ILED-dP} \\
         \sup_{\uptau_1 \leq \uptau \leq \uptau_2}\|\mathbb\Phi_k\|^2_{E(\Sigma_\uptau)} + \|\mathbb\Phi_k\|^2_{LE(\calD_{\uptau_1}^{\uptau_2})} &\lesssim \sum_{j \leq k} \|\mathbb\Phi_j\|^2_{E(\Sigma_{\uptau_1})} + \eps^2 \brk{\uptau_1}^{-\frac{5}{2} + 2\kappa}, \quad |k| \leq M - 20, \tag{ILED-$\mathbb\Phi$}\label{ILED-Phi}\\
         \sup_{\uptau_1 \leq \uptau \leq \uptau_2}\|\p_\uptau\mathbb\Phi_k\|^2_{E(\Sigma_\uptau)} + \|\p_\uptau\mathbb\Phi_k\|^2_{LE(\calD_{\uptau_1}^{\uptau_2})} &\lesssim \sum_{j \leq k} \|\p_\uptau\mathbb\Phi_j\|^2_{E(\Sigma_{\uptau_1})} + \eps^2 \brk{\uptau_1}^{-\frac{9}{2} + 2\kappa}, \quad |k| \leq M - 21. \tag{ILED-$\p_\uptau\mathbb\Phi$} \label{ILED-dPhi}
    \end{align}
\end{proposition}
\begin{proof}
    The details are very alike compared to those in the proof of  \cite[Proposition 9.1]{OS24} and \cite[Proposition 8.8]{LOS22}. We only sketch an outline here and point out the differences. 

    The $k = 0$ case directly follows from Proposition~\ref{prop_ILED}. We estimate the part generated from the original source term $\calF_0$ and the extra part $\p_\uptau g_2$  via bootstrap assumptions, or more specifically, \eqref{BA-p}, \eqref{BA-ddotwp-L2} and \eqref{eqn:L2_est_of_c_im}. For $g_1$, in view of \eqref{eqn:constraint_g12_1} and \eqref{eqn:coeff_c_im}, we estimate the part $p \calP_h Y_m$ by the LE norm of $p$ view H\"older inequality and the rest by bootstrap assumptions. For  $\bfUpomg_j(U)$, it's dealt by orthogonality conditions \eqref{eqn:mod_ortho_for_p}, Lemma~\ref{lemma_decay_of_bfOmg_phi} and Lemma~\ref{lemma_decay_of_F_pm}.
       
    To prove the case that $k \neq 0$ where all $k$ derivatives are $\p_\uptau$'s, the only difference is that one needs to examine the proof of Proposition~\ref{prop_ILED} to make some extra integration by parts to avoid loss of derivatives caused by the errors generated when commuting $\p_\uptau^k$. This only requires for the estimates of $\mathbb\Phi$ since the forcing term of $\bbP$ does not contain a nonlinearity. Also, we mention that due to the smoothing of $c_2$, we are allowed to absorb excessive derivatives. 
    
    Afterwards, standard (non-weighted) elliptic estimates for $\calP_\elliptic$ can be used to conclude the proof for $U_k$ $(k \geq 1)$ when at least one derivative is not $\p_\uptau$. Indeed, this is done thanks to the decomposition \eqref{eqn:form_of_calP_suit_for_ell_est} and the energy estimates for $\p_\uptau^k U$. In particular, the local energy norm is estimated by applying elliptic estimates in each annulus separately and summing them up.
\end{proof}
\begin{remark}
    We note that the source term of $\bbP$ can be put into $LE^*$ type of norms just barely and hence this allows us to use \eqref{BA-ddotwp-L2} throughout the proof above without experiencing the same situation in the $r^p$-estimates for $\bbP$ below. See \eqref{eqn:range_of_p_making_F_p_rp_term_integrable}.
\end{remark}

\subsubsection{Weighted elliptic estimates for $\calP_\elliptic$ on each leaf $\Sigma_\uptau$ of our foliation}
Note that we already set up the weighted elliptic estimates for the operator $L$ in Theorem~\ref{thm_ell_est_for_L_w_lot}. Since $\calP_\elliptic$ (see \eqref{eqn:op_P_elliptic}) can be viewed as a perturbation of $L$, it is expected that a similar estimate can be worked out.

\begin{lemma}[Weighted elliptic estimates for $\calP_\elliptic$]\label{lemma_weight_ell_est_for_P_ell_1}
    Suppose $\calP_\elliptic U = g$ on $\Sigma_\uptau$ and the bootstrap assumptions in Section~\ref{Section:BA} hold, then \[
    \|U\|_{H^{2, \delta}(\Sigma_\uptau)} \lesssim \|g\|_{H^{0, \delta + 2}(\Sigma_\uptau)} + \sum_j \left| \brk{U, Z_j}_{L^2(\Sigma_\uptau)}\right|,
    \]
    where the weighted Sobolev spaces on $\Sigma_\uptau$ are defined via Definition~\ref{defn_weighted_Sobolev} by replacing the $(\rbar, \thetabar)$ coordinates in \eqref{eqn:metric_C_in_r} with $(\uprho, \uptheta)$ global polar coordinates defined in Section~\ref{Section:global_coord}.
\end{lemma}
\begin{proof}
    Examining the expansion of the truncated generalized eigenfunctions of $M$ in Lemma~\ref{lemma_zero_modes} and comparing these with \eqref{eqn:zero_modes_of_L_ej}, one realizes that for $1 \leq j \leq n$, \[
        \brk{U, Z_j} = \brk{U, \ebar_j} + \calO(|\ell|^2 + R_f^{-3}) \|\chi U\|_{L^2(\Sigma_\tau)} 
    \]
    with $\chi$ defined after \eqref{eqn:defn_of_Z_i_s}.
    Therefore, one could conclude \[
        \|U\|_{H^{2, \delta}(\Sigma_\uptau)} \lesssim \|g\|_{H^{0, \delta + 2}(\Sigma_\uptau)} + \sum_j \left|\brk{U, Z_j}_{L^2(\Sigma_\uptau)}\right|
        + \calO(|\ell|^2 + R_f^{-3}) \|U\|_{H^{0, \delta}} + \|\calP_\elliptic^\pert U\|_{H^{0, \delta + 2}},
    \]
    where $\calP_\elliptic^\pert$ is given in \eqref{eqn:op_P_elliptic}.
    Inserting \[\begin{aligned}
        \|\calP_\elliptic^\pert U\|_{H^{0, \delta + 2}} \lesssim o_{\wp, R_f}(1)\|U\|_{H^{2, \delta}} + \|\calO(\dot\wp) \p_\Sigma U\|_{H^{0, \delta + 2}}
        \lesssim o_{\wp, R_f}(1)\|U\|_{H^{2, \delta}} + \eps R_f^{-2}\brk{\uptau}^{-\frac{9}{4} + \kappa} \|U\|_{H^{1, \delta}}
    \end{aligned}
    \]
    into the preceding estimate and absorb terms with smallness to the left hand side, the desired estimate follows.
\end{proof}

Instead of applying the commutation argument in the proof of Theorem~\ref{thm_ell_high_reg}, we prove by directly applying Theorem~\ref{thm_ell_high_reg} to $\Delta_\barcalC$ as it has no zero modes. This can be seen by a simple integration by parts by testing with itself.

\begin{lemma}[Higher order weighted elliptic estimates for $\calP_\elliptic$]\label{lemma_weight_ell_est_for_P_ell_2}
    Suppose $\calP_\elliptic U = g$ on $\Sigma_\uptau$ and the bootstrap assumptions in Section~\ref{Section:BA} hold, then \[
    \|U\|_{H^{s, \delta}(\Sigma_\uptau)} \lesssim \|g\|_{H^{s-2, \delta + 2}(\Sigma_\uptau)} + \sum_j \left| \brk{U, Z_j}_{L^2(\Sigma_\uptau)}\right|, \quad s \geq 2.
    \]
\end{lemma}
\begin{proof}
    We do induction on $s$ and show how to obtain the case $s = 3$ as an example.
    Thanks to Lemma~\ref{lemma_weight_ell_est_for_P_ell_1}, it suffices to bound $\|U\|_{\dot H^{3, \delta}}$ by $\|U\|_{H^{2, \delta}}$ and $\|g\|_{H^{1, \delta + 2}}$.
    We write \[
        \Delta_\barcalC U = - V U - \calP_\elliptic^\pert U + g 
    \]  
    and apply Theorem~\ref{thm_ell_high_reg}, we obtain
    \[
        \|U\|_{\dot H^{3, \delta}}
        \leq \|VU\|_{H^{1, \delta + 2}} + \|\calP_\elliptic^\pert U\|_{H^{1, \delta + 2}} + \|g\|_{H^{1, \delta + 2}}
        \lesssim \|U\|_{H^{1, \delta}} + o_{R_f, \wp}(1) \|U\|_{H^{3, \delta}} + \|g\|_{H^{1, \delta + 2}},
    \]
    which completes the proof.
\end{proof}

\subsubsection{Overview of the following subsections}\label{Section:overview_of_rp}
We next apply the preceding results with $U$ replaced by $\bbP:= sp$ and $\mathbb{\Phi}:= s\phi$. One can refer to \eqref{eqn:relation_varphi_psi} for the definition of $s$. The functions $p$ and $\phi$ are defined in Section~\ref{Section:mod_profile}.
In light of \eqref{eqn:relation_calP_and_calP_graph}, we notice that \[
    \calP_\graph \bbP = \calF_\bbP, \quad \calP_\graph \mathbb{\Phi} = \calF_{\mathbb{\Phi}}, 
\]
where \[\begin{aligned}
    \calF_\bbP &= \calF_{0, \bbP} + s(g_1 + \p_\tau g_2) + s \calP_\pert p, \quad \calF_{0, \bbP} = \calO(\dot\wp r^{-3}), \\ 
    \calF_{\mathbb{\Phi}} &= \calF_{0, \mathbb\Phi} + \calF_2 + \calF_3 - s(g_1 + \p_\tau g_2) - s \calP_\pert p, \quad \calF_{0, \mathbb\Phi} = \calO(\dot\wp r^{-4}),
\end{aligned}
\] 
when $\uprho_0 \gg R_f$. Here, $q$ and $\phi$ agree on the hyperboloidal region due to the support condition on $\vec Z_\pm$. We also notice that the linear contribution in the forcing for both equations all comes with $\delta_\wp$ extra smallness so that we do not need to worry about bootstrap constants as they can always be absorbed by this extra smallness. See also Remark~\ref{rmk_on_a_-_0_in_BA}.

Though the source terms are quite different for $\bbP$ and $\mathbb\Phi$, we expect to prove an almost-$\brk{\uptau}^{-2}$ pointwise decay in a similar fashion. However, the technicalities come from different aspects for $\bbP$ and $\mathbb\Phi$, respectively. 
For $\bbP$, one definitely needs to pay special attention to see which norm we are allowed to use due to the slow decay of $\calF_{0, \bbP}$ in the spatial direction. This is discussed in Section~\ref{Section:decay_for_P}.
For $\mathbb\Phi$, one definitely needs to put some effort to deal with the nonlinearity and we refer to Remark~\ref{rmk_handle_nonlinearity}.
 
However, we expect a better pointwise decay for $\mathbb\Phi$ as the usual range $p \in (0, 2)$ in the $r^p$-estimates does not exhaust the possibility of raising powers of $p$. (See \eqref{eqn:range_of_p_making_F_p_rp_term_integrable} for a clearer explanation.) The limitation of $p \in (0, 2)$ just comes from the requirement of the coercivity of $\calE^p$ and $\calB^p$ as already seen in the proof of \eqref{eqn:main_rp_est}.
Heuristically, if the $p$ range could be expanded, then one expects a better decay rate. This idea is realized by introducing the commutation vector field $K = r^{\frac32}\p_r$ in Section~\ref{section_Y_vf_K}.

One will see in the following part that although we need a better decay rate than $\brk{\tau}^{-2}$ for $\mathbb\Phi$, we may intentionally want to limit ourselves with the decay rate for $\p_\tau^j\mathbb\Phi$ ($j \geq 1$) when setting up the bootstrap assumptions since the rate would be restricted by the contribution $g_1 + \p_\tau g_2$ on the right hand side of the equation. Although $\p_\tau^{j+1} p$ has one more derivative than $\p_\tau^j\phi$, which would be expected to have a better decay rate. However, one may not see an improved decay as expected when $j$ becomes large, as the poor decay of $\calF_\bbP$ could be detrimental to the $r^p$-hierarchy. We should prevent this since $\p_\tau^j(g_1 + \p_\tau g_2)$, as part of $\p_\tau^j \calF_{\mathbb\Phi}$, will be closely related to the decay rate of $\p_\tau^{j+1} p$.

\subsection{Decay estimates for  \texorpdfstring{$\bbP$}{P}}\label{Section:decay_for_P}

We first look into the equation for $U = \bbP$ as it is an almost linear equation and we do not need a sophisticated iteration to deal with the nonlinearity. The dependence on $p$ in $s\calP_\pert p$ has more than enough decay coming from the coefficients of the operator and \eqref{BA-p}. We do not mention explicitly about this term anymore in the subsequent estimates, as they would not be the one which dominates the decay rate. Another term that appears as the forcing term in the equation for $\bbP$ is $s(g_1 + \p_\tau g_2)$, whose decay rate is closely bonded to that of $p$. See \eqref{eqn:all_decay_of_c_im} and \eqref{eqn:L2_est_of_c_im}.

Before we proceed to derive the main $p$-hierarchy estimates, it might be good to recall some notation such as $\tilde U = r^{\frac32} U$ and the ones introduced in \eqref{eqn:notation_tilde_U_N_triple}.
Starting with Proposition~\ref{prop_rp_for_U} and 
focusing on the term 
\[
    I_{\bbP, 1}:= \int_{\calD_{\tau_1}^{\tau_2}} \chi r^p \left|\tilde \calF \p_r \tilde U \right|\, \ud \tau\, \ud r\, \ud \theta, \quad I_{\bbP, 2}:= \int_{\calD_{\tau_1}^{\tau_2}} \chi r^p \left|\tilde \calF \calO(r^{-3}) \p_\theta \tilde U \right|\, \ud \tau\, \ud r\, \ud \theta.
\]
It is easy to bound $I_{\bbP, 2}$ by the Cauchy-Schwarz inequality via $L^2L^2$ type norms \[
    I_{\bbP, 2} \leq \veps \calB^{p-1}[\tilde \bbP](\tau_1, \tau_2) + C_\veps \int_{\calD_{\tau_1}^{\tau_2}} \chi r^{p+1} |\calO(r^{-3})\tilde \calF_\bbP|^2 \, \ud \tau\, \ud r\, \ud \theta \lesssim \veps \calB^{p-1}[\tilde \bbP](\tau_1, \tau_2) + \eps^2\brk{\tau_1}^{-\frac72 + 2\kappa}
\]
and absorb the first term to the left hand side, where $\veps$ is some arbitrary smallness and $\eps$ is the  size of initial data, appearing in the bootstrap assumptions for $\dot\wp$ in \eqref{BA-dotwp}.
However, it turns out to be a bad idea to use $L^2L^2$ type of estimates for $I_{\bbP, 1}$ due to the non-integrability of \begin{equation}\label{eqn:range_of_p_making_F_p_rp_term_integrable}
    \int_{\calD_{\tau_1}^{\tau_2}} \chi r^{p+1} |\tilde \calF_\bbP|^2 \, \ud \tau\, \ud r\, \ud \theta
\end{equation}
in $r$ variable for $p \geq 1$. Instead, we replace by $L^\infty L^2$ and $L^1 L^2$ type estimates \[
    I_{\bbP_1} \leq \veps \sup_{\tau_1 \leq \tau \leq \tau_2} \calE^p[\tilde\bbP](\tau) + C_\veps \left(\int_{\tau_1}^{\tau_2} \big(\int_{\Sigma_\tau} \chi r^p |\tilde \calF_\bbP|^2 \, \ud r\, \ud \theta \big)^{\frac12} \, \ud \tau\right)^2 \lesssim \veps \sup_{\tau_1 \leq \tau \leq \tau_2} \calE^p[\tilde\bbP](\tau) + \eps^2\brk{\tau_1}^{-\frac52 + 2\kappa}
\]
for $0 < p < 2$. It then follows from Proposition~\ref{prop_rp_for_U} that 
\begin{equation}\label{eqn:preliminary_rp_for_bbP_w_Fp_realized}
    \begin{aligned}
    \sup_{\tau_1 \leq \tau \leq \tau_2} \calE^p[\tilde\bbP](\tau) + \calB^p[\tilde\bbP](\tau_1, \tau_2) 
    \lesssim_p &\ \calE^p[\tilde\bbP](\tau_1) + \int_{\calD_{\tau_1}^{\tau_2}} |\Err_\chi[\tilde\bbP]| \, \ud \tau\, \ud r\, \ud\theta + \eps^2\brk{\tau_1}^{-\frac52 + 2\kappa} \\
    &\ + \delta \int_{\calD_{\tau_1}^{\tau_2}} \chi r^{-1-\alpha} |\p_\tau \tilde\bbP|^2\, \ud \tau\, \ud r\, \ud \theta + \delta \sup_{\tau_1 \leq \tau \leq \tau_2} E[\bbP](\tau),
\end{aligned}
\end{equation}
for $\delta < p < 2 - 2\alpha$. 
To deal with the error terms, we add a suitable multiple of \eqref{ILED-P} with $k = 1$ to \eqref{eqn:preliminary_rp_for_bbP_w_Fp_realized}, switching notation back from $\tilde\bbP$ to $\bbP$, and applying Hardy's inequality \eqref{eqn:Hardy_variant}, we obtain \begin{equation}\label{eqn:after_adding_a_multiple_LED_P}
    \begin{aligned}
    \sup_{\tau_1 \leq \tau \leq \tau_2} \scE^p [\bbP](\tau) + \scB^p [\bbP](\tau_1, \tau_2) &\lesssim \sum_{0 \leq |k| \leq 1} \scE^p [\bbP_k](\tau_1) + \eps^2\brk{\uptau_1}^{-\frac52 + 2\kappa}, \quad \delta < p < 2 - 2\alpha, 
\end{aligned}
\end{equation}
where \[\begin{aligned}
    \scE^p[U](\tau) &:= \int_{\Sigma_\tau} \chi_{\leq R} |\p U|^2\, \ud V + \chi_R r^p \left((\p_r + \frac{n-1}{2r})U\right)^2 \, \ud V, \\ 
    \scB^p[U](\tau_1, \tau_2) &:= \int_{\calD_{\tau_1}^{\tau_2}} \chi_{\leq R} |\p U|^2 + \chi_R r^{p-1} \left( \big((\p_r + \frac{n-1}{2r})U\big)^2 + r^{-2} U^2 + r^{-2} |\snabla U|^2\right) \\ & \qquad\qquad + \chi_R r^{-1-\alpha} |\p_\tau U|^2\, \ud V \, \ud \tau =: \int_{\tau_1}^{\tau_2} \scB^p[U](\tau)\, \ud \tau
\end{aligned}
\]
with $n = 4$.
\begin{remark}\label{rmk_on_adding_LED_est}
    We add a few remarks on this estimate here.
    \begin{enumerate}[ wide = 0pt, align = left, label=\arabic*. ]
        \item The last term in the definition of $\scB^p[U](\tau_1, \tau_2)$ is not necessary later when extracting decay. We still keep this specific term in the LE norm but omitting others just to emphasis how we cancel out the contribution of the penultimate term in \eqref{eqn:main_rp_est} as well as \eqref{eqn:main_higher_order_rp_est}.
        \item The contribution of bounded $r$ region of both quantities only involves $\p U$ instead of zeroth order contribution is due to the variant of Hardy's inequality \eqref{eqn:Hardy_variant}.
        \item The one derivative higher appearing on the right hand side of the estimate is due to two things. One is the derivative loss in the definition of LE norm (so we need the local energy estimate with both $k = 0,1$). The additional $\p_\theta$, $\p_\tau$ derivatives in the definition of standard energy $E$ in the far away region compared to $\calE^p$ energy naturally require this $k = 1$ term to be added in as well.
        \item For all estimates involving $\calF_\bbP$, we simply apply the bootstrap assumption \eqref{BA-dotwp}, where the extra $\delta_\wp$ smallness will be the factor that makes all implicit constants in the $r^p$-estimates independent of the bootstrap constants $C_k$'s. Therefore, once we extract pointwise decay estimates from these, we just need to choose $C_k$ large enough to beat these $C_k$-independent implicit constants to close the bootstrap argument.
    \end{enumerate}
\end{remark} 

Following the same lines of proof, higher order estimates could be obtained: \begin{equation}\label{eqn:rp_to_derive_tau_inverse_decay_bbP}
    \sum_{0 \leq |k| \leq N} \sup_{\tau_1 \leq \tau \leq \tau_2} \scE^p [\bbP_k](\tau) + \scB^p [\bbP_k](\tau_1, \tau_2) \lesssim \sum_{0 \leq |k| \leq N+1} \scE^p [\bbP_k](\tau_1) + \eps^2 \brk{\tau_1}^{-\frac52 + 2\kappa},
\end{equation}
where $N + 1 \leq M$ and $\delta < p < 2 - 2\alpha$.

Thanks to the local theory ($\eps$-smallness of initial data and finite speed propagation property, see Appendix~\ref{Section:LWP}), the $p$-energy on the initial slice is bounded by $\eps^2$.
Then it follows from \eqref{eqn:rp_to_derive_tau_inverse_decay_bbP} with $p = 2 - \zeta$ with $\zeta > 2\alpha$, $N = M-1$, $\tau_1 = 0$ that its left hand side is bounded by a bootstrap constant independent multiple of $\eps^2$. Therefore, there exists an increasing sequence of dyadic $\tilde \sigma_m$ such that \begin{equation}\label{eqn:bbP_decay_1}
    \scE^{1-\zeta}[\bbP_k](\tsigma_m) \leq \scB^{2-\zeta}[\bbP_k](\tsigma_m) \lesssim \eps^2 \brk{\tsigma_m}^{-1}, \quad |k| \leq M-1.
\end{equation}
Another application of \eqref{eqn:rp_to_derive_tau_inverse_decay_bbP} with $p = 1 - \zeta$, $N = M-2$, $\tau_1 = \tsigma_{m-1}$, $\tau_2 = \tsigma_m$ would imply that \begin{equation}\label{eqn:bbP_decay_2}
    \scE^{-\zeta}[\bbP_k](\sigma_m) \leq \scB^{1-\zeta}[\bbP_k](\sigma_m) \lesssim \eps^2 \brk{\sigma_m}^{-2}, \quad |k| \leq M-2,
\end{equation}
where $\tilde\sigma_{m-1} \leq \sigma_m \leq \tilde\sigma_m$, i.e., $\sigma_m$ is still an increasing almost dyadic sequence. 
Applying \eqref{eqn:bbP_decay_1} to the right hand side of \eqref{eqn:rp_to_derive_tau_inverse_decay_bbP} would allow us to exchange the dyadic sequence $\{\tsigma_m\}$ for $\{\sigma_m\}$, that is, \[
    \scE^{1-\zeta}[\bbP_k](\sigma_m) \leq \scB^{2-\zeta}[\bbP_k](\sigma_m) \lesssim \eps^2 \brk{\sigma_m}^{-1}, \quad |k| \leq M-2.
\]
Interpolating this with \eqref{eqn:bbP_decay_2} via \eqref{eqn:interpolation_for_rp}, we obtain \[
    \|\bbP_k\|_{E(\Sigma_{\sigma_m})}^2 \lesssim \eps^2 \brk{\sigma_m}^{-2 + \zeta}, \quad |k| \leq M - 3.
\]
Inserting this to \eqref{ILED-P} enables the decay estimates for any arbitrary constant $\uptau$ slice: \begin{equation}\label{eqn:E_bbP_tau_decay}
    \|\bbP_k\|_{E(\Sigma_\tau)}^2 \lesssim \eps^2 \brk{\tau}^{-2 + \zeta}, \quad |k| \leq M - 3.
\end{equation}

For the purpose of proving decay for $\p_\tau \bbP_k$, we interpolate between \eqref{eqn:bbP_decay_1} and \eqref{eqn:bbP_decay_2} again with different ratio to obtain \[
    \calE^{\frac12 + \kappa}[\bbP_k](\sigma_m) \lesssim \eps^2 \brk{\sigma_m}^{-\frac32 + 2\kappa}, \quad |k| \leq M - 2
\]
with a different dyadic sequence $\{\sigma_m\}$.
Putting $p = \frac12 + \kappa$ in \eqref{eqn:rp_to_derive_tau_inverse_decay_bbP}, we obtain \begin{equation}\label{eqn:bbP_decay_3}
    \calE^{-\frac12 + \kappa}[\bbP_k](\sigma_m) \lesssim \eps^2 \brk{\sigma_m}^{-\frac52 + 2\kappa}, \quad |k| \leq M - 3,
\end{equation}
where $\{\sigma_m\}$ is again a new dyadic sequence. Note that we cannot pass to all times for this one since $p = -\frac12 + \kappa$ is negative and hence there is no good p-energy boundedness. We also note that we specify $\zeta = \kappa$ here.

Now we want to mimic the procedure above for $\p_\tau \bbP_k$ as well. We start with the following algebraic relation that if one singles out the first term $-2\p_\tau\p_r \tilde U_N$ in \eqref{eqn:tilde_U_N_calP}, then \begin{equation}\label{eqn:algebraic_relation_between_eng_of_bbP_and_p_tau_bbP}
    \sum_{|j| \leq |k|} \scE^p[\p_\tau \bbP_j](\tau) \lesssim \sum_{|j| \leq |k| + 1} \scE^{p-2}[\bbP_j](\tau) + \eps^2 \brk{\tau}^{-\frac92 + 2\kappa},
\end{equation}
for $p < 2$ and all $\tau$. Here the source term is estimated by \[
    \int_{\Sigma_\tau} \chi r^p |\tilde \calF_\bbP|^2\, \ud r\, \ud \theta \lesssim \eps^2 \brk{\uptau}^{-\frac92 + 2\kappa} \int_{\Sigma_\tau} \chi r^{p-3}\, \ud r\, \ud \theta \lesssim_{p < 2} \eps^2 \brk{\uptau}^{-\frac92 + 2\kappa}.
\]

Putting $p = \frac32 + \kappa$ and $\tau = \sigma_m$, this leads to \begin{equation}\label{eqn:decay_scE_2-zeta_p_tau_P}
    \sum_{j \leq k} \scE^{\frac32 + \kappa}[\p_\tau \bbP_j](\sigma_m) \lesssim \eps^2 \brk{\sigma_m}^{-\frac52 + 2\kappa}, \quad |k| \leq M - 4.
\end{equation}

From \eqref{BA-dotwp}, we know that $\|\ddot\wp\|_{L^1_\uptau} \lesssim \delta_\wp \eps \brk{\uptau}^{-\frac54 + \kappa}$. We remark that this is obtained from directly integrating the pointwise bound. One could in fact obtain a better bound with decay rate $-\frac{7}{4} + \kappa$ by estimating in dyadic $\tau$-region via \eqref{BA-ddotwp-L2}. However, this improved information is not necessary in the subsequent computation.
Next, by putting \begin{equation}\label{eqn:frac72_decay_rate_for_L1L2_p_tau_calFP}
\|r^{\frac{p}2}\p_\tau\calF_\bbP\|_{L^1L^2}^2 \lesssim \eps^2 \brk{\tau_1}^{-\frac52 + 2\kappa}
\end{equation}
with $p = \frac32 + \kappa$ in \eqref{eqn:main_higher_order_rp_est_p_tau_U} with $\p_\tau U$ replaced by $\p_\tau \bbP$, after a suitable ILED of $\p_\tau \bbP_k$ \eqref{ILED-dP} being added to: \begin{equation}\label{eqn:decay_scE_1-zeta_p_tau_P}
    \scE^{\frac12 + \kappa}[\p_\tau \bbP_k](\sigma_m) \lesssim \eps^2 \brk{\sigma_m}^{-\frac72 + 2\kappa}, \quad |k| \leq M - 5
\end{equation}
with a different dyadic sequence $\{\sigma_m\}$, denoted by the same notation.
Now, putting $p = \frac12 + \kappa$ in \eqref{eqn:main_higher_order_rp_est_p_tau_U}, with $\calF_\bbP$ put into $L^2L^2$-type norm $\|r^{\frac{p+1}2}\p_\tau\calF_\bbP\|_{L^2L^2}^2 \lesssim \brk{\tau}^{-\frac92 + 2\kappa}$ thanks to \eqref{BA-ddotwp-L2}. This would lead to \begin{equation}\label{eqn:decay_scE_-zeta_p_tau_P}
    \scE^{-\frac12 + \kappa}[\p_\tau \bbP_j](\sigma_m) \lesssim \eps^2 \brk{\sigma_m}^{-\frac92 + 2\kappa}, \quad |k| \leq M - 6,
\end{equation}
with a different dyadic sequence. We remark that the reason why we need to switch the norms we use is due to poor decay if we use $L^1L^2$ type of norms, as one can see in \eqref{eqn:frac72_decay_rate_for_L1L2_p_tau_calFP}, restricting how much we can achieve.
Interpolating between \eqref{eqn:decay_scE_1-zeta_p_tau_P} and \eqref{eqn:decay_scE_-zeta_p_tau_P} after synchronizing the sequence, we finally achieve that \[
    \|\p_\tau \bbP_k\|_{E(\Sigma_\tau)}^2 \lesssim \eps^2 \brk{\tau}^{-4 + \kappa}, \quad |k| \leq M - 7
\]
and \begin{equation}\label{eqn:p_tau_bbP_k_almost_eng_decay}
    \|\brk{\uprho}^{\frac{\kappa}{2}-1}\p_\tau \bbP_k\|_{L^2(\Sigma_\tau, \ud V)}^2 \lesssim \eps^2 \brk{\tau}^{-4 + 2\kappa}, \quad |k| \leq M - 6.
\end{equation}
However, from \eqref{eqn:decay_scE_-zeta_p_tau_P}, the same type of algebraic relation in \eqref{eqn:algebraic_relation_between_eng_of_bbP_and_p_tau_bbP}, we would know that \[
    \calE^{\frac32 + \kappa}[\p_\tau^2 \bbP_k](\sigma_m) \lesssim \eps^2 \brk{\sigma_m}^{-\frac92 + 2\kappa}, \quad |k| \leq M - 6.
\]
Note that we could not pass to all times at the level of $p = \frac32 + \kappa$, which is due to the fact that for $p = \frac32 + \kappa$, the source term cannot be put into $L^2L^2$-type norm. Instead, we could pass to all times by choosing $p < 1$ and applying \eqref{BA-ddotwp-L2} to $\p_\tau^2 \calF_\bbP$: 
\begin{equation}\label{eqn:p_tau_2_P_eng_decay}
    \calE^{1-}[\p_\tau^2 \bbP_k](\tau) \lesssim \eps^2 \brk{\tau}^{-\frac92 + 2\kappa}, \quad |k| \leq M - 7.
\end{equation}
Though it might seems feasible to run the hierarchy once more for $\p_\tau^2 \bbP_k$, it would cost lost in decay due to the restrictive $p$-range between $0$ and $1$, which is in turn because of the inaccessibility of $L^1L^2$ norm. In particular, this leads to \begin{equation}\label{eqn:LE_norm_bound_for_dtau2_P}
    \|\p_\tau^2 \bbP_k\|^2_{LE(\calD_\tau^{\tau'})} \lesssim \eps^2 \brk{\tau}^{-\frac92 + 2\kappa}, \quad |k| \leq M - 7,  
\end{equation}
which establishes \eqref{IBA-p-LE}.

Finally, we are ready to prove pointwise decay. Writing \[
    \calP_\elliptic \bbP_k = - P_\uptau \bbP_k + s^{-1}\calF_{\bbP_k},
\]
we derive from the elliptic estimates Lemma~\ref{lemma_weight_ell_est_for_P_ell_2} and Lemma~\ref{lemma_Sobolev_embedding} that \[
    \|\bbP_k\|_{L^\infty} \lesssim \|\bbP_k\|_{H^{3, \delta}} \lesssim \|P_\uptau \bbP_k\|_{H^{1, \delta + 2}} + \|\calF_{\bbP_k}\|_{H^{1, \delta + 2}}
\]
for any $-2 < \delta < 0$ and any $k \in \bbN$.
Choosing $\delta + 2 = \frac{\kappa}2$ and applying \eqref{eqn:p_tau_bbP_k_almost_eng_decay}, 
\[
    \|P_\uptau \bbP_k\|_{H^{1, \delta + 2}} \lesssim \sum_{|j| \leq |k| + 1}\|\brk{\rho}^{-1}\p_\uptau\bbP_{j} \|_{H^{1, \delta + 2}} \lesssim \eps \brk{\tau}^{-2 + \kappa}, \quad |k| \leq M - 7.
\]
Combining the inequalities above, we proved
\eqref{IBA-p}. Then \eqref{IBA-p-ext1} and \eqref{IBA-p-ext2} both follow from trace inequality and \eqref{eqn:E_bbP_tau_decay} as in \cite{LOS22}.  
Similarly, applying weighted elliptic estimates to \eqref{eqn:p_tau_2_P_eng_decay} would lead to \begin{equation}\label{eqn:ptwise_decay_of_p_tau_p}
    |\p_\uptau p_k| \lesssim \|\bbP_k\|_{H^{3, \frac{\kappa}2 - 2}} \lesssim \eps \brk{\uptau}^{-\frac94 + \kappa}, \quad |k| \leq M - 8.
\end{equation}
This proves \eqref{IBA-dtau-p}.

\subsection{General \texorpdfstring{$r^p$}{rp}-weighted estimates for \texorpdfstring{$Y = r^{\frac32}\p_r \tilde U$}{Y} in dimension \texorpdfstring{$n = 4$}{n=4}}\label{section_Y_vf_K}

\subsubsection{Equations used to derive \texorpdfstring{$r^p$}{rp}-weighted estimates for \texorpdfstring{$ Y_N = r^{\frac32}\p_r\tilde U_N$}{Y\_N}}

We first provide some commutator computations regarding $K:= r^{\frac32}\p_r$.

\begin{lemma}\label{lemma_commutator}
    Here are some basic commutator identities: \[
        \begin{aligned}
        \,[K, r^{-2}] = -2 r^{-3/2},\quad
            [K, \p_r] = -\frac32 r^{-1}K, \quad  
        [K, -2\p_r\p_\tau] = 3r^{-1}\p_\tau K, \\
            \p_r K = \frac32 r^{-1}K + r^{\frac32}\p_r^2, \quad 
            [K, \p_r^2] = -3r^{-1}\p_r K + \frac{15}{4}r^{-2}K. \quad
        \end{aligned}
    \]
\end{lemma}
\begin{proof}
    We compute \[
        [K, \p_r] = [r^{3/2}\p_r, \p_r] = -\frac32 r^{1/2}\p_r = -\frac32 r^{-1} K, 
    \]
    and hence \[\begin{aligned}
        &[K, -2\p_r\p_\tau] = [K, -2\p_r]\p_\tau = 3r^{-1}\p_\tau K, \quad \p_r K = \frac32 r^{-1} K + K\p_r, 
        \\
        &[K, \p_r^2] = [K, \p_r]\p_r + \p_r[K, \p_r] = -\frac32 r^{-1} K\p_r - \frac32 \p_r(r^{-1} K)
        = -3r^{-1}\p_r K + \frac{15}{4} r^{-2}K.
    \end{aligned}
    \]
\end{proof}
\begin{remark}\label{remark_exact_cancel}
    In the computation above, we should keep in mind that all the commutators can be expressed in terms of some functions/vector fields applied to $K$ except the commutator $[K, r^{-2}]$. However, one can rewrite this as \begin{equation}\label{eqn:crucial_observation_cancel}
        [K, V] = -2r^{\frac12} V, \quad \text{ where } V = A(\theta, \tau)r^{-2}\p_\theta^\alpha\p_\tau^\beta \text{ or } A(\theta, \tau)r^{-1}\p_r \p_\theta^\alpha \p_\tau^\beta, 
    \end{equation}
    where $A(\theta, \tau)$ denotes a function independent of $r$. 
    \begin{enumerate}[ wide = 0pt, align = left, label=\arabic*. ]
        \item This leads to the choice of $2r^{\frac12}$ in the following lemma and the absence of inverse square potentials in the new operator $Q_1$ defined in \eqref{eqn:Q_1_defn}. This significant exact cancellation would be the key to improving the decay rate later.
        \item More generally, for any differential operator of the form $V = \calO(r^{-2})\p_\theta^\alpha\p_\tau^\beta$,  \begin{equation}\label{eqn:K_comm_general_comp}
        [K, V] = -2r^{\frac12}V + \calO(r^{-\frac52}) \p_\theta^\alpha\p_\tau^\beta.
    \end{equation}
    To show a sketch of the proof, we first make a simple expansion 
        $\calO(r^{-2}) = \calO_{\theta, \tau}(1) r^{-2} + \calO(r^{-3})$ in $r$
    and apply \eqref{eqn:crucial_observation_cancel}. It turns out to be crucial that we could single out the main term $-2 r^{\frac12} V$ of the commutator in $r$ and leave the error terms one order lower in $r$. In particular, for $\tilde V = r^{-1}\p_r \p_\theta^\alpha \p_\tau^\beta$, $[K, \tilde V] = -2r^{\frac12}\tilde V + \calO(r^{-\frac12})\p_r \p_\theta^\alpha \p_\tau^\beta$, where we do not see an improvement in the big-$\calO$ term.
    
    See the derivation of \eqref{eqn:calG_2N_defn} and the discussion in Remark~\ref{rmk_calG_2_treatment_precise_decay_crucial}.
    \end{enumerate}
\end{remark}

\begin{lemma}
    In $(\tau,r,\theta)$ coordinates, we define the operator $Q_1$ as \begin{equation}\label{eqn:Q_1_defn}
        Q_1:= Q_0 - r^{-1}\p_\tau - r^{-1}\p_r + \frac34 r^{-2} = - 2\p_r\p_\tau + \p_r^2 - r^{-1}\p_\tau - r^{-1}\p_r + r^{-2} \sDelta_{\bbS^{n-1}}.
    \end{equation}
    Then $Q_1$ satisfies the commutator identity \begin{equation}\label{eqn:Q_1_K}
        (K + 2r^{\frac12})Q_0 = Q_1 K
    \end{equation}
    where $Q_0$ is defined in \eqref{eqn:Q_0_defn}.
\end{lemma}
\begin{proof}
    Thanks to Lemma~\ref{lemma_commutator}, we compute \[
    \begin{aligned}
    \left[K, Q_0\right] &= 3r^{-1}\p_\tau K - 3r^{-1}\p_r K + \frac{15}{4}r^{-2}K + \frac32 r^{-3/2} -2r^{-3/2}\sDelta_{\bbS^{n-1}} \\ 
    &= 3r^{-1}\p_\tau K - 3r^{-1}\p_r K + \frac{15}{4}r^{-2}K -2r^{\frac12}(Q_0 + 2\p_r\p_\tau - \p_r^2) \\
        &= 3r^{-1}\p_\tau K - 3r^{-1}\p_r K + \frac{15}{4}r^{-2}K -2r^{\frac12}Q_0 - 4r^{-1}\p_\tau K + 2r^{-1}(\p_r K - \frac32 r^{-1}K)\\
        &= -2r^{\frac12} Q_0 + ( - r^{-1} \p_\tau - r^{-1}\p_r + \frac34 r^{-2})K,
    \end{aligned}
    \]
    which gives the desired identity.
\end{proof}

We now have two routes when commuting $K$: \begin{enumerate}[ wide = 0pt, align = left, label=\arabic*. ]
    \item deriving the equation satisfied by $K\tilde U$ first and then for higher order derivatives $(K\tilde U)_N$ following induction;
    \item taking advantage of the equation for $\tilde U_N$ in \eqref{eqn:tilde_U_N_calP} and commute it with $K$.
\end{enumerate}
These two will be almost equivalent for our purpose in view of the Hardy's inequality. Moreover, the analysis for error terms are identical. Though the first option might be more pedagogical, we choose the second one for simplicity. 

\begin{lemma}\label{lemma_computation_for_R_c}
    Suppose $\calP_\graph U = \calF$. Then one could derive an equation for $\tilde U_k$ as in \eqref{eqn:tilde_U_N_calP}. We denote the operator applied to $\tilde U_k$ in the last three lines of \eqref{eqn:tilde_U_N_calP} by $R_{3, k}$.
    Then for $\tilde U:= r^{\frac32} U$, any given triple $N$, the function $Y_N:= K\tilde U_N$ (see \eqref{eqn:notation_tilde_U_N_triple}) satisfies \begin{equation}\label{eqn:eqn_for_rp_for_Y_N}
    \begin{aligned}
        &\quad (Q_1 + \frac{(n_3 + 1)n_3}{2r^2} - n_3 r^{-1}\p_r) Y_N - \frac{n_3}{r^2}\sDelta_S Y_{n_1, n_2, n_3 - 1} \\ 
        &+ n_3 \sum_{|k|=0}^{|N|-1} O_\infty(r^{-2}) (Y_k + \calO(r^{-\frac12}) \tilde U_k) + n_3(n_3-1) \sum_{|k|=0}^{|N|-2}O_\infty(r^{-2}) \sDelta_S (Y_k + \calO(r^{-\frac12}) \tilde U_k) \\
        &+ \sum_{|k| = 0}^{|N|} \left(R_{3,k} + \calO(r^{-2})\p_\tau + \calO(\wp + r^{-1}) r^{-1}\p_r\right) Y_k
        = \calG_{1, N} + \calG_{2, N},
    \end{aligned}
    \end{equation}
    where $\calG_{1, N}:= (K + 2r^{\frac12})\tilde\calF_N$ and \begin{equation}\label{eqn:calG_2N_defn}
        \begin{aligned}
        \calG_{2, N} &:= \sum_{|k| = 0}^{|N|} \left( \calO(\wp r^{-\frac52} + r^{-\frac72}) \tilde U_k
        + \calO(r^{-\frac52}) \p_a\p_b \tilde U_k \right.\\
        & \qquad\qquad + \calO(r^{-\frac52}) \p_\tau^2 \tilde U_k +
         \calO(\wp r^{-\frac12} + r^{-\frac52})\p_r\p_\theta \tilde U_k 
         + \calO(r^{-\frac52})\p_\tau\p_\theta \tilde U_k \\
        & \qquad\qquad\left. + \calO(\wp r^{-\frac12} + r^{-\frac32}) \p_r \tilde U_k
        + \calO(r^{-\frac52}) \p_a \tilde U_k + \calO(r^{-\frac52}) \p_\tau \tilde U_k\right).
    \end{aligned}
    \end{equation}
    In particular, when $N = 0$, we record the equation and denote the left hand side by $R_c Y$ with $Y = Y_N$: \begin{equation}\label{eqn:operator_Rc_defn}
        R_c = - 2\p_r\p_\tau + \p_r^2 - r^{-1}\p_\tau - r^{-1}\p_r + r^{-2} \sDelta_{\bbS^{n-1}} + \calO(r^{-2})\p_\tau + \calO(\wp r^{-1} + r^{-2})\p_r + R_{3,0},
    \end{equation}
    where $R_{3,0}$ is given by \[\begin{aligned}
        R_{3, 0} &:= \calO(\wp r^{-2} + r^{-4}) + \calO(\wp + r^{-2})\big(\p_r\p_\tau + \p_r^2\big)
    + \calO(\wp r^{-2} + r^{-3}) \p_a\p_b  \\
    &\ \quad + \tilde m_1^{\tau\tau} \p_\tau^2  +
     \calO(\wp r^{-1} + r^{-3})\p_r\p_\theta 
     + \calO(r^{-3})\p_\tau\p_\theta \\
    &\ \quad + \calO(\wp r^{-1} + r^{-2}) \p_r + \calO(\wp r^{-2} + r^{-3}) \p_a + \calO(\dot\wp r^{-2} + r^{-3}) \p_\tau.
    \end{aligned}
    \]
    One can see \eqref{eqn:defn_m1_tt} for the formula of $\tilde m_1^{\tau\tau}$.
\end{lemma}
\begin{proof}
    We start with \eqref{eqn:tilde_U_N_calP}. 
    Rewriting the first line as \[
        R_{1,N}\tilde U_N + R_{1,N-1} \tilde U_{N-1} = \left(Q_0 + \frac{(n_3 + 1)n_3}{2r^2} - n_3r^{-1}\p_r\right)\tilde U_N - \frac{n_3}{r^2}\sDelta_S \tilde U_{N-1}.
    \]
    It then follows from \eqref{eqn:Q_1_K} and \eqref{eqn:crucial_observation_cancel},
    that \[
        [K, R_{1,N}] = -2r^{\frac12}R_{1,N} + (-r^{-1}\p_\tau - r^{-1}\p_r + \frac34 r^{-2})K, \qquad\quad [K, R_{1, N-1}] = -2r^{\frac12} R_{1, N-1}.
    \]
    Then we know that \[
        (K + 2r^{\frac12}) (R_{1,N}\tilde U_N + R_{1, N-1} \tilde U_{N-1}) = (Q_1 + \frac{(n_3 + 1)n_3}{2r^2} - n_3r^{-1}\p_r) K\tilde U_N + R_{1, N-1} K \tilde U_{N-1}.
    \]
    For the second line of \eqref{eqn:tilde_U_N_calP}, we use the observation \eqref{eqn:K_comm_general_comp} so that \[\begin{aligned}
        &(K + 2r^{\frac12}) \left(n_3 \sum_{|k|=0}^{|N|-1} O_\infty(r^{-2}) \tilde U_k + n_3(n_3-1) \sum_{|k|=0}^{|N|-2}O_\infty(r^{-2}) \sDelta_S \tilde U_k\right) \\
        =&\ n_3 \sum_{|k|=0}^{|N|-1} O_\infty(r^{-2}) (K\tilde U_k + \calO(r^{-\frac12}) \tilde U_k) + n_3(n_3-1) \sum_{|k|=0}^{|N|-2}O_\infty(r^{-2}) \sDelta_S (K\tilde U_k + \calO(r^{-\frac12}) \tilde U_k).
    \end{aligned}
    \]

    The main error parts occupy the last three lines of \eqref{eqn:tilde_U_N_calP}. Among these, the first line consists of the small perturbation of the main contribution in $Q_0$ and hence they are still small perturbations of main contribution in $Q_1$ except for the zeroth order term due to the absence of inverse square potential in $Q_1$. We rewrite the last three lines as $\sum_{|k| = 0}^{|N|} R_{3,k} \tilde U_k$ by denoting each operator by $R_{3, k}$.
    Bearing this heuristic in mind and keeping track of \eqref{eqn:K_comm_general_comp}, it follows that \begin{equation}\label{eqn:improved_decay_for_error_terms_last_three_lines_in_tilde_U_N_calP}
        \begin{aligned}
        (K + 2 r^{\frac12}) R_{3,k} \tilde U_k
        &= \left(R_{3,k} + \calO(r^{-2})\p_\tau + \calO(\wp + r^{-1})r^{-1}\p_r\right) K\tilde U_k \\
        & \quad + 
        \calO(\wp r^{-\frac52} + r^{-\frac72}) \tilde U_k
        + \calO(r^{-\frac52}) \p_a\p_b \tilde U_k \\
        & \quad + \calO(r^{-\frac52}) \p_\tau^2 \tilde U_k +
         \calO(\wp r^{-\frac12} + r^{-\frac52})\p_r\p_\theta \tilde U_k 
         + \calO(r^{-\frac52})\p_\tau\p_\theta \tilde U_k \\
        & \quad + \calO(\wp r^{-\frac12} + r^{-\frac32}) \p_r \tilde U_k
        + \calO(r^{-\frac52}) \p_a \tilde U_k + \calO(r^{-\frac52}) \p_\tau \tilde U_k.
    \end{aligned}
    \end{equation}
    Note that \begin{equation}\label{eqn:terms_w_worst_decay_in_calG_2}
        \calO(\wp r^{-\frac12} + r^{-\frac52})\p_r\p_\theta \tilde U_k, \quad \calO(\wp r^{-\frac12} + r^{-\frac32}) \p_r \tilde U_k
    \end{equation}
    are now the terms with worse decay due to the comment after \eqref{eqn:K_comm_general_comp}.
    All the other terms in the last three lines \eqref{eqn:improved_decay_for_error_terms_last_three_lines_in_tilde_U_N_calP} picks up an additional order of decay due to \eqref{eqn:K_comm_general_comp}.
\end{proof}

\subsubsection{$r^p$-weighted estimates for $Y$}
In this part, we develop the Dafermos–Rodnianski $r^p$-hierarchy for $Y = K \tilde U$, where $U$ solves the equation $\calP_\graph U = \calF$. Then one knows from Lemma~\ref{lemma_computation_for_R_c} the equation solved by $Y$ and we will derive the $r^p$-weighted estimates for $Y$. Higher order estimates for $Y_N$ can be obtained without difficulty by dealing with the extra terms following exactly the same method in the proof of \eqref{eqn:main_higher_order_rp_est}. 
Let $\chi = \chi_R = \chi_{\geq R}$ be a smooth cut-off function with support in the hyperboloidal region and it also satisfies $\chi \equiv 1$ when $r \gtrsim R$, $\supp\chi \subset \{r \gtrsim R\}$.

Before proceeding, we modify the notation introduced in \eqref{eqn:shorthand_notation_Ep_Bp}: \[
    \tilde\calE^p[Y](\tau):= \int_{\Sigma_\tau} \chi_R r^p \big((\p_r Y)^2 + r^{-2} Y^2 + r^{-4} |\snabla Y|^2\big)\, \ud r\, \ud \theta.
\]

\begin{proposition}\label{prop_rp_est_for_Y_w_G12}
    Let $U$ be a solution to $\calP_\graph U = \calF$ and $Y = K\tilde U$. We assume that the bootstrap assumption \eqref{BA-dotwp} and the smallness assumption of $\calO(\wp)$ are satisfied. With the same assumptions on the smallness as in Proposition~\ref{prop_rp_for_U}, for all $\delta < p < \frac32 - 2\alpha$,
\begin{equation}\label{eqn:main_rp_est_for_Y_w_G12}
        \begin{aligned}
            \sup_{\tau_1 \leq \tau \leq \tau_2} \tilde\calE^p[ Y](\tau) &+ \calB^p[ Y](\tau_1, \tau_2) 
            \lesssim \ \tilde\calE^p[ Y](\tau_1) + \sum_{0 \leq |k| \leq 1} 
            \left(\calB^{p}[\tilde U_k](\tau_1, \tau_2) + \sup_{\tau_1 \leq \tau \leq \tau_2}  \calE^{p-1}[\tilde U_k](\tau)\right) \\ 
            &\quad + \int_{\calD_{\tau_1}^{\tau_2}} |\Err_\chi[ Y]| + \chi r^p \left|(\calG_1 + \calG_2) \left(\p_r Y + r^{-1} Y + r^{-3} \p_\theta Y + r^{-2}\p_\tau Y\right)\right|\, \ud \tau\, \ud r\, \ud\theta.
        \end{aligned}
    \end{equation}
    where $\delta$ is a small constant with $|\delta| \lesssim \eps + R^{-\alpha}$  and $\Err_\chi[ Y]$ denotes all terms multiplied by $\chi'$ and squares of derivatives of $ Y$ up to order $1$. Here, $\calG_1:= \calG_{1,0}$ and $\calG_2:= \calG_{2,0}$ are given in Lemma~\ref{lemma_computation_for_R_c}.
\end{proposition}
\begin{proof}
    We multiply $R_c Y = \calG_1 + \calG_2$ (see \eqref{eqn:operator_Rc_defn}) by $\chi r^p \p_r Y$ and obtain \begin{equation}\label{eqn:preliminary_rp_IBP_for_Y}
        \begin{aligned}
    &\int_{\calD_{\tau_1}^{\tau_2}} \chi\left(- r^p \p_\tau (\p_r  Y)^2 + \frac12 r^p \p_r (\p_r Y)^2
     - \frac12 r^{p-2} \p_r|\snabla Y|^2 - r^{p-1} \p_\tau Y \p_r Y - r^{p-1}|\p_r Y|^2\right)\, \ud \tau\, \ud r\, \ud \theta \\
    &\quad\quad\quad + \sum_{j=1}^4 \calR_{1j} 
    = \int_{\calD_{\tau_1}^{\tau_2}} \chi r^p (\calG_1 + \calG_2) \p_r  Y\, \ud \tau\, \ud r\, \ud \theta, 
\end{aligned}
    \end{equation}
where the remainder $\calR_{1j}$'s are given by 
\[
    \begin{aligned}
        \calR_{11} &= \int_{\calD_{\tau_1}^{\tau_2}} \chi r^p \left(\calO(r^{-2}) \p_\tau Y + 
        \calO(\wp + r^{-1})r^{-1}\p_r Y + \calO(\wp r^{-2} + r^{-4}) Y \right. \\ 
        & \quad\qquad \qquad\left. + \calO(\wp + r^{-2})\big(\p_r\p_\tau Y + \p_r^2 Y\big) + \calO(\wp r^{-2} + r^{-3}) \p_a\p_b Y \right)\p_r Y\, \ud \tau\, \ud r\, \ud \theta, \\
        \calR_{12} &= \int_{\calD_{\tau_1}^{\tau_2}} \chi r^p \tilde m_1^{\tau\tau} \p_\tau^2 Y \p_r Y\, \ud \tau\, \ud r\, \ud \theta, \\
        \calR_{13} &= \int_{\calD_{\tau_1}^{\tau_2}} \chi r^p \left( \calO(\wp r^{-1} + r^{-3})\p_r\p_\theta Y 
     + \calO(r^{-3})\p_\tau\p_\theta Y \right)\p_r Y\, \ud \tau\, \ud r\, \ud \theta, \\
        \calR_{14} &= \int_{\calD_{\tau_1}^{\tau_2}} \chi r^p \left( \calO(\wp r^{-1} + r^{-2}) \p_r Y + \calO(\wp r^{-2} + r^{-3}) \p_a Y + \calO(\dot\wp r^{-2} + r^{-3}) \p_\tau Y \right) \p_r Y\, \ud \tau\, \ud r\, \ud \theta.
    \end{aligned}
\]
Here, $\calR_{11}$ combines the error in $R_c$ and the first line in $R_{3,0}$, which can be viewed as a small perturbation of the major contribution in $R_c$ when performing $r^p$-estimates, like what we did for $\calR_1$ and $\calR_2$ in the proof of Proposition~\ref{prop_rp_for_U}. 

We shall use $\delta$ throughout the proof to denote a small constant satisfying \[
        |\delta| \leq C(\eps + R^{-\alpha}),
    \]
where $\alpha$ is a small constant $0 < \alpha \ll 1$.

    \textit{Step 1: Obtain \eqref{eqn:preliminary_rp_est_for_Y}, a preliminary $r^p$-weighted estimates for $ Y$. }
    In \eqref{eqn:preliminary_rp_IBP_for_Y},
    the fourth term is worth a separate discussion. Thanks to divergence theorem, \begin{equation}\label{eqn:preliminary_rp_IBP_for_Y_4th_term}
    \begin{aligned}
        &-\int_{\calD_{\tau_1}^{\tau_2}} \chi r^{p-1}\p_\tau Y \p_r Y\, \ud \tau\, \ud r\, \ud \theta \\
        &\qquad\qquad\qquad = \int_{\calD_{\tau_1}^{\tau_2}} \left(\chi r^{p-1}\p_\tau\p_r Y  + (p-1)\chi r^{p-2}\p_\tau Y  + \chi'(r) r^{p-1}\p_\tau Y\right) Y\, \ud \tau\, \ud r\, \ud \theta,
    \end{aligned}
    \end{equation}
    then we need to make use of the knowledge that $ Y$ solves the equation $R_c  Y = \calG_1 + \calG_2$ to replace the first term so that we have \begin{equation}\label{eqn:p_u_p_r_Phi_term}
        \begin{aligned}
        \int_{\calD_{\tau_1}^{\tau_2}} \chi r^{p-1} Y \p_\tau\p_r Y
        & = \frac12\int_{\calD_{\tau_1}^{\tau_2}} \chi r^{p-1}  Y \left( \p_r^2 Y + r^{-2} \sDelta_{\bbS^{n-1}} Y - r^{-1}\p_\tau Y - r^{-1}\p_r Y\right)\, \ud \tau\, \ud r\, \ud \theta \\
        &\quad + \sum_{j=1}^4\calR_{2j} - \frac12\int_{\calD_{\tau_1}^{\tau_2}} \chi r^{p-1}  Y (\calG_1 + \calG_2)\, \ud \tau\, \ud r\, \ud \theta \\ 
    \end{aligned}
    \end{equation}
    where the error terms are given by \[
    \begin{aligned}
        \calR_{21} &= \frac12\int_{\calD_{\tau_1}^{\tau_2}} \chi r^{p-1} Y \left(\calO(r^{-2}) \p_\tau Y + 
        \calO(\wp + r^{-1})r^{-1}\p_r Y + \calO(\wp r^{-2} + r^{-4}) Y \right. \\ 
        & \qquad\qquad\qquad\qquad\left. + \calO(\wp + r^{-2})\big(\p_r\p_\tau Y + \p_r^2 Y\big) + \calO(\wp r^{-2} + r^{-3}) \p_a\p_b Y \right)\, \ud \tau\, \ud r\, \ud \theta, \\
        \calR_{22} &= \frac12\int_{\calD_{\tau_1}^{\tau_2}} \chi r^{p-1} Y \tilde m_1^{\tau\tau} \p_\tau^2 Y \, \ud \tau\, \ud r\, \ud \theta, \\
        \calR_{23} &= \frac12\int_{\calD_{\tau_1}^{\tau_2}} \chi r^{p-1} Y \left( \calO(\wp r^{-1} + r^{-3})\p_r\p_\theta Y + \calO(r^{-3})\p_\tau\p_\theta Y \right)\, \ud \tau\, \ud r\, \ud \theta, \\
        \calR_{24} &= \frac12 \int_{\calD_{\tau_1}^{\tau_2}} \chi r^{p-1} Y \left( \calO(\wp r^{-1} + r^{-2}) \p_r Y + \calO(\wp r^{-2} + r^{-3}) \p_a Y \right.\\
        & \qquad\qquad\qquad\qquad\qquad\qquad\qquad\qquad\ \left. + \calO(\dot\wp r^{-2} + r^{-3}) \p_\tau Y \right) \, \ud \tau\, \ud r\, \ud \theta.
    \end{aligned}
    \]
    Furthermore, we integrate by parts to manipulate the first line of \eqref{eqn:p_u_p_r_Phi_term} as follows: \begin{equation}\label{eqn:p_u_p_r_Phi_term_first_line}
        \begin{aligned}
        & \frac12\int_{\calD_{\tau_1}^{\tau_2}} \chi r^{p-1}  Y \left( \p_r^2 Y + r^{-2} \sDelta_{\bbS^{n-1}} Y - r^{-1}\p_\tau Y - r^{-1}\p_r Y\right)\, \ud \tau\, \ud r\, \ud \theta \\
        =& \frac12\int_{\calD_{\tau_1}^{\tau_2}} -\chi r^{p-1} (\p_r Y)^2 - \frac{p-1}{2}\chi r^{p-2} \p_r Y^2 - \chi r^{p-3}|\snabla_S  Y|^2 - \frac12 \p_\tau\big(\chi r^{p-2} Y^2\big) - \frac12 \chi r^{p-2} \p_r Y^2\, \ud \tau\, \ud r\, \ud \theta \\
        =& \frac12\int_{\calD_{\tau_1}^{\tau_2}} - \frac12 \p_\tau\big(\chi r^{p-2} Y^2\big) -\chi r^{p-1} (\p_r Y)^2 - \frac{p(2-p)}{2}\chi r^{p-3}  Y^2 - \chi r^{p-3}|\snabla_S  Y|^2 \, \ud \tau\, \ud r\, \ud \theta + \Err_\chi[ Y].
    \end{aligned}
    \end{equation}
    Therefore, combining \eqref{eqn:preliminary_rp_IBP_for_Y_4th_term}, \eqref{eqn:p_u_p_r_Phi_term} and \eqref{eqn:p_u_p_r_Phi_term_first_line},
    we obtain \begin{equation}\label{eqn:anamolous_term_in_Rc_IBP}
        \begin{aligned}
        &-\int_{\calD_{\tau_1}^{\tau_2}} \chi r^{p-1}\p_\tau Y \p_r Y\, \ud \tau\, \ud r\, \ud \theta  \\
        &= \frac12\int_{\calD_{\tau_1}^{\tau_2}} 
        -(\frac32 - p) \p_\tau(\chi r^{p-2} Y^2) - \chi r^{p-1} (\p_r Y)^2 - \frac{p(2-p)}{2}\chi r^{p-3}  Y^2 - \chi r^{p-3}|\snabla_S  Y|^2
        \, \ud \tau\, \ud r\, \ud \theta  \\
        & \quad + \Err_\chi[ Y] + \sum_{j=1}^4\calR_{2j} - \frac12\int_{\calD_{\tau_1}^{\tau_2}} \chi r^{p-1}  Y (\calG_1 + \calG_2)\, \ud \tau\, \ud r\, \ud \theta.
    \end{aligned}
    \end{equation}
    Plugging this computation result into \eqref{eqn:preliminary_rp_IBP_for_Y}, one would reach \begin{equation}\label{eqn:preliminary_rp_est_for_Y}
        \begin{aligned}
    &-\int_{\calD_{\tau_1}^{\tau_2}} \chi \p_\tau \big(r^p(\p_r  Y)^2 + \frac{3-2p}4 r^{p-2} Y^2\big) \, \ud \tau\, \ud r\, \ud \theta \\
    &- \int_{\calD_{\tau_1}^{\tau_2}} \big(\frac{p+3}2 r^{p-1} (\p_r Y)^2
     + \frac{3-p}2 r^{p-3} |\snabla Y|^2  + \frac{p(2-p)}{4} r^{p-3} Y^2\big)\, \ud \tau\, \ud r\, \ud \theta 
    + \sum_{j=1}^4 \calR_{1j} + \sum_{j=1}^4\calR_{2j} \\
    &+ \Err_\chi[Y] - \frac12\int_{\calD_{\tau_1}^{\tau_2}} \chi r^{p-1}  Y (\calG_1 + \calG_2)\, \ud \tau\, \ud r\, \ud \theta 
    = \int_{\calD_{\tau_1}^{\tau_2}} \chi r^p (\calG_1 + \calG_2) \p_r  Y\, \ud \tau\, \ud r\, \ud \theta.
    \end{aligned}
    \end{equation}

    \textit{Step 2: Estimates regarding $\calR_{12}$.}
    As mentioned in Remark~\ref{rmk_calR_12_not_error_term}, producing any error term like \begin{equation}\label{eqn:bad_terms_from_R_12}
        \int_{\calD_{\tau_1}^{\tau_2}} r^{p - 3} |\p_\tau Y|^2\, \ud \tau\, \ud r\, \ud \theta \leq \int_{\calD_{\tau_1}^{\tau_2}} r^{-\frac32-\alpha} |\p_\tau Y|^2\, \ud \tau\, \ud r\, \ud \theta
    \end{equation}
    would be bad as there is no tool for estimating them. 
    For $\calR_{12}$, we need to exploit the nature of $\tilde m_1^{\tau\tau}$. We compute \begin{equation}\label{eqn:DBP_for_calR12}
        \chi r^p \tilde m_1^{\tau\tau} \p_\tau^2 Y \p_r Y
        = \p_\tau\left(\chi r^p \tilde m_1^{\tau\tau} \p_\tau Y \p_r Y\right) - \chi r^p \calO(\dot\wp r^{-2} + r^{-3}) \p_\tau Y \p_r Y - \chi r^p \tilde m_1^{\tau\tau} \p_\tau Y \p_r\p_\tau Y.
    \end{equation}
    After integrating over $\calD_{\tau_1}^{\tau_2}$, 
    the first term can be estimated by \begin{equation}\label{eqn:IBP_for_flux_term_in_calR12}
        \sup_{\tau_1 \leq \tau \leq \tau_2 }\left(\veps\int_{\Sigma_\tau} r^p (\p_r Y)^2\, \ud r\, \ud \theta + C_\veps \int_{\Sigma_\tau} r^{p-4} (\p_\tau Y)^2\, \ud r\, \ud \theta\right),
    \end{equation}
    where the first term can be absorbed and the second one can be estimated by $\sum_{|k| = 1}\calE^{p-1}[\tilde U_k](\tau)$. The second term in \eqref{eqn:DBP_for_calR12} can be estimated by integration by parts as in \eqref{eqn:preliminary_rp_IBP_for_Y_4th_term} with one power less in $r$. Hence, they would produce terms $\calR_{1j}$ ($1 \leq j \leq 4$) with $r$ exponent lowered by one. Then we would deal with the similar term as $\calR_{12}$, the following estimate will be applied to replace the techniques in \eqref{eqn:preliminary_rp_IBP_for_Y_4th_term} to avoid endless replacements: 
    \begin{equation}\label{eqn:p_tau_Y_p_r_Y_est}
        \begin{aligned}
        \int_{\calD_{\tau_1}^{\tau_2}} \chi r^{p-1} \calO(\dot\wp r^{-2} + r^{-3}) \p_\tau Y \p_r Y\, \ud \tau\, \ud r\, \ud \theta &\lesssim \delta \int_{\calD_{\tau_1}^{\tau_2}} \chi\left(r^{p-1}|\p_r Y|^2 + r^{p-5}|\p_\tau Y|^2\right)\, \ud \tau\, \ud r\, \ud\theta \\ &\lesssim \delta \left(\calB^p[Y] + \sum_{0\leq |k|\leq 1} \calB^{p-1}[\tilde U_k]\right).
    \end{aligned}
    \end{equation}
    For the third term in \eqref{eqn:DBP_for_calR12}, the obvious integration by parts will lead to the wrong sign (see Remark~\ref{rmk_calR_12_not_error_term}). Instead, we replace $\p_r\p_\tau Y$ by using $R_c Y = \calG_1 + \calG_2$. This leads to the following \begin{equation}\label{eqn:third_term_in_calR12}
        \begin{aligned}
        &- \int_{\calD_{\tau_1}^{\tau_2}} \chi r^p \tilde m_1^{\tau\tau} \p_\tau Y \p_r\p_\tau Y\, \ud \tau\, \ud r\, \ud \theta \\ 
        =& - \frac12 \int_{\calD_{\tau_1}^{\tau_2}} \chi r^p \tilde m_1^{\tau\tau} \p_\tau Y \left( \p_r^2 Y - r^{-1}\p_\tau Y - r^{-1} \p_r Y + r^{-2} \sDelta_{\bbS^{n-1}}Y \right. \\ &\quad\qquad\left. +  \calO(r^{-2})\p_\tau Y + \calO(\wp + r^{-1})r^{-1}\p_r Y) + R_{3,0} Y - \calG_1 - \calG_2\right)\, \ud \tau\, \ud r\, \ud \theta \\
        =&\ \int_{\calD_{\tau_1}^{\tau_2}} \frac12 \chi r^{p-1} \tilde m_1^{\tau\tau} (\p_\tau Y)^2 + \frac14 \chi r^{p-2} \tilde m_1^{\tau\tau} \p_\tau |\snabla Y|^2 + \frac12 \chi \calO(r^{p-2}) \p_\tau Y (\calG_1 + \calG_2)\, \ud \tau\, \ud r\, \ud \theta \\
        &- \int_{\calD_{\tau_1}^{\tau_2}} \chi r^p\left( \calO(r^{-2})\p_\tau Y \p_r^2 Y + \calO(r^{-3})\p_\tau Y \p_r Y + \calO(r^{-4})|\p_\tau Y|^2 + \calO(r^{-2})\p_\tau Y R_{3,0}Y \right)\, \ud \tau\, \ud r\, \ud \theta.
    \end{aligned}
    \end{equation}
    Here, the first term (with the bad decay rate mentioned in \eqref{eqn:bad_terms_from_R_12}) will have exact cancellation with \eqref{eqn:calR22_IBP} and the second term has favorable sign ($\tilde m_1^{\tau\tau} \sim -r^{-2}$) for putting into the flux term. We simply put the third term on the right hand side. For the last line, we integrate by parts and estimate like in \eqref{eqn:p_tau_Y_p_r_Y_est} for similar terms. The third term in the last line is worth a separate remark that it is bounded by $\sum_{0\leq |k|\leq 1} \calB^{p}[\tilde U_k]$, which is due to the absence of $\calO(\wp)r^{-1}\p_\tau Y$ term in $R_c$, mentioned in Remark~\ref{rmk_on_eqn_of_P_graph_coeff}.

    \textit{Step 3: Estimates regarding $\calR_{22}$.}
    Analogous to $\calR_{12}$, it will produce bad terms like \eqref{eqn:bad_terms_from_R_12} by integration by parts: \begin{equation}\label{eqn:calR22_IBP}
        \calR_{22} = -\frac12 \int_{\calD_{\tau_1}^{\tau_2}} \chi r^{p-1} \tilde m_1^{\tau\tau} (\p_\tau Y)^2 \, \ud \tau\, \ud r\, \ud \theta - \int_{\calD_{\tau_1}^{\tau_2}} \chi \calO(\dot\wp r^{p-3}) \p_\tau Y^2 + \p_\tau (\chi \calO(r^{p-3}) Y \p_\tau Y),
    \end{equation}
    where the first term causes exact cancellation with the first term in \eqref{eqn:third_term_in_calR12} produced by $\calR_{12}$ (see \eqref{eqn:DBP_for_calR12}). The following two terms in \eqref{eqn:calR22_IBP} can be easily handled by Cauchy-Schwarz after converting to flux terms like in \eqref{eqn:IBP_for_flux_term_in_calR12}.
    
    \textit{Step 4: Estimate $\calR_{1j}$ and $\calR_{2j}$ ($j = 1, 3, 4$).}
   
    For $\calR_{11}$, we would like to handle it using divergence theorem and exploiting the smallness of $\wp$ and $R^{-1}$ to absorb to the left hand side of \eqref{eqn:preliminary_rp_est_for_Y}. This works for all the terms in $\calR_{11}$ except the first one. However, we could still mimic the computation in \eqref{eqn:anamolous_term_in_Rc_IBP} with one power less in the exponent of $r$ for all terms. The significance of the improvement on decay in $r$ by $1$ will give us the freedom to treat the similar term of $\calR_{22}$ in this process as an error term $\sum_{0\leq |k| \leq 1} \calB^p[\tilde U_k]$ in the final estimate.

    A similar discussion would suffice to treat the last term in $\calR_{14}$ as well.
    
    For other terms we haven't mentioned, all could be bounded in a similar spirit of the estimates for $\calR_3, \calR_4, \calR_5, \calR_6$.
\end{proof}

\begin{remark}\label{rmk_calR_12_not_error_term}
    It might be tempting to skip the discussions of $\calR_{12}$, $\calR_{13}$ and $\calR_{14}$ as they look quite similar to $\calR_3$, $\calR_4 + \calR_5$ and $\calR_6$, respectively. However, this may lead to the following estimate \[
        \begin{aligned}
            \sup_{\tau_1 \leq \tau \leq \tau_2} \tilde\calE^p[ Y](\tau) &+ \calB^p[ Y](\tau_1, \tau_2) 
            \lesssim \ \tilde\calE^p[ Y](\tau_1) + \int_{\calD_{\tau_1}^{\tau_2}} \chi r^p \left|(\calG_1 + \calG_2) \left(\p_r  Y + \calO(r^{-3}) \p_\theta Y\right)\right|\, \ud \tau\, \ud r\, \ud\theta \\
            &\quad + \int_{\calD_{\tau_1}^{\tau_2}} |\Err_\chi[ Y]|\, \ud \tau\, \ud r\, \ud \theta + \delta \int_{\calD_{\tau_1}^{\tau_2}} \chi r^{-\frac32-\alpha} |\p_\tau  Y|^2\, \ud \tau\, \ud r\, \ud \theta + \delta \sum_{0 \leq |k| \leq 1}  \sup_{\tau_1 \leq \tau \leq \tau_2} \calE^{\frac32}[U_k](\tau).
        \end{aligned}
    \]
    The change in $r$ weights in the penultimate term is because the manipulation for $\calR_3$ through $\calR_6$ is highly dependent on the $p$ range. Though it is improved at the level of $Y$, it turns out to be problematic 
    \[
        \int_{\calD_{\tau_1}^{\tau_2}} \chi r^{-\frac32-\alpha} |\p_\tau  Y|^2\, \ud \tau\, \ud r\, \ud \theta \sim \int_{\calD_{\tau_1}^{\tau_2}} \chi r^{-\frac12-\alpha} |\p_\tau \tilde U|^2\, \ud \tau\, \ud r\, \ud \theta
    \]
    in view of Hardy's estimates and the absence of ILED for the equation satisfied by $Y$.
    Instead, we need to perform a more delicate analysis and exploit the negative sign of $\tilde m_1^{\tau\tau}$ in \eqref{eqn:third_term_in_calR12} so that such problematic terms do not appear.
\end{remark}

Now we state the higher order $r^p$-weighted estimates for $Y$. The proof is omitted and will follow from a similar derivation as in Proposition~\ref{eqn:main_higher_order_rp_est}.

\begin{proposition}
    Let $U$ be a solution to $\calP_\graph U = \calF$ and $Y_k = K\tilde U_k$ with $\tilde U_k$ defined in \eqref{eqn:notation_tilde_U_N_triple} for any triple $k$. We assume that the bootstrap assumption \eqref{BA-dotwp} and the smallness assumption of $\calO(\wp)$ are satisfied. With the same assumptions on the smallness as in Proposition~\ref{prop_rp_for_U}, for all $\delta < p < \frac32 - 2\alpha$, for any $N \in \bbN$,
\begin{equation}
        \begin{aligned}
            &\sum_{0 \leq |k| \leq N} \sup_{\tau_1 \leq \tau \leq \tau_2} \tilde\calE^p[ Y_k](\tau) + \calB^p[ Y_k](\tau_1, \tau_2) \\
            \lesssim& \ \sum_{0 \leq |k| \leq N} \tilde\calE^p[ Y_k](\tau_1) + \sum_{0 \leq |k| \leq N+1} 
            \left(\calB^{p}[\tilde U_k](\tau_1, \tau_2) + \sup_{\tau_1 \leq \tau \leq \tau_2}  \calE^{p-1}[\tilde U_k](\tau)\right) \\ 
            +& \int_{\calD_{\tau_1}^{\tau_2}} |\Err_\chi[ Y_N]| + \sum_{0 \leq |k| \leq N} \chi r^p \left|  (\calG_{1,k} + \calG_{2,k}) \left(\p_r Y_k + r^{-1} Y_k + r^{-3} \p_\theta Y_k + r^{-2}\p_\tau Y_k\right)\right|\, \ud \tau\, \ud r\, \ud\theta,
        \end{aligned}
    \end{equation}
    where $\delta$ is a small constant with $|\delta| \lesssim \eps + R^{-\alpha}$  and $\Err_\chi[ Y]$ denotes all terms multiplied by $\chi'$ and squares of derivatives of $ Y$ up to order $1$. Here, $\calG_{1,k}$ and $\calG_{2,k}$ are defined in Lemma~\ref{lemma_computation_for_R_c}.
\end{proposition}

\subsection{Decay estimates for \texorpdfstring{$\mathbb{\Phi}$}{Phi}}\label{Section:decay_for_Phi}

Now we focus on the equation satisfied by $U = \mathbb\Phi$. Due to the presence of nonlinearity and the coupling between $\bbP$ and $\mathbb\Phi$ in the nonlinearity, we would expect to restrict the highest derivatives to start with by $|k| \leq M - 20$ so that we could apply any bootstrap assumptions regarding $p$ in the nonlinearity. On the other hand, following the same arguments as in Section~\ref{Section:decay_for_P}, without resorting to the $L^1 L^2$ type norm, one could prove the same kind of decay estimates with different $k$ bounds easily. In fact, we will stick to $L^2L^2$ type of norms in this subsection.
To this end, we give a holistic discussion of how we deal with the nonlinearity.
\begin{remark}\label{rmk_handle_nonlinearity}
    \textbf{A summary of how the source term and the nonlinearity are handled: }
\begin{enumerate}[ wide = 0pt, align = left, label=\arabic*. ]
    \item Since we are using $L^2 L^2$ type of norms, for $\|\calF_{0, \mathbb\Phi_k}\|^2_{L^2_\uptau}$ and $\|\p_\uptau^j \calF_{0, \mathbb\Phi_k}\|^2_{L^2_\uptau}$ ($j \geq 1$), \eqref{BA-dotwp},  \eqref{BA-ddotwp-L2} allow us to bound them by $\eps^2\brk{\uptau}^{-\frac72 + 2\kappa}$ and $\eps^2\brk{\uptau}^{-\frac{9}2 + 2\kappa}$ with suitable decay in $r$, respectively. If one examine \eqref{eqn:calP_q_source_term} and the relation between $q$ and $\phi$ in \eqref{eqn:q_decomp}, the terms $g_1 + \p_\uptau g_2$ and $\calP_\pert p$ (see \eqref{eqn:calP_pert_in_C_hyp}) also appear as the source term. Depending whether there is $\p_\uptau$ falling on them or not, they can be bounded by the same bounds stated above thanks to \eqref{eqn:all_decay_of_c_im},  \eqref{eqn:L2_est_of_c_im} and \eqref{BA-dotwp},  \eqref{BA-ddotwp-L2}.
    \item For the $L^2L^2$ type of norms of $\calF_{2, k}:= (\calF_2)_k$ with $|k| \leq M - 20$, we examine the exact formula in \eqref{eq:calF2_1} and classify quadratic terms into the following two types (expanding $\varphi = \bbP + \mathbb\Phi$): 
    \begin{enumerate}
        \item if the two slots for $\varphi$ both have contribution from $\bbP$, that is, for terms like $\nabla^\mu Q \nabla^\nu \bbP \nabla_{\mu\nu}^2 \bbP$, we apply \eqref{BA-p} and \eqref{BA-p-ext1} and it is of the form $\calO(r^{-4}) \brk{\uptau}^{-3 + \frac{3\zeta}2}$, which is even better than $\calF_{0, \mathbb\Phi_k} = \calO(\dot\wp r^{-4})$; (we can find an even better bound $\calO(r^{-5}) \brk{\uptau}^{-4 + 2\zeta}$ thanks to the null condition)
        \item all other terms with $\mathbb\Phi$ involved, we single out one copy and estimate whatever left by bootstrap assumptions, then this can be seen as a perturbation (an error term) of the operator as in \eqref{eqn:tilde_U_N_calP}, i.e., they are handled by integration by parts of \[
            \int_{\calD_{\tau_1}^{\tau_2}} \chi r^p \p_r \mathbb{\Phi} \tilde\calF_2\, \ud \tau\, \ud r\, \ud \theta
        \]
        without putting anything into $L^2L^2$-type norm. The necessity of applying integration by parts is due to the loss of derivative issue caused by the quasilinearity so that we could not apply \eqref{BA-q} if the top order derivative falls on $\mathbb\Phi$.
    \end{enumerate}
    \item For $\calF_{3, k}$, one needs to exploit the null condition by examining the closed form of $\calF_3$
    (see \eqref{eq:calF3_1}) and the inverse metric $m^{-1} = m_0^{-1} + \tilde m_1$ (see \eqref{eqn:m_0_inv_formula} and \eqref{eqn:tilde_m_1_formula}). It turns out that $\nabla^\mu$ either contains good derivative $\p_r$ or one will gain $r^{-2}$ decay and as a matter of fact, one knows that a typical term \[
        \nabla^\mu \bbP \nabla^\nu \bbP \nabla_{\mu\nu} \bbP = \calO(r^{-2})(r^{-1}\brk{\uptau}^{-1 + \frac{\zeta}2})^2 \brk{\uptau}^{-2+\zeta} = \calO(r^{-4}) \brk{\uptau}^{-4 + 2\zeta},
    \]
    which is much more better than $\calF_{0, \mathbb\Phi}$.
    On the other hand, if there is $\mathbb\Phi$ involved, one will still single out this term and estimate the others using the null condition and bootstrap assumptions, with the term finally dealt by integration by parts as other error terms.
\end{enumerate}
\end{remark}

Therefore, with these in mind, like in \eqref{eqn:preliminary_rp_for_bbP_w_Fp_realized}, one can show that 
\[
    \begin{aligned}
     \sup_{\tau_1 \leq \tau \leq \tau_2} \calE^p[\tilde{\mathbb{\Phi}}](\tau) + \calB^p[\tilde{\mathbb{\Phi}}](\tau_1, \tau_2) \lesssim_p &\ \calE^p[\tilde{\mathbb{\Phi}}](\tau_1) + \int_{\calD_{\tau_1}^{\tau_2}} |\Err_\chi[\tilde{\mathbb{\Phi}}]| \, \ud \tau\, \ud r\, \ud\theta + \eps^2 \brk{\tau_1}^{-\frac72 + 2\kappa} \\
     &\ + \delta \int_{\calD_{\tau_1}^{\tau_2}} \chi r^{-1-\alpha} |\p_\tau \tilde{\mathbb{\Phi}}|^2\, \ud \tau\, \ud r\, \ud \theta + \delta \sup_{\tau_1 \leq \tau \leq \tau_2} E[\mathbb{\Phi}](\tau).
    \end{aligned}
\]
Then adding a suitable multiple of \eqref{ILED-Phi} with one more order, 
we obtain a similar estimate as in \eqref{eqn:after_adding_a_multiple_LED_P}: \begin{equation}\label{eqn:rp_to_derive_tau_inverse_decay_Phi}
    \begin{aligned}
    \sum_{0 \leq |k| \leq N} \sup_{\tau_1 \leq \tau \leq \tau_2} \scE^p [\mathbb{\Phi}_k](\tau) + \scB^p [\mathbb{\Phi}_k](\tau_1, \tau_2) &\lesssim \sum_{0 \leq |k| \leq N+1} \scE^p [\mathbb{\Phi}_k](\tau_1) + \eps^2 \brk{\tau_1}^{-\frac72 + 2\kappa},
\end{aligned}
\end{equation}
where $N + 1 \leq M - 20$ and $\delta < p < 2 - 2\alpha$. 

Instead of pursuing almost $\brk{\tau}^{-2}$ decay of the standard energy like in Section~\ref{Section:decay_for_P}, we start with an $r^p$-weighted estimates for $\Phi:= K \mathbb{\Phi}$.

We still start with the $N = 0$ case for simplicity.
For $0 < p < \frac32$, it follows from Proposition~\ref{prop_rp_est_for_Y_w_G12} that \begin{equation}\label{eqn:preliminary_rp_for_Y_Phi}
    \begin{aligned}
     \sup_{\tau_1 \leq \tau \leq \tau_2} \tilde\calE^p[\Phi](\tau) + \calB^p[\Phi](\tau_1, \tau_2) 
            \lesssim &\ \tilde\calE^p[\Phi](\tau_1) + \int_{\calD_{\tau_1}^{\tau_2}} \chi r^{p+1} (|\calG_{\Phi,1}|^2 + |\calG_{\Phi,2}|^2) + |\Err_\chi[\Phi]| \, \ud \tau\, \ud r\, \ud\theta \\
            &\quad + \sum_{0 \leq |k| \leq 1} 
            \left(\calB^{p}[\tilde{\mathbb\Phi}_k](\tau_1, \tau_2) + \sup_{\tau_1 \leq \tau \leq \tau_2} \calE^{p-1}[\tilde{\mathbb\Phi}_k](\tau)\right).  \\
    \end{aligned}
\end{equation}

Compared to the estimates in Section~\ref{Section:decay_for_P}, the most tricky part is the $\calG_2$ portion in the source terms \begin{equation}\label{eqn:est_G_2_terms_by_U_ests}
    \int_{\calD_{\tau_1}^{\tau_2}} \chi r^{p+1} |\calG_{\Phi,2}|^2 \, \ud \tau\, \ud r\,\ud\theta, \quad 0 < p < \frac32.
\end{equation}

\begin{remark}\label{rmk_calG_2_treatment_precise_decay_crucial}
    To treat the term \eqref{eqn:est_G_2_terms_by_U_ests} involving $\calG_2$, we make use of the precise decay rate in \eqref{eqn:calG_2N_defn}. To illustrate, we mention that $\calO(r^{-\frac52})\p_\tau \tilde U_k$ is harmless while $\calO(r^{-\frac32})\p_\tau \tilde U_k$ is problematic: \begin{equation}\label{eqn:bdd_difficult_term_in_G_2}
    \begin{aligned}
    &\int_{\calD_{\tau_1}^{\tau_2}} \chi_R r^{p+1} |\calO(r^{-\frac52}) \p_\tau \tilde U|^2 \, \ud \tau\, \ud r\, \ud\theta 
    \simeq \int_{\calD_{\tau_1}^{\tau_2}} \chi_R r^{p-4} |\p_\tau \tilde U|^2 \, \ud \tau\, \ud r\, \ud\theta \lesssim \calB^{p-1}[\tilde U].
\end{aligned} 
\end{equation}
\end{remark}

It is then easy to estimate all the terms in $\calG_2$ in a similar fashion (see \eqref{eqn:bdd_difficult_term_in_G_2}) except the two terms in  \eqref{eqn:terms_w_worst_decay_in_calG_2}. For these two, we rewrite $\p_r \mathbb\Phi$ into $r^{-\frac32} \Phi$ and then absorb it to the left hand side due to smallness of size $\wp$ or $R^{-1}$.
Then adding a multiple of \eqref{ILED-Phi} to deal with the $\Err_\chi$ term in \eqref{eqn:preliminary_rp_for_Y_Phi} and bounding $\calB^p$ by $\calE^p$ in terms of \eqref{eqn:rp_to_derive_tau_inverse_decay_Phi}, we obtain \[
\begin{aligned}
    &\sup_{\tau_1 \leq \tau \leq \tau_2} \uscE^p [\mathbb{\Phi}](\tau) + \uscB^p [\mathbb{\Phi}](\tau_1, \tau_2) \\
    \lesssim &\ \uscE^p [\mathbb{\Phi}](\tau_1) + \int_{\calD_{\tau_1}^{\tau_2}} \chi_{R} r^{p+1} (|\calG_{\Phi, 1}|^2 + \sum_{|k| = 0}^2 |\calF_{0, \mathbb\Phi_k}|^2)\, \ud V\, \ud \tau + \sum_{0 \leq |k| \leq 2} \sup_{\tau_1 \leq \tau \leq \tau_2} \calE^{p}[\tilde{\mathbb\Phi}_k](\tau) + \eps^2 \brk{\tau_1}^{-\frac72 + 2\kappa} \,
\end{aligned}
\]
where \[
    \begin{aligned}
    \uscE^p[U](\tau) &:= \int_{\Sigma_\tau} \chi_{\leq R} |\p U|^2 + \chi_R r^p \left((\p_r + \frac{3}{2r})\p_r \tilde U\right)^2 + \chi_R r^{p-4}|\snabla \p_r \tilde U|^2 \, \ud V, \\ 
    \uscB^p[U](\tau_1, \tau_2) &:= \int_{\calD_{\tau_1}^{\tau_2}} \chi_{\leq R} |\p U|^2 + \chi_R r^{p-1} \left( \big((\p_r + \frac{3}{2r})\p_r \tilde U\big)^2 + r^{-2} (\p_r \tilde U)^2 + r^{-2} |\snabla \p_r \tilde U|^2\right) \\ & \qquad\qquad + \chi_R r^{-1-\alpha} |\p_\tau U|^2\, \ud V \, \ud \tau =: \int_{\tau_1}^{\tau_2} \uscB^p[U](\tau)\, \ud \tau.
\end{aligned}
\]
One can see the discussion in Remark~\ref{rmk_on_adding_LED_est} to facilitate the understanding of the details. The following higher order estimates can be proved in the same spirit: \begin{equation}\label{eqn:higher_order_rp_for_Y_in_Phi}
    \begin{aligned}
    &\sum_{0 \leq k \leq N} \sup_{\tau_1 \leq \tau \leq \tau_2} \uscE^p [\mathbb{\Phi}_k](\tau) + \uscB^p [\mathbb{\Phi}_k](\tau_1, \tau_2) \\
    \lesssim &\sum_{0 \leq k \leq N} \uscE^p [\mathbb{\Phi}_k](\tau_1) + \int_{\calD_{\tau_1}^{\tau_2}} \chi_{R} r^{p+1} (|\calG_{\Phi_k, 1}|^2 + \sum_{|k| = 0}^{N+2} |\calF_{0, \mathbb\Phi_k}|^2)\, \ud V\, \ud \tau \\
    &\ + \sum_{|k| = 0}^{N+2} \sup_{\tau_1 \leq \tau \leq \tau_2} \calE^p[\tilde{\mathbb\Phi}_k](\tau) + \eps^2 \brk{\tau_1}^{-\frac72 + 2\kappa},
\end{aligned}
\end{equation}
where $N \leq M - 23$ and $\alpha < p < \frac32$. 

\begin{remark}
    The method in this subsection does not work for $P:= K\tilde\bbP$ due to the non-integrability of \[
    \int_{\calD_{\tau_1}^{\tau_2}} \chi_{R} r^{p+1} |\calG_{P_k, 1}|^2\, \ud V\, \ud \tau \lesssim \int_{\calD_{\tau_1}^{\tau_2}} \chi_{R} r^{p+1} |\calO(r^{\frac12}\tau^{-2}r^{-3})|^2\, \ud V\, \ud \tau
    \]
    for any $p > 0$.
\end{remark}

Choose $p = \frac32 - \kappa$, $N = M - 23$, $\tau_1 = 0$ in \eqref{eqn:higher_order_rp_for_Y_in_Phi} and applying the uniform boundedness of $\calE^p$ obtained through \eqref{eqn:rp_to_derive_tau_inverse_decay_Phi}, we know that the left hand side of \eqref{eqn:higher_order_rp_for_Y_in_Phi} is then bounded by a bootstrap constant independent multiple of $\eps^2$. Therefore, there exists an increasing sequence of dyadic $\{\usigma_m\}$ such that \[
    \uscE^{\frac12 - \kappa}[\mathbb\Phi_k](\usigma_m) \leq \uscB^{\frac32 - \kappa}[\mathbb\Phi_k](\usigma_m) \lesssim \eps^2 \brk{\usigma_m}^{-1}, \quad |k| \leq M - 23.
\]
Furthermore, a direct computation by Hardy's inequality reveals that \[\begin{aligned}
    \uscE^{\frac12 - \kappa}[U](\tau) \gtrsim& \int_{\Sigma_\tau} \chi_{\leq R} |\p U|^2 \, \ud V + \chi_R r^{-\frac32-\kappa}(\p_r \tilde U)^2\, \ud V \\
    \gtrsim& \int_{\Sigma_\tau} \chi_{\leq R} |\p U|^2 \, \ud V + \chi_R r^{\frac32-\kappa}\big((\p_r + \frac3{2r})U\big)^2\, \ud V = \scE^{\frac32 - \kappa}[U](\tau),
\end{aligned}
\]
which shows that \begin{equation}\label{eqn:p_frac32_eng_decay_Phi}
    \scE^{\frac32 - \kappa}[\mathbb\Phi_k](\usigma_m) \lesssim \eps^2 \brk{\usigma_m}^{-1}, \qquad |k| \leq M - 23.
\end{equation}
Now we plug this result back into \eqref{eqn:rp_to_derive_tau_inverse_decay_Phi} with $p = \frac32 - \kappa$, $N = M - 24$, $\tau_1 = \usigma_{m-1}$, $\tau_2 = \usigma_m$ would imply that the left hand side of \eqref{eqn:rp_to_derive_tau_inverse_decay_Phi} is bounded by $\eps^2 \brk{\usigma_{m-1}}^{-1}$. This further tells us that there exists a slightly different increasing dyadic sequence $\{\tsigma_m\}$ with $\tsigma_m \simeq \usigma_m$ such that \begin{equation}\label{eqn:sigma_decay_Phi_2}
    \scE^{\frac12 - \kappa}[\mathbb\Phi_k](\tsigma_m) \lesssim \scB^{\frac32 - \kappa}[\mathbb\Phi_k](\tsigma_m) \lesssim \eps^2 \brk{\tsigma_m}^{-2}, \qquad |k| \leq M - 24.
\end{equation}
Repeat the former step again with $p = \frac12 - \kappa$, $N = M - 25$, $\tau_1 = \tsigma_{m-1}$ and $\tau_2 = \tsigma_m$, we obtain \begin{equation}\label{eqn:sigma_decay_Phi_3}
    \scE^{-\frac12 - \kappa}[\mathbb\Phi_k](\sigma_m) \lesssim \scB^{\frac12 - \kappa}[\mathbb\Phi_k](\sigma_m) \lesssim \eps^2 \brk{\sigma_m}^{-3}, \qquad |k| \leq M - 25,
\end{equation}
where $\{\sigma_m\}$ is an increasing dyadic sequence with $\tsigma_{m-1} \leq \sigma_m \leq  \tsigma_m$. In this step, it is crucial that the decay rate in \eqref{eqn:rp_to_derive_tau_inverse_decay_Phi} satisfies $\frac72 - 2\kappa > 3$.
With another application of \eqref{eqn:rp_to_derive_tau_inverse_decay_Phi} with $p = \frac12 - \kappa$, $N = M - 25$, $\tau_1 = \tsigma_{m-1}$, $\tau_2 = \sigma_m$, we could synchronize the slight different dyadic sequence in \eqref{eqn:sigma_decay_Phi_2} with \eqref{eqn:sigma_decay_Phi_3}: \begin{equation}\label{eqn:sigma_decay_Phi_2_sync}
    \scE^{\frac12 - \kappa}[\mathbb\Phi_k](\sigma_m) \lesssim \eps^2 \brk{\tsigma_{m-1}}^{-2} \lesssim \eps^2 \brk{\sigma_m}^{-2}, \qquad |k| \leq M - 25.
\end{equation}
Interpolating \eqref{eqn:sigma_decay_Phi_3} and \eqref{eqn:sigma_decay_Phi_2_sync} via \eqref{eqn:interpolation_for_rp} with $p = 0$, $\veps = \frac12 + \kappa$, $s = 1$ and $q = 3$ , we obtain \[
    \sum_{|j| \leq |k|} E[\mathbb\Phi_j](\sigma_m) \lesssim \sum_{|j| \leq |k| + 1} \scE^0[\mathbb\Phi_j](\sigma_m) \lesssim \eps^2 \brk{\sigma_m}^{-\frac52 + \kappa}, \quad |k| \leq M - 26.
\]
Then by \eqref{ILED-Phi}, it has the same decay rate at all times and hence \eqref{IBA-q-ext1} and \eqref{IBA-q-ext2} follow.
Though we might not need, we remark that once we pass to $\calE^p$ norm with nonnegative $p$, we could upgrade from a dyadic sequence to all times.
The algebraic structure observed in \eqref{eqn:algebraic_relation_between_eng_of_bbP_and_p_tau_bbP} can be adapted here, applying \eqref{eqn:sigma_decay_Phi_3}, one shows that \[
    \sum_{|j| \leq |k|} \scE^{\frac32 - \kappa}[\p_\tau \mathbb\Phi_j](\sigma_m) \lesssim \sum_{|j| \leq |k| + 1} \scE^{-\frac12-\kappa}[\bbP_j](\sigma_m) + \eps^2 \brk{\sigma_m}^{-\frac92 + 2\kappa} \lesssim \eps^2 \brk{\sigma_m}^{-3}, \quad |k| \leq M - 26.
\]
One could run the procedure with $p = \frac32 - \kappa$ and interpolate to obtain decay rate for $\calE^{1 + \kappa}[\p_\tau \mathbb\Phi_k]$, then run the procedure with $p = 1 + \kappa$ to conclude that \[
    \calE^\kappa[\p_\uptau \mathbb\Phi_k] \lesssim \eps^2 \brk{\tau}^{-\frac92 + 2\kappa}, \quad |k| \leq M - 28,
\]
which is of the same decay rate as that of the contribution of source term. This again explains why we need to interpolate first before we assign the final value $p = 1 + \kappa$ to avoid loss of decay. Besides, it is also worthwhile to make sure that the square of the $L^2L^2$ type contribution from both the nonlinearity won't restrict the decay rate at this stage. In particular, we would see at least $\brk{\tau}^{-\frac92 + 2\kappa}$ decay could be achieved based on bootstrap assumptions and all previous decay estimates.
Recall the discussion in Remark~\ref{rmk_handle_nonlinearity}, we note that when $\p_\tau$ hits the nonlinearity, after distributing it by product rule, we could still follow the procedure mentioned before except for the terms not containing $(\p_\tau \mathbb\Phi)_k$ as $\p_\tau$ might falls on other slots. For this, we do not try to absorb it to left hand side again but for this kind of terms, we could bound the quadratic ones by \begin{equation}\label{eqn:how_to_handle_quadratic_after_p_tau}
    \begin{aligned}
    &\int_{\calD_{\tau_1}^{\tau_2}} \chi r^{p+1} \sum_{|i| +|j| \leq |k|} |\calO(r^{-3}) \nabla \p_\tau \bbP_i \nabla^2 \tilde{\mathbb\Phi}_j|^2 + |\calO(\dot\wp r^{-3}) \nabla \bbP_i \nabla^2 \tilde{\mathbb\Phi}_j|^2\, \ud \tau\, \ud r\, \ud \theta \\
    \lesssim&\ \eps^2 \brk{\tau_1}^{-\frac72 + 2\kappa} \sup_{\tau_1 \leq \tau \leq \tau_2} E[\mathbb\Phi_k](\tau) \lesssim \eps^2\brk{\tau_1}^{-5}, \quad 0 < p < \frac32 - \kappa, \quad |k| \leq M  - 28, 
\end{aligned}
\end{equation}
thanks to \eqref{BA-dtau-p} and \eqref{BA-dotwp} while the cubic ones are bounded by \begin{equation}\label{eqn:bound_for_cubic_after_p_tau}
     \lesssim \eps^2 \brk{\tau_1}^{-\frac72 + 2\kappa} \sup_{\tau_1 \leq \tau \leq \tau_2} \calB^{\frac32 - \kappa}[\tilde{\mathbb\Phi}](\tau) \lesssim \eps^2 \brk{\tau_1}^{-\frac92 + 2\kappa}, \quad 0 < p \leq \frac32 - \kappa, \quad |k| \leq M  - 28,
\end{equation}
where the first estimate is due to \eqref{BA-dtau-p} and the null condition while the last one follows from \eqref{eqn:p_frac32_eng_decay_Phi} after passing to all times via \eqref{eqn:rp_to_derive_tau_inverse_decay_Phi}. In this estimate, one of the three entries is simply bounded by $\eps$ without any decay, which makes the extra spatial decay exploited from the null condition to be important. The way we bound cubic terms in \eqref{eqn:bound_for_cubic_after_p_tau} is in the same spirit of how we deal with the last term in \eqref{eqn:typical_term_after_take_derivative_in_tau}. Note that the $k$ range for nonlinearity matches with the range in \eqref{eqn:almost_frac92_decay_for_p_tau_Phi}.
This completes the discussion for the decay rate of nonlinearity.

Therefore, we obtain the following decay for energy
\[
    \|\p_\tau \mathbb\Phi_k\|^2_{LE(\calD_{\tau}^{\tau'})} + E[\p_\tau\mathbb\Phi_k] \lesssim \eps^2 \brk{\tau}^{-\frac92 + 2\kappa}, \quad |k| \leq M - 29,
\]
which proves \eqref{IBA-q-LE}.
Moreover, we also obtain \begin{equation}\label{eqn:almost_frac92_decay_for_p_tau_Phi}
    \|\brk{\uprho}^{\frac{\kappa}{2} - 1}\p_\tau\mathbb\Phi_k\|_{L^2(\Sigma_\tau, \ud V)}^2 \lesssim \eps^2 \brk{\tau}^{-\frac92 + 2\kappa}, \quad |k| \leq M - 28
\end{equation}
and could conclude the pointwise decay of $\mathbb\Phi_k$ \eqref{IBA-q} by the weighted elliptic estimate Lemma~\ref{lemma_weight_ell_est_for_P_ell_2}.

Note that we would not go one step further to the $p$-energy of $\p_\tau^2 \mathbb\Phi_k$ anymore since the decay rate in \eqref{eqn:almost_frac92_decay_for_p_tau_Phi} already matches \eqref{eqn:ptwise_decay_of_p_tau_p}.

\appendix
\section*{Appendices}
% \addcontentsline{toc}{section}{Appendices}
\renewcommand{\thesubsection}{\Alph{subsection}}
\numberwithin{theorem}{subsection}
\numberwithin{equation}{subsection}

\subsection{Local existence result}\label{Section:LWP}
Since the HVMC equation \eqref{eqn:HVMC} is a wave equation, it is natural to study it as a Cauchy problem. In particular, we are interested in the case when our initial data is close to that of the catenoid $\calC$. Set $g = \Phi^* \eta$ and hence $g_{\alpha\beta} = \frac{\p \Phi^\mu}{\p s^\alpha} \frac{\p \Phi_\mu}{\p s^\beta}$. Following the ideas of \cite{AC79}, we summarize the local existence result as follows.

Recall that the catenoid solution $\calC = \R \times \barcalC$, there is a natural parametrization \[
    \Phi(s^0, s') = \Phi(s^0, s^1, \cdots, s^n) = (s^0, F(s')),
\]
where $s'$ denotes $\rrho,\omg$ in \eqref{eqn:defn_F_for_Rm_catenoid}. It then has the Cauchy data $\Phi(s^0 = 0) = (0, F(s')) \in \{0\} \times \R^{1+n}$, $\p_{s^0} \Phi (s^0 = 0) = (1, 0) \in R^{1+(n+1)}$. 
Now for any arbitrary solution $\calM$ to \eqref{eqn:HVMC} with $(V, (s^0, s^1, \cdots, s^n))$ being any coordinate chart in $\calM$. Let $V_0 \subset \{(s^0, s') \in V: s^0 = 0\}$ be an open set in $\R^n$. Given data $\Phi^\mu(0, s')$ and $\p_{s^0} \Phi^\mu(0, s')$, we could use the idea of \cite[Theorem 14.2]{Ring09} to show the local existence of \begin{equation}\label{eqn:HVMC_wave_coord}
    g^{\alpha\beta} \frac{\p^2 \Phi^\mu}{\p s^\alpha \p s^\beta} = 0
\end{equation}
by replacing $g^{\alpha\beta}$ by another Lorentz matrix-valued function $A^{\alpha\beta}$ depending smoothly on the components of $g$. Let $A_{00} = g_{00}$ for all $-\frac32 \leq g_{00} \leq -\frac12$ and have the property that the range of $A_{00}$ is contained in $[-2, -\frac14]$. Let $A_{0i} = g_{0i}$ for all $-1 \leq g_{0i} \leq 1$ and have the property that the range of $A_{0i}$ is contained in $[-2, 2]$. With localized initial data on $V_0$ by applying a cutoff, the standard local existence argument shows the existence of a solution to the modified equation. Then due to smoothness of the solution, there exists an open neighborhood $W$ of $V_0$ such that $A_{\mu\nu} = g_{\mu\nu}$ on $W$ and $\pi(W) \subset V_0$, where $\pi$ is the projection to $\{s^0 = 0\}$. Therefore, we have a development of \eqref{eqn:HVMC_wave_coord}. In particular, when $V$ is compact, it shows local existence in a short time $s^0 \in I$ in $W$.

Now, note that \eqref{eqn:HVMC_wave_coord} is equivalent to \eqref{eqn:HVMC} if one chooses wave coordinate $\Box_g s^\alpha = 0$, which is in turn equivalent to $\Gmm^\alpha = g^{\mu\nu} \Gmm^{\alpha}_{\mu\nu} = 0$. By applying \cite[B. Lemma 1]{AC79}, which proves the existence of wave coordinates and the invariance of Cauchy data under a coordinate transform, we could conclude the local existence in $W_0$. Now, if one considers the following particular initial data \begin{equation}\label{eqn:data_on_const_t_slice}
    \Phi|_{X^0= 0} = F + (\tilde \psi_0 \circ F) N \circ F, \quad \p_{X^0} \Phi|_{X^0 = 0} = (1, 0) + (\tilde \psi_1 \circ F) N \circ F,
\end{equation}
where $N$ is the gauge we choose in \eqref{eqn:defn_for_N}, which is a vector with vanishing $0$-component. Even though $N$ is defined with respect to hyperboloidal foliations instead of constant $X^0$-foliations, the preceding initial data is well-defined according to the support condition of $\tilde\psi_0, \tilde\psi_1$: \[
    \tilde \psi_0,\ \tilde \psi_1 \in C^\infty_c(\barcalC), \qquad \supp \tilde \psi_0,\ \supp \tilde \psi_1 \subset \{|\Xbar| \leq R_f/2 \}.
\]
We note that $X^0$ shall play the role of $s^0$ in each local chart we choose. Also, we also mention that $N$ defined in \eqref{eqn:defn_for_N} is a vector on the unringed foliation, such $N$ on our simplified foliation still exists thanks to implicit function theorem, which is applicable due to smallness of $\ell, \xi$. In fact, we could simply choose $\xi(0) = \ell(0) = 0$ so that on initial slice, the ringed foliation and the unringed one coincide. See Remark~\ref{rmk_ambiguity_of_N} as well.

Recall that $\barcalC$ has two ends and there are only compactly supported perturbations involved, we could patch all local existence results above in any compact coordinate chart and the coordinate charts for the two ends together thanks to the finite speed of propagation phenomenon. 

Note that $\barbsUpsigma_\tau$ is a spacelike hypersurface (and hence $\bsUpsigma_\tau$ as well) by computing from \eqref{eqn:m_0_inv_formula} and \eqref{eqn:tilde_m_1_formula}: \[
    (d\tau)^\sharp = - \frac{\brk{r}^{-2}}{(1 - \Theta \cdot \eta')^2}\p_\tau - \frac1{1 - \Theta \cdot \eta'}\p_r + \calO(r^{-2})\p_r + \calO(r^{-3}) \p_\theta
\]
and \[
    \brk{(d\tau)^\sharp, (d\tau)^\sharp} = - \frac{\brk{r}^{-2}}{(1 - \Theta \cdot \eta')^2} < 0,
\]
which is timelike but asymptotically null.
Therefore, with the initial data \eqref{eqn:data_on_const_t_slice} and finite speed of propagation, the solution is determined between $\Sigma_{\tau = 0}$ and $\{X^0 = 0\}$, which is identical to $\calC$. Hence, we can now start with Cauchy data on $\bsUpsigma_\tau$ type hyperboloidal hypersurfaces to construct hyperboloidal foliations designed in \eqref{eqn:defn_calQ}. 
Instead of letting $\ell, \xi$ to be time-dependent, it suffices to fix them for the purpose of a local existence result since for each constant $\uptau$, $\ell$ and $\xi$ are constant. Moreover, the variance of parameters shall be small enough for an inverse function argument to work in a local neighborhood sufficiently close to $\Sigma_\uptau$. 

Given fixed $\ell_0, \xi_0 \in \R^n \subset \R^{n+1}$, let $\ringcalC_\uptau(\ell_0, \xi_0)$ and $\ringbsUpsigma_\uptau(\ell_0, \xi_0)$ be the submanifolds defined in \eqref{eqn:calC_sigma} and \eqref{eqn:parametrize_bsupsigma_in_tau_r_theta} with \begin{equation}\label{eqn:prescribe_fixed_ell_xi}
    \ell(\uptau) \equiv \ell_0, \quad \xi(\uptau) = \xi_0 + \uptau \ell_0
\end{equation}
In other words, we treat $\dot\wp$ as vanishing in this local existence part. The corresponding choice of $N$ in \eqref{eqn:defn_for_N} is denoted by $\ringN_{\ell_0, \xi_0}$. The region we are concerning about is given by \[
\calD_0^\uptau(\ell_0, \xi_0):= \cup_{\tau \in [0, \uptau)} \ringSigma_\tau(\ell_0, \xi_0):= \cup_{\tau \in [0, \uptau)} \left(\ringcalC_\tau(\ell_0, \xi_0) \cap \ringbsUpsigma_\tau(\ell_0, \xi_0)\right).
\]
Here, we simply let the shift and translation independent of $\uptau$ but all the discussions in the setup work with $\dot\wp \equiv 0$ in mind and we use notation with a ring on top of it to distinguish. Note that if we prove the existence of solution to \eqref{eqn:HVMC} forward in $\tau$ starting from $\ringSigma_\tau(\ell_0, \xi_0)$ locally via the gauge $\ringN_{\ell_0, \xi_0}$, we could parametrize this local solution via the gauge $N$ (recall \eqref{eqn:defn_for_N}) adapted to the moving and boosted foliations with now $\ell(\tau) = \ell_0$, $\xi(\tau) = \xi_0 + \tau \ell$ we use throughout the paper. This tells us the local existence of $\ringpsi$ as a function over $\ringSigma_\tau(\ell_0, \xi_0)$ implies local existence of $\psi$ as a function over $\Sigma_\tau$. See \cite[Lemma 2.2]{LOS22}.

We will denote $\mathbf{T} = \Lambda_{-\ell_0}(\p_{X^0}) = \p_{X^0} + \ell_0 \p_{X^j}$. This is compatible with the boost we applied to construct our foliations and it makes sure that the solution we obtain is the corresponding catenoid if perturbation becomes trivial. 
Now we have necessary notions to set up the local existence lemma. First, we reiterate how we transform a vector value problem to a scalar one by the gauge choice $\ringN_{\ell_0, \xi_0}$ in the following lemma. 

\begin{lemma}
    Consider $\ringPhi_0(p) = p + \ringeps \ringpsi_0 \ringN_{\ell_0, \xi_0}$ and $\ringPhi_1(p) = (1, \ell_0) + \ringeps \ringpsi_1 \ringN_{\ell_0, \xi_0}$. Then the Cauchy problem for an embedding $\ringPhi: \calD_0^\uptau(\ell_0, \xi_0) \to \R^{1 + (n+1)}$ satisfying \eqref{eqn:HVMC} with data $\ringPhi|_{\uptau = c} = \ringPhi_0, \p_{X^0}\ringPhi|_{\uptau = c} = \ringPhi_1$ is equivalent to consider the Cauchy problem of $\ringpsi: \calD_0^\uptau(\ell_0, \xi_0) \to \R$ satisfying \begin{equation}\label{eqn:scalar_reduction_for_LWP}
        (\Box_h + |\II|^2) \ringpsi = \ringf:= \calO((\ringpsi^2, \p \ringpsi^2)), \quad \ringpsi|_{\uptau = c} = \ringpsi_0, \quad \mathbf{T}\ringpsi|_{\uptau = c} = \ringpsi_1
    \end{equation}
    with $h$ is the metric defined in  \eqref{eqn:defn_h_mu_nu} (see also \eqref{eqn:metric_h_in_Cflat}) with $\ell$, $\xi$ set to be \eqref{eqn:prescribe_fixed_ell_xi} and $\II$ is the second fundamental form of $\ringSigma_\sigma(\ell_0, \xi_0)$. In the far away region, $\ringf = \rings^{-1}(\ringcalF_2 + \ringcalF_3)$ with $\rings$, $\ringcalF_2$ and $\ringcalF_3$ all defined in the same way as \eqref{eq:calF2_1} and \eqref{eq:calF3_1} with the prescribed information \eqref{eqn:prescribe_fixed_ell_xi}.
\end{lemma}
\begin{proof}
    The derivation of the scalar equation is exactly the same as Section~\ref{Section:2nd_order_formulation_whole_region}. 
    The form of nonlinearity is exactly the same as Section~\ref{Section:2nd_order_formulation_whole_region} with $\dot\wp \equiv 0$ prescribed and more specifically, $\ell(\tau) \equiv \ell_0, \xi(\tau) = \xi_0 + \tau \ell_0$. In particular, the equation and the source term are of the same form as in \eqref{eqn:form_of_P_P_0_P_pert} and \eqref{eqn:relation_calP_and_calP_graph}.
\end{proof}

In what follows, we will use $\p_{\ringSigma}$ to denote any order $1$ derivative tangent to the leaf $\ringSigma_\uptau$, i.e., $\p_\uprho, \brk{\uprho}^{-1}\p_\uptheta$ in the $(\uptau, \uprho, \uptheta)$ coordinates. Even if the foliation is different and simplified a bit from the one used in the main part, such coordinate still exists and a slight abuse notation occurs here. 

\begin{remark}\label{rmk_decay_assumption_on_id}
    Note that the smallness assumption in the setting of the lemma comes from the smallness in the Cauchy data on $\{X^0 = 0\}$ as we introduced earlier. We want to justify the following: if we start with data \eqref{eqn:data_on_const_t_slice} and $\tilde\psi_0, \tilde\psi_1$ having $\eps$-smallness, then the new Cauchy data on $\Sigma_{\uptau = 0}$ also has $\eps$-smallness in suitable norms.
    In the far away region, from \eqref{eqn:length_tilde_r} and \eqref{eqn:Q'_tilde_r}, we note that $\left.\left(Q'(r) - Q'(\tilde r)\right)\right|_{\uptau = 0} = \calO(|\ell(0)|r^{-3})$ by Taylor expansion. Other derivatives share similar estimates with $\ell$-smallness. Recall that $Q$ represents the solution by parametrizing them with respect to a graph over the hypersurface $\{X^{n+1} = S\}$, we know that $\ringpsi|_{\ringSigma_{\uptau = 0}} \sim \left(Q(r) - Q(\tilde r)\right)|_{\ringSigma_{\uptau = 0}}$ and hence $\|\chi_{r > R_f}  (\brk{\uprho}\p_{\ringSigma})^k \ringpsi\|_{L^2(\ringSigma_0(\ell(0),\xi(0)))}$ and $\|\chi_{r > R_f} (\brk{\uprho}\p_{\ringSigma})^{k-1} \mathbf{T}\ringpsi\|_{L^2(\ringSigma_0(\ell(0),\xi(0)))}$ are of order $\eps$. We will later fix $|k| \leq M$ in the local existence statement and also our main theorem (Theorem~\ref{main_thm}). 

    Note that we not only need $\eps$-smallness at the level of standard energy, we require $\eps$-smallness after inserting suitable $\brk{\uprho}$-weights into the norm for our purpose of $r^p$-weighted estimates. To this end, we require $\ell(0) = 0$ such that $\left.\left(Q'(r) - Q'(\tilde r)\right)\right|_{\uptau = 0} = 0$, which implies that the initial data on $\ringSigma_0(\ell(0), \xi(0))$ is compactly supported and hence satisfies $\eps$-smallness with proper $\brk{\uprho}$-weights.
    
    Finally, we shall mention that even if the initial data is rapid decaying but not compactly supported, the $\eps$-smallness from $X^0 = 0$ still carries over to these norms on $\ringSigma_0(\ell(0), \xi(0))$ as long as the spatial decay of Cauchy data on $X^0 = 0$ is sufficient. One can easily adapt the proof of \cite[Theorem 2.27, exterior stability theorem]{LO24} to establish this claim.
\end{remark}

\begin{remark}\label{rmk_ambiguity_of_N}
    In the near region, the $\eps$-smallness of the new Cauchy data on $\Sigma_0$ is trivial as the constant $\uptau$ is equivalent to constant $X^0$ due to our choice of foliation. In particular, we could choose $\ell(0) = \xi(0) = 0$ to avoid any ambiguity of the $N$ used in \eqref{eqn:data_on_const_t_slice}.
\end{remark}

To prepare ourselves to state the local existence result, we introduce some notation below. We shall use $\bfGmm^k$ to denote a string of $k$ vector fields consisting of $\{\p_{X^\mu}, S, \Omg, L\}$, where \[
    S = \tau\p_\tau + r \p_r, \quad \Omg_{ij} = (X^i - \eta) \p_{X^j} - (X^j - \eta) \p_{X^i}, \quad L_j = X^0 \p_{X^j} + (X^j - \eta) \p_{X^0}.
\]
Since $\eta$ is constant and independent of $\tau$ in this local wellposedness setting, it is easy to read off the following \[
    \p_{X^0} = \p_\tau, \quad \p_{X^j} = - \frac{r\Theta^j}{\brk{r}}\p_\tau + \Theta^j \p_r + \frac{\ringsg^{ab}\Theta^j_b}{r}\p_a.
\]
See \cite[Section 4.2]{LOS22} for a detailed computation. One could see that $\brk{r}\p_r$ can be reconstructed from $\Gmm$'s by observing that \[
    \frac{(X^j - \eta)L_j}{r} = r\p_\tau + \frac{\tau + \brk{r}}{r}(\p_r - \frac{r}{\brk{r}}\p_\tau) = (\tau + \brk{r})\p_r - \tau \frac{r}{\brk{r}}\p_\tau.
\]
Therefore, \begin{equation}\label{eqn:representation_of_rpr_by_bfGmm}
    \brk{r}\p_r = \frac{\brk{r}}{\tau + 2\brk{r}} \left(\frac{\brk{r}}{r}S + \frac{(X^j - \eta)}{r}L_j\right) = \calO_\Gmm^\infty(1)S + \sum_j \calO_\Gmm^\infty(1)L_j,
\end{equation}
where we use $\ringpsi = \calO_\Gmm^N(m(\tau, r))$ to denote \[
    |(\p_\tau)^{k_1}(\brk{r}\p_r)^{k_2}\Omg^{k_3}\ringpsi| \lesssim m(\tau, r), \quad \forall k \text{ such that } |(k_1, k_2, k_3)| \leq N.
\]

Moreover, a simple computation reveals that \begin{equation}\label{eqn:commutation_relation_of_Box_w_bfGmm}
    |[\Box_m, \bfGmm^k] \ringpsi| \lesssim |\Box_m \bfGmm^{k-1} \ringpsi|.
\end{equation}

Now we could formulate the local existence result, which will be used in the proof of our main theorem. The proof is adapted from that of \cite[Proposition 12.16]{LO24}.

\begin{lemma}[Local existence theorem]
    Given smooth functions $\ringpsi_0, \ringpsi_1$ on $\ringSigma_{\uptau_1}(\ell_0, \xi_0)$ such that \[
    \|\p_{\ringSigma}^k \ringpsi\|_{L^2(\ringSigma_{\uptau_1}(\ell_0,\xi_0))}, \ \  \|\p_{\ringSigma}^{k-1} \mathbf{T}\ringpsi\|_{L^2(\ringSigma_{\uptau_1}(\ell_0,\xi_0))} \lesssim 1, \quad k \leq M
    \]
    with $M$ sufficiently large, 
    the Cauchy problem \eqref{eqn:scalar_reduction_for_LWP} has a regular solution $\ringpsi$ on $\calD_{\uptau_1}^{\uptau_2}(\ell_0, \xi_0)$. In the region $\calS_{\uptau_1}^{\uptau_2}(\infty)$ to be defined below, the following pointwise bound holds for $M$ sufficiently large $M$ ($M > 9$ when $n = 4$): \begin{equation}\label{eqn:ptwise_bound_LWP}
        \ringpsi(\tau, r, \theta) = \calO_{\Gmm}^{M - 5}(\ringeps r^{-\frac{3}{2}}).
    \end{equation}
    We also mention that $\uptau_2 - \uptau_1$ depends solely on suitable norms of initial data.
\end{lemma}
\begin{proof}
    Given $C^\infty$ initial data on spacelike asymptotically null leaf $\Sigma_{\uptau = c}$, the standard local existence result gives a unique development of initial data. In the near region of the foliation, the local in $\uptau$ existence follows from standard theory. Therefore, it suffices to deal with the local existence in the hyperboloidal region, where a graph formulation is accessible. In fact, instead of parametrizing and writing the perturbation $\ringpsi$ as a function $\ringvarphi$ over $ \barbsUpsigma_\tau$, we use the unmodified version $\tilde\barbsUpsigma_\tau$. With an abuse of notation, the parametrization is still denoted by $(\tau, r, \theta)$ but with formula \eqref{eqn:parametrize_tilde_bsUpsigma}. The new parametrization is the same for $\calC_\hyp$ but will be different of the old one for $\calC_\tran \cup \calC_\flatfar$.
    
    We consider \[
        \calP_\graph \ringvarphi = \ringcalF_2 + \ringcalF_3
    \]
    in the region \[
        \calS_{\uptau_1}^{\uptau_2}(T_0) := \{X \in \calD_{\uptau_1}^{\uptau_2}(\ell_0, \xi_0): \uptau_1 + \gmm(\uptau_1)R_f =: \sigma(\uptau_1) \leq X^0 \leq T_0:= \uptau_1 + \brk{R}, \uptau_1 \leq X^0 - \brk{r} \leq \uptau_2 \}.
    \]
    We will later bootstrap on $T_0$ to show the bound and hence local existence in $\uptau$ at the same time.
    The derivation of the operator is the same as in Section~\ref{Section:derivation_of_calP_graph_op} and $\eta'$ is still nonvanishing but a fixed small constant. Therefore, the metric now becomes stationary (i.e. $\tau$-independent).

    To prove local existence, it suffices to show \eqref{eqn:ptwise_bound_LWP} on $\calS_{\uptau_1}^{\uptau_2}(\infty)$ for some sufficient small $\uptau_2 - \uptau_1$ depending on $\ringeps$. Assuming that \begin{equation}\label{eqn:BA-LWP}
         |\ringpsi_k(\tau, r, \theta)| \leq 2C_k\ringeps r^{-\frac32}, \quad |k| \leq M - 5, \quad (\tau, r, \theta) \in \calS_{\uptau_1}^{\uptau_2}(T_0),
    \end{equation}
    it suffices to show that \begin{equation}\label{eqn:IBA-LWP}
        |\ringpsi_k(\tau, r, \theta)| \leq C_k\ringeps r^{-\frac32}, \quad |k| \leq M - 5, \quad (\tau, r, \theta) \in \calS_{\uptau_1}^{\uptau_2}(T_0).
    \end{equation}
    Here, $\ringpsi_k = (\p_\tau)^{k_1}(\brk{r}\p_r)^{k_2}\Omg^{k_3} \ringpsi$ for $k = (k_1, k_2, k_3)$.
    Indeed, by applying standard local existence result with initial data on $\{X^0 = T_0, \uptau_1 \leq \uptau \leq \uptau_2\} \cup \{\tau = \uptau_1, T_0 \leq X^0 \leq T_0 + \delta\}$ with $\delta$ small enough, one knows that the solution can be extended to $\calS_{\uptau_1}^{\uptau_2}(T_0 + \delta')$ on which the bound \eqref{eqn:BA-LWP} still holds for some small $\delta'$, which completes the proof of \eqref{eqn:ptwise_bound_LWP}.

    We compute \begin{equation}\label{eqn:DBP_identity_for_LWP}
        \begin{aligned}
            \sqrt{|m|}\Box_m \ringvarphi \p_\tau \ringvarphi &= \p_\alpha \left( \sqrt{|m|} m^{\alpha\beta}  \p_\beta\ringvarphi \p_\tau \ringvarphi\right) - \sqrt{|m|} m^{\alpha\beta}  \p_\beta \ringvarphi \p_\alpha \p_\tau \ringvarphi \\ 
            &= \p_\alpha \left( \sqrt{|m|} m^{\alpha\beta}  \p_\beta\ringvarphi \p_\tau \ringvarphi - \frac12 \sqrt{|m|} m^{\gmm\beta}  \p_\beta \ringvarphi T^\alpha\p_\gmm \ringvarphi\right), 
        \end{aligned}
    \end{equation}
    where $T^\gmm \p_\gmm:= \p_\tau$.
    Applying divergence theorem to \eqref{eqn:DBP_identity_for_LWP} in the region $\calS^{\uptau_2}_{\uptau_1}(T_0)$ with $\ringvarphi$ replaced by $\ringvarphi_k$, we obtain the following energy estimates \begin{equation}\label{eqn:eng_est_LWP}
        \begin{aligned}
        &\sup_{\uptau_1 \leq \uptau_2' \leq \uptau_2}\int_{X^0 = T_0, \uptau_1 \leq \tau \leq \uptau'_2} \left((\p_\tau \ringvarphi_k)^2 + (\p_r \ringvarphi_k)^2 + r^{-2}(\snabla \ringvarphi_k)^2\right) r^3 \, \ud X' \\
        +& \sup_{\uptau_1 \leq \uptau_2' \leq \uptau_2}\int_{\tau = \uptau_2', \sigma(\uptau_1) \leq X^0 \leq T_0} \left(r^{-2}(\p_\tau \ringvarphi_k)^2 + (\p_r \ringvarphi_k)^2 + r^{-2}(\snabla \ringvarphi_k)^2\right) r^3 \, \ud r\, \ud \theta \\
        \lesssim &\,
        \int_{X^0 = \sigma(\uptau_1), \uptau_1 \leq \tau \leq \uptau_2} \left((\p_\tau \ringvarphi_k)^2 + (\p_r \ringvarphi_k)^2 + r^{-2}(\snabla \ringvarphi_k)^2\right) r^3 \, \ud X' \\
        +& \int_{\tau = \uptau_1, \sigma(\uptau_1) \leq X^0 \leq T_0} \left(r^{-2}(\p_\tau \ringvarphi_k)^2 + (\p_r \ringvarphi_k)^2 + r^{-2}(\snabla \ringvarphi_k)^2\right) r^3 \, \ud r\, \ud \theta + \left| \iint_{\calS^{\uptau_2}_{\uptau_1}(T_0)} \Box_m \ringvarphi_k \p_\tau \ringvarphi_k\, \ud V \right| \\
        \lesssim&\, \ringeps^2 + \left| \iint_{\calS^{\uptau_2}_{\uptau_1}(T_0)} \Box_m \ringvarphi_k \p_\tau \ringvarphi_k\, \ud V \right|, \quad |k| \leq M - 2,
    \end{aligned}
    \end{equation}
    where the first two terms on the right hand side of the first inequality in \eqref{eqn:eng_est_LWP} have some overlapping and are both restricted to the initial slice, which then possess $\ringeps$-smallness due to the choice of initial data and our selection of $k$. Note that even if $\mathbf{T}$ is not exactly $\p_\tau$, the $\eps$-smallness is still ensured for the norm we require given the transversality.

    To estimate the bulk term, we first notice that it follows from \eqref{eqn:BA-LWP} that \[
        |\bfGmm^k(\ringcalF_2 + \ringcalF_3)| \lesssim  C_k\ringeps 
        r^{-3} |\p \bfGmm^{\leq k} \ringvarphi|, \quad \bfGmm \in \{\p_{X^\mu}, S, \Omg, L\}, \quad |k| \leq M - 2,
    \]
    when $M \geq 9$ (so that $\frac12(M - 2) < M - 5$). Here, we need to keep in mind that the estimate is not bootstrap constant independent.
    Recall the form of $\calP_\graph$ \eqref{eqn:calP_graph} and the commutator relation \eqref{eqn:commutation_relation_of_Box_w_bfGmm}, we could write \[
        |\Box_m \bfGmm^k \ringvarphi| \lesssim (C + C_k\ringeps) 
        r^{-3} |\p \bfGmm^{\leq k} \ringvarphi|, \quad |k| \leq M - 2,
    \]
    where $C$ is an absolute constant determined from $\calP_\graph - \Box_m$. By choosing $\ringeps$ small, the constant shall be independent of bootstrap constants. Thanks to \eqref{eqn:representation_of_rpr_by_bfGmm}, we would not distinguish $\ringpsi_{\leq k}$ from $\bfGmm^{\leq k} \ringpsi$.
    Then the bulk term in \eqref{eqn:eng_est_LWP} satisfies that \[
    \begin{aligned}
        \left| \iint_{\calS^{\uptau_2}_{\uptau_1}(T_0)} \Box_m \ringvarphi_k \p_\tau \ringvarphi_k\, \ud V \right| \lesssim&  \int_{\uptau_1}^{\uptau_2} \int_{\tau = \uptau_2', \sigma(\tau_1) \leq X^0 \leq T_0} r^{-2} (\p_\tau \ringvarphi_k)^2\, r^3\, \ud r\, \ud \theta\, \ud \tau_2' \\&+  \sup_{\uptau_1 \leq \uptau_2' \leq \uptau_2} \int_{\tau = \uptau_2', \sigma(\tau_1) \leq X^0 \leq T_0} r^{-4} (\p \bfGmm^{\leq k}\ringvarphi)^2 \, r^3\, \ud r\, \ud \theta\, \ud \tau_2'.
    \end{aligned}
    \]
    We first sum up \eqref{eqn:eng_est_LWP} with order  $\leq k$. Then the second term can be absorbed to the left hand side by taking the transition radius of foliation $R_f$ sufficiently large to gain smallness. On the other hand, the first term can be estimated by sup norm in $\uptau_2'$ and gain smallness from $\uptau_2 - \uptau_1$. This implies that the left hand side of \eqref{eqn:eng_est_LWP} is bounded by $\ringeps^2$ and in particular, \[
        \sup_{\uptau_1 \leq \uptau_2' \leq \uptau_2}\int_{\tau = \uptau_2', \sigma(\uptau_1) \leq X^0 \leq T_0} \left(r^{-2}(\p_\tau \ringvarphi_{\leq k})^2 + (\p_r \ringvarphi_{\leq k})^2 + r^{-2}(\snabla \ringvarphi_{\leq k})^2\right) r^3 \, \ud r\, \ud \theta \lesssim \ringeps^2, \quad |k| \leq M - 2.
    \]

    We note that \[\begin{aligned}
        \p_{X^j}\big(u(\sqrt{|X' - \eta|^2 + 1} + \tau, X')\big) = (\frac{X^j - \eta}{\sqrt{|X' - \eta|^2 + 1}} -\frac{r\Theta^j}{\brk{r}})\p_{X^0} u + \p_{X^j} u) = \p_{X^j} u|_{\ringSigma_\tau}.
    \end{aligned}
    \]
    By parametrizing $\ringSigma_\tau$ via $X'$, one can apply the standard Sobolev inequality on $\ringSigma_\tau$ in a ball of radius $\simeq T_0$, one obtains that for any $k$, we have \begin{equation}\label{eqn:ptwise_est_LWP_prelim}
        |\ringvarphi_{\leq k}(\tau_2', r, \theta)| \lesssim T_0^{-2} \left(\int_B \left((\p_\tau \ringvarphi_{\leq k + 3})^2 + (\p_r \ringvarphi_{\leq k + 3})^2 + r^{-2}(\snabla \ringvarphi_{\leq k + 3})^2\right) r^3 \, \ud r\, \ud \theta\right)^{\frac12} 
    \end{equation}
    where $B \subset \ringSigma_{\tau} \cap \{\sigma(\uptau_1) \leq X^0 \leq T_0\}$ contains $(\tau, r, \theta)$ and is of radius $O(T_0)$. This estimate is uniform in $\tau \in [\uptau_1, \uptau_2]$. Note that from the standard local existence and corresponding estimates from constant $X^0$-slices, we could choose $\ringeps$ small enough so that \eqref{eqn:ptwise_bound_LWP} is already satisfied on when $T_0 \leq \eta + R_f$. Therefore, we could assume $T_0 > \eta + R_f$ and hence $T_0 > \frac{r}{2}$ thanks to 
    $\eta + R_f \lesssim r \lesssim \eta + R_f + T_0$. Therefore, we know from \eqref{eqn:ptwise_est_LWP_prelim} that 
    \[
    |\ringvarphi_{\leq k}(\tau_2', r, \theta)| \lesssim r^{-\frac32} \left(\int_B \left(r^{-2}(\p_\tau \ringvarphi_{\leq k + 3})^2 + (\p_r \ringvarphi_{\leq k + 3})^2 + r^{-2}(\snabla \ringvarphi_{\leq k + 3})^2\right) r^3 \, \ud r\, \ud \theta\right)^{\frac12} \lesssim \ringeps r^{-\frac32}
    \]
    for all $|k| \leq M - 5$ and all $(\tau_2', r, \theta) \in \ringSigma_{\uptau'_2} \cap \{\sigma(\uptau_1) \leq X^0 \leq T_0\}$ and all $\uptau_1 \leq \uptau_2' \leq \uptau_2$. 
    This then completes the proof of \eqref{eqn:IBA-LWP}.
\end{proof}

\begin{remark}
    For our purpose of applying this lemma to prove our main theorem and justify the existence of the parametrization by $\psi$, one can actually view that $\ringpsi$ is also determined in the region in the causal past of $\ringSigma_{\uptau_1}$. Bearing this in mind, one can actually invoke the spacetime version of Sobolev inequality as in \cite{LO24} to get a more delicate decay estimate. This thinking is also helpful when one wants to extend from $\calS_{\uptau_1}^{\uptau_2}(T_0)$ to $\calS_{\uptau_1}^{\uptau_2}(T_0 + \delta')$ so that a classical local existence on constant $X^0$ will suffice to meet the needs.
\end{remark}

\subsection{A summary of the foliations used in \texorpdfstring{\cite{LOS22}, \cite{OS24}}{LOS22, OS24}} \label{Section:different_foliation_LOS22}
We note that the foliation we select in Section~\ref{Section:foliations} is slightly different from the one in \cite{LOS22} and \cite{OS24}, where these two works used different foliations as well. The advantage of the one we adopt here is to have a explicit parametrization even in the transition region.

We now highlight the foliations used in \cite{LOS22} so that one can compare with Section~\ref{Section:foliations}.
Fix $\frakm \in C^\infty(\R \times \R \to \R)$ to be a smoothed out version of the minimum function such that for some small $\delta_1 > 0$, \[
    \frakm(t_1, t_2) = \min(t_1, t_2), \quad \text{when } |t_1 - t_2| > \delta_1.
\]
Define $\sigma, \sigma_\temp: \{-S \leq X^{n+1} \leq S\} \to \R$ by \[\begin{aligned}
    &\sigma_\temp(X) = \sigma' \quad \text{ if } X \in \tilde\bsUpsigma_{\sigma'}, \\
    &\sigma(X) = \frakm(X^0, \sigma_\temp(X)),
\end{aligned}
\]
where the constant $S$ in the domain of $\sigma, \sigma_\temp$ is defined in \eqref{eqn:S_bound_X_n+1}. This range is fixed since our parameters $\ell, \xi$ corresponding to the boost and translation are in $\R^n \hookrightarrow \R^{n+1}$.
Finally, we define $\bsUpsigma_\sigma$ as follows: \[
    \bsUpsigma_{\sigma'}:= \{X: \sigma(X) = \sigma'\}.
\]
One would keep in mind that starting at a large radius away from $\xi(\sigma)$ , $\bsUpsigma_\sigma$ coincides with $\tilde\bsUpsigma_\sigma$. 

\subsection{Hardy inequality and an interpolation inequality}
\begin{lemma}[Hardy inequality]
    Let $p \neq 1$ and $\phi: [r_0, r_1] \to \R$ be a $C^1$ function, then \begin{equation}\label{eqn:Hardy}
        \frac{(p-1)^2}{4} \int_{r_0}^{r_1} \phi^2 r^{p-2}\, \ud r \leq \int_{r_0}^{r_1} (\p_r \phi)^2 r^p\, \ud r + \frac{p-1}{2} r^{p-1}\phi^2|_{r_0}^{r_1}.
    \end{equation}
    For $q \neq n$, the following variant holds \begin{equation}\label{eqn:Hardy_variant}
        \frac{(q-n)^2}{4} \int_{r_0}^{r_1} \phi^2 r^{q-2}\, \ud r \leq \int_{r_0}^{r_1} (\p_r \phi + \frac{n-1}{2r}\phi)^2 r^q\, \ud r + \frac{q-n}{2} r^{q-1}\phi^2|_{r_0}^{r_1}
    \end{equation}
\end{lemma}
\begin{proof}
    We write \[\begin{aligned}
        \int_{r_0}^{r_1} (\p_r \phi)^2 r^p\, \ud r
        &= \int_{r_0}^{r_1} r^p\left(\p_r \phi - \frac{c}{r}\phi\right)^2 +  2c r^{p-1}\phi\p_r\phi - c^2 r^{p-2}\phi^2\, \ud r \\
        &\geq \int_{r_0}^{r_1} c r^{p-1}\p_r(\phi)^2 - c^2 r^{p-2}\phi^2\, \ud r
        = c r^{p-1}\phi^2|_{r_0}^{r_1} - c(p-1+c) \int_{r_0}^{r_1}  r^{p-2}\phi^2\, \ud r,
    \end{aligned}
    \]
    where the best constant $c$ to choose is $c = -\frac{p-1}{2}$.
    For the second one, we argue similarly \[
        \int_{r_0}^{r_1} (\p_r \phi + \frac{n-1}{2r}\phi)^2 r^q\, \ud r
        \geq cr^{q-1}\phi^2|_{r_0}^{r_1} - c(q-n+c) \int_{r_0}^{r_1} r^{q-2}\phi^2\, \ud r.
    \]
    The proof is done by choosing $c = -\frac{q-n}{2}$.
\end{proof}

\begin{lemma}[Interpolation inequality]
    Let $f: 
    \R_+ \times [R, \infty) \rightarrow \R$ be a function such that the following inequalities hold:
$$
\begin{gathered}
\int_R^{\infty} r^{p-\veps} f^2(\tau, r)\, \ud r \leq D_1(1+\tau)^{-q}, \\
\int_R^{\infty} r^{p+s-\veps} f^2(\tau, r)\, \ud r \leq D_2(1+\tau)^{-q+1},
\end{gathered}
$$
for some $\tau$-independent constants $D_1, D_2>0, q \in \R$, $s \in (0,1]$  and $\veps \in (0, s)$.
Then \begin{equation}
    \label{eqn:interpolation_for_rp}
    \int_R^\infty r^p f^2(\tau, r)\, \ud r \lesssim \max{\{D_1, D_2\}} (1 + \tau)^{-q + 1 - s + \veps}.
\end{equation}
\end{lemma}
\begin{proof}
    We split \[
        \int_R^\infty r^p f^2(\tau, r)\, \ud r = \int_R^{R + \tau} r^p f^2(\tau, r)\, \ud r + \int_{R+\tau}^\infty r^p f^2(\tau, r)\, \ud r.
    \]
    Then we use the two assumptions separately to obtain \[
    \begin{aligned}
        \int_R^{R + \tau} r^p f^2(\tau, r)\, \ud r
        &= \int_R^{R + \tau} r^\veps r^{p-\veps} f^2(\tau, r)\, \ud r
        \leq D_1(R + \tau)^\veps(1+\tau)^{-q} \lesssim (1 + \tau)^{-q+\veps}, \\
        \int_{R+\tau}^\infty r^p f^2(\tau, r)\, \ud r &= \int_{R+\tau}^\infty r^{-s + \veps}r^{p+s-\veps} f^2(\tau, r)\, \ud r \leq D_2(R+\tau)^{-s + \veps}(1+\tau)^{-q+1}\lesssim (1 + \tau)^{-q + 1 - s + \veps}.
    \end{aligned}
    \]
    Hence, \eqref{eqn:interpolation_for_rp} follows by adding the above two inequalities together.
\end{proof}

% \stoptoc
% \subsection*{Declarations}
% The author has no competing interests to declare that are relevant to the content of this article.
% \resumetoc

\bibliographystyle{alphaurl}
\bibliography{cite}

\end{document}

%% file: figure_1.tex
\begin{tikzpicture}
    \draw[->] (-6,0) -- (6,0) node[right] {$X'$};
    \draw[->] (-2,-0.5) -- (-2,3) node[right] {$X^0$};
    % Left part: (y - 1)^2 = (x + c)^2 + 1 (with c>0)
    \draw[thick, black, domain=-6:-3, smooth, variable=\x] 
        plot ({\x}, {sqrt((\x + 3)^2 + 1)});
    % Right part: (y - 1)^2 = (x - c)^2 + 1 (with c>0)
    \draw[thick, black, domain=3:6, smooth, variable=\x] 
        plot ({\x}, {sqrt((\x - 3)^2 + 1)}) node[right]{$\bsUpsigma_{\tau}$};
    % Middle part: horizontal line
    \draw[thick, black] (-3,1) -- (3,1);

    \draw[thick, blue, <->] (-1.5, 0.7) -- (1.5, 0.7) node[midway, below]{$\bsUpsigma_{\tau, \flatnear}$};
    \draw[thick, blue, <->] (1.5, 0.7) -- (2.5, 0.7) node[midway, below]{$\bsUpsigma_{\tau, \flatfar}$};
    \draw[thick, blue, <->] (2.5, 0.7) -- (4.5, 0.7) node[midway, below]{$\bsUpsigma_{\tau, \tran}$};
    \draw[thick, blue, <-] (4.5, 0.7) -- (6, 0.7) node[midway, below]{$\bsUpsigma_{\tau, \hyp}$};

    \draw[thick, blue, <->] (-1.5, -0.2) -- (1.5, -0.2) node[midway, below]{$\chi_g \equiv 1$};
    \draw[thick, blue, <-] (2.5, -0.2) -- (6, -0.2) node[midway, below]{$\chi_g \equiv 0$};
\end{tikzpicture}

%\end{document}